\documentclass[reqno,10pt]{article}

\usepackage[english]{babel}
\usepackage[T1]{fontenc}

\usepackage{enumerate}
\usepackage{enumitem}

\usepackage{a4wide}
\usepackage{color}
\usepackage{mathrsfs}
\usepackage{mathtools}
\usepackage{amsmath}
\usepackage{amsthm}
\usepackage{amssymb}
\usepackage{bbm}
\usepackage{tikz-cd}
\usepackage{esint}
\usepackage{nicefrac}
\numberwithin{equation}{section}
\usepackage[colorlinks,citecolor=green,linkcolor=red]{hyperref}
\usepackage{cleveref}
\usepackage{longtable}
\usepackage{setspace}

\usepackage[nottoc,numbib]{tocbibind}

\usepackage{bbm}

\usepackage[utf8]{inputenc}

\makeatletter
\newcommand{\subjclass}[2][1991]{%
  \let\@oldtitle\@title%
  \gdef\@title{\@oldtitle\footnotetext{#1 \emph{Mathematics subject classification.} #2}}%
}
\newcommand{\keywords}[1]{%
  \let\@@oldtitle\@title%
  \gdef\@title{\@@oldtitle\footnotetext{\emph{Key words and phrases.} #1.}}%
}
\makeatother

\newcommand{\nchi}{{\raise.3ex\hbox{\(\chi\)}}}
\newcommand{\N}{\mathbb{N}}
\newcommand{\R}{\mathbb{R}}

\newcommand{\AC}{{\rm AC}}

\newcommand{\sfd}{{\sf d}}
\renewcommand{\d}{{\mathrm d}}
\newcommand{\e}{{\rm e}}
\newcommand{\X}{{\rm X}}
\newcommand{\Y}{{\rm Y}}

\newcommand{\mm}{\mathfrak{m}}
\newcommand{\1}{\mathbbm 1}
\newcommand{\Mod}{{\rm Mod}}
\newcommand{\Modp}{{\rm Mod}_{p}}

\newcommand{\LIP}{{\rm LIP}}
\newcommand{\Lip}{{\rm Lip}}
\newcommand{\lip}{{\rm lip}}

\newcommand{\ppi}{{\mbox{\boldmath\(\pi\)}}}

\newcommand{\sppi}{{\mbox{\scriptsize\boldmath\(\pi\)}}}

\newcommand{\limi}{\varliminf}
\newcommand{\lims}{\varlimsup}
\renewcommand{\div}{{\rm div}}

\newcommand{\diam}{{\rm diam}}

\newcommand{\fr}{\penalty-20\null\hfill\(\blacksquare\)}

\newtheorem{theorem}{Theorem}[section]
\newtheorem{corollary}[theorem]{Corollary}
\newtheorem{lemma}[theorem]{Lemma}
\newtheorem{proposition}[theorem]{Proposition}
\newtheorem{definition}[theorem]{Definition}

\newtheorem{remark}[theorem]{Remark}

\title{Metric Sobolev spaces I: equivalence of definitions}


\author{Luigi Ambrosio \thanks{Scuola Normale Superiore, Piazza dei Cavalieri 7, 56126 Pisa.
\textit{Email:} {\sf luigi.ambrosio@sns.it}}
\and Toni Ikonen \thanks{Department of Mathematics and Statistics, P.O.\ Box 68 (Pietari Kalmin katu 5),
FI-00014 University of Helsinki, Finland. \textit{Email:} {\sf toni.ikonen@helsinki.fi}}
\and Danka Lu\v{c}i\'{c} \thanks{University of Jyvaskyla, Department of Mathematics and Statistics,
P.O.\ Box 35 (MaD) FI-40014, Finland. \textit{Email:} {\sf danka.d.lucic@jyu.fi}}
\and Enrico Pasqualetto \thanks{University of Jyvaskyla, Department of Mathematics and Statistics,
P.O.\ Box 35 (MaD) FI-40014, Finland. \textit{Email:} {\sf enrico.e.pasqualetto@jyu.fi}}}
\begin{document}
\date{\today}
\maketitle

\begin{abstract}
This is the first of two works concerning the Sobolev calculus on metric measure spaces and its 
applications. In this work, we focus on several 
notions of metric Sobolev space and on their equivalence. More precisely, we give a systematic presentation 
of first-order \(p\)-Sobolev spaces,
with \(p\in[1,\infty)\), defined over a complete and separable metric space equipped with a boundedly-
finite Borel measure. We focus on three
different approaches: via approximation with Lipschitz functions; by studying the behaviour along curves, 
in terms either of the curve modulus or
of test plans; via integration-by-parts, using Lipschitz derivations with divergence. Eventually, we show 
that all these approaches are fully equivalent.
We emphasise that no doubling or Poincar\'{e} assumption is made, and that we allow also for the exponent 
\(p=1\).

A substantial part of this work consists of a self-contained and partially-revisited exposition of known 
results, which are scattered across the existing
literature, but it contains also several new results, mostly concerning the equivalence of metric Sobolev 
spaces for \(p=1\).
\end{abstract}
\bigskip

\tableofcontents
\section{Introduction}
\subsection{General overview}
\subsubsection{The classical Euclidean setting}\label{sec:Euclidean_setting}
The theory of spaces of \emph{weakly differentiable functions} in the Euclidean setting, today known as 
\emph{Sobolev spaces} (named after the Russian mathematician Sergei Sobolev),
goes back to the beginning of the 20th century. The main motivation for their study comes from seeking to 
give the meaning to the `weak' solutions to many important classes of PDEs.
We refer e.g.\ to \cite{adams:sobolev, leoni:sobolev, mazya:sobolev} for a thorough account of the Sobolev 
space theory in this classical setting.
Several equivalent definitions of Sobolev space have been established, three of which we recall next as 
relevant in the forthcoming discussion and we refer to them as the
{\bf H}-, {\bf W}- and {\bf BL}-approaches. If not specified differently, throughout the paper the exponent 
\(p\) ranges in the interval \([1,\infty)\), while \(q\in(1,\infty]\)
denotes its conjugate exponent. 
\begin{itemize}
\item The {\bf H-approach}, based on the approximation by smooth, compactly-supported functions 
\(C^{\infty}_c(\R^n)\) led to the definition of the space \(H^{1,p}(\R^n)\),
consisting of those \(p\)-integrable functions \(f\) admitting a weak gradient \(\nabla_w f\), obtained as 
the \(L^p\)-limit of \(\nabla f_i\), for some approximating sequence
\((f_i)_i\) of smooth functions. This notion was firstly considered by Hilbert in 1900, when trying to find 
a suitable class of functions which minimize the Dirichlet integral
\cite{Hil:04}. The `H' in the notation reminds of the \emph{Hilbertianity} of the space \(H^{1,2}(\R^n)\).
\item  The {\bf W-approach}, based on the existence of the weak gradient in terms of an 
integration-by-parts formula, while testing against vector fields having distributional divergence.
This approach led to the definition of the space \(W^{1,p}(\R^n)\), which is the space of those 
\(p\)-integrable functions \(f\) associated with a weak gradient \(\nabla_w f\) for which the formula
\begin{equation}\label{eq:ibp}
\int v\cdot \nabla_w f\,\d\mathcal L^n=-\int f\div(v)\,\d\mathcal L^n\quad 
\text{ for }v\in C^{\infty}_c(\R^n;\R^n)
\end{equation}
holds. 
The notation `W' reminds of the existence of \emph{weak} gradients. This notion has been introduced by 
Sobolev in \cite{Sob:06}, originally using the notation `L' and in 1950s switching to `W'.
\item The {\bf BL-approach}, based on the property of functions to be absolutely continuous `along curves', 
led to the definition of the space \(BL^{1,p}(\R^n)\). It was firstly introduced by \emph{Beppo Levi} in 
\cite{Lev:06}, considering only curves in the coordinate directions; this is where the `BL' in the notation 
comes from and came into use in 1950s, by Nikod\'{y}m. A drawback of this approach was the dependence on 
the choice of coordinates. This drawback was removed by the refined approach of Fuglede \cite{Fug:57}, 
looking at functions absolutely continuous along `$p$-almost every' curve. Namely, a \(p\)-integrable 
function \(f\) is an element of
\(BL^{1,p}(\R^n)\) if for `\(p\)-almost every' rectifiable curve \(\gamma\colon [0,1]\to \R^n\) it holds 
that \(f\circ \gamma\) is absolutely continuous and its weak gradient \(\nabla_w f\), determined by the 
validity of the identity
\begin{equation*}
(f\circ \gamma)'_t=\nabla_w f(\gamma_t)\cdot \gamma'_t\quad \text{ for a.e.\ }t\in (0,1),
\end{equation*} 
is \(p\)-integrable.
In particular, \(|\nabla_w f|\) has the following variational characterization: it is the a.e.\ minimal 
\(p\)-integrable non-negative function satisfying for `\(p\)-almost every' rectifiable
curve \(\gamma\) the inequality
\begin{equation}\label{eq:ac}
|(f\circ \gamma)'_t|\leq |\nabla_w f|(\gamma_t)|\gamma'_t|\quad \text{for a.e.\ }t\in (0,1).
\end{equation}
 We will recall below the concept of `\(p\)-almost every' (in terms of the \emph{\(p\)-modulus}) used by 
 Fuglede,
which transfers also to the more general setting of our interest.
\end{itemize}
The equivalence between the approaches {\bf H} and {\bf BL} has been proved in 1957 by Fuglede 
\cite{Fug:57}.
Another important contribution to the theory was the paper `{\bf H}={\bf W}' by Meyers and Serrin from 1964 
\cite{Mey:Ser:64}, proving the equivalence of the approaches {\bf H} and {\bf W}. 
More precisely, for \(p\in [1,\infty)\)
\begin{equation}\label{eq:equiv_Rn}
 H^{1,p}(\R^n)= W^{1,p}(\R^n) =BL^{1,p}(\R^n)\; \textit{ and the respective weak gradients coincide}.
\end{equation}
Thus we can refer to any of the above spaces as to the \(p\)-Sobolev space and state equivalently that
\begin{equation}\label{eq:MS}
\textit{smooth functions are strongly dense in the \(p\)-Sobolev space, for all \(p\in [1,\infty)\).}
\end{equation}
\subsubsection{Brief history of metric Sobolev spaces}
In the last three decades, the analysis on \emph{metric measure spaces} has been an intense research area, 
developing at a fast pace. In this work, by a metric measure space
we mean a complete and separable metric space \((\X,\sfd)\) endowed with a non-negative Borel measure 
\(\mm\) that is finite on bounded sets. A significant role in the study
of metric measure spaces is played by the investigation of first-order \emph{Sobolev spaces}. Several 
definitions of metric Sobolev space can be found in the literature, starting
from the one proposed by Haj\l asz in \cite{Haj:96}. Our aims in this work are the following:
\begin{itemize}
    \item [1)] to provide a detailed and systematic overview of the approaches {\bf H}, {\bf W} and 
    {\bf BL} 
    in the metric setting;
    \item [2)] to prove the equivalence of these approaches for all \(p\in[1,\infty)\);
    \item [3)] to lay the groundwork for the follow-up work \cite{AILP}, 
    where we concentrate on the study of the underlying differential calculus, dual energies and their
    applications to potential analysis (see Section \ref{sec:future_work} below).
\end{itemize}

The family of metric measure spaces includes, among others, the following classes of spaces: Riemannian,
Finsler, or more general topological manifolds, Carnot groups and sub-Riemannian structures, 
Alexandrov spaces, RCD spaces (containing all Ricci-limit spaces) and weighted Banach spaces.
We next recall the definitions that are the metric counterparts to the Euclidean ones presented in Section 
\ref{sec:Euclidean_setting}:
\begin{itemize}
\item The {\bf H-approach}, essentially due to Cheeger \cite{Ch:99} --- the role of \(|\nabla f|\) for smooth functions being 
now played by any \emph{upper gradient} of the function \(f\).
This theory has been later on revisited in \cite{Amb:Gig:Sav:13,Amb:Gig:Sav:14}, considering a more restrictive class of approximating functions $f$,
namely Lipschitz functions with bounded support, and a specific and possibly larger choice
of upper gradients in the approximation, namely the \emph{asymptotic slope} \(\lip_a(f)\) of \(f\); we 
refer to the latter as to the {\bf H} approach. This led to the definition of the space \(H^{1,p}(\X)\),
each element \(f\) being associated with a \emph{minimal relaxed slope} \(|Df|_H\), which plays the role of 
the modulus of the weak gradient.
\item The {\bf W-approach}, due to Di Marino \cite{DiMar:14} -- based on an integration-by-parts formula 
\eqref{eq:ibp}, the role of vector fields being played by \emph{derivations with divergence}.
In this way, we obtain the corresponding space \(W^{1,p}(\X)\) whose elements \(f\) are associated with a 
minimal object playing the role of the `modulus of the differential', denoted by \(|Df|_W\).
\end{itemize}
The {\bf BL-approach} has (essentially) two different viewpoints, with respect to the way of selecting the 
negligible curve families where \eqref{eq:ac} does not hold.
\begin{itemize}
\item  The {\bf B-approach}, due to the first named author together with Gigli and Savar\'e 
\cite{Amb:Gig:Sav:13, Amb:Gig:Sav:14} -- the set of negligible curves determined by using the notion of 
\emph{\(q\)-test plan},
providing the notion of Sobolev space \(B^{1,p}(\X)\). Each element \(f\) is associated with a minimal 
function \(|Df|_B\), called the \emph{minimal \(B\)-weak upper gradient},
which plays the role of \(|\nabla_w f|\) on the right-hand side of \eqref{eq:ac}. Here, the `B' in the 
notation reminds of the relation with Beppo Levi's approach. Similarly,
Savar\'e \cite{Sav:22} used the notion of \emph{plan with barycenter}, previously introduced in 
\cite{Amb:Mar:Sav:15}, for the same purpose. 
\item The {\bf N-approach}, due to Shanmugalingam \cite{Sha:00} after Koskela--MacManus \cite{Kos:Mac:98} 
-- generalization of the original Fuglede's approach on \(\R^n\) to the metric measure space setting.
The concept of \emph{p-modulus} introduced by Fuglede in \cite{Fug:57} (being directly connected to the 
concept of \emph{extremal length} due to Ahlfors and Beurling \cite{Ahl:Beu:50}) has been used to determine
the set of negligible curves. This approach led to the definition of the \emph{Newtonian Sobolev space} 
\(N^{1,p}(\X)\), where the notation `N' comes from. Similarly as in B-approach, we associate
to each \(f\) a minimal function \(|Df|_N\) and we call it the \emph{minimal \(N\)-weak upper gradient} of 
\(f\).
\end{itemize}
 Already in the Euclidean theory, the interest in proving the equivalence between notions that
seem substantially different (approximation, integration by parts, behaviour along almost every curve) is self-evident.
In connection with the curve-based Beppo-Levi and Fuglede's approaches, particularly interesting for the development of a good differentiable calculus
in metric measure spaces (see for instance \cite{Gig:15} and \cite{Gig:18}) is the possibility not only to single out classes of exceptional curves, but also
the possibility to identify ``nice'' families of curves, as the test plans or the plans with finite energy and/or bounded compression; the good behaviour of
a function along these families turns out to be sufficient to provide Sobolev regularity for any of the other definitions.

When restricted to the Euclidean setting, all the above-mentioned notions of Sobolev space coincide with 
the usual notion of Sobolev space.
A great attention has been devoted to the study of their equivalence in full generality, or, in other 
words, to obtaining a result analogous to \eqref{eq:equiv_Rn}:
\begin{equation}\label{eq:equivalence_mms}
H^{1,p}(\X)=W^{1,p}(\X)=B^{1,p}(\X)=N^{1,p}(\X)\;\text{ and }\; |Df|_H=|Df|_W=|Df|_B=|Df|_N.
\end{equation}
In what follows, we shall walk briefly through a list of several instances in which the above statement is 
true, depending on the assumptions of the
underlying metric measure space, or on the exponent \(p\) under consideration. It is clear that, with no 
linear structure on the underlying space,
in the metric measure space setting the role of smooth functions is, in general, played by the Lipschitz 
ones. 
Thus, the key fact behind 
\eqref{eq:equivalence_mms} is given by the following statement:
\begin{equation}\label{eq:MS_mms}
    \textit{Lipschitz functions are dense in energy in the \(p\)-Sobolev space, for every \(p\in 
    [1,\infty)\),}
\end{equation}
a result analogous to \eqref{eq:MS} in the Euclidean setting.
Above, by the \emph{density in energy} we mean that 
for every \(p\)-Sobolev \(f\) there is a sequence of boundedly-supported Lipschitz functions \((f_i)_i\) 
approximating \(f\) in the following sense:
\begin{equation}
f_i\to f\, \text{ in } L^p(\mm)\quad \text{ and }\quad |Df_i|\to |Df|\, \text{ in } L^p(\mm).
\end{equation}
The notation \(|Df|\) is understood as any of the four minimal functions associated with the Sobolev \(f\).
This form of density may be in general weaker than the strong density, but still sufficient for many 
applications of the theory.
In the case $p\in (1,\infty)$, the equivalence between the ${\bf N}$ approach and the original Cheeger's approach, 
has been proven in \cite{Sha:00}.
The first proof of the energy density of Lipschitz functions in \(B^{1,p}(\X)\) appeared in 
\cite{Amb:Gig:Sav:13} in the case \(p\in (1,\infty)\), using techniques based on the metric Hopf-Lax 
semigroup. As a consequence,  the equivalence between the approaches {\bf H}, {\bf B} and {\bf N} is 
proven. By means of the same techniques, the equivalence of the four listed notions has been then proven in 
\cite{DiMaPhD:14,DiMar:14} for \(p\in (1,\infty)\). Therein, similar techniques have been used also to show 
the equivalence of the respective notions of the space of functions of bounded variation. 
In the setting of \emph{extended metric measure spaces} (which includes all metric measure spaces), in 
\cite{Sav:22} an approach based on duality arguments and von Neumann min-max principle has been used in 
proving the equivalence of {\bf H}, {\bf B} and {\bf N}. Another route has been taken in 
\cite{EB:20:published}, proving the energy density in the Newtonian Sobolev space and thus establishing the 
equalities ${\bf H} = {\bf N}$ and $|Df|_H = |Df|_N$. This technique covers also the case \(p=1\).

A condition under which the density in energy can be improved to the \emph{strong density} is given by the 
\emph{reflexivity} of the Sobolev space in the case \(p\in (1,\infty)\); see Section 
\ref{sec:Functional_properties}. We mention, as a particular example, the class of  \emph{PI spaces}, 
namely, those metric measure spaces with \(\mm\) being doubling
and satisfying a suitable form of Poincar\' e inequality; see \cite{Ch:99, HK:00}.

Another proof of the equivalence of the four notions in the case \(p\in (1,\infty)\) has been recently 
provided in \cite{Lu:Pa:23}, reducing the problem to the study of the same question in the setting of a 
\emph{weighted Banach space}, namely, a separable Banach space \(\mathbb B\) endowed with an arbitrary non-
negative boundedly-finite Borel measure.
In this setting we still have density in energy, but we can actually gain the word `smooth' appearing in 
the statement \eqref{eq:MS_mms}. It turns out that the class of \emph{cylindrical functions} (in particular 
of class \(C^\infty\)) is shown to be dense in energy in the Sobolev space; see \cite{Sav:22} and 
\cite{Lu:Pa:23}. 
 
Without the aim of being exhaustive, we also mention other definitions of metric Sobolev space and some 
other functional spaces present in the literature -- obtained via approaches analogous to the above -- 
which will not be considered in this work.
\begin{itemize}
\item Haj\l asz--Sobolev space 
\cite{Haj:96} and 
Korevaar--Schoen Sobolev space (the latter introduced first in the setting of Riemannian manifolds in 
\cite{Kor:Sch:93}).
In the PI setting, both notions turn out to be equivalent to the four notions considered in this paper.
\item Characterizations (in the PI setting) of the Sobolev space via non-local functionals \emph{\` a la} 
Bourgain--Brezis--Mironescu (see e.g.\ 
\cite{DiMar:Sq:19}). 
\item The spaces of functions of bounded variation (see e.g.\ \cite{Mir:03,DiMaPhD:14}).
\item The spaces of metric-valued Sobolev maps (cf.\ \cite{HKST:15}).
\item First-order Sobolev spaces in the setting of extended metric measure spaces \cite{Sav:22}, comprising 
for instance abstract Wiener spaces or configuration spaces. 
\end{itemize} 
\subsection{Contents of this work}
As can be seen from the above discussion, the question of the equivalence has been solved in the case 
\(p\in (1,\infty)\) in several different ways, while  
the density in energy of Lipschitz functions in \(N^{1,p}(\X)\) from \cite{EB:20:published} gives 
\(N^{1,p}(\X) = H^{1,p}(\X)\) and $|Df|_N = |Df|_H$ for all \(p\in [1,\infty)\). In order to obtain 
\eqref{eq:equivalence_mms}, i.e.\ the equivalence of the four notions in a general metric measure space, 
comprising also the case  \(p=1\), we 
prove in Section \ref{sec:Equivalence} that
\[
H^{1,p}(\X)\subseteq W^{1,p}(\X)\subseteq B^{1,p}(\X)\subseteq N^{1,p}(\X)\;\text{ and }\;
|Df|_N\leq|Df|_B\leq |Df|_W\leq |Df|_H.
\]
This is then enough to conclude the equivalence, taking into account the previous observation. The first 
two inclusions are proven via well-established techniques. The first one by writing the 
integration-by-parts formula for Lipschitz functions and passing it to the limit suitably.
The second one 
comes from the
relation between test plans and derivations, the objects we investigate in Section \ref{sec:mod_plan} and 
Section \ref{sec:derivations} respectively. The most involved part of the proof 
is the third inclusion, which relies on showing that, for a given function \(f\in B^{1,p}(\X)\),
the set of curves along which the absolute continuity of \(f\) is violated (which is negligible with 
respect to any \(q\)-test plan) is negligible for the \(p\)-modulus -- we deal with this issue in Section 
\ref{sec:mod_plan}. 
\smallskip

We now describe  section-by-section the contributions of this work. At the end of each section, we will 
provide some bibliographical
notes and a list of the symbols introduced in that section.

\paragraph{Section \ref{sec:preliminaries}.}
We devote this section to set up the notation used throughout the paper. A significant part of it is 
devoted to the study of \emph{curves} in metric spaces, their \emph{reparametrizations} and the related 
notions of \emph{path integrals}. 
Most of the results are already present in the literature -- here we collect, state and prove them in great 
generality.
We draw the reader's attention to Section~\ref{sec:reparametrization} and 
Section~\ref{sec:integration_along_paths}, where, respectively, the properties of reparametrization maps of curves 
are considered and the notion of \emph{integration along rectifiable curves} of any Borel measurable 
function is provided. We prefer to develop the curve theory in the generality of rectifiable curves for future developments, also related to the fine theory of $BV$ functions in metric measure spaces.

\paragraph{Section \ref{sec:mod_plan}.}
This section deals with the two notions of measures on the space of rectifiable curves \(\mathscr R(\X)\) 
on \(\X\) needed to define the Sobolev spaces following the {\bf B} and {\bf N} approaches. 
The first one is the \(p\)-modulus \(\Modp\), an outer measure on \(\mathscr R(\X)\), in the sense of 
Fuglede \cite{Fug:57}. It is defined, for every \(\Gamma\subseteq \mathscr R(\X)\), as the infimum of 
\(\|\rho\|^p_{L^p(\mm)}\) where \(\rho\) varies in the set of admissible functions, i.e.\ among non-
negative Borel functions satisfying 
\(\int_\gamma \rho\,\d s\geq 1\)
for all \(\gamma\in \Gamma\). The second notion consists of plans with $q$-barycenter (introduced in 
\cite{Amb:Mar:Sav:15,Sav:22}), where a plan is a non-negative Borel measure $\ppi$ on the space of 
continuous curves, concentrated on the subset of the rectifiable ones, and a plan has a  $q$-barycenter 
\({\rm Bar}(\ppi)\in L^q(\mm)\) if
\[
\int \rho\, {\rm Bar}(\ppi)\, \d\mm=\int\!\!\!\int_\gamma \rho\,\d s\,\d\ppi(\gamma) \quad\text{ for every 
Borel function }\rho\colon \X\to [0,\infty].
\]
The \(q\)-test plans, used in the definition of \(B^{1,p}(\X)\), are in particular plans with \(q\)-
barycenter.

A significant part of this section is devoted to the preparation of the proof of the inclusion 
\(B^{1,p}(\X)\subseteq N^{1,p}(\X)\).
In the case \(p\in (1,\infty)\), the above inclusion comes as a consequence of the duality formula between 
\(\Modp\) and plans with \(q\)-barycenter: namely, for every Souslin family of curves \(\Gamma\) with 
$0 < \Modp(\Gamma) < +\infty$ there exists a plan with \(q\)-barycenter \(\ppi_{\Gamma}\) so that
\[
(\Modp(\Gamma))^{\frac{1}{p}}=  \frac{\ppi_{\Gamma}(\Gamma)}{\|{\rm Bar}(\ppi_{\Gamma})\|_{L^q(\mm)}}.
\]
The duality formula is first  verified for compact families and then the identity for Souslin $\Gamma$
follows from the Choquet capacity property of the modulus \cite{Amb:Mar:Sav:15}. 
However, ${\rm Mod_1}$ is not a Choquet capacity, see \cite{Ex:Ka:Ma:Ma:21}. 
Thus, in order take into account all \(p\in [1,\infty)\), our argument is based on showing that on compact 
sets \(\Gamma\) of positive and finite \(p\)-modulus the above duality holds true (see Proposition 
\ref{prop:pi_gamma}), and then  on arguing  by contradiction and on finding such a compact family of curves 
inside the family of our interest (see Lemma \ref{lem:gamma_f_rho_2}). In order to prove the existence of 
the mentioned compact family, we rely on the notion of modified modulus (introduced in \cite{Sav:22}), 
taking into account the value of \(\rho\) at the curve's endpoints when testing 
admissibility.

\paragraph{Section \ref{sec:derivations}.}
In this section, we recall the concept of \emph{Lipschitz derivation}, which comes in when defining the \(p\)-Sobolev space via the {\bf W}
approach introduced by Di Marino in \cite{DiMaPhD:14}. In that case, the role of \(q\)-integrable vector fields with divergence is played
by linear maps $b \colon \LIP_{bs}(\X) \rightarrow L^{q}( \mm )$. Such \(b\) are also required to satisfy a suitable \emph{locality} property and 
\emph{Leibniz rule}, and to admit a \emph{divergence} \(\div(b)\in L^q(\mm)\), i.e.
\[
\int b(f)\,\d \mm=-\int f\,\div(b)\, \d\mm\quad \text{ holds for all }f \text{ Lipschitz with bounded support}.
\]
We denote by \({\rm Der}^q_q(\X)\) the space of all \(q\)-integrable Lipschitz derivations with \(q\)-integrable divergence as above. 
The notion itself is inspired by that of Weaver \cite{Wea:00};  we study the relation between the two in Section \ref{sec:Weaver_derivation}.
Although the class \({\rm Der}_q^q(\X)\)  is sufficient in order to define the space \(W^{1,p}(\X)\), we consider also a more general class of derivations 
having measure-valued divergence (cf.\ Definition \ref{def:derivations_with_divergence}). 
In the follow-up work \cite{AILP}, one of our aims is related to the dual energies of such measures and to the associated notion of a $p$-Laplacian. 
Regarding the equivalence of Sobolev spaces, the inclusion \(W^{1,p}(\X)\subseteq B^{1,p}(\X)\) is based on the fact that every plan with \(q\)-barycenter 
induces an element of \({\rm Der}_q^q(\X)\) (see Proposition \ref{prop:plan_induces_derivation}).

\paragraph{Section \ref{sec:Sobolev_H}.}
This section is devoted to the presentation of the precise definitions of the Sobolev spaces obtained via the four approaches {\bf H}, {\bf W}, {\bf B} 
and {\bf N} listed above, relying on the preparatory material on Lipschitz functions, plans, derivations and modulus we obtained in the previous sections.

\paragraph{Section \ref{sec:calculus_rules}.}
This section is focused on the properties of Newtonian Sobolev functions, which -- after proving the equivalence -- will be transferred to \(p\)-Sobolev 
functions in the sense of any of the other three approaches. We study in Section \ref{sec:calculus_rules_dirichlet} the calculus rules for the elements of 
the Newtonian Sobolev space \(N^{1,p}(\X)\), recall the notion of \(p\)-capacity and of the good representative of Newtonian Sobolev functions, needed in 
order to have the energy density result proven in \cite{EB:20:published}.  In order to be self contained, we provide in Section \ref{sec:energy_density} a 
sketch of the proof of the latter result. 
We also emphasize that in Section \ref{sec:calculus_rules_dirichlet} we prove calculus rules not only for Newtonian Sobolev functions, but for all 
measurable functions admitting weak upper gradients (in particular, for all Dirichlet functions). The key idea lies in Lemma \ref{lemm:newt:leibniz}, 
showing that the Leibniz rule `along curves'  is satisfied by the mentioned functions. This result is based on the notion of integration along paths we 
present in Section \ref{sec:integration_along_paths}. Analyzing Dirichlet functions is motivated by constructions of capacitary potentials; we refer e.g.\ 
to \cite{Ki:Sha:01, Bj:Bj:Le:20, Bj:Bj:11, Es:Iko:Raj} and references therein for related topics.

\paragraph{Section \ref{sec:Equivalence}.}
This section contains the proof of our main result Theorem \ref{thm:equivalence_Sobolev_spaces}, namely of the equivalence of the approaches {\bf H}, 
{\bf W}, {\bf B} and {\bf N} for all \(p\in [1,\infty)\). Once the equivalence is proven, we turn to gathering some of the important functional properties 
of \(p\)-Sobolev spaces, e.g.\ the property of being Banach spaces, sufficient and necessary conditions for reflexivity, the dependence on \(p\) of the 
minimal weak upper gradients, as well as some stability and localization results. 
\subsection{Contents of the follow-up work}\label{sec:future_work}

All the different approaches to the definition of Sobolev functions shared a common aspect: any Sobolev function 
is associated with a minimal object that
plays the role of the `modulus of the distributional differential' of
the given function.  However, many of the standard functional-analytic tools have been generalized to the metric measure space setting within the 
work of Gigli in \cite{Gig:18}, giving the possibility to talk about the
linear differential \(\d f\) of a Sobolev function \(f\) as an element of the 
\(p\)-\emph{cotangent module} -- a generalization of the space of \(p\)-integrable sections of the cotangent bundle. The \(q\)-\emph{tangent module}, 
consisting of abstract \(q\)-integrable  vector fields, is defined as the module dual of the cotangent module. As such, it gives the possibility of 
talking about \emph{gradients} of Sobolev functions. 
For instance, it is not difficult to see that the elements of \({\rm Der}_q^q(\X)\) belong to the 
\(q\)-tangent module. 
Our follow-up work \cite{AILP} will consist of:
\begin{itemize}
    \item exploring deeper the relations between the above-mentioned  differential objects, and their relations to currents (revisiting partially the work
    of Schioppa \cite{Sch:16:der});
    \item proving a representation of the dual energy functional defined on (boundedly-finite) signed measures \(\nu\) by 
    \[
    {\sf F}_p(\nu)\coloneqq \sup\left\{\int f\,\d \nu\,\Big|\, \int \lip_a^p(f)\,\d \mm\leq 1 \text{ for } f\text{ Lipschitz with bounded support}\right\},
    \]
    in terms of the minimization of the \(L^q\)-norm of the derivations with \({\bf div}(b)=\nu\). When ${\sf F}_p( \nu ) <+ \infty$, we call $\nu$ a 
    \emph{divergence measure}.
    Dual energies have been studied in \cite{Amb:Sav:21,Sav:22} for \(p\in (1,\infty)\) in terms of the duality with plans with barycenters (to make a 
    comparison, take into account Lemma \ref{prop:plan_induces_derivation});
    \item studying the properties of divergence measures and in particular their relations to capacity measures;
    \item considering more specific divergence measures, namely, being a \(p\)-Laplacian of some Dirichlet function in a suitable sense;
    \item applying the above understandings to study the pre-dual of the Sobolev space in the case \(p\in (1,\infty)\) and obtain several 
    characterizations of the reflexivity property.
\end{itemize}

\paragraph{Acknowledgements.} 
The authors wish to thank Sylvester Eriksson-Bique for several helpful discussions on the contents of this paper.

L.A.\ and E.P.\ have been supported by the MIUR-PRIN 202244A7YL project “Gradient Flows and Non-Smooth
Geometric Structures with Applications to Optimization and Machine Learning”.

T.I.\ has been supported by the Academy of Finland, project numbers 308659 and 332671 and by the Vilho, Yrjö and Kalle Väisälä Foundation.
Part of the research was initiated when T.I.\ was visiting Scuola Normale Superiore di Pisa. T.I.\ thanks the institute for their hospitality.
\section{Preliminaries}\label{sec:preliminaries}
Aim of this section is to fix the terminology and recall some key results concerning metric spaces, measure theory, and functional analysis,
which will be used throughout the whole paper.
\subsection{General notation}\label{ss:gen_term}
Given any exponent \(p\in[1,\infty]\), we will implicitly denote by \(q\in[1,\infty]\) its \textbf{conjugate exponent}, and vice versa.
Namely, it will be always tacitly understood that \(\frac{1}{p}+\frac{1}{q}=1\).
\medskip

Given a non-empty set \(\X\) and a function \(f\colon\X\to\R\), we will denote by \(f^+\colon\X\to[0,\infty)\) and \(f^-\colon\X\to[0,\infty)\)
the \textbf{positive part} and the \textbf{negative part} of \(f\), respectively. Namely,
\begin{equation}\label{eq:pos_neg_part}
f^+\coloneqq\max\{f,0\},\qquad f^-\coloneqq\max\{-f,0\}.
\end{equation}
Notice that \(f=f^+ -f^-\) and \(|f|=f^+ +f^-\). The collection of all subsets of \(\X\) will be denoted by \(\mathscr P(\X)\).
Given any subset \(E\) of \(\X\), we denote by \(\1_E\colon\X\to\{0,1\}\) its \textbf{characteristic function}, i.e.
\begin{equation}\label{eq:char_fct}
\1_E(x)\coloneqq\left\{\begin{array}{ll}
1\\
0
\end{array}\quad\begin{array}{ll}
\text{ if }x\in E,\\
\text{ if }x\in\X\setminus E.
\end{array}\right.
\end{equation}
We denote by \(\mathcal L^n\) the Lebesgue measure on \(\R^n\) (defined on the Borel \(\sigma\)-algebra of \(\R^n\)), while \(\mathcal L_1\)
will denote the restriction of \(\mathcal L^1\) to the Borel sets of the unit interval \([0,1]\subseteq\R\).
\subsection{Metric geometry}\label{s:metric_geometry}
In a metric space \((\X,\sfd)\), we denote by \(B_r(x)\) (resp.\ \(\bar B_r(x)\)) the open (resp.\ closed) ball of radius \(r>0\) 
and center \(x\in\X\),
i.e.\ \(B_r(x)\coloneqq\{y\in\X\,:\,\sfd(x,y)<r\}\) and \(\bar B_r(x)\coloneqq\{y\in\X\,:\,\sfd(x,y)\leq r\}\).
The diameter of a non-empty set \(E\subseteq\X\) is defined as \(\diam(E)\coloneqq\sup\big\{\sfd(x,y)\,\big|\,x,y\in E\big\}\).
\medskip

Let \((\X,\sfd_\X)\) and \((\Y,\sfd_\Y)\) be complete, separable metric spaces. We denote by \(C(\X;\Y)\) the set of all continuous
maps from \(\X\) to \(\Y\). We denote by \(C_b(\X;\Y)\) the subset of \(C(\X;\Y)\) consisting of all \(f\colon\X\to\Y\) continuous
and bounded, meaning that \(f(\X)\) has finite diameter.

The set \(C_b(\X;\Y)\) becomes a complete metric space (which is also separable in the case where \(\X\) is compact) if endowed
with the \textbf{supremum distance} \(\sfd_{C_b(\X;\Y)}\), which is defined as
\begin{equation}\label{eq:def_sup_dist}
\sfd_{C_b(\X;\Y)}(f,g)\coloneqq\sup_{x\in\X}\sfd_\Y\big(f(x),g(x)\big),\quad\text{ for every }f,g\in C_b(\X;\Y).
\end{equation}
Notice that if \(\X\) is compact, then \(C_b(\X;\Y)=C(\X;\Y)\). In the case where \(\Y=\R\) with the Euclidean distance, we opt for the
shorthand notation \(C(\X)\) and \(C_b(\X)\), instead of \(C(\X;\R)\) and \(C_b(\X;\R)\), respectively. We denote by \(C_{bs}(\X)\)
the space of \(f\in C(\X)\) with bounded support. The space of all non-negative continuous functions will be denoted by \(C_+(\X)\). 
We remark that hereafter the same subscripts will be used to denote the corresponding property of the elements in some subspace of \(C(\X)\). 
\medskip

A map \(f\in C(\X;\Y)\) is said to be \textbf{Lipschitz} (or \textbf{\(L\)-Lipschitz}) if there exists \(L\geq 0\) with
\[
\sfd_\Y\big(f(x),f(y)\big)\leq L\,\sfd_\X(x,y),\quad\text{ for every }x,y\in\X.
\]
We denote by \(\LIP(\X;\Y)\subseteq C(\X;\Y)\) the set of all Lipschitz maps from \(\X\) to \(\Y\). 
Given a map \(f\in\LIP(\X;\Y)\) and \(\varnothing\neq E\subseteq\X\), we define the \textbf{Lipschitz constant} of \(f\) on \(E\) as
\begin{equation}\label{eq:Lip_const}
\Lip(f;E)\coloneqq\sup\bigg\{\frac{\sfd_\Y(f(x),f(y))}{\sfd_\X(x,y)}\;\bigg|\;x,y\in E,\,x\neq y\bigg\}.
\end{equation}
When \(E=\X\), we use the shorthand notation \(\Lip(f)\coloneqq\Lip(f;\X)\). Notice that \(\Lip(f)\) is the minimal value \(L\geq 0\)
for which \(f\) is \(L\)-Lipschitz. Given any function \(f\in\LIP(\X)\), we define the \textbf{slope} \(\lip(f)\colon\X\to[0,\Lip(f)]\)
and the \textbf{asymptotic slope} \(\lip_a(f)\colon\X\to[0,\Lip(f)]\) of \(f\) as
\begin{equation}\label{eq:def_slopes}
\lip(f)(x)\coloneqq\lims_{y\to x}\frac{|f(x)-f(y)|}{\sfd_\X(x,y)},\qquad\lip_a(f)(x)\coloneqq\inf_{r>0}\Lip\big(f;B_r(x)\big),
\end{equation}
if \(x\in\X\) is an accumulation point, while \(\lip(f)(x)=\lip_a(f)(x)\coloneqq 0\) if \(x\in\X\) is an isolated point.
Notice that \(\lip(f)\leq\lip_a(f)\) and that the infimum in \eqref{eq:def_slopes} can be replaced by \(\lim_{r\to 0}\).
We will use the notations \(\LIP_b(\X)\), \(\LIP_{bs}(\X)\) and \(\LIP_+(\X)\) to denote the spaces of bounded, boundedly-supported and non-negative 
elements of \(\LIP(\X)\), respectively.
\smallskip

We now briefly recall the concept of Souslin sets. 
Firstly, for a Hausdorff topological space 
\((\Y,\tau)\) we say that it is a \textbf{Polish space} if there exists a distance \(\sfd_\Y\) on \(\Y\)
that induces the same topology as \(\tau\) and such that
the metric space \((\Y,\sfd_\Y)\) is complete and separable. 
A subset \(S\subseteq \Y\) of a Hausdorff topological space \((Y,\tau)\) is said to be a \textbf{Souslin set} 
if it is the image of some Polish space under a continuous map.
The following property will be useful for our purposes:
given two complete and separable metric spaces \((\X,\sfd_\X)\) and \((\Y,\d_\Y)\) and a Borel map \(f\colon \X\to \Y\), 
it holds that \(f(S)\subseteq \Y\) is a Souslin set whenever \(S\subseteq \X\) is so.
\subsection{Measure theory}\label{s:measure_theory}
By a \textbf{measure space} we mean a triple \((\X,\Sigma,\mu)\), where \((\X,\Sigma)\) is a measurable space
(i.e.\ \(\X\neq\varnothing\) is a set and \(\Sigma\subseteq\mathscr P(\X)\) is a \(\sigma\)-algebra on \(\X\)),
while \(\mu\colon\Sigma\to[0,\infty]\) is a (countably-additive, non-negative) measure on \((\X,\Sigma)\).
We say that $\mu$ is \textbf{concentrated} on a set $E\in\Sigma$ if \(\mu(A)=0\) for every \(A\in\Sigma\)
such that \(A\cap E=\varnothing\), or equivalently if \(\mu(A)=\mu(A\cap E)\) for every \(A\in\Sigma\).
Given any set \(G\in\Sigma\), we define the \textbf{restriction} \(\mu|_G\colon\Sigma\to[0,\infty]\) of \(\mu\) to \(G\) as
\begin{equation}\label{eq:def_restr_meas}
\mu|_G(E)\coloneqq\mu(E\cap G)\quad\text{ for every }E\in\Sigma.
\end{equation}
Then \(\mu|_G\) is a non-negative measure on \((\X,\Sigma)\). Given a measure space \((\X,\Sigma_\X,\mu)\), a measurable space
\((\Y,\Sigma_\Y)\) and a measurable map \(\varphi\colon\X\to\Y\) (i.e.\ \(\varphi^{-1}(F)\in\Sigma_\X\) for every
\(F\in\Sigma_\Y\)), we define the \textbf{pushforward} \(\varphi_\#\mu\colon\Sigma_\Y\to[0,\infty]\) of \(\mu\) as
\begin{equation}\label{eq:def_pushforward_meas}
(\varphi_\#\mu)(F)\coloneqq\mu\big(\varphi^{-1}(F)\big),\quad\text{ for every }F\in\Sigma_\Y.
\end{equation}
Then \(\varphi_\#\mu\) is a measure on \((\Y,\Sigma_\Y)\) and \((\varphi_\#\mu)(\Y)=\mu(\X)\). In particular,
if \(\mu\) is finite (i.e.\ \(\mu(\X)<+\infty\)), then \(\varphi_\#\mu\) is finite as well. However, if
\(\mu\) is \(\sigma\)-finite (i.e.\ there exists a partition \((E_n)_{n\in\N}\subseteq\Sigma\) of \(\X\)
such that \(\mu(E_n)<+\infty\) for every \(n\in\N\)), then \(\varphi_\#\mu\) is not necessarily \(\sigma\)-finite.
For an arbitrary collection \(\{\mu_i\}_{i\in I}\) of measures on \((\X,\Sigma)\), we define their
\textbf{supremum} \(\bigvee_{i\in I}\mu_i\) as
\begin{equation}\label{eq:def_sup_meas}
\bigg(\bigvee_{i\in I}\mu_i\bigg)(E)\coloneqq\sup\sum_{i\in C}\mu_i(E_i)\quad\text{ for every }E\in\Sigma,
\end{equation}
where the supremum is over all countable subsets \(C\subseteq I\) and all partitions \((E_i)_{i\in I}\subseteq\Sigma\)
of \(E\). Then \(\bigvee_{i\in I}\mu_i\colon\Sigma\to[0,\infty]\) is a measure on \((\X,\Sigma)\). Moreover, it is the
smallest measure on \((\X,\Sigma)\) satisfying \(\mu_i\leq\mu\) for every \(i\in I\), in the sense that if \(\nu\) is
any measure on \((\X,\Sigma)\) such that \(\mu_i\leq\nu\) for every \(i\in I\), then \(\bigvee_{i\in I}\mu_i\leq\nu\).
Analogous definition can be given for the \textbf{infimum} of such a family, denoted by \(\bigwedge_{i\in I}\mu_i\).
\medskip

Given a metric space \((\X,\sfd)\), we denote by \(\mathscr B(\X)\) its \textbf{Borel \(\sigma\)-algebra}, i.e.\ the
\(\sigma\)-algebra generated by the topology of \(\X\). Moreover, we denote by \(\mathcal M_+(\X)\) the set of all
finite (non-negative) Borel measures on \(\X\).  The space of all \textbf{probability measures} on \(\X\), i.e.\ of
all \(\mu\in \mathcal M_+(\X)\) such that \(\mu(\X)=1\), will be denoted by \(\mathcal P(\X)\). A non-negative Borel
measure $\mu$ on $\X$ is said to be \textbf{boundedly-finite} if \(\mu(B)<+\infty\) holds for every bounded set
$B \in \mathscr{B}(\X)$. Notice that each boundedly-finite Borel measure is in particular \(\sigma\)-finite.
In this work, we are mostly concerned with boundedly-finite Borel measures on a complete and separable metric space:

\begin{definition}[Metric measure space]\label{def:mms}
A \textbf{metric measure space} is a triple \((\X,\sfd,\mm)\), where
\[\begin{split}
(\X,\sfd)&\quad\text{ is a complete and separable metric space},\\
\mm\geq 0&\quad\text{ is a boundedly-finite Borel measure on }\X.
\end{split}\]
When \(\mm\) is finite, we refer to $( \X, \sfd, \mm )$ as a \textbf{finite metric measure space}.
\end{definition}

At times, we will rather consider measures defined on the completion of the Borel \(\sigma\)-algebra.
Given a metric measure space \((\X,\sfd,\mm)\), we say that a set \(A\subseteq\X\) is \textbf{\(\mm\)-measurable}
if there exist Borel sets \(E,\,F\in\mathscr B(\X)\) such that \(E\subseteq A\subseteq F\) and \(\mm(F\setminus E)=0\).
The family of \(\mathscr B_\mm(\X)\) of all \(\mm\)-measurable subsets of \(\X\) is a \(\sigma\)-algebra, called
the \textbf{completion} of the Borel \(\sigma\)-algebra \(\mathscr B(\X)\). Recall that all Souslin sets are \(\mm\)-measurable.
Moreover, \(\mm\) can be uniquely extended to a measure \(\bar\mm\colon\mathscr B_\mm(\X)\to[0,\infty]\), which we
call the \textbf{completion} of \(\mm\). 
\subsubsection{Lebesgue spaces}\label{s:Lebesgue_spaces}
Let \((\X,\Sigma)\) be a measurable space. Then we define \(\mathcal L^0_{\rm ext}(\X,\Sigma)\), \(\mathcal L^0(\X,\Sigma)\),
and \(\mathcal L^\infty(\X,\Sigma)\) as
\begin{equation}\label{eq:def_mathcal_L0}\begin{split}
\mathcal L^0_{\rm ext}(\X,\Sigma)&\coloneqq\big\{f\colon\X\to\R\cup\{-\infty,+\infty\}\;\big|\;f\text{ is $\Sigma$-measurable}\big\},\\
\mathcal L^0(\X,\Sigma)&\coloneqq\big\{f\colon\X\to\R\;\big|\;f\text{ is $\Sigma$-measurable}\big\},\\
\mathcal L^\infty(\X,\Sigma)&\coloneqq\big\{f\in\mathcal L^0(\X,\Sigma)\;\big|\;\sup|f|<+\infty\big\},
\end{split}\end{equation}
respectively. Notice that \(\mathcal L^0(\X,\Sigma)\) and \(\mathcal L^\infty(\X,\Sigma)\) are vector spaces
if endowed with the usual pointwise operations. The positive cones of \(\mathcal L^0_{\rm ext}(\X,\Sigma)\),
\(\mathcal L^0(\X,\Sigma)\), and \(\mathcal L^\infty(\X,\Sigma)\) are defined as
\begin{equation}\label{eq:def_mathcal_L0+}\begin{split}
\mathcal L^0_{\rm ext}(\X,\Sigma)^+&\coloneqq\big\{f\in\mathcal L^0_{\rm ext}(\X,\Sigma)\;\big|\;f(x)\geq 0\text{ for every }x\in\X\big\},\\
\mathcal L^0(\X,\Sigma)^+&\coloneqq\big\{f\in\mathcal L^0(\X,\Sigma)\;\big|\;f(x)\geq 0\text{ for every }x\in\X\big\},\\
\mathcal L^\infty(\X,\Sigma)^+&\coloneqq\big\{f\in\mathcal L^\infty(\X,\Sigma)\;\big|\;f(x)\geq 0\text{ for every }x\in\X\big\},
\end{split}\end{equation}
respectively. Moreover, given a \(\sigma\)-finite measure space \((\X,\Sigma,\mu)\) and \(p\in[1,\infty)\), we define
\begin{equation}\label{eq:def_mathcal_Lp}\begin{split}
\mathcal L^p_{\rm ext}(\X,\Sigma,\mu)&\coloneqq\bigg\{f\in\mathcal L^0_{\rm ext}(\X,\Sigma)\;\bigg|\;\int|f|^p\,\d\mu<+\infty\bigg\},\\
\mathcal L^p(\X,\Sigma,\mu)&\coloneqq\mathcal L^p_{\rm ext}(\X,\Sigma,\mu)\cap\mathcal L^0(\X,\Sigma).
\end{split}\end{equation}
The space \(\mathcal L^p(\X,\Sigma,\mu)\) is a vector subspace of \(\mathcal L^0(\X,\Sigma)\). 
\medskip

We say that two functions \(f,g\in\mathcal L^0_{\rm ext}(\X,\Sigma)\) agree \textbf{\(\mu\)-almost everywhere},
or \textbf{\(\mu\)-a.e.}, if
\[
\mu(\{f\neq g\})=0,\quad\text{ where we set }\{f\neq g\}\coloneqq\big\{x\in\X\;\big|\;f(x)\neq g(x)\big\}.
\]
The \(\mu\)-a.e.\ equality induces an equivalence relation on \(\mathcal L^0_{\rm ext}(\X,\Sigma)\):
given any \(f,g\in\mathcal L^0_{\rm ext}(\X,\Sigma)\), we declare that \(f\sim_\mu g\) if and only if
\(f=g\) holds \(\mu\)-a.e.\ on \(\X\). Then we consider the quotient space
\begin{equation}\label{eq:def_L0}
L^0(\X,\Sigma,\mu)\coloneqq\mathcal L^0(\X,\Sigma)/\sim_\mu,
\end{equation}
which inherits a vector space structure. The canonical projection map \(\pi_\mu\colon\mathcal L^0(\X,\Sigma)\to L^0(\X,\Sigma,\mu)\)
is a linear operator. For any \(p\in[1,\infty)\), we define the \textbf{\(p\)-Lebesgue space} \(L^p(\X,\Sigma,\mu)\) as
\begin{equation}\label{eq:def_Lp}
L^p(\X,\Sigma,\mu)\coloneqq\pi_\mu\big(\mathcal L^p(\X,\Sigma,\mu)\big).
\end{equation}
Notice that \(L^p(\X,\Sigma,\mu)\) is a vector subspace of \(L^0(\X,\Sigma,\mu)\). Also, \(L^p(\X,\Sigma,\mu)\) is a Banach space if endowed with
\[
\|f\|_{L^p(\X,\Sigma,\mu)}\coloneqq\bigg(\int|f|^p\,\d\mu\bigg)^{1/p}\quad\text{ for every }f\in L^p(\X,\Sigma,\mu).
\]
Moreover, the \textbf{\(\infty\)-Lebesgue space} \(L^\infty(\mu)\) is defined as
\begin{equation}\label{eq:def_Linfty}
L^\infty(\X,\Sigma,\mu)\coloneqq\pi_\mu\big(\mathcal L^\infty(\X,\Sigma)\big),
\end{equation}
is a vector subspace of \(L^0(\X,\Sigma,\mu)\), and is a Banach space if endowed with the norm
\[
\|f\|_{L^\infty(\X,\Sigma,\mu)}\coloneqq\underset{\X}{\rm ess\,sup\,}|f|\coloneqq
\inf\big\{\lambda\geq 0\;\big|\;|f|\leq\lambda\text{ holds }\mu\text{-a.e.}\big\}
\quad\text{ for every }f\in L^\infty(\X,\Sigma,\mu).
\]
It holds that \(L^p(\X,\Sigma,\mu)\) is separable if \(p\in[1,\infty)\), reflexive if \(p\in(1,\infty)\), and Hilbert if \(p=2\).

The symbols $\X$ and/or $\Sigma$ may be dropped from the $\mathcal L^p$ or $L^p$ notation when their role is clear from the context.
This will be the case when $(\X,\sfd)$ is a metric space: it is understood that $\Sigma=\mathscr B(\X)$, unless otherwise stated
(occasionally the $\mm$-completion $\mathscr B_\mm(\X)$ and $\mm$-measurable functions, i.e.\ the elements of \(\mathcal L^0(\X,\mathscr B_\mm(\X))\), 
will be needed).
For brevity, we keep using the notation $\mathcal L^p(a,b)$, $L^p(a,b)$ when $X=(a,b)$ is an interval of the real line 
and $\mathcal L$ is the Lebesgue measure. 
Still for brevity, we will also use the following shorthand notation:
\begin{equation}\label{eq:mu-a.e._char_fct}
\1_E^\mu\coloneqq\pi_\mu(\1_E)\in L^\infty(X,\Sigma,\mu),\quad\text{ for every }E\in\Sigma.
\end{equation}
At times, for a given set \(E\subseteq\X\) we will use the notation \(\infty\cdot\1_E\)
to denote the function
\[
(\infty\cdot\1_E)(x)\coloneqq\left\{\begin{array}{ll}
+\infty,\\
0,
\end{array}\quad\begin{array}{ll}
\text{ if }x\in E,\\
\text{ if }x\in\X\setminus E.
\end{array}\right.
\]
In particular, when the set $E\in\Sigma$, we have that $\infty \cdot \1_E \in \mathcal L_{\rm ext}^0(\X,\Sigma)^+$.
\begin{lemma}
\label{lem:lip_dense_linfty}
Let \((\X,\sfd,\mm)\) be a metric measure space. Then the following properties hold:
\begin{itemize}
\item[\(\rm i)\)] Given any \(p\in[1,\infty)\), the space \(\LIP_{bs}(\X)\) is strongly dense in \(L^p(\mm)\).
\item[\(\rm ii)\)] Given any \(f\in L^\infty(\mm)\), there exists a sequence \((f_n)_n\subseteq\LIP_{bs}(\X)\)
with \(\sup_n\|f_n\|_{C_b(\X)}<+\infty\) such that \(f_n(x)\to f(x)\) holds for \(\mm\)-a.e.\ \(x\in\X\)
and \(f_n\) weakly\(^*\) converges to \(f\). In particular, the space \(\LIP_{bs}(\X)\) is sequentially weakly\(^*\)
dense in \(L^\infty(\mm)\).
\end{itemize}
\end{lemma}
\begin{proof}
We omit the proof of i), since it is a standard knowledge. Let us verify ii). Fix \(f\in L^\infty(\mm)\). Choose a point
\(\bar x\in\X\). For every \(n\in\N\) we have that \(\1_{B_n(\bar x)}^\mm f\in L^1(\mm)\), thus by i) we can find a function
\(\tilde f_n\in\LIP_{bs}(\X)\) such that \(\|\tilde f_n-\1_{B_n(\bar x)}^\mm f\|_{L^1(\mm)}\leq\frac{1}{n^2}\).
It follows that the function \(f_n\coloneqq(\tilde f_n\wedge\sup_\X f)\vee\min_\X f\in\LIP_{bs}(\X)\) satisfies
\(\|f_n-\1_{B_n(\bar x)}^\mm f\|_{L^1(\mm)}\leq\frac{1}{n^2}\).
Therefore, \(\sup_n\|f_n\|_{C_b(\X)}\leq\|f\|_{L^\infty(\mm)}\) and \(f_n(x)\to f(x)\) for \(\mm\)-a.e.\ \(x\in\X\) because
for any integer $k$ the series $\sum_n \|(f_n- f)\1_{B_k(\bar x)}^\mm\|_{L^1(\mm)}$ is convergent.
The weak\(^*\) convergence \(f_n\rightharpoonup f\) follows thanks to the dominated convergence theorem. Therefore, ii) is proved.
\end{proof}
\subsubsection{Finite signed measures}\label{s:finite_signed_measures}
Let \((\X,\sfd)\) be a complete and separable metric space. Then we denote by \(\mathcal M(\X)\) the vector
space of all \textbf{finite signed Borel measures} on \(\X\).
The space \(\mathcal M(\X)\) is Banach if endowed with the \textbf{total variation norm} \(\|\mu\|_{\rm TV}\coloneqq|\mu|(\X)\),
where the \textbf{total variation measure} \(|\mu|\in\mathcal M_+(\X)\) of a given \(\mu\in\mathcal M(\X)\) is defined as
\begin{equation}\label{eq:def_tv_meas_std}
|\mu|(E)\coloneqq\sup_{(E_n)_n}\sum_{n\in\N}\big|\mu(E_n)\big|,\quad\text{ for every }E\in\mathscr B(\X),
\end{equation}
the supremum being taken over all countable partitions \((E_n)_{n\in\N}\subseteq\mathscr B(\X)\) of the set \(E\).
However, on \(\mathcal M(\X)\) we always consider the weaker (metrisable) topology induced by narrow convergence: given
\((\mu_n)_n\subseteq\mathcal M(\X)\) and \(\mu\in\mathcal M(\X)\), we declare that \(\mu_n\rightharpoonup\mu\) \textbf{in the narrow sense} provided
\begin{equation}\label{eq:def_narrow_conv}
\int f\,\d\mu_n\to\int f\,\d\mu,\quad\text{ for every }f\in C_b(\X).
\end{equation}

\begin{remark}\label{rmk:narrow_weak_star}
{\rm
Note that, in the case in which the space \((\X,\sfd)\) is also locally compact, the narrow convergence  implies
the weak* convergence on 
$\mathcal M(\X)$, since the space \((\mathcal M(\X), \|\cdot\|_{\rm TV})\) is isometric to the dual of the closure of
$C_c(\X)$ in $C(\X)$, endowed with the sup norm. In this case, a direct application of the Uniform Boundedness Principle gives that any sequence of measures \((\mu_n)_n\subseteq \mathcal M(\X)\) narrowly converging to some \(\mu\in \mathcal M(\X)\), is bounded with  respect to the total variation norm, i.e.\ it holds \(\sup_n\|\mu_n\|_{\rm TV}<\infty\).}
\end{remark}

We also recall that each \(\mu\in\mathcal M(\X)\) admits a \textbf{Hahn decomposition} \((P,N)\),
i.e.\ \(P,N\in\mathscr B(\X)\) are disjoint sets with \(P\cup N=\X\) such that
\[\begin{split}
\mu(E)\geq 0,&\quad\text{ for every }E\in\Sigma\text{ with }E\subseteq P,\\
\mu(E)\leq 0,&\quad\text{ for every }E\in\Sigma\text{ with }E\subseteq N.
\end{split}\]
The Hahn decomposition \((P,N)\) is \emph{essentially unique}, in the following sense: if \((\tilde P,\tilde N)\)
is another Hahn decomposition of \(\mu\), then \((P\Delta\tilde P)\cup(N\Delta\tilde N)\) has null \(\mu\)-measure.  The Hahn decomposition naturally induces a \textbf{Jordan decomposition} $\mu=\mu^+-\mu^-$ 
(with $\mu^+$ and $\mu^-$ called positive and negative parts of $\mu$, respectively) by
\begin{equation}\label{eq:def_Jord_dec_std}
\mu^+(E)=\mu(E\cap P),\quad \mu^-(E)=-\mu(E\cap N),\qquad\text{ for every }E\in\mathscr B(\X).
\end{equation}
Notice that $|\mu|=\mu^++\mu^-$.
The Jordan decomposition can also be characterized as the unique one with $\mu^+,\,\mu^-\in\mathcal M_+(\X)$ 
and $\mu^+$ singular with respect to $\mu^-$.

\begin{remark}\label{rmk:tv_Borel}{\rm
We claim that the map
\begin{equation}\label{eq:tv_Borel}
\mathcal M(\X)\ni\mu\mapsto|\mu|\in\mathcal M_+(\X),\quad\text{ is Borel measurable}.
\end{equation}
To this end, fix a sequence of functions \((f_n)_n\subseteq C_b(\X)\) with \(\|f_n\|_{C_b(\X)}\leq 1\) such that
\begin{equation}\label{eq:aux_formula_tv}
\|\mu\|_{\rm TV}=\sup_{n\in\N}\int f_n\,\d\mu,\quad\text{ for every }\mu\in\mathcal M(\X),
\end{equation}
whose existence is well-known, cf. \cite{Bog:07}.
Moreover, given any \(f\in C_b(\X)\) with \(f\geq 0\), it holds \(|f\mu|=f|\mu|\). Therefore,
for any \(\mu\in\mathcal M(\X)\) and \(f\in C_b(\X)\) we have
\begin{equation}\label{eq:aux_formula_tv_2}\begin{split}
\int f\,\d|\mu|&=\int f^+\,\d|\mu|-\int f^-\,\d|\mu|=(f^+|\mu|)(\X)-(f^-|\mu|)(\X)\\
&=\|f^+\mu\|_{\rm TV}-\|f^-\mu\|_{\rm TV}\overset{\eqref{eq:aux_formula_tv}}=\sup_{n\in\N}\int f^+ f_n\,\d\mu-\sup_{n\in\N}\int f^- f_n\,\d\mu.
\end{split}\end{equation}
Given that \(\mathcal M(\X)\ni\mu\mapsto\int f^+ f_n\,\d\mu\in\R\) and \(\mathcal M(\X)\ni\mu\mapsto\int f^- f_n\,\d\mu\in\R\)
are continuous, we deduce from \eqref{eq:aux_formula_tv_2} that \(\mathcal M(\X)\ni\mu\mapsto\int f\,\d|\mu|\) is Borel measurable
for every \(f\in C_b(\X)\). Consequently, the claim \eqref{eq:tv_Borel} is achieved.
\fr}\end{remark}
\subsubsection{Boundedly-finite signed measures}\label{s:boundedly-finite_signed_measures}
Given a set \(\X\), we say that a family \(\mathcal R\subseteq\mathscr P(\X)\) is a \textbf{\(\delta\)-ring}
if it is closed under finite unions, under countable intersections, and under relative complementation
(i.e.\ \(E\setminus F\in\mathcal R\) whenever \(E,F\in\mathcal R\)). On a metric space \((\X,\sfd)\),
we consider the following \(\delta\)-ring:
\[
\mathscr B_b(\X)\coloneqq\big\{B\in\mathscr B(\X)\;\big|\;B\text{ is bounded}\big\}.
\]
\begin{definition}[Boundedly-finite signed measure]\label{def:bdd-fin_signed}
Let \((\X,\sfd)\) be a complete and separable metric space. Then by a \textbf{boundedly-finite signed measure} on \(\X\)
we mean a set-function \(\mu\colon\mathscr B_b(\X)\to\R\) having the following property: given any \(B\in\mathscr B_b(\X)\),
the set-function \(\mu_B\colon\mathscr B(\X)\to\R\) defined as
\[
\mu_B(E)\coloneqq\mu(B\cap E)\quad\text{ for every }E\in\mathscr B(\X)
\]
is a finite signed Borel measure on \(\X\). We denote by \(\mathfrak M(\X)\) the space of all boundedly-finite signed measures on \(\X\).
We also define \(\mathfrak M_+(\X)\coloneqq\{\mu\in\mathfrak M(\X)\,|\,\mu\geq 0\}\).
\end{definition}
\begin{remark}{\rm
A set-function \(\nu\colon\mathcal R\to\R\) on a \(\delta\)-ring \(\mathcal R\) is said to be a measure if
\(\nu(\varnothing)=0\) and
\[
\nu(E)=\sum_{n=1}^\infty\nu(E_n)\quad\text{ whenever }(E_n)_{n\in\N}\subseteq\mathcal R\text{ are pairwise disjoint and }
E\coloneqq\bigcup_{n\in\N}E_n\in\mathcal R.
\]
Definition \ref{def:bdd-fin_signed} ensures that a boundedly-finite signed measure is a measure in this sense.
\fr}\end{remark}
\begin{definition}[Jordan decomposition]\label{def:Jor_decomp}
Let \((\X,\sfd)\) be a complete and separable metric space. Let \(\mu\in\mathfrak M(\X)\) be given. Then we define
the \textbf{positive part} and the \textbf{negative part} of \(\mu\) as
\[
\mu^+\coloneqq\bigvee_{B\in\mathscr B_b(\X)}\mu^+_B,\qquad\mu^-\coloneqq\bigvee_{B\in\mathscr B_b(\X)}\mu^-_B,
\]
respectively. Then \(\mu^+,\mu^-\colon\mathscr B(\X)\to[0,\infty]\) are (possibly infinite) non-negative Borel measures.
\end{definition}

One can readily check that whenever \((B_n)_{n\in\N}\subseteq\mathscr B_b(\X)\) is a partition of \(\X\), we have that
\[
\mu^+=\sum_{n\in\N}\mu^+_{B_n},\qquad\mu^-=\sum_{n\in\N}\mu^-_{B_n}.
\]
Furthermore, we are entitled to write \(\mu=\mu^+-\mu^-\), since it holds that
\[
\mu(B)=\mu^+(B)-\mu^-(B)\quad\text{ for every }B\in\mathscr B_b(\X).
\]
However, note that the quantity \(\mu^+(E)-\mu^-(E)\) might be undefined for an arbitrary (non-bounded) Borel
set \(E\subseteq\X\), since it can happen that both \(\mu^+(E)\) and \(\mu^-(E)\) are equal to \(+\infty\).
\medskip

The \textbf{total variation measure} of a given \(\mu\in\mathfrak M(\X)\) is then defined as
\begin{equation}\label{eq:def_tv_meas}
|\mu|\coloneqq\mu^+ +\mu^-.
\end{equation}
\subsection{Curves in metric spaces}\label{s:curves}
Let us fix a complete and separable metric space \((\X,\sfd)\).
Given $a, \,b \in \R$ with $a < b$, any continuous map $\gamma \colon [a,b] \to \X$ is said to be a \textbf{curve}. 
The family of all curves in \(\X\) is denoted by
\begin{equation}\label{eq:def_space_curves}
\mathscr C(\X)\coloneqq\bigcup_{\substack{a,b\in\R: \\ a<b}}C([a,b];\X).
\end{equation}
Given any curve \(\gamma\in\mathscr C(\X)\), we denote by \(I_\gamma=[a_\gamma,b_\gamma]\subseteq\R\) its compact interval of definition.
We say that \(\gamma\)
is \textbf{constant} provided its image \(\gamma(I_\gamma)\subseteq\X\) is a singleton, \textbf{non-constant} otherwise. 
The image $\gamma( I_\gamma )$ is denoted by ${\rm im}(\gamma)$. By a \textbf{subcurve} of the curve \(\gamma\) we 
mean the restriction of $\gamma$ to a closed subinterval.
\medskip

Given a non-trivial, compact interval \(I\subseteq\R\), we denote by \(\e\colon C(I;\X)\times I\to\X\) the \textbf{evaluation map}, which is defined as
\begin{equation}\label{eq:def_ev_map}
\e(\gamma,t)\coloneqq\gamma_t\quad\text{ for every }\gamma\in C(I;\X)\text{ and }t\in I.
\end{equation}
Moreover, for any \(t\in I\) we denote by \(\e_t\colon C(I;\X)\to\X\) the evaluation map at time \(t\), namely
\begin{equation}\label{eq:def_ev_map_time_t}
\e_t(\gamma)\coloneqq\e(\gamma,t)=\gamma_t,\quad\text{ for every }\gamma\in C(I;\X).
\end{equation}
Notice that $\e$ and $\e_t$ are continuous.

Recall that a \textbf{partition} $P$ of an interval $[a,b]$ is a nondecreasing family of elements 
$\left\{ t_{i} \right\}_{ i = 0 }^{ n }$ satisfying $t_0 = a$ and $t_{n} = b$. 
The \textbf{mesh} of $P$ is the largest diameter among the intervals $[ t_{i-1}, t_{i} ]$.

Given any function $\gamma \colon [a,b] \rightarrow \X$ and partition 
$P = \left\{ t_{i} \right\}_{ i = 0 }^{ n }$ of $[a,b]$, the \textbf{variation} of $\gamma$ relative to $P$ is the number
\begin{equation}\label{eq:variation:partition}
    V( \gamma; P )
    \coloneqq
    \sum_{ i = 1 }^{ n } \sfd( \gamma_{t_{i}}, \gamma_{t_{i-1}} ).
\end{equation}
The \textbf{length} of $\gamma$ is the number
\begin{equation}\label{eq:partition:upperbound}
    \ell( \gamma )
    \coloneqq
    \sup_{P} V( \gamma; P ).
\end{equation}
We say that $\gamma$ is \textbf{rectifiable} if $\gamma$ is continuous and the length in \eqref{eq:partition:upperbound} is finite.
Observe that $\gamma \mapsto V( \gamma; P )$ in \eqref{eq:variation:partition} is continuous on $C( [a,b]; \X )$ for any partition $P$. 
Therefore \eqref{eq:partition:upperbound} implies that $\gamma \mapsto \ell( \gamma )$ is lower semicontinuous on $C( [a,b]; \X )$. 
We denote by $R( [a,b]; \X )$ the collection of \textbf{rectifiable} elements in $C( [a,b]; \X )$. 
When $\X = \mathbb{R}$, we omit $\X$ from the notation.
Notice that
\[
R([a,b];\X)\quad\text{ is a Borel subset of }C([a,b];\X).
\]
Indeed, we have \(R([a,b];\X)=\bigcup_{n\in\N}\{\gamma\in C([a,b];\X)\,|\,\ell(\gamma)\leq n\}\), 
which shows that \(R([a,b];\X)\) is a countable union of closed sets
(thanks to the lower semicontinuity of \(\ell\)). Moreover, we set 
\begin{equation}\label{eq:def_space_rect_curves}
\mathscr R(\X)\coloneqq \bigcup_{\substack{a,b\in\R: \\ a<b}}R([a,b];\X).
\end{equation}
We next define the map \({\sf ms}\colon C([0,1];\X)\times (0,1)\to[0,\infty]\) as
\begin{equation}\label{eq:ms}
{\sf ms}(\gamma,t)\coloneqq\left\{\begin{array}{ll}
\lim\limits_{h\to 0}\frac{\sfd(\gamma_{t+h},\gamma_t)}{|h|},\\
+\infty
\end{array}\quad\begin{array}{ll}
\text{ if such limit exists,}\\
\text{ otherwise.}
\end{array}\right.
\end{equation}
Further, for a given  \(s,t\in[0,1]\) with \(s<t\), we define
\begin{equation}\label{eq:Restr}
{\rm Restr}_s^t(\gamma)_r\coloneqq\gamma_{(1-r)s+rt},\quad\text{ for every }\gamma\in C([0,1];\X)\text{ and }r\in[0,1].
\end{equation}
Notice that the resulting mapping \({\rm Restr}_s^t\colon C([0,1];\X)\to C([0,1];\X)\) is \(1\)-Lipschitz, thus Borel.

\begin{definition}[Absolutely continuous curve]\label{def:AC^q}
Let \((\X,\sfd)\) be a complete, separable metric space. Let \(q\in[1,\infty]\) be a given exponent.
Then we say that a curve \(\gamma\in C([0,1];\X)\) is \textbf{\(q\)-absolutely continuous} if there exists a non-negative
function \(g\in L^q(0,1)\) such that
\begin{equation}\label{eq:def_AC}
\sfd(\gamma_s,\gamma_t)\leq\int_s^t g(r)\,\d r,\quad\text{ for every }s,t\in[0,1]\text{ with }s<t.
\end{equation}
We denote by \({\rm AC}^q([0,1];\X)\) the space of all \(q\)-absolutely continuous curves in \(\X\).
\end{definition}

\begin{lemma}\label{lem:equiv_AC}
Let \((\X,\sfd)\) be a complete, separable metric space. Let \(q\in[1,\infty]\) and \(\gamma\in C([0,1];\X)\) be given.
Then, if \(\gamma\in {\rm AC}^q([0,1];\X)\) the following three conditions hold:
\begin{itemize}
\item[\(\rm i)\)] The \textbf{metric speed} of \(\gamma\) at \(t\), which is given by
\begin{equation}\label{eq:def_metric_speed}
|\dot\gamma_t|\coloneqq\lim_{h\to 0}\frac{\sfd(\gamma_{t+h},\gamma_t)}{|h|},
\end{equation}
exists and is finite for \(\mathcal L_1\)-a.e.\ \(t\in (0,1)\).
\item[\(\rm ii)\)] The resulting \(\mathcal L_1\)-a.e.\ defined function 
\(|\dot\gamma|\colon (0,1)\to[0,\infty)\) belongs to \(L^q(0,1)\).
\item[\(\rm iii)\)] It holds that
\[
\sfd(\gamma_s,\gamma_t)\leq\int_s^t|\dot\gamma_r|\,\d r,\quad\text{ for every }s,t\in[0,1]\text{ with }s<t.
\]
\end{itemize}
Moreover, \(|\dot\gamma|\) is the \(\mathcal L_1\)-a.e.\ minimal function 
\(g\in L^q(0,1)\) verifying \eqref{eq:def_AC}.
\end{lemma}

Notice that if \(\gamma\in {\rm AC}^q([0,1];\X)\) for  some 
\(q\in [1,\infty]\), then for the 
function \(\sf ms\) defined in \eqref{eq:ms} it holds that \({\sf ms}(\gamma,t)=|\dot\gamma_t|\) 
for \(\mathcal L_1\)-a.e.\ \(t\in (0,1)\). 
Also, \(\sf ms\) is a Borel function.

\begin{proposition}\label{prop_possibly remove}
Let \((\X,\sfd)\) be a complete, separable metric space. Let \(q\in[1,\infty]\) be given. Then
\[
{\rm AC}^q([0,1];\X)\;\text{ is a Borel subset of }\;C([0,1];\X).
\]
\end{proposition}
\begin{proof}
Suppose \(q<\infty\). Given any \(\gamma\in C([0,1];\X)\) and \(s,t\in[0,1]\) with \(s<t\), we define
\[
\Phi_q(\gamma;s,t)\coloneqq\int_s^t\lims_{h\to 0}\frac{\sfd(\gamma_{r+h},\gamma_r)^q}{|h|^q}\,\d r.
\]
Moreover, for any \(n\in\N\) we also define
\[
\Phi_{q,n}(\gamma;s,t)\coloneqq\int_s^t\lims_{h\to 0}\frac{\sfd(\gamma_{r+h},\gamma_r)^q}{|h|^q}\wedge n\,\d r.
\]
An application of the monotone convergence theorem ensures that
\begin{equation}\label{eq:AC_Borel_1}
\Phi_{q,n}(\gamma;s,t)\to\Phi_q(\gamma;s,t),
\quad\text{ as }n\to\infty.
\end{equation}
By  Definition~\ref{def:AC^q} and Lemma~\ref{lem:equiv_AC}, the space \({\rm AC}^q([0,1];\X)\) can be equivalently characterised as follows:
\[
{\rm AC}^q([0,1];\X)=\bigg\{\gamma\in C([0,1];\X)\;\bigg|\;
\Phi_q(\gamma;0,1)<+\infty\bigg\}\cap\bigcap_{\substack{s,t\in\mathbb Q:\\0\leq s<t\leq 1}}\mathscr A(\gamma;s,t),
\]
where we define \(\mathscr A(\gamma;s,t)\coloneqq\big\{\gamma\in C([0,1];\X)\,:\,\sfd(\gamma_s,\gamma_t)\leq\Phi_1(\gamma;s,t)\big\}\). 
Therefore, to prove that \({\rm AC}^q([0,1];\X)\) is a Borel set
amounts to showing that \(C([0,1];\X)\ni\gamma\mapsto\Phi_q(\gamma;s,t)
\) 
are Borel measurable functions for every \(s,t\in[0,1]\) with \(s<t\). Bearing \eqref{eq:AC_Borel_1}
in mind, it is sufficient to check that \(C([0,1];\X)\ni\gamma\mapsto\Phi_{n,q}(\gamma;s,t)
\) 
are Borel measurable for every \(n\in\N\). 
For any \(k\in\N\), 
fix an enumeration \((h^k_i)_{i\in\N}\) of \(\big(-\frac{1}{k},\frac{1}{k}\big)\cap(\mathbb Q\setminus\{0\})\).
The dominated convergence theorem yields
\[
\Phi_{n,q}(\gamma;s,t)=\lim_{k\to\infty}\lim_{j\to i}\int_s^t\sup_{\substack{i\in\N: \\ i\leq j}}
\frac{\sfd(\gamma_{r+h^k_i},\gamma_r)^q}{|h^k_i|^q}\wedge n\,\d r.
\]
Since each function \(C([0,1];\X)\ni\gamma\mapsto\int_s^t\sup_{i\leq j}\big(\sfd(\gamma_{r+h^k_i},\gamma_r)^q/|h^k_i|^q\big)\wedge n\,\d r\) 
is continuous,
we conclude that \(C([0,1];\X)\ni\gamma\mapsto\Phi_{n,q}(\gamma;s,t)\) is Borel measurable, so that \({\rm AC}^q([0,1];\X)\) is Borel.

In the case where \(q=\infty\), just observe that \({\rm AC}^\infty([0,1];\X)=\LIP([0,1];\X)\) can be written as
\[
\LIP([0,1];\X)=\bigcup_{n\in\N}\bigcap_{\substack{s,t\in\mathbb Q: \\ 0\leq s<t\leq 1}}
\Big\{\gamma\in C([0,1];\X)\;\Big|\;\sfd(\gamma_s,\gamma_t)\leq n|s-t|\Big\},
\]
whence it follows that \(\LIP([0,1];\X)\) is a Borel subset of \(C([0,1];\X)\).
\end{proof}

\begin{definition}[\(q\)-energy of a curve]\label{def:q-energy_curve}
Let \((\X,\sfd)\) be a complete and separable metric space. 
Given any \(\gamma\in C([0,1];\X)\) we define 
for every \(q\in [1,\infty)\) the \(q\)-\textbf{energy} of \(\gamma\) as
\begin{equation}
E_q(\gamma)\coloneqq \int_0^1|\dot\gamma_t|^q\,\d t\, \text{ for }\,\gamma\in {\rm AC}^q([0,1];\X),\\
\quad \text{ and }\quad 
E_q(\gamma)=+\infty\text{ otherwise,} 
\end{equation}
and for \(q=\infty\) we set
\begin{equation}
E_\infty(\gamma)\coloneqq \mathop{\mathrm{esssup}}_{t\in [0,1]}|\dot\gamma_t|=\Lip(\gamma)\,
\text{ for }\,\gamma\in \LIP([0,1];\X),\quad \text{ and }\quad E_\infty(\gamma)=+\infty\text{ otherwise.}
\end{equation}
\end{definition}

Notice that the energy functionals \(E_q\) are Borel measurable, due to the Borel measurability of the map \( {\sf ms}\) 
and Fubini's theorem in the case \(q\in [1,\infty)\), or to the definition of the essential supremum in the case \(q=\infty\).
Moreover, it is not difficult to see that for \(q\in (1,\infty)\) the energies \(E_q\) are lower semicontinuous 
in the topology of uniform convergence on \(C([0,1];\X)\). In fact, one can show that for every \(\gamma\in {\rm AC}^q([0,1];\X)\) it holds that 
\[
\int_0^1 |\dot\gamma_t|^q\,\d t=\sup\sum_i\frac{\sfd(\gamma(t_i),\gamma(t_{i+1}))^q}{|t_{i+1}-t_i|^{q-1}},\]
the supremum being taken over all finite partitions of the unit interval. Then the claimed lower semicontinuity follows.
In the case \(q=\infty\), 
the lower semicontinuity of the functional \(E_\infty\) follows from the the fact that the 
sublevel sets \(\{\gamma\in \LIP([0,1];\X)|\,E_\infty(\gamma)\leq k\}\), for \(k\in\R\), are closed with respect to the uniform convergence.
\medskip

We recall here an elementary lemma proven in \cite[Lemma 2.1]{Amb:Gig:Sav:13}.
\begin{lemma}
\label{lem:AGS13_Lemma2.1}
Let \(f\colon (0,1)\to \R\) be a Borel function and assume that there exists a non-negative function 
\(g\in L^1(0,1)\) such that
\[
|f(t)-f(s)|\leq \left|\int_s^tg(r)\,\d r\right|\quad \text{ holds for }\mathcal L^2\text{-a.e. } (s,t)\in (0,1)^2.
\]
Then there exists an absolutely continuous 
representative \(\hat f\colon [0,1]\to \R\)
of \(f\) such that 
\[|\hat f'_t|\leq g(t)\quad \text{holds for }\mathcal L_1\text{-a.e.\ }  t\in (0,1).\]
\end{lemma}
\subsubsection{Constant-speed reparametrization}\label{sec:reparametrization}

\begin{definition}[Reparametrization of curves]\label{def:rep}
Let \((\X,\sfd)\) be a complete and separable metric space.
Given \(\gamma\in R([a,b];\X)\), we say that \(\gamma^{\sf rp}\in R([c,d];\X)\)
is a \textbf{reparametrization} of \(\gamma\) if there exists a non-decreasing, continuous 
and surjective function \({\sf R}\colon [a,b]\to [c,d]\) such that 
\[\gamma=\gamma^{\sf rp}\circ {\sf R}.\]
\end{definition}

Consider a curve \(\gamma\in R([a,b];\X)\). We say that it has \textbf{constant speed} 
\(L\geq 0\) if 
\[\ell(\gamma|_{[s,t]})=L\,(t-s),\quad \text{ for all }\,a\leq s<t\leq b.\]
Let us define the function \({\sf S}_\gamma\colon [0,1]\to [a,b]\) by
\begin{equation}\label{eq:def_S_gamma}
{\sf S}_{\gamma}(t)\coloneqq 
\sup\big\{s\in [a,b]\,|\;\ell(\gamma|_{[a,s]})\leq t\,\ell(\gamma)
\big\},\quad \text{ for all }t\in [0,1],
\end{equation}
and the function \({\sf R}_\gamma\colon [a,b]\to [0,1]\) by
\begin{equation}\label{eq:def_R_gamma}
    {\sf R}_\gamma(s)
    =
    \left\{
    \begin{aligned}
        &\frac{\ell(\gamma|_{[a,s]})}{\ell(\gamma)}, \quad &&\text{ for every }s\in [a,b], \text{ if  }\ell(\gamma)>0,
        \\
        &0,
    \quad 
    &&\text{ if  }\ell(\gamma)=0.
    \end{aligned}
    \right.
\end{equation}
Observe that \({\sf R}_\gamma\) is continuous, non-decreasing and surjective and that 
\begin{equation}\label{eq:ScircR}
{\sf R}_\gamma\circ {\sf S}_\gamma={\rm id}_{[0,1]} \text{ whenever }\ell(\gamma)>0.
\end{equation}
\begin{lemma}[Constant-speed reparametrization]\label{lem:csrep}
Let \((\X,\sf d)\) be a complete and separable metric space.
Given \(\gamma\in R([a,b];\X)\), define 
\begin{equation}\label{eq:cs}
\gamma^{\sf cs}\coloneqq \gamma\circ {\sf S}_\gamma\in R([0,1];\X).
\end{equation} 
Then \(\gamma^{\sf cs}\) has constant speed equal to \(\ell(\gamma)\) and 
it holds that 
\(
\gamma=\gamma^{\sf cs}\circ {\sf R}_\gamma.
\)
\end{lemma}
\begin{proof}
In the case \(\ell(\gamma)=0\), we have that
\(\gamma(t)=\gamma(b)\) for all \(t\in [a,b]\) and 
\(\gamma^{\sf cs}(s)=\gamma(b)\) for all \(s\in [0,1]\) and thus \(\gamma=\gamma^{\sf cs}\circ R_\gamma\) trivially holds.
Let us now consider the case \(\ell(\gamma)>0\).
Since \({\sf R}_\gamma\) is continuous, non-decreasing and surjective
there exists a unique  \(\sigma\in R([0,1];\X)\)
such that \(\gamma=\sigma\circ R_\gamma\). 
Indeed, it is easy to check that \(\sigma(s)\coloneqq \gamma(t)\) for every \(s=R_\gamma(t)\in [0,1]\) does the job.

Taking into account
\eqref{eq:ScircR}, it follows from the very definition of \(\gamma^{\sf cs}\) that 
\(\gamma^{\sf cs}=\sigma\). 
To see that \(\gamma^{\sf cs}\) has constant speed \(\ell(\gamma)\), we argue as follows:
fix any \(0\leq s<t\leq 1\) and a partition \(Q=(t_i)_{i=0}^n\) of the interval \([{\sf S}_\gamma(s),{\sf S}_\gamma(t)]\). 
Then \(P\coloneqq (s_i)_{i=0}^n\) with \(s_i\coloneqq {\sf R}_\gamma(t_i)\) is a partition of \([a,b]\) such that 
\(V(\gamma;Q)=V(\gamma^{\sf cs}; P)\). Consequently, we have that 
\(\ell(\gamma^{\sf cs}|_{[s,t]})=\ell(\gamma|_{[{\sf S}_\gamma(s),{\sf S}_\gamma(t)]})\).
Taking into account that (cf. \eqref{eq:ScircR})
\[
\ell(\gamma|_{[{\sf S}_\gamma(s),{\sf S}_\gamma(t)]})
=\ell(\gamma)\big({\sf R}_{\gamma}({\sf S}_\gamma(t))-{\sf R}_{\gamma}({\sf S}_\gamma(s))\big)=\ell(\gamma)(t-s),
\]
the claim follows.
\end{proof}
We refer to \(\gamma^{\sf cs}\) defined in \eqref{eq:cs}
as the \textbf{constant-speed reparametrization} of \(\gamma\).
We shall sometimes use the notation \(\hat\gamma\) instead of \(\gamma^{\sf cs}\).
We shall denote by \(R_{\sf cs}([0,1];\X)\) (resp. \({\rm AC}^q_{\sf cs}([0,1];\X)\) for \(q\in [1,\infty]\)) 
the set of curves in \(R([0,1];\X)\) (resp. \({\rm AC}^q([0,1];\X)\) for \(q\in [1,\infty]\)) having constant speed.
We also define the constant-speed-reparametrization map \({\sf CSRep}\colon R(I;\X)\to R_{cs}([0,1];\X)\) as
\begin{equation}\label{eq:cs_reparam_map}
    {\sf CSRep}(\gamma)\coloneqq \gamma^{\sf cs},\quad \text{ for every }\gamma\in R(I;\X).
\end{equation}
The following measurability result will be needed in the sequel (for a similar result see also \cite[Theorem 2.2.13]{Sav:22}).
\begin{lemma}\label{lemm:continuity:constantspeed}
Let \((\X,\sfd)\) be a complete and separable metric space.
Let \(I\subseteq\R\) be a non-trivial, compact interval. We define the distance \(\hat\sfd\) on \(R(I;\X)\) as
\[
\hat\sfd(\gamma,\sigma)\coloneqq\max\big\{\sfd_{C(I;\X)}(\gamma,\sigma),|\ell(\gamma)-\ell(\sigma)|\big\},\quad\text{ for every }\gamma,\sigma\in R(I;\X).
\]
Then the constant-speed reparametrization map 
\({\sf CSRep}\colon(R(I;\X),\hat\sfd)\to(R_{\sf cs}([0,1];\X),\sfd_{C([0,1];\X)})\) is continuous. In particular,
\({\sf CSRep}\colon(R(I;\X),\sfd_{C(I;\X)})\to(R_{\sf cs}([0,1];\X),\sfd_{C([0,1];\X)})\) is Borel.
\end{lemma}
\begin{proof}
Let \((\gamma^n)_n\subseteq R(I;\X)\) and \(\gamma\in R(I;\X)\) satisfy \(\hat\sfd(\gamma^n,\gamma)\to 0\) as \(n\to\infty\). For brevity, we denote
\(\hat\gamma^n\coloneqq{\sf CSRep}(\gamma^n)\) for every \(n\in\N\). 

Since \(\gamma^n\to\gamma\) uniformly, there exists a
compact set \(K\subseteq\X\) such that \(\gamma^n_t\in K\) for every \(n\in\N\) and \(t\in I\). In particular, \(\hat\gamma^n_t\in K\) for every \(n\in\N\)
and \(t\in[0,1]\). Moreover, since \(\Lip(\hat\gamma^n)\leq\ell(\hat\gamma^n)=\ell(\gamma^n)\) for every \(n\in\N\) and \(\ell(\gamma^n)\to\ell(\gamma)\)
as \(n\to\infty\), we deduce that \((\hat\gamma^n)_n\) is an equiLipschitz family of curves. 
Hence, an application of the Arzel\'{a}--Ascoli theorem ensures
that any subsequence of \((\hat\gamma^n)_n\) admits a further subsequence such that \(\hat\gamma^n\to\sigma\) converging uniformly to some 
limit curve $\sigma \colon [0,1]\to\X$. We will show that $\sigma \equiv \hat\gamma$. This will then imply that the original sequence $\hat\gamma^n$ 
itself converges to $\hat\gamma$ uniformly, by uniqueness of the limit, thereby establishing the desired continuity property.

To ease notation, we relabel the subsequence and assume that $\hat\gamma^n$ converge to $\sigma$ uniformly as $n \rightarrow \infty$.

To this aim, fix a countable dense subset \(D\) of \(I=[a,b]\). Up to passing to a further subsequence and relabeling, 
we may assume that both \(\lim_n\ell(\gamma^n|_{[a,t]})=\limi_n\ell(\gamma^n|_{[a,t]})\)
and \(\lim_n\ell(\gamma^n|_{[t,b]})=\limi_n\ell(\gamma^n|_{[t,b]})\) for every \(t\in D\). 
The lower semicontinuity of \(\ell\) ensures that
\[
\ell(\gamma)=\ell(\gamma|_{[a,t]})+\ell(\gamma|_{[t,b]})
\leq\lim_{n\to\infty}\ell(\gamma^n|_{[a,t]})+\lim_{n\to\infty}\ell(\gamma^n|_{[t,b]})
=\lim_{n\to\infty}\ell(\gamma^n)=\ell(\gamma)
\]
for every \(t\in D\), which yields the identities \(\ell(\gamma|_{[a,t]})
=\lim_n\ell(\gamma^n|_{[a,t]})\) and \(\ell(\gamma|_{[t,b]})=\lim_n\ell(\gamma^n|_{[t,b]})\).
Letting \(h_n(t)\coloneqq\frac{\ell(\gamma^n|_{[a,t]})}{\ell(\gamma^n)}\) 
and \(h(t)\coloneqq\frac{\ell(\gamma|_{[a,t]})}{\ell(\gamma)}\) for all \(n\in\N\) and \(t\in I\),
we have verified that \(h_n(t)\to h(t)\) for every \(t\in D\). Recalling that \(\hat\gamma^n_{h_n(t)}=\gamma^n_t\) 
for every \(t\in I\) and that \(\hat\gamma^n\to\sigma\) uniformly, we deduce
\[
\sigma_{h(t)}=\lim_{n\to\infty}\hat\gamma^n_{h_n(t)}=\lim_{n\to\infty}\gamma^n_t=\gamma_t,\quad\text{ for every }t\in D.
\]
By continuity, we conclude that \(\sigma_{h(t)}=\gamma_t\) for every \(t\in I\) and thus \(\sigma=\hat\gamma\), as required.

Let us now pass to the verification of the last part of the statement. We endow the product space \(R(I;\X)\times\R\) 
with the distance \(\sf D\), defined as
\[
{\sf D}\big((\gamma,\lambda),(\tilde\gamma,\tilde\lambda)\big)\coloneqq\max\big\{\sfd_{C(I;\X)}(\gamma,\tilde\gamma),|\lambda-\tilde\lambda|\big\},
\quad\text{ for every }(\gamma,\lambda),(\tilde\gamma,\tilde\lambda)\in R(I;\X)\times\R.
\]
Notice that the map \(\iota\colon R(I;\X)\to R(I;\X)\times\R\) given by \(\iota(\gamma)\coloneqq(\gamma,\ell(\gamma))\) is 
Borel measurable if its domain \(R(I;\X)\)
is endowed with the distance \(\sfd_{C(I;\X)}\). Moreover, if \(R(I;\X)\) is endowed with the distance \(\hat\sfd\), 
then \(\iota\) (which we denote by \(\hat\iota\)
in this case) is an isometric embedding. We also denote by \(\hat{\sf CSRep}\) the constant-speed reparametrization map when 
its domain is endowed with \(\hat\sfd\).
The first part of the statement says that \(\hat{\sf CSRep}\) is continuous. Therefore, we can finally conclude that the 
map \({\sf CSRep}\colon(R(I;\X),\sfd_{C(I;\X)})\to(R_{\sf cs}([0,1];\X),\sfd_{C([0,1];\X)})\),
which can be written as \(\hat{\sf CSRep}\circ\hat\iota^{-1}\circ\iota\), is Borel measurable. The proof of the statement is achieved.
\end{proof}
\subsubsection{Path integral}\label{sec:integration_along_paths}
We start by making some considerations about continuous real valued curves with finite length. Recalling that our curves are defined on
compact intervals of the real line, let \(\theta\in\mathscr R(\R)\) be a rectifiable curve \(\theta\colon I_\theta\to\R\). 
Given any \(\varepsilon>0\)
and any continuous function \(a\colon I_\theta\to\R\), we define the set \(D_\varepsilon(a,\theta)\subseteq\R\) as
\[
D_\varepsilon(a,\theta)\coloneqq\bigg\{\sum_{i=1}^n a(t_i)\big(\theta(t_i)-\theta(t_{i-1})\big)\;\bigg|
\;P=(t_i)_{i=0}^n\text{ is a partition of }I_\theta\text{ with }|P|\leq\varepsilon\bigg\}.
\]
Notice that each \(D_\varepsilon(a,\theta)\) is a non-empty set and \(D_{\varepsilon'}(a,\theta)\subseteq D_\varepsilon(a,\theta)\) holds
for every \(\varepsilon,\varepsilon'>0\) with \(\varepsilon'<\varepsilon\). Moreover,
denoting by \(\delta_a(\cdot)\) the modulus of continuity of the function \(a\), namely
\[
\delta_a(\varepsilon)\coloneqq\sup\big\{|a(t)-a(s)|\;\big|\;t,s\in I_\theta,\,|t-s|\leq\varepsilon\big\},\quad\text{ for every }\varepsilon>0,
\]
one can easily obtain the following estimate:
\begin{equation}\label{eq:est_diam_D_eps}
\diam\big(D_\varepsilon(a,\theta)\big)\leq
\delta_a(\varepsilon)\ell(\theta),\quad\text{ for every }\varepsilon>0.
\end{equation}
Since \(I_\theta\) is compact and thus \(a\) is uniformly continuous, we deduce that 
\(\delta_a(\varepsilon)\to 0\) as \(\varepsilon\to 0\),
so that accordingly the diameter of \(D_\varepsilon(a,\theta)\) converges to \(0\) as \(\varepsilon\to 0\). 
All in all, we have shown that the
intersection \(\bigcap_{\varepsilon>0}{\rm cl}_\R\big(D_\varepsilon(a,\theta)\big)\) consists of 
a unique element, which we denote by \(\phi_\theta(a)\). The number \(\phi_\theta(a)\in\R\) is called the 
\textbf{Riemann--Stieltjes integral} of \(a\) \textbf{over} \(\theta\) (cf. \cite[2.5.17]{Fed:69}). The resulting function
\(\phi_\theta\colon C(I_\theta)\to\R\) is linear and
\(|\phi_\theta(a)|\leq \ell(\theta)\|a\|_{C(I_\theta)}\) for every \(a\in C(I_\theta)\).
\begin{definition}[Measures induced by rectifiable curves in \(\R\)]\label{def:signedvariation}
Let \(\theta\in\mathscr R(\R)\) be given. Then we denote by \(\mu_\theta\in\mathcal M(I_\theta)\) the unique 
signed Borel measure on \(I_\theta\) such that
\[
\int_{I_\theta}a\,\d\mu_\theta=\phi_\theta(a),\quad\text{ for every }a\in C(I_\theta).
\]
We call $\mu_\theta$ the \textbf{signed variation} of $\theta$, characterized by the property 
\(\mu_\theta([a,b])=\theta(b)-\theta(a)\) whenever $[a,b]\subset I_\theta$.
\end{definition}
When $\theta$ is absolutely continuous, then $\mu_{\theta} = \theta'\mathcal{L}^1$, where $\theta'$ 
is the classical derivative of $\theta$. This follows from the fundamental theorem of calculus.

\begin{lemma}\label{lem:cont_mu_theta}
Consider \(R_L\coloneqq\big\{\theta\in C([0,1])\,\big|\,\ell(\theta)\leq L\big\}\) for every \(L>0\). 
Then the map
\[
R_L\ni\theta\mapsto\mu_\theta\in\mathcal M([0,1])
\]
is continuous, where the domain \(R_L\) is endowed with the supremum distance \(\sfd_{C([0,1])}\) defined in \eqref{eq:def_sup_dist},
while the codomain \(\mathcal M([0,1])\) is endowed with the narrow topology \eqref{eq:def_narrow_conv}.
\end{lemma}
\begin{proof}
Let \((\theta_k)_k\subseteq R_L\) and \(\theta\in R_L\) satisfy \(\lim_k\sfd_{C([0,1])}(\theta_k,\theta)=0\). Fix any \(a\in C([0,1])\).
We want to show that \(\int a\,\d\mu_{\theta_k}\to\int a\,\d\mu_\theta\) as \(k\to\infty\). To this aim, fix any \(\varepsilon>0\) and
pick \(\varepsilon'>0\) such that \(2\varepsilon'+\delta_a(\varepsilon')L<\varepsilon\). Then there exist \(\delta\in(0,\varepsilon')\)
and a partition \(P=(t_i)_{i=0}^n\) of \([0,1]\) such that \(|P|\leq\delta\) and
\(\big|\phi_a(\theta)-\sum_{i=1}^n a(t_i)\big(\theta(t_i)-\theta(t_{i-1})\big)\big|<\varepsilon'\). Now pick \(k_0\in\N\) such that
\[
\sfd_{C([0,1])}(\theta_k,\theta)\leq\frac{\varepsilon'}{2(1+\max_{[0,1]}|a|)n},\quad\text{ for every }k\geq k_0.
\]
Therefore, for any \(k\geq k_0\) we can estimate
\[
\bigg|\sum_{i=1}^n a(t_i)\big(\theta(t_i)-\theta(t_{i-1})\big)-\sum_{i=1}^n a(t_i)\big(\theta_k(t_i)-\theta_k(t_{i-1})\big)\bigg|
\leq 2n\,\sfd_{C([0,1])}(\theta_k,\theta)\max_{[0,1]}|a|\leq\varepsilon',
\]
whence it follows that \(\phi_a(\theta)\) belongs to the \(2\varepsilon'\)-neighbourhood of \(D_{\varepsilon'}(a,\theta_k)\)
for every \(k\geq k_0\). Given that \(\diam\big(D_{\varepsilon'}(a,\theta_k)\big)\leq\delta_a(\varepsilon')L\) holds for every \(k\geq k_0\)
by \eqref{eq:est_diam_D_eps}, we can finally deduce that \(\sup_{k\geq k_0}\big|\phi_a(\theta)-\phi_a(\theta_k)\big|\leq\varepsilon\).
This proves that \(R_L\ni\theta\mapsto\mu_\theta\in\mathcal M([0,1])\) is continuous.
\end{proof}
\begin{corollary}\label{cor:aux_plan_gives_der}
Let \((\X,\sfd,\mm)\) be a metric measure space. Let \(f\in\LIP_{bs}(\X)\) and \(g\in\LIP_b(\X)\) be given. Denote
\(\bar R_L\coloneqq\big\{\gamma\in C([0,1];\X)\,\big|\,\ell(\gamma)\leq L\big\}\subseteq R([0,1];\X)\) for every \(L>0\). Then
\[
\bar R_L\ni\gamma\mapsto\int g\circ\gamma\,\d\mu_{f\circ\gamma}\in\R\quad\text{ is continuous}.
\]
In particular, the function \(R([0,1];\X)\ni\gamma\mapsto\int g\circ\gamma\,\d\mu_{f\circ\gamma}\in\R\) is Borel.
\end{corollary}
\begin{proof}
The map \(C([0,1];\X)\ni\gamma\mapsto f\circ\gamma\in C([0,1])\) is \(\Lip(f)\)-Lipschitz and \(f\circ\gamma\in R_{\Lip(f)L}\) 
for every \(\gamma\in\bar R_L\)
(where \(R_{\Lip(f)L}\) is defined as in Lemma \ref{lem:cont_mu_theta}), thus 
\(\bar R_L\ni\gamma\mapsto\mu_{f\circ\gamma}\in\mathcal M([0,1])\)
is continuous by Lemma \ref{lem:cont_mu_theta}. Now let \(\bar R_L\ni\gamma_n\to\gamma\in\bar R_L\) be fixed. 
Since \(\mu_{f\circ\gamma_n}\rightharpoonup\mu_{f\circ\gamma}\)
in the narrow sense, we have that \(M\coloneqq\sup_n|\mu_{f\circ\gamma_n}|([0,1])<+\infty\) (cf.\ Remark \ref{rmk:narrow_weak_star}). Therefore,
\[
\begin{split}
&\bigg|\int g\circ\gamma_n\,\d\mu_{f\circ\gamma_n}-\int g\circ\gamma\,\d\mu_{f\circ\gamma}\bigg|\\
\leq\,&\int|g\circ\gamma_n-g\circ\gamma|\,\d|\mu_{f\circ\gamma_n}|+\bigg|\int g\circ\gamma\,\d\mu_{f\circ\gamma_n}
-\int g\circ\gamma\,\d\mu_{f\circ\gamma}\bigg|\\
\leq\,&M\Lip(g)\sfd_{C([0,1];\X)}(\gamma_n,\gamma)+\bigg|\int g\circ\gamma\,\d\mu_{f\circ\gamma_n}
-\int g\circ\gamma\,\d\mu_{f\circ\gamma}\bigg|\longrightarrow 0
\end{split}
\]
as \(n\to\infty\), which proves that the map \(\bar R_L\ni\gamma\mapsto\int g\circ\gamma\,\d\mu_{f\circ\gamma}\in\R\) is continuous. 
Finally, given that \(R([0,1];\X)=\bigcup_{k\in\N}\bar R_k\)
and each set \(\bar R_L\) is closed (and thus Borel) by the lower semicontinuity of \(\ell\), 
we conclude that \(R([0,1];\X)\ni\gamma\mapsto\int g\circ\gamma\,\d\mu_{f\circ\gamma}\in\R\) is a Borel function.
\end{proof}
\begin{definition}[Measures induced by rectifiable curves in a metric space and path integrals]\label{eq:lengthmeasure}
Let \((\X,\sf d)\) be a complete and separable metric space.
Let \(\gamma\in \mathscr C(\X)\) be given with $\ell(\gamma)<+\infty$. The \textbf{total variation measure} 
$s_\gamma$ of $\gamma$ is the measure $\mu_\theta$ induced by $\theta(t) \coloneqq \ell( \gamma|_{ [a_\gamma,t] } )$,
$t \in [a_\gamma, b_\gamma]$.
For every \(\rho\in \mathcal L^0_{\rm ext}(\X,\mathscr B(\X))^+\), the \textbf{path integral of} \(\rho\) \textbf{over} \(\gamma\) as
\begin{equation}\label{eq:pathintegral}
    \int_\gamma \rho \,\d s \coloneqq \int_{I_\gamma} (\rho \circ \gamma) \,\d s_\gamma.
\end{equation}
If $\ell( \gamma ) = +\infty$, by convention, \(\int_\gamma \rho \,\d s = +\infty\) 
holds for every \(\rho\in \mathcal L^0_{\rm ext}(\X,\mathscr B(\X))^+\).
\end{definition}

For rectifiable $\gamma$, the same measure $s_\gamma$ is obtained if the usual Carath\' eodory's construction is applied to 
the premeasure $\psi( [s,t] ) \coloneqq \ell( \gamma|_{ [s,t] } )$ for intervals $[s,t] \subseteq [a_\gamma, b_\gamma]$, cf. \cite[Section 2.10]{Fed:69}.
Moreover, if \(\gamma\in\mathscr R(\X)\) and \(f\in\LIP(\X)\) are given, then it holds that \(f\circ\gamma\in\mathscr R(\R)\) and
\begin{equation}\label{eq:chain_rule_s_gamma}
s_{f\circ\gamma}\leq(\lip(f)\circ\gamma)s_\gamma.
\end{equation}

In the next lemma  (see for instance \cite[Proposition 4.4.25]{HKST:15}) 
we collect a useful characterization of the above defined integral in the special case of absolutely continuous curves.

\begin{lemma}[Path integral over an absolutely continuous curve]
\label{lem:integral_along_ac}
Let \((\X,\sfd)\) be a complete and separable metric space and let \(\gamma\in {\rm AC}^q([0,1];\X)\), for some \(q\in [1,\infty]\). 
Then for every Borel function $\rho:X\to [0,\infty]$, it holds that
\[
\int_\gamma \rho\,\d s=\int_0^1\rho(\gamma_t)|\dot\gamma_t|\,\d t.
\]
In particular,
\begin{equation}\label{eq:s_gamma_abs}
s_\gamma=|\dot\gamma_t|\mathcal L_1.
\end{equation}
\end{lemma}
\begin{remark}\label{rmk:prop_real_valued_curves}
{\rm 
We observe the following properties for a rectifiable $\theta \colon [a,b] \rightarrow \mathbb{R}$; we apply them later. 
\begin{itemize}
    \item [1)] As an immediate consequence of the definitions, we have that
    \begin{equation}\label{eq:signed:vs:total}
    |\mu_{ \theta }|\leq s_{\theta}\qquad \forall\, \theta\in\mathscr R(\R).
\end{equation}
In fact, $| \mu_\theta | \equiv s_{\theta}$.
\item [2)] From  the area formula for paths \cite[Theorem 2.10.13]{Fed:69}, we have that $s_{ \theta }( \theta^{-1}(N) ) = 0$, 
whenever $N \subseteq \mathbb{R}$ is $\mathcal{L}^1$-negligible.
\item [3)] If \(\theta\) is also absolutely continuous and \(\theta(t)=0\) for all \(t\in E\), for some Borel set 
\(E\subseteq [a,b]\), then  $s_\theta(E)=0$.
\item [4)] The map \(R([a,b])\ni\theta\mapsto\mu_\theta\in\mathcal M([a,b])\) is linear. Indeed, if \(\theta_1,\theta_2\in R([a,b])\) 
and \(\lambda\in\R\) are given,
then it can be readily checked that 
\(D_\varepsilon(c,\lambda\theta_1+\theta_2)\subseteq\lambda D_\varepsilon(c,\theta_1)+D_\varepsilon(a,\theta_2)\) 
for every \(c\in C([a,b])\)
and \(\varepsilon>0\), so that by taking the intersection over all \(\varepsilon>0\) we obtain that
\[
\int c\,\d\mu_{\lambda\theta_1+\theta_2}
=\phi_{\lambda\theta_1+\theta_2}(c)=\lambda\phi_{\theta_1}(c)+\phi_{\theta_2}(c)=\lambda\int c\,\d\mu_{\theta_1}+\int c\,\d\mu_{\theta_2}.
\]
By the arbitrariness of \(c\in C([a,b])\), it follows that \(\mu_{\lambda\theta_1+\theta_2}=\lambda\mu_{\theta_1}+\mu_{\theta_2}\), as claimed.
\item [5)] If \(\theta_1,\theta_2\in R([a,b])\), then \(\theta_1\theta_2\in R([a,b])\) and
\begin{equation}\label{eq:chain_rule_mu_theta}
\mu_{\theta_1\theta_2}=\theta_1\mu_{\theta_2}+\theta_2\mu_{\theta_1}.
\end{equation}
Indeed, for any \(\delta>0\) we can take \(\varepsilon(\delta)\in(0,\delta)\) so small that 
\(|\theta_1(t)-\theta_1(s)|\leq\delta\) whenever \(t,s\in[a,b]\) and \(|t-s|\leq\varepsilon(\delta)\),
whence it readily follows that for any \(c\in C([a,b])\) it holds that
\[
D_{\varepsilon(\delta)}(c,\theta_1\theta_2)\subseteq D_{\varepsilon(\delta)}(\theta_1 c,\theta_2)
+D_{\varepsilon(\delta)}(\theta_2 c,\theta_1)+[-\lambda(\delta),\lambda(\delta)],
\]
where \(\lambda(\delta)\coloneqq\|c\|_{C([a,b])}\ell(\theta_2)\delta\). Letting \(\delta\to 0\) we get \(\int c\,\d\mu_{\theta_1\theta_2}=
\int\theta_1 c\,\d\mu_{\theta_2}+\int\theta_2 c\,\d\mu_{\theta_1}\), whence \eqref{eq:chain_rule_mu_theta} follows,
given the arbitrariness of \(c\in C([a,b])\).
\end{itemize}
}
\end{remark}

For the convenience of the reader, we recall the following result and its proof. 
The proof can be also found for instance in \cite[Lemma 2.2]{Ja:Ja:Ro:Ro:Sha:07} or \cite[Lemma 2.2.11]{Sav:22}.
\begin{lemma}[Lower semicontinuity of the integral with respect to curves]\label{lem:lsc_integral}
Let \((\X,\sfd)\) be a complete and separable metric space.
Given any lower semicontinuous function \(\rho\colon \X\to [0,\infty]\), 
the function \(\mathscr R(\X)\ni \gamma\mapsto \int_\gamma \rho\, \d s\in [0,\infty]\) is lower semicontinuous. 
\end{lemma}
\begin{proof}
Since \(\rho\) is lower semicontinuous, we find an increasing sequence of continuous bounded functions \((\rho_i)_{i\in \N}\) such that
\(\rho=\sup_{i\in \N}\rho_i\). Consequently, by Monotone Convergence Theorem, we have that
\(\int_\gamma \rho\,\d s=\lim_{i\to\infty}\int_{\gamma}\rho_i\,\d s\),
and therefore it is enough to prove that \(\int_{\gamma}\rho_i\,\d s\) is lower semicontinuous, for every \(i\in \N\). 
So, fix \(i\in \N\). As \(\rho_i\) is continuous, we have that
\[
\int_\gamma \rho_i\, \d s= 
\sup
\left\{
\sum_{j=1}^n\min_{t\in [t_{j-1},t_j]}\rho_i(\gamma_t) \,
\ell(\gamma\vert_{[\gamma_{t_{j-1}},\gamma_{t_j}]})
\big|\, P=(t_j)_{j=0}^n \text{ is a partition of }I_\gamma
\right\},
\]
This yields the claimed lower semicontinuity, given that the right hand side in the 
formula above is the supremum of continuous functions and thus lower semicontinuous.
\end{proof}
\begin{corollary}\label{cor:gamma_measure_Borel}
Let \((\X,\sfd)\) be a complete and separable metric space.
The map $$\mathscr R(\X)\ni\gamma\mapsto\int_\gamma f\,\d s\in \R$$ is Borel whenever
\(f\colon \X\to \R\) is bounded and Borel. Moreover, whenever $f \colon \X \to [0,\infty]$ is Borel,
\begin{equation*}
    \mathscr C( \X ) \ni \gamma \mapsto \int_\gamma f \,\d s \in [0,\infty]
\end{equation*}
is Borel.
\end{corollary}
\begin{proof}
By a standard reduction argument it suffices to prove that $\gamma\mapsto\int_\gamma \chi_B\,\d s$ is 
Borel for any Borel set $B$. The class $\mathcal L$ of Borel sets having this property is stable under complement, under
monotone countable unions and contains all open sets $U$   
(as a consequence of Lemma~\ref{lem:lsc_integral} and the fact that \(\chi_U\) is lower semicontinuous whenever \(U\) is open).
Since $B_1,\,B_2,\,B_1\cup B_2\in\mathcal L$ imply $B_1\cap B_2\in\mathcal L$ 
we can apply the monotone class theorem \cite[Chapter 1, Theorem 21]{Probability}
to conclude that \(\mathcal L=\mathscr B(\X)\).

Whenever $f \colon \X \to [0,\infty]$ is Borel, so is $f_n = \min\left\{ n, f \right\}$ for each $n \in \mathbb{N}$. 
Now if $\gamma \in \mathscr{C}( \X ) \setminus \mathscr{R}( \X )$, then according to our convention one has
\(\int_\gamma f \,\d s=+\infty=\lim_{ n \rightarrow \infty }\int_{ \gamma } f_n \,\d s.\)
Otherwise $\gamma \in \mathscr{R}( \X )$, so monotone convergence yields
\(\int_\gamma f \,\d s=\lim_{ n \rightarrow \infty }\int_{ \gamma } f_n \,\d s.\)
Either case, the map
\(\mathscr C( \X ) \ni \gamma \mapsto \int_\gamma f \,\d s \in [0,\infty]\)
is a pointwise limit of Borel functions by the Borel measurability of $\mathscr{R}( \X ) \subseteq \mathscr{C}( \X )$ and 
the first half of the claim.
\end{proof}
\begin{corollary}\label{cor:lsc_int_rho_pi}
Let \((\X,\sfd)\) be a complete and separable metric space.
    Fix any lower semicontinuous function \(\rho\colon \X\to [0,\infty]\).
Then the function \(\mathcal M_+(C([0,1];\X))\ni \ppi\mapsto \int \int_\gamma \rho\, \d s\,\d \ppi\in [0,\infty]\) is lower semicontinuous with respect to 
the narrow topology on measures. 
\end{corollary}
\begin{proof}
Let us denote for brevity \(F(\gamma)\coloneqq \int_{\gamma}\rho \,\d s\) and consider an increasing sequence of 
bounded continuous functions \((F_i)_{i\in \N}\) such that \(F=\sup_{i\in \N}F_i\). 
Fix also a sequence of measures \((\ppi_n)_n\subseteq \mathcal M_+(C([0,1];\X))\)
converging in the narrow topology to some \(\ppi\in  \mathcal M_+(C([0,1];\X))\). 
By applying 
Lemma \ref{lem:lsc_integral}, we deduce
\[
\limi_n \int F(\gamma)\,\d \ppi_n(\gamma)
\geq \lim_n \int F_i(\gamma)\,\d \ppi_n(\gamma)=
\int F_i(\gamma)\, \d\ppi(\gamma) \quad\text{for every $i \in \mathbb{N}$}.
\]
Passing to the supremum over \(i\in \N\), we deduce the claimed lower semicontinuity of the map 
\(\mathcal M_+(C([0,1];\X))\ni \ppi\mapsto \int \int_\gamma \rho\, \d s\,\d \ppi\in [0,\infty]\).
\end{proof}
\begin{lemma}[Invariance of the path integral under reparameterization] \label{lem:integral_rep_invariant}
Let \((\X,\sfd)\) be a complete and separable metric space.
Let \(\gamma\in R(I_{\gamma};\X)\) be given. Then  \(\gamma_\#s_\gamma=\gamma^{\sf rp}_\#s_{\gamma^{\sf rp}}\), 
where \(\gamma^{\sf rp}\in R(I_{\gamma^{\sf rp}};\X)\) is any  reparametrization of \(\gamma\) (cf.\ Definition \ref{def:rep}).

In particular, for every Borel function $\rho:X\to [0,\infty]$, it holds that
\[\int_\gamma\rho\,\d s=\int_{\gamma^{\sf rp}}\rho\,\d s.\]  
\end{lemma}
\begin{proof}
Let us recall that \(\gamma=\gamma^{\sf rp}\circ {\sf R}\), where  \({\sf R}\colon I_{\gamma}\to I_{\gamma^{\sf rp}}\) 
is non-decreasing, continuous and surjective. 
First, we observe that 
\[s_{\gamma}([a,b])
= \ell\big(\gamma_{\sf rp}\circ {\sf R}|_{[a,b]}\big)
=\ell\big(\gamma|_{[{\sf R}(a),{\sf R}(b)]}\big)
=s_{\gamma^{\sf rp}}([{\sf R}(a),{\sf R}(b)]),\quad \text{ for every }[a,b]\subseteq I_{\gamma}.\]
Since \(\gamma_{\#}s_{\gamma}
=(\gamma^{\sf rp}\circ {\sf R})_{\#}s_{\gamma}=\gamma^{\sf rp}_\#({\sf R}_\#s_{\gamma})\), 
in order to conclude the proof, we need to show that \({\sf R}_\#s_{\gamma}=s_{\gamma^{\sf rp}}\).
Pick any interval \([c,d]\subseteq I_{\gamma^{\sf rp}}\) and call \(c_{\rm min}\coloneqq \min {\sf R}^{-1}(c)\) and 
\(d_{\rm min}\coloneqq \min {\sf R}^{-1}(d)\).
Note that \({\sf R}^{-1}([c,d])=[c_{\rm min},d_{\rm min}]\). Consequently, we have that 
\[
{\sf R}_\#s_{\gamma}([c,d])=s_{\gamma}({\sf R}^{-1}([c,d]))
=s_{\gamma}([c_{\rm min},d_{\rm min}])
=s_{\gamma^{\sf rp}}([{\sf R}(c_{\rm min}),{\sf R}(d_{\rm min})])=s_{\gamma^{\sf rp}}([c,d]),
\]
proving the claim.
\end{proof}
\subsection{Bibliographical notes}
The material in this section is mainly taken from the following sources: 
\begin{itemize}
    \item For the reminder about measure theory, we refer e.g.\ to Bogachev's monograph \cite{Bog:07}.
    \item Many of the results regarding reparametrization maps of paths can be found in \cite{Fed:69,HKST:15,Amb:Mar:Sav:15}, respectively. See also \cite{Dud:07}.
    \item When $q \in (1,\infty)$, most of the measure theory regarding families of curves and path integrals in the  metric measure space setting can be found in  \cite{Amb:Mar:Sav:15} and in extended metric measure space setting in  \cite{Sav:22}.
    \item The argument used to prove Proposition \ref{prop_possibly remove} 
    is new and different from the standard approach (used e.g.\ in \cite{Lisini2007}). The proof we present here is suitable to cover the entire range of exponents \(q\in [1,\infty]\). For another proof in the case where \(q=1\), we refer to \cite[Section 2.2]{Amb:Gig:Sav:14}.
\end{itemize}
\subsection{List of symbols}
\begin{center}
\begin{spacing}{1.2}
\begin{longtable}{p{2.2cm} p{11.8cm}}
\(\N\) & the set of natural numbers\\
\(\mathbb Q\) & the set of rational numbers\\
\(\R\) & the set of real numbers\\
\(\mathcal L^n\) & the \(n\)-dimensional Lebesgue measure on \(\R^n\); Section \ref{ss:gen_term}\\
\(\mathcal L_1\) & restriction of \(\mathcal L^1\) to \([0,1]\); Section \ref{ss:gen_term}\\
\(\mu\otimes\nu\) & product of two measures \(\mu\) and \(\nu\)\\
\(a\wedge b\) & minimum of \(a,b\in\R\)\\
\(a\vee b\) & maximum of \(a,b\in\R\)\\
\(f^+\) & positive part of a function \(f\colon\X\to\R\); \eqref{eq:pos_neg_part}\\
\(f^-\) & negative part of a function \(f\colon\X\to\R\); \eqref{eq:pos_neg_part}\\
\(\mathscr P(\X)\) & the set of all subsets of \(\X\); Section \ref{ss:gen_term}\\
\(\1_E\) & characteristic function of a set \(E\); \eqref{eq:char_fct}\\
\(B_r(x)\) & open ball of radius \(r>0\) and center \(x\in\X\); Section \ref{s:metric_geometry}\\
\(\bar B_r(x)\) & closed ball of radius \(r>0\) and center \(x\in\X\); Section \ref{s:metric_geometry}\\
\({\rm diam}(E)\) & diameter of a set \(E\); Section \ref{s:metric_geometry}\\
\(C(\X;\Y)\) & space of continuous maps from \(\X\) to \(\Y\); Section \ref{s:metric_geometry}\\
\(C_b(\X;\Y)\) & space of bounded continuous maps from \(\X\) to \(\Y\); Section \ref{s:metric_geometry}\\
\(C(X)\) & shorthand notation for \(C(\X;\R)\); Section \ref{s:metric_geometry}\\
\(C_b(\X)\) & shorthand notation for \(C_b(\X;\R)\); Section \ref{s:metric_geometry}\\
\(C_{bs}(\X)\) & space of boundedly-supported continuous functions from \(\X\) to \(\R\); Section \ref{s:metric_geometry}\\
\(C_+(\X)\) & space of non-negative continuous functions from \(\X\) to \(\R\); Section \ref{s:metric_geometry}\\
\(\LIP(\X;\Y)\) & space of Lipschitz maps from \(\X\) to \(\Y\); Section \ref{s:metric_geometry}\\
\(\Lip(f;E)\) & Lipschitz constant of a map \(f\in\LIP(\X;\Y)\) on a set \(E\subseteq\X\); \eqref{eq:Lip_const}\\
\(\Lip(f)\) & shorthand notation for \(\Lip(f;\X)\); Section \ref{s:metric_geometry}\\
\(\LIP(X)\) & shorthand notation for \(\LIP(\X;\R)\); Section \ref{s:metric_geometry}\\
\(\lip(f)\) & slope of a function \(f\in\LIP(\X)\); \eqref{eq:def_slopes}\\
\(\lip_a(f)\) & asymptotic slope of a function \(f\in\LIP(\X)\); \eqref{eq:def_slopes}\\
\(\LIP_b(\X)\) & space of bounded Lipschitz functions from \(\X\) to \(\R\); Section \ref{s:metric_geometry}\\
\(\LIP_{bs}(\X)\) & space of boundedly-supported Lipschitz functions from \(\X\) to \(\R\); Section \ref{s:metric_geometry}\\
\(\LIP_+(\X)\) & space of non-negative Lipschitz functions from \(\X\) to \(\R\); Section \ref{s:metric_geometry}\\
\(\mu|_G\) & restriction of a measure \(\mu\) to a set \(G\); \eqref{eq:def_restr_meas}\\
\(\varphi_\#\mu\) & pushforward of a measure \(\mu\) with respect to a map \(\varphi\); \eqref{eq:def_pushforward_meas}\\
\(\bigvee_{i\in I}\mu_i\) & supremum of a collection of measures \(\{\mu_i\}_{i\in I}\); \eqref{eq:def_sup_meas}\\
\(\bigwedge_{i\in I}\mu_i\) & infimum of a collection of measures \(\{\mu_i\}_{i\in I}\); Section \ref{s:measure_theory}\\
\(\mathscr B(\X)\) & Borel \(\sigma\)-algebra of \(\X\); Section \ref{s:measure_theory}\\
\(\mathcal M_+(\X)\) & space of finite non-negative Borel measures on \(\X\); Section \ref{s:measure_theory}\\
\(\mathcal P(\X)\) & space of probability measures on \(\X\); Section \ref{s:measure_theory}\\
\((\X,\sfd,\mm)\) & a metric measure space; Definition \ref{def:mms}\\
\(\mathscr B_\mm(\X)\) & completion of the Borel \(\sigma\)-algebra with respect to a measure \(\mm\); Section \ref{s:measure_theory}\\
\(\bar\mm\) & completion of the measure \(\mm\); Section \ref{s:measure_theory}\\
\(\mathcal L^0_{\rm ext}(\X,\Sigma)\) & space of measurable functions from \((\X,\Sigma)\) to \(\R\cup\{-\infty,+\infty\}\); \eqref{eq:def_mathcal_L0}\\
\(\mathcal L^0(\X,\Sigma)\) & space of measurable functions from \((\X,\Sigma)\) to \(\R\); \eqref{eq:def_mathcal_L0}\\
\(\mathcal L^\infty(\X,\Sigma)\) & space of bounded measurable functions from \((\X,\Sigma)\) to \(\R\); \eqref{eq:def_mathcal_L0}\\
\(\mathcal L^0_{\rm ext}(\X,\Sigma)^+\) & positive cone of \(\mathcal L^0_{\rm ext}(\X,\Sigma)\); \eqref{eq:def_mathcal_L0+}\\
\(\mathcal L^0(\X,\Sigma)^+\) & positive cone of \(\mathcal L^0(\X,\Sigma)\); \eqref{eq:def_mathcal_L0+}\\
\(\mathcal L^\infty(\X,\Sigma)^+\) & positive cone of \(\mathcal L^\infty(\X,\Sigma)\); \eqref{eq:def_mathcal_L0+}\\
\(\mathcal L^p_{\rm ext}(\X,\Sigma,\mu)\) & space of \(p\)-integrable elements of 
\(\mathcal L^0_{\rm ext}(\X,\Sigma)\) for \(p\in[1,\infty)\); \eqref{eq:def_mathcal_Lp}\\
\(\mathcal L^p(\X,\Sigma,\mu)\) & space of \(p\)-integrable elements of \(\mathcal L^0(\X,\Sigma)\) for \(p\in[1,\infty)\); \eqref{eq:def_mathcal_Lp}\\
\(L^0(\X,\Sigma,\mu)\) & quotient of \(\mathcal L^0(\X,\Sigma)\) up to \(\mu\)-a.e.\ equality; \eqref{eq:def_L0}\\
\(\pi_\mu\) & the canonical projection map \(\pi_\mu\colon\mathcal L^0(\X,\Sigma)\to L^0(\X,\Sigma,\mu)\); Section \ref{s:Lebesgue_spaces}\\
\(L^p(\X,\Sigma,\mu)\) & \(p\)-Lebesgue space over \((\X,\Sigma,\mu)\) for \(p\in[1,\infty]\); \eqref{eq:def_Lp}, \eqref{eq:def_Linfty}\\
\(\1_E^\mu\) & equivalence class of \(\1_E\) up to \(\mu\)-a.e.\ equality; \eqref{eq:mu-a.e._char_fct}\\
\(\mathcal M(\X)\) & space of finite signed Borel measures on \(\X\); Section \ref{s:finite_signed_measures}\\
\(|\mu|\) & total variation measure of a measure \(\mu\); \eqref{eq:def_tv_meas_std}, \eqref{eq:def_tv_meas}\\
\(\|\mu\|_{\rm TV}\) & total variation norm of \(\mu\in\mathcal M(\X)\); Section \ref{s:finite_signed_measures}\\
\(\mu_n\rightharpoonup\mu\) & narrow convergence of \((\mu_n)_n\subseteq\mathcal M(\X)\) to \(\mu\in\mathcal M(\X)\); \eqref{eq:def_narrow_conv}\\
\((\mu^+,\mu^-)\) & Jordan decomposition of \(\mu\); \eqref{eq:def_Jord_dec_std}, Definition \ref{def:Jor_decomp}\\
\(\mathscr B_b(\X)\) & the \(\delta\)-ring of bounded Borel subsets of \(\X\); Section \ref{s:boundedly-finite_signed_measures}\\
\(\mathfrak M(\X)\) & space of boundedly-finite signed measures on \(\X\); Definition \ref{def:bdd-fin_signed}\\
\(\mathfrak M_+(\X)\) & positive cone of \(\mathfrak M(\X)\); Definition \ref{def:bdd-fin_signed}\\
\(\mathscr C(\X)\) & space of continuous curves on \(\X\) parametrized on a compact interval; \eqref{eq:def_space_curves}\\
\(I_\gamma=[a_\gamma,b_\gamma]\) & domain of definition of a curve \(\gamma\in\mathscr C(\X)\); Section \ref{s:curves}\\
\(\e\) & the evaluation map \(\e\colon C(I;\X)\times I\to\X\); \eqref{eq:def_ev_map}\\
\(\e_t\) & the evaluation map \(\e_t\colon C(I;\X)\to\X\) at time \(t\in I\); \eqref{eq:def_ev_map_time_t}\\
\(V(\gamma;P)\) & variation of a curve \(\gamma\colon[a,b]\to\X\) relative to a partition \(V\) of \([a,b]\); \eqref{eq:variation:partition}\\
\(\ell(\gamma)\) & length of a curve \(\gamma\); \eqref{eq:partition:upperbound}\\
\(R([a,b];\X)\) & space of all rectifiable curves \(\gamma\colon[a,b]\to\X\); Section \ref{s:curves}\\
\(\mathscr R(\X)\) & space of all rectifiable curves in \(\X\); \eqref{eq:def_space_rect_curves}\\
\(\sf ms\) & the metric speed functional \({\sf ms}\colon C([0,1];\X)\times[0,1]\to[0,+\infty]\); \eqref{eq:ms}\\
\({\rm Restr}_s^t\) & the map that restricts \(\gamma\) to \([s,t]\) and stretches it to \([0,1]\); \eqref{eq:Restr}\\
\({\rm AC}^q([0,1];\X)\) & space of \(q\)-absolutely continuous curves in \(\X\); Definition \ref{def:AC^q}\\
\(|\dot\gamma|\) & metric speed of a curve \(\gamma\in{\rm AC}^q([0,1];\X)\); \eqref{eq:def_metric_speed}\\
\(E_q(\gamma)\) & \(q\)-energy of a curve \(\gamma\colon[0,1]\to\X\); Definition \ref{def:q-energy_curve}\\
\({\sf S}_\gamma\) & constant-speed reparametrizing function of \(\gamma\); \eqref{eq:def_S_gamma}\\
\({\sf R}_\gamma\) & left inverse of \({\sf S}_\gamma\); \eqref{eq:def_R_gamma}\\
\(\gamma^{\sf cs}\), \(\hat\gamma\) & constant-speed reparametrization of \(\gamma\); \eqref{eq:cs}\\
\(R_{\sf cs}([0,1];\X)\) & curves in \(R([0,1];\X)\) having constant speed; Section \ref{sec:reparametrization}\\
\({\rm AC}^q_{\sf cs}([0,1];\X)\) & curves in \({\rm AC}^q([0,1];\X)\) having constant speed; Section \ref{sec:reparametrization}\\
\(\sf CSRep\) & the constant-speed-reparametrization map; \eqref{eq:cs_reparam_map}\\
\(\mu_\theta\) & signed variation of a curve \(\theta\in\mathscr R(\R)\); Definition \ref{def:signedvariation}\\
\(s_\gamma\) & total variation measure of a curve \(\gamma\in\mathscr R(\X)\); Definition \ref{eq:lengthmeasure}\\
\(\int_\gamma\rho\,\d s\) & path integral of \(\rho\) over \(\gamma\); \eqref{eq:pathintegral}
\end{longtable}
\end{spacing}
\end{center}
\section{Modulus and plans}\label{sec:mod_plan}
This section is devoted to various notions of measures on the space $\mathscr C(\X)$ of continuous curves on \(\X\). In the first subsection,
we recall the \(p\)-modulus \(\Modp\) and its properties, used in the definition of the Newtonian Sobolev space \(N^{1,p}(\X)\),
and its modified version \(\Modp^1\). In the second subsection, a notion of plan with barycenter in \(L^q(\mm)\) is presented.
A particular example of a plan with barycenter, known as a test plan, is used in the definition of the Sobolev space \(B^{1,p}(\X)\).
The last subsection deal with the relation between the two and provides all the technical results needed in order to show that
\(B^{1,p}(\X)\subseteq N^{1,p}(\X)\) in Section \ref{sec:Equivalence}.
\subsection{Modulus of a family of curves}
\begin{definition}\label{def:Adm_mod}
Let \((\X,\sfd,\mm)\) be a metric measure space and
let \(\lambda\in \{0,1\}\). Given a family \( \Gamma \subseteq \mathscr C(\X) \), a Borel function 
\(\rho:\X\to [0,\infty]\) is said to be \(\lambda\)-\textbf{admissible} for \(\Gamma\) if  
\[
\lambda\big(\rho(\gamma_{a_\gamma})+\rho(\gamma_{b_\gamma})\big)+\int_{\gamma} \rho \,\d s \geq 1\quad \text{ for every } \gamma \in \Gamma.
\]
We denote by \({\rm Adm}_{\lambda}(\Gamma)\) the set of all functions $\rho$ that are \(\lambda\)-admissible for \(\Gamma\). 
When $\lambda = 0$, we use the standard convention \(\lambda\cdot \infty=0\) above.
\end{definition}
\begin{definition}
Let \((\X,\sfd,\mm)\) be a metric measure space, \(p\in [1,\infty)\) and \(\lambda\in \{0,1\}\). Given a family
\(\Gamma\subseteq\mathscr C(\X)\), the \textbf{\((p,\lambda)\)-modulus} $\Modp^{\lambda}(\Gamma)$ of \(\Gamma\)
is defined as
\begin{equation}\label{eq:Mod}
\Modp^{\lambda}(\Gamma)\coloneqq\inf\bigg\{\int\rho^p\,\d\mm\;\bigg|\;
\rho\in{\rm Adm}_{\lambda}(\Gamma)
\bigg\}\in [0,\infty]
\end{equation}  
 with the convention $\inf\varnothing=\infty$.
\end{definition}
\begin{remark}\label{rem:basicproperties}
{\rm 
Notice that the choice of \(\lambda=0\) corresponds to the standard definition of \(p\)-modulus,
while the choice of \(\lambda=1\) corresponds to  the new modulus \(\widetilde{\rm Mod}_p\) introduced in \cite{Sav:22}.
Both definitions can be formulated in terms of modulus on the space of measures,  along the
lines of Fuglede's original paper \cite{Fug:57}. Indeed, as noticed in \cite{Sav:22}, 
each curve can generate a measure via the maps \(\gamma\mapsto \lambda ((\e_0)_\#\gamma+(\e_1)_{\#}\gamma)+\gamma_\#(|\dot\gamma|\mathcal L^1)\), 
for \(\lambda\in \{0,1\}\).

In the case \(\lambda=0\), we shall use the standard notation \(\Modp\) in place of \(\Modp^0\). 
Given \(\Gamma\subseteq \mathscr C(\X)\), since \({\rm Adm}_0(\Gamma)\subseteq {\rm Adm}_1(\Gamma)\) one
has
\[
\Modp^1(\Gamma)\leq \Modp(\Gamma).
\]
}
\end{remark}

We do not impose any integrability conditions for $\rho$ in the definition of ${\rm Adm}_{\lambda}( \Gamma )$. For this reason, it is natural to consider the class of Borel and $p$-integrable extended functions $\mathcal{L}^{p}_{ \rm ext }( \mm )^{+}$ when computing $\Mod_{p}^{\lambda}( \Gamma )$. This plays a role, especially in the study of Newtonian Sobolev space in \Cref{sec:N_def}.

In the following lemma we list some standard properties of \((p,\lambda)\)-modulus - these statements essentially follow from the 
work by Fuglede \cite{Fug:57}; cf.\ \cite{Amb:Mar:Sav:15}.  
See also the $\lambda = 0$ case from \cite{HKST:01}.
\begin{lemma}[Some properties of \(\Modp^\lambda\)]\label{lem:properties_Mod}
Let \((\X,\sfd,\mm)\) be a metric measure space, \(p\in [1,\infty)\) and \(\lambda\in \{0,1\}\). Then the following statements hold:
   \begin{itemize}
       \item [\rm (1)] \(\Modp^\lambda\) vanishes on any family of curves which are not rectifiable.
       \item [\rm (2)] \(\Modp^\lambda(\varnothing)=0\) and $\Modp(\Gamma)=\infty$ whenever $\Gamma$ contains a constant curve.
         \item [\rm (3)] Given any \(\Gamma_1\subseteq \Gamma_2\subseteq \mathscr C(\X)\) it holds that 
         \(\Modp^\lambda(\Gamma_1)\leq\Modp^\lambda(\Gamma_2)\).
       \item [\rm (4)] Given any  family \((\Gamma_i)_{i\in \N}\subseteq \mathscr C(\X)\) it holds that 
       \(\Modp^\lambda(\cup_{i\in \N}\Gamma_i)\leq \sum_{i\in \N}\Modp^\lambda(\Gamma_i)\).
       \item [\rm (5)] Given any \(h\in \mathcal L^p_{\rm ext}(\mm)^+\) it holds that 
       \[\Modp^\lambda\left(\Big\{\gamma\in \mathscr C(\X)\big |
       \,\lambda\big(h(\gamma_{a_\gamma})+h(\gamma_{b_\gamma})\big) +\int_\gamma h\,\d s=+\infty\Big\}\right)=0.\]
       \item [\rm (6)] Given any 
       \(\Gamma\subseteq \mathscr C(\X)\)  such that 
       \(\Modp^\lambda(\Gamma)=0\) there exists \(h\in \mathcal L^p_{\rm ext}(\mm)^+\) such that 
       \[\Gamma\subseteq 
       \Big\{\gamma\in \mathscr C(\X)\big| \, \lambda\big(h(\gamma_{a_\gamma})+h(\gamma_{b_\gamma})\big)
       +\int_\gamma h\,\d s=+\infty\Big\}.\]
       \item [\rm (7)] If $\Gamma,\,\Gamma' \subseteq \mathscr{C}( \X )$ and for every $\gamma \in \Gamma$ there exists
       a subinterval \(I\subseteq I_\gamma\) such that $\gamma|_{I} \in \Gamma'$ and $\gamma_I$ has the same end points as $\gamma$, 
       then \(\Modp^\lambda( \Gamma ) \leq \Modp^\lambda( \Gamma' )\). If \(\lambda=0\), the same inequality holds under the condition that \(\gamma|_{I} \in \Gamma'\).     
   \end{itemize}
\end{lemma}
In particular, the properties (2), (3) and (4) imply that for every \(\lambda \in \{0,1\}\) the function 
\(\mathscr C(\X)\ni\Gamma\mapsto\Modp^\lambda(\Gamma)\in [0,\infty]\) is an outer measure on \(\mathscr C(\X)\). We also use the following lemma later.

\begin{lemma}\label{lem:rough_est_Modp}
Let \((\X,\sfd,\mm)\) be a metric measure space and \(p\in [1,\infty)\). For \(E\subseteq\X\) Borel, we define
\[
\Gamma^{a,b}_E\coloneqq\big\{\gamma\in\mathscr C(\X)\;\big|\;\gamma_{a_\gamma},\gamma_{b_\gamma}\in E\big\}.
\]
Then it holds that
\[
\Modp^1(\Gamma^{a,b}_E)\leq\frac{\mm(E)^p}{2^p}.
\]
\end{lemma}
\begin{proof}
Define \(\rho\coloneqq\frac{1}{2}\1_E\) and notice that 
\(\rho(\gamma_{a_\gamma})+\rho(\gamma_{b_\gamma})+\int_\gamma\rho\,\d s\geq
\frac{1}{2}\1_E(\gamma_{a_\gamma})+\frac{1}{2}\1_E(\gamma_{b_\gamma})=1\) for every \(\gamma\in\Gamma^{a,b}_E\). 
This shows that the function \(\rho\)
is 1-admissible for \(\Gamma^{a,b}_E\), thus accordingly we can 
conclude that \(\Modp^1(\Gamma^{a,b}_E)\leq\int\rho^p\,\d\mm=\mm(E)^p/2^p\), as desired.
\end{proof}
\begin{theorem}[Fuglede's lemma]\label{thm:Fuglede}
Let \((\X,\sfd,\mm)\) be a metric measure space and \(p\in [1,\infty)\). 
Let \((f_n)_n\subseteq \mathcal L^p(\mm)\) be a sequence of non-negative functions 
converging in the \(L^p(\mm)\)-seminorm to \(0\). Then there exist a subsequence \((f_{n_k})_k\subseteq (f_n)_n\) and a family 
\(\Gamma_N\subseteq \mathscr C(\X)\) with \(\Modp(\Gamma_N)=0\) such that 
\[
\lim_{k\to +\infty}\int_\gamma f_{n_k}\,\d s= 0
\qquad\text{for every $\gamma\in \mathscr C(\X)\setminus \Gamma_N$.} 
\]
\end{theorem}
\begin{proof}
See, for instance, \cite[page 131]{HKST:15}.
\end{proof}

\begin{lemma}[Invariance of \(\Modp^\lambda\) under reparametrization]\label{lem:Mod_param_invariant}
Let \((\X,\sfd,\mm)\) be a metric measure space, \(p\in [1,\infty)\) and $\lambda\in\{0,1\}$. Then
$$
\Modp^\lambda(\Gamma)= \Modp^\lambda( {\sf CSRep}( \Gamma \cap \mathscr R( \X ) ) )
\quad\text{for any curve family $\Gamma\subseteq \mathscr C(\X)$.}
$$
\end{lemma}
\begin{proof}
    Follows by the very definition and taking into account Lemma \ref{lem:integral_rep_invariant}. 
\end{proof}
\begin{definition}[\(p\)-exceptional sets]\label{def:exceptional}
Let \((\X, \sfd,\mm)\) be a metric measure space and let \(p\in [1,\infty)\). A set \(E\subseteq \X\) is said to be 
\(p\)-\textbf{exceptional} if \(\Modp(\Gamma_E)=0\), where \(\Gamma_E\) denotes the family of non-constant curves in \(\X\)
intersecting \(E\).
\end{definition}
\subsection{Plans with barycenter}
\begin{definition}[Plan]
Let \((\X,\sfd,\mm)\) be a metric measure space. 
A \textbf{plan} on \(\X\) is any non-negative Borel measure \(\ppi\in \mathcal M_+(C([0,1];\X))\) that 
is concentrated on \(R([0,1];\X)\).
\end{definition}
Let \(\lambda\in \{0,1\}\) and a  plan \(\ppi\) on \(\X\) be fixed. 
Then, taking into account Corollary \ref{cor:gamma_measure_Borel}, 
we can associate to \(\ppi\) a non-negative measure 
\(\Pi_\sppi^\lambda\in \mathcal M_+(\X)\) defined as 
\begin{equation}\label{eq:Pi_lambda}
\Pi_\sppi^\lambda\coloneqq\lambda(\e_0)_\#\ppi+\lambda(\e_1)_\#\ppi+\int\gamma_\# s_\gamma\,\d\ppi(\gamma).
\end{equation}
Notice that
the measure \(\Pi_\sppi^\lambda\) can be alternatively characterised as the unique measure satisfying
\[
\int \rho\,\d \Pi_{\sppi}^\lambda=\int\bigg(\lambda\,\rho(\gamma_0)+\lambda\,\rho(\gamma_1)+\int_\gamma\rho\,\d s\bigg)\,\d \ppi(\gamma)
\]
for every bounded Borel function \(\rho\colon \X\to\R\), or equivalently for every \(\rho\colon\X\to[0,\infty]\) Borel.
\begin{definition}[Barycenter]\label{def_plan_barycenter}
Let \((\X,\sfd,\mm)\) be a metric measure space, \(q\in (1,\infty]\) and
\(\lambda\in\{0,1\}\). A plan  \(\ppi\) on \(\X\)
admits \(\lambda\)-\textbf{barycenter} ${\rm Bar}_\lambda(\ppi)$ in \(L^q(\mm)\) if 
\begin{equation}\label{eq:def_barycenter}
\Pi_\sppi^\lambda\ll \mm 
\quad\text{ and }\quad
{\rm Bar}_\lambda(\ppi)\coloneqq \frac{\d\Pi_\sppi^\lambda}{\d\mm}\in L^q(\mm).
\end{equation}
We refer to \({\rm Bar}_\lambda(\ppi)\) as the \(\lambda\)-barycenter of \(\ppi\). We also shorten \(0\)-barycenter to \textbf{barycenter}
and we write \({\rm Bar}(\ppi)\) instead of \({\rm Bar}_0(\ppi)\).
\end{definition}

\begin{lemma}\label{lem:barycenter}
Let \((\X,\sfd, \mm)\) be a metric measure space, \(p\in [1,\infty)\) and
let \(\lambda\in\{0,1\}\). 
A plan \(\ppi\) on \(\X\) admits \(\lambda\)-barycenter in \(L^q(\mm)\) if and only if there exists \(C\geq 0\) such that 
\begin{equation}\label{eq:equiv_bar}
\int\rho\,\d\Pi^\lambda_\sppi\leq C\|\rho\|_{L^p(\mm)}\quad \text{ for every non-negative }\rho\in \LIP_{bs}(\X).
\end{equation}
In this case, \(\|{\rm Bar}_\lambda(\ppi)\|_{L^q(\mm)}\) is the minimal constant \(C\geq 0\) satisfying \eqref{eq:equiv_bar}.
\end{lemma}
\begin{proof}
If \(\ppi\) admits \(\lambda\)-barycenter in \(L^q(\mm)\), then an application of H\"{o}lder's inequality gives
\[
\bigg|\int\rho\,\d\Pi^\lambda_\sppi\bigg|\leq\int|\rho|{\rm Bar}_\lambda(\ppi)\,\d\mm\leq\|{\rm Bar}_\lambda(\ppi)\|_{L^q(\mm)}\|\rho\|_{L^p(\mm)}
\quad\text{ for every }\rho\in\LIP_{bs}(\X),
\]
which proves that \eqref{eq:equiv_bar} holds with \(C\coloneqq\|{\rm Bar}_\lambda(\ppi)\|_{L^q(\mm)}\). 
Conversely, suppose \eqref{eq:equiv_bar} holds.
Then the map \(T\colon\pi_\mm\big(\LIP_{bs}(\X)\big)\to\R\) 
given by \(T\big(\pi_\mm(\rho)\big)\coloneqq\int\rho\,\d\Pi^\lambda_\sppi\) for every
\(\rho\in\LIP_{bs}(\X)\) is well-defined, linear, and continuous 
by applying \eqref{eq:equiv_bar} to the positive and negative parts of $\rho$ separately. 
Given that \(\pi_\mm\big(\LIP_{bs}(\X)\big)\) is a dense linear subspace of \(L^p(\mm)\), 
the map \(T\) can be uniquely extended to a linear, continuous map \(\bar T\colon L^p(\mm)\to\R\) with $\|\bar T\|\leq C$.

Recalling that the dual Banach space of \(L^p(\mm)\) is \(L^q(\mm)\), we deduce that there exists a unique function \(g\in L^q(\mm)\) 
with $\|g\|_{L^q(\mm)}\leq C$
such that \(\int\rho g\,\d\mm=\bar T(\rho)=\int\rho\,\d\Pi^\lambda_\sppi\) for every \(\rho\in\LIP_{bs}(\X)\). 
By the arbitrariness of \(\rho\), it follows that \(\Pi^\lambda_\sppi\ll\mm\) and that
\({\rm Bar}_\lambda({\ppi})=\frac{\d\Pi_\sppi^\lambda}{\d\mm}=g\in L^q(\mm)\), which also proves the minimality property
of \(\|{\rm Bar}_\lambda(\ppi)\|_{L^q(\mm)}\).
\end{proof}
\begin{remark}\label{rem:plan_properties}
{\rm In the setting of Definition~\ref{def_plan_barycenter} and for \(\ppi\) being a plan on 
\(\X\) with \(\lambda\)-barycenter in \(L^q(\mm)\), we collect here some properties of plans with barycenter.
\begin{itemize}
    \item [1)] Consider any \(\ppi\)-measurable set \(\Gamma\subseteq C([0,1];\X)\) with \({\rm Adm}_\lambda(\Gamma)\neq \varnothing\).
    Pick any \(\rho\in {\rm Adm}_\lambda(\Gamma)\) and notice that 
    \[\begin{split}
    \ppi(\Gamma)&\leq\int\bigg(\lambda\,\rho(\gamma_0)+\lambda\,\rho(\gamma_1)+\int_\gamma\rho\,\d s\bigg)\,\d \ppi(\gamma) 
    \leq \int|\rho|{\rm Bar}_\lambda(\ppi)\,\d\mm\\
    &\leq\|{\rm Bar}_\lambda(\ppi)\|_{L^q(\mm)}\|\rho\|_{L^p(\mm)}.
    \end{split}\]
    Passing to the infimum among all \(\rho\in {\rm Adm}_\lambda(\Gamma)\), we get the estimate
    \begin{equation}\label{eq:pi_less_modulus}
        \ppi(\Gamma)\leq \|{\rm Bar}_\lambda(\ppi)\|_{L^q(\mm)}\big(\Modp^\lambda(\Gamma)\big)^{1/p}.
    \end{equation}
    \item [2)] It follows from 1) and from Lemma \ref{lem:properties_Mod}(5) that
    \[\ppi\left(\Big\{\gamma\in C([0,1];\X)\big | \,\lambda\big(h(\gamma_{a_\gamma})+h(\gamma_{b_\gamma})\big) 
    +\int_\gamma h\,\d s=+\infty\Big\}\right)=0\]
    for any \(h\in \mathcal L^p_{\rm ext}(\mm)^+\).
\end{itemize}
}
\end{remark}
Next, we introduce the notion of a test plan (first considered in \cite{Amb:Gig:Sav:13}).
\begin{definition}[q-energy of a plan]
Let \((\X,\sfd,\mm)\) be a metric measure space and let \(q\in (1,\infty]\).
We say that a plan \(\ppi\) on \(\X\) 
has \textbf{finite} \(q\)-\textbf{energy} if it is concentrated on \({\rm AC}^q(C[0,1];\X))\) and 
\begin{equation}\label{eq:def_energy_plan}
{\sf E}_q(\ppi)\coloneqq \int E_q(\gamma)\,\d \ppi(\gamma)=\int\!\!\!\int_0^1 |\dot\gamma_t|^q\,\d t\,\d \ppi(\gamma)<+\infty.
\end{equation}
\end{definition}

\begin{definition}[\(q\)-test plans]\label{def:test_plan}
Let \((\X,\sfd,\mm)\) be a metric measure space and \(q\in (1,\infty]\).
A plan $\ppi$ on $\X$ is said to be a \(q\)-\textbf{test plan}
if it satisfies the following two conditions:
\begin{itemize}
    \item [\rm (TP1)] there exists \(C\geq 0\) such that \((\e_t)_\#\ppi\leq C\mm\) holds for every \(t\in [0,1]\);
    \item [\rm (TP2)] \(\ppi\) is a probability measure and it has finite \(q\)-energy.
\end{itemize}
The \textbf{compression constant} \({\rm Comp}(\ppi)\) of \(\ppi\) is defined as the minimal non-negative \(C\) satisfying \(\rm (TP1)\).
Since $\ppi$ is a probability, the compression constant is strictly positive.
\end{definition}
\begin{lemma}\label{lem:ppi_bdd_compr}
Let \((\X,\sfd,\mm)\) be a metric measure space. 
Let \(\ppi\) be a \(q\)-test plan on \(\X\), for some exponent \(q\in(1,\infty]\). Then 
\begin{equation}\label{eq:ppi_bdd_compr}
\e_\#(\ppi\otimes\mathcal L_1)\leq{\rm Comp}(\ppi)\mm.
\end{equation}
Moreover, recalling the definition \eqref{eq:ms}, the plan \(\ppi\) admits barycenter in \(L^q(\mm)\) and 
\[
{\rm Bar}(\ppi)=\frac{\d\e_\#({\sf ms}(\ppi\otimes\mathcal L_1))}{\d\mm}
\quad\text{with}\quad\|{\rm Bar}(\ppi)\|_{L^q(\mm)}\leq{\rm Comp}(\ppi){\rm E}_q(\ppi)^{1/q}.
\]
\end{lemma}
\begin{proof}
Given any non-negative Borel function \(\rho\colon\X\to[0,\infty)\), we can estimate
\[
\int\rho\,\d\e_\#(\ppi\otimes\mathcal L_1)=\int\!\!\!\int_0^1\rho(\gamma_t)\,\d t\,\d\ppi(\gamma)
=\int_0^1\!\!\!\int\rho\,\d(\e_t)_\#\ppi\,\d t\leq{\rm Comp}(\ppi)\int\rho\,\d\mm,
\]
thus proving \eqref{eq:ppi_bdd_compr}. Moreover, still for every Borel function \(\rho\colon\X\to[0,\infty)\) 
we can compute (using the fact that \(\ppi\) is concentrated on \(q\)-absolutely continuous curves and 
the characterisation of the integral along such curves given in Lemma \ref{lem:integral_along_ac})
\[
\int\!\!\!\int_\gamma\rho\,\d s\,\d\ppi(\gamma)=\int\!\!\!\int_0^1\rho(\gamma_t)|\dot\gamma_t|\,
\d t\,\d\ppi(\gamma)=\int\rho\circ\e\,{\sf ms}\,\d(\ppi\otimes\mathcal L_1)
=\int\rho\,\d\e_\#({\sf ms}(\ppi\otimes\mathcal L_1)),
\]
which proves that ${\rm Bar}(\ppi)$ is the density of $\e_\#({\sf ms}(\ppi\otimes\mathcal L_1))$ with respect to $\mm$.
To see that \({\rm Bar}(\ppi)\) is \(q\)-integrable, let us fix any non-negative \(g\in L^p(\mm)\). 
Then, by using H\"{o}lder's inequality and \eqref{eq:ppi_bdd_compr}, we get
\[
\int {\rm Bar}(\ppi)g\,\d\mm=\int \int_{\gamma} g\, \d s\,\d\ppi(\gamma)=\int g\circ \e \,{\sf ms}\,\d (\ppi\otimes \mathcal L_1)
\leq {\rm Comp}(\ppi)\|g\|_{L^p(\mm)}{\rm E}_q(\ppi)^{ 1/q}.
\]
Then applying \Cref{lem:barycenter} concludes the proof.
\end{proof}
The proof of the following two propositions is straightforward.
\begin{proposition}\label{prop:restr_x_ppi}
Let \((\X,\sfd,\mm)\) be a metric measure space. Let \(\ppi\) be a \(q\)-test plan on \(\X\), 
for some exponent \(q\in(1,\infty]\). Let \(\Gamma\subseteq C([0,1];\X)\)
be a Borel set such that \(\ppi(\Gamma)>0\). Then
\[
\ppi_\Gamma\coloneqq\frac{\ppi|_\Gamma}{\ppi(\Gamma)},\quad\text{ is a $q$-test plan on $\X$ with ${\rm Comp}(\ppi_\Gamma)\leq 
\frac{{\rm Comp}(\ppi)}{\ppi(\Gamma)}$.}
\]
\end{proposition}
\begin{proposition}\label{prop:restr_t_ppi}
Let \((\X,\sfd,\mm)\) be a metric measure space. Let \(\ppi\) be a \(q\)-test plan on \(\X\) with \(q\in(1,\infty]\). Let \(s,\,t\in[0,1]\) with \(s<t\)
be given. Then
\[
({\rm Restr}_s^t)_\#\ppi,\quad\text{ is a $q$-test plan on $\X$ with ${\rm Comp}(({\rm Restr}_s^t)_\#\ppi)\leq {\rm Comp}(\ppi)$}.
\]
\end{proposition}
The notion of barycenter introduced above is usually referred to as 'non-parametric barycenter' 
(see  \cite{Sav:22} and references therein). We shall also need the notion of 'parametric barycenter' defined below, where the term $\int \gamma_\#(s_\gamma)\,\d\ppi(\gamma)$ is replaced by 
${\rm e}_\#(\ppi\otimes\mathcal L_1)$.

\begin{definition}[Parametric barycenter]\label{def:parametric_bar}
Let \((\X,\sfd, \mm)\) be a metric measure space, \(q\in (1,\infty]\) and \(\lambda\in \{0,1\}\). 
We say that a plan \(\ppi\) on \(\X\) admits \textbf{parametric \(\lambda\)-barycenter} ${\rm pBar_\lambda}(\ppi)$
in \(L^q(\mm)\) if the non-negative 
measure 
\[
\rho_\sppi^\lambda\coloneqq \lambda ({\rm e}_0)_\#\ppi+ \lambda ({\rm e}_1)_\#\ppi + {\rm e}_\#(\ppi\otimes 
\mathcal L_1)\in \mathcal M_+(\X)
\]
satisfies 
\[
\rho_\sppi^\lambda\ll \mm\quad \text{ and }\quad {\rm pBar_\lambda}(\ppi)\coloneqq \frac{\d\rho_\sppi^\lambda}{\d \mm}\in L^q(\mm).
\]
\end{definition}
\subsection{Relations between modulus and plans with barycenter}
In this section we will show the relation between \(p\)-modulus and plans admitting a \(q\)-integrable barycenter.
First, we show in Proposition \ref{prop:pi_gamma} that on the compact family of curves the two measures coincide (up to a scaling). 
Before stating and proving the result, we recall Sion's min-max theorem, which will be our key tool in the proof.

\begin{theorem}[Sion's minmax theorem]\label{thm:minmax}
Let \(V,W\) be  topological vector spaces, let \(K\subseteq V\) be a compact and convex set, and let 
\(C\subseteq W\) be a convex set. Suppose that the function \(L\colon C\times K\to \R\) satisfies
\begin{itemize}
\item [a)] for each \(v\in C\), the function \(L(v,\cdot)\colon K\to \R\) is upper semicontinuous and 
concave,
\item [b)] for each \(w\in K\), the function \(L(\cdot, w)\colon  C\to \R\) is convex.
\end{itemize}
Then we have that 
\begin{equation}
\max_{w\in K}\inf_{v\in C}L(v,w)=\inf_{v\in C}\max_{w\in K}L(v,w).
\end{equation}
\end{theorem}
For the proof of this result we refer the reader for instance to \cite{Si:58} or \cite{Sorin}.

\begin{proposition}\label{prop:pi_gamma}
Let \((\X,\sfd,\mm)\) be a metric measure space, \(p\in [1,\infty)\) and \(\lambda\in \{0,1\}\). 
Given a compact set \(\Gamma\subseteq R([0,1];\X)\) such that  
\(0<\Modp^\lambda(\Gamma)<+\infty\),
there exists \(\ppi_{\Gamma}\in \mathcal P(\Gamma)\) that admits \(\lambda\)-barycenter in \(L^q(\mm)\)
with
\[
\|{\rm Bar}_\lambda(\ppi_{\Gamma})\|_{L^q(\mm)}=(\Modp^\lambda(\Gamma))^{-1/p}.
\]
\end{proposition}
\begin{proof}
Let us start by observing that, as the evaluation map 
\(\e\) is continuous and \(\Gamma\) is compact, the set \(K\coloneqq \e(\Gamma\times [0,1])\subseteq \X\) is compact. 
We consider the restriction map
\( R \colon \LIP_{b}( \X ) \to \LIP( K )\), given by  \(R(\rho)\coloneqq \rho|_{K}\).
By the compactness of $K$ and by McShane extension theorem, the map $R$ is onto.  
Given $\rho \in \LIP( K )$ in the sequel we will denote by \(\hat \rho\in \LIP(\X)\) the McShane extension of \(\rho\), namely:
\begin{equation*}
    \hat{\rho}( x ) \coloneqq \inf_{ y \in K }\left\{ \rho(y) + \Lip(\rho) \sfd( y, x ) \right\}, \quad \text{ for all } x \in \X.
\end{equation*}
Note that $\hat\rho \1_K \in \mathcal{L}^{p}( \mm )$ and that $\hat\rho\geq 0$ whenever $\rho \geq 0$.

\noindent
{\bf Step 1.} Given any
\(\rho\in \LIP_+(K)\) and \(\tau>0\), the function 
\[\tilde \rho\coloneqq (\Modp^\lambda(\Gamma))^{1/p}\frac{ \widehat{\rho} \1_{K} }{ \|\widehat{\rho} \1_{K} \|_{L^p(\mm)} + \tau 
 } \]
is not  \(\lambda\)-admissible for \(\Gamma\). Indeed, the claim is immediate from the definition of \(\Modp^\lambda\), 
the observation that the image of every $\gamma \in \Gamma$ lies in $K$, and the fact that \(0<\Modp^\lambda(\Gamma)<+\infty\).
Since $\tilde{\rho}$ is not $\lambda$-admissible for $\Gamma$, there exists \(\gamma\in \Gamma\) such that
\begin{equation}\label{eq:non_adm_rho}
\lambda\big(\tilde\rho(\gamma_0)
+\tilde\rho(\gamma_1)\big)
+\int_\gamma\tilde \rho\,\d s<1\; 
\text{ or equivalently }\; \lambda\big(\rho(\gamma_0)+\rho(\gamma_1)\big)+\int_\gamma\rho\,\d s <
\frac{ \| \widehat{\rho} \1_{K} \|_{L^p(\mm)} + \tau }{(\Modp^\lambda(\Gamma))^{1/p}}.
\end{equation}

\noindent
{ \bf Step 2.}
Consider the functional \(\Phi\colon \LIP_+(K) \times \mathcal P(\Gamma)\to \R\) defined by 
\[\Phi(\rho, \ppi)\coloneqq \| \widehat{\rho} \1_{K} \|^p_{L^p(\mm)}- (\Modp^\lambda(\Gamma))^{1/p}
\biggl(\int\lambda\big(\rho(\gamma_0)+\rho(\gamma_1)\big)+\int_\gamma \rho\,\d s\,\d\ppi(\gamma)
\biggr).\]
Using the definition of \(\Phi\), the condition \eqref{eq:non_adm_rho} can be equivalently written as
\[
\tau +\Phi(\rho, \delta_\gamma)> 0,
\]
where $\delta_{ \gamma }$ is the Dirac measure at the path $\gamma$. This implies that 
\(\tau+\sup_{\ppi\in \mathcal P(\Gamma)}\Phi(\rho,\ppi)>0\), thus by letting \(\tau\to 0\) and 
passing to the infimum over \(\rho\in \LIP_+(K)\), we obtain
\[\inf_{\rho\in \LIP_+(K)}\sup_{\ppi\in \mathcal P(\Gamma)}\Phi(\rho, \ppi)\geq 0.\]

\noindent
{\bf Step 3.} We wish to apply \Cref{thm:minmax} to $\Phi$. To this end, whenever \(\rho\in \LIP_+(K)\), the functional
\(\Phi(\rho,\cdot)\) is concave and upper semicontinuous with respect to the narrow-topology on measures, due to Lemma \ref{cor:lsc_int_rho_pi}. 
Also, the convex set \(\mathcal P(\Gamma) \subset \mathcal M(C([0,1];\X))\) is compact with respect to the narrow-topology, by Prokhorov theorem. 
The set \(\LIP_+(K)\) is a convex subset of the topological vector space \(C(K)\) and for every \(\ppi \in \mathcal P( \Gamma )\), 
the mapping \(\Phi(\cdot,\ppi)\) is convex.
We have now verified the assumptions of \Cref{thm:minmax}, so
\[
\sup_{\ppi\in \mathcal P(\Gamma)} \inf_{\rho\in \LIP_+(K)}\Phi(\rho, \ppi)
=\inf_{\rho\in \LIP_+(K)}\sup_{\ppi\in \mathcal P(\Gamma)}\Phi(\rho, \ppi) \geq  0.
\]
Due to the upper semicontinuity of  \(\Phi(\rho,\cdot)\) and compactness of $\mathcal P( \Gamma )$, 
the suprema above are maxima. Thus, there exists \(\ppi_{\Gamma}\in \mathcal P(\Gamma)\) such that for every \(\rho\in \LIP_+(K)\) 
it holds
\[ \int \biggl(\lambda\big(\rho(\gamma_0)+\rho(\gamma_1)\big)+\int_\gamma \rho\,\d s\biggr)\,\d\ppi_{\Gamma}(\gamma)
\leq(\Modp^\lambda(\Gamma))^{-1/p}\| \hat\rho \1_{K} \|^p_{L^p(\mm)}.
\]
Since the restriction map $R$ is onto, we obtain
\[ \int\biggl(\lambda\big(\rho(\gamma_0)+\rho(\gamma_1)\big)+\int_\gamma \rho\,\d s\biggr)\,\d\ppi_{\Gamma}(\gamma)
\leq(\Modp^\lambda(\Gamma))^{-1/p}\| \rho \|^p_{L^p(\mm)}
\]
for any non-negative $ \rho\in \LIP_{bs}( \X )$.
We may now apply \Cref{lem:barycenter}, showing that \(\ppi_\Gamma\) admits a \(\lambda\)-barycenter in \(L^q(\mm)\) and  
\[\|{\rm Bar}_\lambda(\ppi_{\Gamma})\|_{L^q(\mm)}\leq (\Modp^\lambda(\Gamma))^{-1/p}.\]

\noindent
{\bf Step 4.} To get the reverse inequality, we observe that every \(\rho\in {\rm Adm}_\lambda(\Gamma)\) satisfies
\[
 1=\ppi_{\Gamma}(\Gamma)\leq\int\biggl(\lambda\big(\rho(\gamma_0)+\rho(\gamma_1)\big)+\int_\gamma \rho\,\d s\biggr)\,\d\ppi_{\Gamma}
\leq \|\rho\|_{L^p(\mm)}\|{\rm Bar}_\lambda(\ppi_{\Gamma})\|_{L^q(\mm)}.
\]
Taking the infimum over such $\rho$ shows
\[
(\Modp^\lambda(\Gamma))^{-1/p}\leq\|{\rm Bar}_\lambda(\ppi_{\Gamma})\|_{L^q(\mm)}
\]
as claimed.
\end{proof}
The rest of this section is devoted to the technical results needed to prove the inclusion \(B^{1,p}(\X)\subseteq N^{1,p}(\X)\) in Section 
\ref{sec:Equivalence}. In order to make it easier for the reader to follow the argument we first outline our strategy for proving the latter. 
The proof is based, essentially, on the fact that 
\begin{equation}\label{eq:proof_idea}
\ppi(\Gamma^q_{f,\rho})=0\;\text{ for every \(q\)-test plan }\ppi\quad \Longrightarrow \quad\Modp(\LIP([0,1];\X)\setminus C_{f,\rho})=0,
\end{equation}
where the sets \(\Gamma_{f,\rho}^q\) and \(C_{f,\rho}\) are introduced below in \eqref{eq:gamma_f_rho} and \eqref{eq:C_f_rho}, respectively.
\begin{itemize}
    \item In the case \(p>1\) the right-hand side holds true due to the fact that 
    the family \(\LIP([0,1],\X)\setminus C_{f,\rho})\) is a stable family of curves 
    (in the sense of \cite[Definition 9.3]{Amb:Mar:Sav:15}) and  \(\ppi\)-null, for any \(q\)-test plan \(\ppi\) and thus it 
    is negligible also for \(\Modp\) due to \cite[Theorem 9.4]{Amb:Mar:Sav:15}. The above result comes as a consequence of the 
    fact that the result of our Proposition \ref{prop:pi_gamma} in the case \(p>1\) holds true not only for compact, 
    but for all Souslin sets - indicating in the same time that \(\Modp\) in this case is a Choquet capacity. In the case  \(p=1\) 
    this property fails (cf.\ \cite{Ex:Ka:Ma:Ma:21}). To overcome this difficulty, we proceed as follows, providing a 
    proof that covers the range \([1,\infty)\) for the exponent \(p\).
    \item We show \eqref{eq:proof_idea} in two steps: first one, contained in Lemma \ref{lem:test_plan_mod_negligible} showing that 
    \begin{equation}\label{eq:proof_idea_1}
\ppi(\Gamma^q_{f,\rho})=0\;\text{ for every \(q\)-test plan }\ppi\quad \Longrightarrow \quad\Modp^1(\Gamma_{f,\rho}^q)=0,
 \end{equation}
    and the second one contained in Lemma \ref{lem:Gamma_tilde_negligible}  proving that 
    \begin{equation}\label{eq:proof_idea_2}
\quad\Modp^1(\Gamma_{f,\rho}^q)=0 \quad \Longrightarrow \quad \Modp(\LIP([0,1];\X)\setminus C_{f,\rho})=0.
 \end{equation}  
 \item To prove \eqref{eq:proof_idea_1}, we argue by contradiction, and show in Lemma~\ref{lem:gamma_f_rho_2} 
 that \(\Modp^1(\Gamma_{f,\rho}^q)>0\) implies the existence of a compact subfamily \(\Gamma\subseteq \Gamma_{f,\rho}^q\) having 
 positive and finite \(\Modp^1\).  As in \cite{Sav:22}, the use of \(\Modp^1\) is crucial for getting the compactness here. 
 Then, by means of Proposition \ref{prop:pi_gamma} we associate to such \(\Gamma\) the plan \(\ppi_\Gamma\) with  
 \(1\)-barycenter in \(L^q(\mm)\), which in particular has the property \(\ppi_\Gamma(\Gamma)>0\). 
 
 In order to get to a contradiction with the assumption in \eqref{eq:proof_idea}, we aim at constructing, 
 starting from  \(\ppi_\Gamma\), a \(q\)-test plan  with the latter property. Lemma \ref{lem:gamma_f_rho_1}
 says that equivalently, we may show such a property for a plan with parametric barycenter in \(L^\infty\), 
 having finite \(q\)-energy. We show its existence in Theorem \ref{thm:AGS_8.5}.
\end{itemize}
\begin{lemma}
\label{lem:gamma_f_rho_2}
Let \((\X,\sfd,\mm)\) be a metric measure space and let \(p\in [1,\infty)\). For $r\in [1,\infty]$,
\(f\in \mathcal L^p(\mm)\) and \(\rho\in \mathcal L^p_{\rm ext}(\mm)^+\), define
\begin{equation}\label{eq:gamma_f_rho}
\Gamma^r_{f,\rho}\coloneqq\bigg\{\gamma\in\AC^r([0,1];\X)\;\bigg|\;|f(\gamma_1)-f(\gamma_0)|>\int_\gamma\rho\,\d s\bigg\}.
\end{equation}
If \(\Modp^1(\Gamma^r_{f,\rho})>0\), there exist a compact family \(\Gamma\subseteq\Gamma^r_{f,\rho}\) and $k>0$
such that 
\[0<\Modp^1(\Gamma)<+\infty\qquad\text{and}\qquad
\Gamma\subseteq\big\{\gamma\in\LIP([0,1];\X)\,:\, \Lip(\gamma)\leq k\big\}.\]
\end{lemma}
\begin{proof}
Define \(\widehat\Gamma\coloneqq\Gamma^r_{f,\rho}\cap\LIP([0,1];\X)\). Letting \({\sf CSRep}\) be the
constant-speed-reparametrization map defined as in \eqref{eq:cs_reparam_map}, we
have that \({\sf CSRep}(\Gamma^r_{f,\rho})\subseteq\widehat\Gamma\) by Lemma \ref{lem:integral_rep_invariant}.
Since \({\sf CSRep}(\Gamma^r_{f,\rho})\) and \(\Gamma^r_{f\rho}\) have the same \(\Modp^1\)-modulus
by Lemma~\ref{lem:Mod_param_invariant}, we deduce that \(\Modp^1(\widehat\Gamma)=\Modp^1(\Gamma^r_{f,\rho})>0\).
\smallskip

\noindent
In {\textbf{Steps 1}} to {\textbf{3}}, we construct the claimed subfamily $\Gamma$ of $\widehat\Gamma \subset \Gamma^{r}_{f,\rho}$ of 
uniformly Lipschitz paths, whose images lie in some closed and bounded set. Then in {\textbf{Step 4}} we prove that the constructed 
family $\Gamma$ is compact.
\smallskip

\noindent
\textbf{Step 1.} As a consequence of the Vitali--Carath\'{e}odory theorem (see e.g.\ \cite[Section 4.2]{HKST:15}), we find a sequence
\((\rho_n)_n\subseteq\mathcal L^p_{\rm ext}(\mm)^+\) of lower semicontinuous functions satisfying \(\int|\rho_n-\rho|^p\,\d\mm\to 0\) 
and $\rho_n \geq \rho_{n+1} \geq \rho$ everywhere. By Fuglede's lemma, cf. \Cref{thm:Fuglede}, we may assume the existence of 
$\Gamma_0 \subset \mathscr{C}( \X )$ satisfying $\Mod_p( \Gamma_0 ) = 0$ such that
\begin{equation*}
    \inf_{ n \in \mathbb{N} } \int_\gamma \rho_n \,\d s
    =
    \lim_{ n \rightarrow \infty } \int_\gamma \rho_n \,\d s
    =
    \int_\gamma \rho \,\d s
    \quad\text{for every $\gamma \in \mathscr{C}( \X ) \setminus \Gamma_0$.}
\end{equation*}
By \Cref{rem:basicproperties}, we have $\Mod_p^1( \Gamma_0 ) = 0$, so the subadditivity of $\Mod_p^1$ implies
\begin{equation*}
    \Mod_p^1(\widehat{\Gamma}) = \Mod_p^1( \widehat{\Gamma} \setminus \Gamma_0 ).
\end{equation*}
By the subadditivity and monotonicity of $\Mod_p^1$, there is $\mu > 0$ such that
\begin{equation*}
    \widehat{\Gamma}_\mu
    \coloneqq
    \left\{ \gamma \in \widehat{\Gamma} \setminus \Gamma_0 \colon | f( \gamma_1 ) - f( \gamma_0 ) | > \mu + \int_\gamma \rho \,\d s \right\}
\end{equation*}
has positive $\Mod_{p}^1$-modulus. By definition of $\Gamma_0$, for every $\gamma \in \widehat{\Gamma}_\mu$, there is $n_0 \in \mathbb{N}$ satisfying
\begin{equation*}
    |f(\gamma_1)-f(\gamma_0)| \geq \mu+\int_\gamma\rho_n\,\d s
    \quad\text{for every $n \geq n_0$.}
\end{equation*}
Thus $\widehat{\Gamma}_\mu$ is contained in the union of
\[
\Gamma_n\coloneqq \bigg\{\gamma\in \LIP( [0,1]; \X ) \;\bigg|\;|f(\gamma_1)-f(\gamma_0)| \geq \mu+\int_\gamma\rho_n\,\d s\bigg\}
\quad\text{for $n \in \mathbb{N}$.}
\]
Since $\rho_n \geq \rho$ everywhere, we have $\Gamma_n \subset \widehat{\Gamma}$ for every $n \in \mathbb{N}$.
By monotonicity and subadditivity of $\Mod_p^1$, we have
\begin{equation*}
    \sum_{n\in\N}\Modp^1(\Gamma_n)
    \geq
    \Modp^1(\widehat\Gamma_\mu\setminus\tilde\Gamma)
    =
    \Modp^1(\widehat{\Gamma}_\mu)>0,
\end{equation*}
so there is $n \in \mathbb{N}$ for which $\Mod_p^1( \Gamma_n ) > 0$. We fix such an $n$ for the rest of the claim.
The subfamily $\Gamma_n$ of $\widehat{\Gamma}$ is the first simplification we make. Our next goal is to reduce to 
uniformly Lipschitz paths contained in a closed ball.
\smallskip

\noindent
\textbf{Step 2.} Fix an arbitrary point \(\bar x\in\X\). Given any \(k\in\N\), we define
\[
\Sigma_k\coloneqq\big\{\gamma\in C([0,1];\X)\;\big|\;\Lip(\gamma)\leq k,\, \gamma_t \in \bar B_k(\bar x) \text{ for every $t \in [0,1]$} \big\}.
\]
Since uniform limits of \(k\)-Lipschitz curves are \(k\)-Lipschitz, the set \(\Sigma_k\) is closed in \(C([0,1];\X)\). 
Observe that \(\Gamma_n=\bigcup_{k\in\N}\Gamma_n\cap\Sigma_k\), so there is $k \in \mathbb{N}$ for which \(\Modp^1(\Gamma_n\cap\Sigma_k)>0\). 
Moreover, by Lemma \ref{lem:rough_est_Modp} and Lemma~\ref{lem:properties_Mod}(3), we also have that
\(\Modp^1(\Gamma_n\cap\Sigma_k)<+\infty\). We fix $k \in \mathbb{N}$ with this property for the rest of the proof.

In the next two steps, our goal is to find a closed subfamily $\Gamma$ of $\Gamma_n \cap \Sigma_k$ such that $f$ is 
continuous relative to the end points of $\Gamma$. More precisely, for every sequence $( \gamma^n )_{ n = 1 }^{ \infty }$ in $\Gamma$ 
converging uniformly to some $\gamma$, we have that $f( \gamma^n_i ) \rightarrow f( \gamma_i )$ for $i \in  \{0,1\}$ as $n \to \infty$.
\smallskip

\noindent
\textbf{Step 3.} An application of Lusin's theorem gives an increasing sequence \((K_j)_j\) of compact subsets of \(\bar B_k(\bar x)\)
such that \(f|_{K_j}\) is continuous and \(\mm\big(\bar B_k(\bar x)\setminus K_i\big)\leq 2^{-jp}\) for every \(j\in\N\). In particular,
\[
\sum_{j=0}^\infty\mm\big(\bar B_k(\bar x)\setminus K_j\big)^{ 1/p}\leq\sum_{j=0}^\infty\frac{1}{2^j}=2,
\]
whence it follows that \(\eta\coloneqq\sum_{j=0}^\infty\1_{\bar B_k(\bar x)\setminus K_j}\in\mathcal L^p_{\rm ext}(\mm)^+\). 
Since each function \(\1_{\bar B_k(\bar x)\setminus K_j}\)
is lower semicontinuous in \(\bar B_k(\bar x)\) (due to the fact that \(K_j\) is closed), we deduce that the function \(\eta\) 
is lower semicontinuous in \(\bar B_k(\bar x)\) as well.
Consequently, the function \(h\coloneqq\rho_n+\eta\colon\bar B_k(\bar x)\to[0,\infty]\) is lower semicontinuous. 
We extend $h$ as infinite to $\X \setminus\bar B_k( \bar x )$, thereby obtaining a lower semicontinuous function on $\X$. 
We denote the extension by $h$ as well.

Now define the sets \((\Sigma_{k,m})_m\) as
\[
    \Sigma_{k,m}
    \coloneqq
    \bigg\{
        \gamma\in\Sigma_k
        \;\bigg|\;
        h(\gamma_0) + h(\gamma_1)
        +\int_\gamma h\,\d s
        \leq
    m\bigg\}
    \quad\text{ for every }m\in\N,
\]
and let $\mathcal{N} \coloneqq \Sigma_k \setminus \bigcup_{ m \in \mathbb{N} } \Sigma_{k,m}$.

Using Lemma \ref{lem:lsc_integral}, we see that \(C( [0,1]; \X ) \ni \gamma\mapsto h(\gamma_0)+ h(\gamma_1)+\int_\gamma h\,\d s\) 
is lower semicontinuous and thus \(\Sigma_{k,m}\) is a closed subset of \(C([0,1];\X)\). 
Moreover, from Lemma \ref{lem:properties_Mod} (5), we have that 
\(\Modp^{ 1}(\mathcal N)=0\).
Since \((\Gamma_n\cap\Sigma_k)\setminus\mathcal N=\bigcup_{m\in\N}(\Gamma_n\cap\Sigma_{k,m})\setminus\mathcal N\),
we find \(m\in\N\) such that the set \(\Gamma\coloneqq\Gamma_n\cap\Sigma_{k,m}\) satisfies \(\Modp^1(\Gamma)>0\).
\smallskip

\noindent
\textbf{Step 4.} It remains to verify that \(\Gamma\) is compact. Let us begin with its closedness.
We claim that
\begin{equation}\label{eq:contr_Gamma_cpt_aux}
\gamma_0,\gamma_1\in K_m,\quad\text{ for every }\gamma\in\Gamma.
\end{equation}
To this end, fix \(i\in\{0,1\}\) and observe that for every \(\gamma\in\Gamma\) we have
\[
    \#\{
        j\in\N
        \;|\;
        \gamma_i
        \notin K_j
    \}
    =
    \eta(\gamma_i)
    \leq
    h(\gamma_i)
    \leq m.
\]
Since the sets \((K_j)_j\) are increasing, we deduce that \(\gamma_i\in K_m\), proving \eqref{eq:contr_Gamma_cpt_aux}.

We are in a position to show that \(\Gamma\) is closed: fix a sequence \((\gamma^i)_i\subseteq\Gamma\)
uniformly converging to some \(\gamma\in C([0,1];\X)\). Since \(\Sigma_{k,m}\) is closed, we have that
\(\gamma\in\Sigma_{k,m}\), so we need only to prove that $\gamma \in \Gamma_n$. To this end, we obtain that
\[
\mu+\int_\gamma\rho_n\,\d s\leq\limi_{i\to\infty}\bigg( \mu +\int_{\gamma^i}\rho_n\,\d s\bigg)
\leq\lim_{i\to\infty}\big|f(\gamma^i_1)-f(\gamma^i_0)\big|=\big|f(\gamma_1)-f(\gamma_0)\big|,
\]
where the first inequality follows from Lemma \ref{lem:lsc_integral},  the second inequality from the definition of 
$\Gamma_n$, while the last equality follows from \eqref{eq:contr_Gamma_cpt_aux} and the continuity of \(f|_{K_m}\). 
Therefore $\gamma \in \Gamma$.

We have deduced that $\Gamma$ is closed. To establish compactness, we wish to apply Arzel\`{a}--Ascoli theorem. 
Since the elements of $\Gamma$ are uniformly Lipschitz, the compactness follows from the following claim:
\begin{equation}\label{eq:contr_Gamma_cpt_aux2}
[\Gamma]\coloneqq\big\{\gamma_t\;\big|\;\gamma\in\Gamma,\,t\in[0,1]\big\}\subseteq\X,\quad\text{ is totally bounded.}
\end{equation}
Towards proving \eqref{eq:contr_Gamma_cpt_aux2}, fix any \(\varepsilon>0\) and an $\varepsilon$-separated set 
$S \subseteq [\Gamma]$. The claim follows if we can show that $S$ is finite.

To this end, fix $j_0 \in \mathbb{N}$ satisfying $j_0 \geq m$ and \(\frac{m}{j_0-m+1}<\frac{\varepsilon}{4}\). We obtain a cardinality 
upper bound for $S$ by showing \([\Gamma]\subseteq B_{\varepsilon/4}(K_{j_0})\). Indeed, by total boundedness of $K_{ j_0 }$, 
we find \(x_1,\ldots,x_\ell\in K_{j_0}\) satisfying \(K_{j_0}\subseteq\bigcup_{i=1}^\ell B_{\varepsilon/4}(x_i)\) and 
$d( x_i, x_j ) \geq \varepsilon/4$ for every $i, j$ with $i \neq j$. 
Now if $S \subset [\Gamma]$ is $\varepsilon$-separated, then for every distinct pair $x, y \in S$
satisfying $x \in B_{\varepsilon/2}(x_i)$ and $y \in B_{\varepsilon/2}(x_j)$, we have $i \neq j$. Indeed, otherwise we would have
\begin{equation*}
    \varepsilon \leq \sfd( x, y ) \leq \sfd( x, x_i ) + \sfd( x_i, y ) < \varepsilon;
    \quad\text{ a contradiction}.
\end{equation*}
We deduce, therefore, that the cardinality of $S$ is at most $\ell$.

We have reduced \eqref{eq:contr_Gamma_cpt_aux2} to showing that \([\Gamma]\subseteq B_{\varepsilon/4}(K_{j_0})\). 
We argue by contradiction and suppose the existence of \(\gamma\in\Gamma\) and \(t_1\in[0,1]\) such that
\(\sfd(\gamma_{t_1},K_{j_0})\geq \varepsilon/4\). Given any \(j=m,\ldots,j_0\), 
we have that \(\gamma_0\in K_m\subseteq K_j\subseteq K_{j_0}\) by \eqref{eq:contr_Gamma_cpt_aux},
so letting \(t_0\coloneqq\max\{t\in[0,t_1)\,:\,\gamma_t\in K_j\}\), 
we have that \(\gamma((t_0,t_1])\subseteq\bar B_k(\bar x)\setminus K_j\) and thus
\[
\frac{\varepsilon}{4}\leq\sfd(\gamma_{t_1},K_{j_0})\leq\sfd(\gamma_{t_1},K_j)\leq\sfd(\gamma_{t_1},\gamma_{t_0})\leq\ell(\gamma|_{[t_0,t_1]}) 
\leq \int_\gamma\1_{\bar B_k(\bar x)\setminus K_j}\,\d s.
\]
Summing the resulting inequalities over \(j=m,\ldots,j_0\) yields
\[
\frac{(j_0-m+1)\varepsilon}{4}\leq\sum_{j=m}^{j_0}\int_\gamma\1_{\bar B_k(\bar x)\setminus K_j}\,\d s
\leq\int_\gamma\eta\,\d s\leq\int_\gamma  h\,\d s\leq m,
\]
leading to a contradiction with the choice of \(j_0\). So \eqref{eq:contr_Gamma_cpt_aux2} follows.

Given that $\Gamma$ is a compact subset of $\Gamma_n \cap \Sigma_k \subseteq \widehat{\Gamma}$ satisfying 
$\Mod_p^1(\Gamma)>0$, the proof is complete.
\end{proof}
The proof of the next result follows by the same lines of the  proof  of  
\cite[Lemma 5.1.38]{Sav:22}. We include it for the sake of completeness.
\begin{lemma}\label{lem:gamma_f_rho_1}
Let \((\X,\sfd,\mm)\) be a metric measure space and \(p\in [1,\infty)\).
Let \(f\in \mathcal L^p(\mm)\) and \(\rho\in \mathcal L^p_{\rm ext}(\mm)^+\) be given and let \(\Gamma_{f,\rho}^q\)
be defined as in \eqref{eq:gamma_f_rho}. 
Suppose that the family \(\Gamma_{f,\rho}^q\) is $\ppi$-negligible for every \(q\)-test plan $\ppi$. 
Then \(\ppi(\Gamma_{f,\rho}^q)=0\)
for every plan \(\ppi\) on \(\X\) admitting parametric 1-barycenter in \(L^\infty(\mm)\) (see Definition \ref{def:parametric_bar}),
and  having finite \(q\)-energy.
\end{lemma}
\begin{proof}
Fix a plan \(\ppi\) on \(\X\) having finite $q$-energy and with parametric \(1\)-barycenter in \(L^{\infty}(\mm)\).
We also set \(M\coloneqq \|{\rm pBar}_1(\ppi)\|_{L^\infty(\mm)}>0\).
Given \(0\leq r<s\leq 1\), we consider the Borel maps 
\(D_r^+, D_s^-\colon C([0,1];\X)\times [0,1]\to C([0,1];\X)\) given by
\[
D^+(\gamma, r)_t\coloneqq \gamma((r+t)\wedge 1)\quad \text{ and } \quad D^-(\gamma, s)_t\coloneqq \gamma((t-s)\vee 0).
\]
Now we set \(\lambda\coloneqq 3\mathcal L^1|_{(1/3,2/3)}\) and define 
\[
\ppi^+\coloneqq D^+_\#(\ppi\otimes \lambda)\quad \text{ and } \quad \ppi^-\coloneqq D^-_\#(\ppi\otimes \lambda).
\]
We will show that \(\ppi^+\) and \(\ppi^-\) are \(q\)-test plans.
It is not difficult to check that \((\e_t)_\#\ppi^+=(\e_1)_\#\ppi\leq M\mm\) for all \(t\geq 2/3\).
For \(t\in [0,2/3)\) and every non-negative Borel function \(f\colon \X\to \R\) we have that 
\[
\begin{aligned}
\int f(\e_t(\gamma))\,\d \ppi^+(\gamma)
& = 3\int_{\frac{1}{3}}^{\frac{2}{3}} \int f(\gamma_{(r+t)\wedge 1})\,\d \ppi(\gamma)\,\d r\\
& = 3 \int_{\frac{1}{3}}^{(1-t)\wedge \frac{2}{3}} \int f(\gamma_{r+t})\,\d\ppi(\gamma)\d r + 3 (\frac{1}{3}-t)^+ \int f(\gamma_1)\,\d \ppi(\gamma)\\
&\leq 3 \int f\,\d \e_\#(\ppi\otimes \mathcal L^1)+\int f\,\d (\e_1)_\#\ppi\leq 4 M\int f\,\d\mm.
\end{aligned}
\]
Therefore, we get for every \(t\in [0,1]\) that \((\e_t)_\#\ppi^+\leq 4 M\mm\). An analogous computation shows the same property of \(\ppi^-\). Also, for every \(r\in [0,1]\) the inequality \(|\dot{D^{\pm}(\gamma, r)}_t|\leq |\dot\gamma_t|\) holds for 
\(\mathcal L_1\)-a.e.\ every \(t\in (0,1)\), implying that \(E_q( D^{\pm}(\gamma, r))\leq E_q(\gamma)\) and thus that \(\ppi^+\) and \(\ppi^-\) have finite \(q\)-energy. All in all, we have proven that 
\(\ppi^+\) and \(\ppi^-\) are \(q\)-test plans and thus, by the assumptions of the theorem, we have that
\(\ppi^\pm(\Gamma^{q}_{f,\rho})=0\). Thus,  applying Fubini's theorem we obtain the existence of a Borel set \(N\subseteq C([0,1];\X)\) with \(\ppi(N)=0\) and 
such that for every \(\gamma\in {\rm AC}^q([0,1];\X)\setminus N\) it holds 
\[
\begin{split}
|f(\gamma_{1-s})-f(\gamma_0)|&=|f(D^-(\gamma, s)_1)-f(D^-(\gamma, s)_0)|\leq \int_0^1 \rho(\gamma_{(t-s)\vee 0})|\dot \gamma_{(t-s)\vee 0}|\,\d t\\
&=\int _0^{1-s}\rho(\gamma_t)|\dot \gamma_t|\,\d t,
\end{split}
\]
for \(\mathcal L_1\)-a.e.\ \(s\in (1/3,2/3)\), and similarly for \(\mathcal L_1\)-a.e.\ \(r\in (1/3,2/3)\) it holds
\[
|f(\gamma_{1})-f(\gamma_r)|=|f(D^+(\gamma, r)_1)-f(D^+(\gamma, r)_0)|\leq \int_0^1 \rho(\gamma_{(t+r)\wedge 1})|\dot \gamma_{(t+r)\wedge 1}|\,\d t=\int _r^{1}\rho(\gamma_t)|\dot \gamma_t|\,\d t.
\]
Thus, for \(\ppi\)-a.e.\ curve \(\gamma\in {\rm AC}^q([0,1];\X)\) we can find a common value \(r=1-s\in (1/3,2/3)\) such that
\[
|f(\gamma_1)-f(\gamma_0)|\leq \int_0^{1-s}\rho(\gamma_t)|\dot\gamma_t|\,\d t+\int_r^1\rho(\gamma_t)|\dot\gamma_t|\,\d t=\int_0^1\rho(\gamma_t)|\dot\gamma_t|\,\d t,
\]
proving that \(\ppi(\Gamma^q_{f,\rho})=0\), as claimed.
\end{proof}

We next prove a slight variant of the result proven in \cite[Theorem 8.5]{Amb:Mar:Sav:15}, an improvement, obtained via reparameterization,
of the $1$-barycenter from $L^q(\mm)$ to parametric $1$-barycenter in $L^\infty(\mm)$.

\begin{theorem}\label{thm:AGS_8.5}
Let \((\X,\sfd,\mm)\) be a metric measure space and \(q\in (1,\infty]\).
Let \(\ppi\in \mathcal P\big(C([0,1];\X)\big)\) be a plan with 1-barycenter in \(L^q(\mm)\). 
Suppose that \(\ppi\) is 
 concentrated on a Souslin set \(\Gamma\subseteq \LIP([0,1];\X)\) which consists of 
 curves having the following properties: 
 \begin{itemize}
     \item [a)] there exists \(C>0\) so that \(\Lip(\gamma)\leq C\) for \(\ppi\)-a.e.\ \(\gamma\in \Gamma\);
     \item [b)] there exists \(\ell>0\) so that \(\ell(\gamma)\geq \ell\) for \(\ppi\)-a.e.\ \(\gamma\in \Gamma\);
     \item [c)] there exists a bounded Borel set \(B\subseteq\X\) so that \({\rm im}(\gamma)\subseteq B\) for \(\ppi\)-a.e.\ \(\gamma\in \Gamma\).
 \end{itemize}
 
 Then there exist a Borel
 reparametrization map \({\sf H}\colon AC([0,1];\X)\to AC([0,1];\X)\) and a plan
 \(\sigma\in \mathcal P\big(C([0,1];\X)\big)\) having parametric 1-barycenter in 
 \(L^\infty(\mm)\), finite \(q\)-energy and being concentrated on the Souslin set 
 \({\sf H}(\Gamma)\).
\end{theorem}
\begin{proof} 
We assume $q<\infty$, the statement being trivially true with $H$ equal to the identity map in the case $q=\infty$.
Since the $1$-barycenter is invariant under reparameterizations, we can assume with no
loss of generality that $\Gamma$ consists of nonconstant curves with constant speed.
We denote by $L \leq C$ the $L^\infty(\ppi)$-norm of the length of the curves in $\Gamma$.
Let $g\in \mathcal L^q(\mm)^+$ be a Borel representative of the $1$-barycenter of $\ppi$ and let us set 
$h:=1/(g\lor 1)$, so that $h$ takes its values in $(0,1]$. We then set
$$
G(\gamma):=\int_0^1h(\gamma(r))\,\d r\in [0,1],\qquad
t_\gamma(s):=\frac{1}{G(\gamma)}\int_0^sh(\gamma(r))\,\d r:[0,1]\to [0,1],
$$
and define the reparameterization map
$$
H_\gamma(t):=\gamma(s_\gamma(t))
$$
where $s_\gamma \colon [0,1] \rightarrow [0,1]$ is the inverse of $t_\gamma \colon [0,1] \rightarrow [0,1]$. 
We postpone the proof of Borel measurability of $\gamma \mapsto H_\gamma$ and $\gamma \mapsto G(\gamma)$ to the end of the proof.

We set $\sigma_1\coloneqq z_1^{-1} H_\#(G\ppi)$, where $z_1\in (0,1]$ is the normalization constant $\int G(\gamma)\,\d\ppi(\gamma)$. 
Let us evaluate, first, the $q$-energy of $\sigma$, using $\eta=H\gamma$
as the dummy variable. We notice that, since $t_\gamma'(s)=h(\gamma(s))/G(\gamma)$, one can compute
$$
|\dot H_\gamma(t)|=|\dot\gamma(s_\gamma(t))||s_\gamma'(t)|=
\frac{|\dot\gamma(s_\gamma(t))|G(\gamma)}{h(\gamma(s_\gamma(t)))}.
$$
Using the differential identity above, the very definition of $1$-barycenter and
finally the inequality $G\leq 1$, one has
\begin{eqnarray*}
z_1\int \int_0^1 |\dot\eta(t)|^q\,\d t\,\d\sigma_1(\eta)&=&
\int\int_0^1  |\dot H_\gamma(t)|^q\,\d t G(\gamma)\,\d\ppi(\gamma)
\\
&\leq &L^{q-1}\int G^{q+1}(\gamma)\int_0^1 h^{-q}(\gamma(s_\gamma(t)))|\dot\gamma(s_\gamma(t))| \,\d t\,\d\ppi(\gamma)
\\
&= &L^{q-1}\int G^{q+1}(\gamma)\int_0^1 h^{-q}(\gamma(s))|\dot\gamma(s)|t_\gamma'(s)\,\d s\,\d\ppi(\gamma)
\\
&\leq& L^{q-1}\int G^{q}(\gamma)\int_0^1 h^{1-q}(\gamma(s))|\dot\gamma(s)|\,\d s\,\d\ppi(\gamma)
\\
&\leq& L^{q-1}\int g(g\lor 1)^{q-1}\,\d\mm <+\infty,
\end{eqnarray*}
 where property a) is used in the first inequality and the finiteness of the last integral follows from property c).
Next, let us compute the parametric $1$-barycenter of $\sigma$. With analogous computations, for Borel $f:X\to [0,\infty)$, one has
\begin{eqnarray*}
z_1\int\int_0^1 f(\eta(t))\,\d t\,\d\sigma_1(\eta)&=&
\int\int_0^1f(\gamma(\sigma(t)))\,\d t\,G(\gamma)\d\ppi(\gamma)
\\
&=&\int\int_0^1 f(\gamma(s))t_\gamma'(s)\,\d s\,G(\gamma)\d\ppi(\gamma)
\\
&=&\int\int_0^1 f(\gamma(s))h(\gamma(s))\,\d s\d\ppi(\gamma)
\\
&\leq& \frac{1}{\ell} \int f gh\,\d\mm,
\end{eqnarray*}
where we have used property b) in the last inequality.
Since $gh\leq 1$, this shows that $$\frac{\d{\rm e}_\#(\sigma_1\otimes \mathcal L^1)}{\d\mm}\in L^\infty(\mm).$$
Finally, in order to get an $L^\infty$-control also on the marginals at $t=0$, $t=1$, we set
$$
\alpha\coloneqq\max\left\{1,d_0+d_1\right\}
$$
where $d_i\in L^q(\mm)$ are the densities of $(\e_i)_\#\ppi$, $i=0,\,1$. Since $\alpha \geq 1$, setting
$$
\sigma\coloneqq z^{-1}\frac{1}{\alpha(\e_0)+\alpha(\e_1)}\sigma_1,
$$
where $z\leq 1$ is the normalization constant, we conclude that the previous properties of $\sigma_1$ are retained by $\sigma$.
Since $\alpha\geq d_i$, the marginals at $i$ of $\sigma$ have density in $L^\infty(\mm)$.

{\it Borel measurability of the reparametrization map:}
We claim that $\gamma \mapsto G(\gamma)$ and $\gamma \mapsto H_\gamma$ considered above are Borel functions. 
This then shows that $H( \Gamma )$ is Souslin \cite[Theorem 6.7.3]{Bog:07} and that $\sigma_1$ is a Borel measure.
This yields also that $\sigma$ is a Borel measure. We first prove the claim under the assumption that $h$ is continuous and then prove stability of the conclusion under pointwise convergence; then the measurability in our case follows from a standard monotone class argument, cf. \cite[Chapter 1, Theorem 21]{Probability}.
Given that $h \colon \X \to (0,1]$ is continuous,
\begin{equation*}
    \gamma \mapsto ( s \mapsto \int_{ 0}^{ s } h( \gamma(r) ) \,\d r )
\end{equation*}
is continuous from $C( [0,1]; \X )$ to $C( [0,1] )$. This, in turn, yields the continuity of
\begin{equation*}
    C( [0,1]; \X ) \rightarrow C( [0,1] ), \quad \gamma \mapsto t_\gamma
\end{equation*}
Since every $t_\gamma \colon [0,1] \rightarrow [0,1]$ is a homeomorphism, the above map takes values in the topological 
group of self-homeomorphisms of $[0,1]$ (with composition being the group operation and topology arising from uniform convergence). 
So
\begin{equation*}
    C( [0,1]; \X ) \rightarrow C( [0,1] ), \quad \gamma \mapsto s_\gamma
\end{equation*}
is continuous by the continuity of the inverse operation. Hence $\gamma \mapsto H_\gamma$ is continuous as a composition of continuous maps. 
Next, if $h_n \colon \X \rightarrow (0,1]$ are Borel functions for which the corresponding $H_\gamma^n$ are Borel 
and $h = \lim_{n\rightarrow\infty} h_n$ pointwise everywhere, then we claim that the limiting $H_\gamma$ is also Borel. 
To this end, observe that
\begin{equation*}
    ( s \mapsto \int_{ 0}^{ s } h_n( \gamma(r) ) \,\d r )
    \rightarrow
    ( s \mapsto \int_{ 0 }^{ s } h( \gamma(r) ) \,\d r )
\end{equation*}
in $C( [0,1] )$ by dominated convergence. This implies that the corresponding $t_\gamma^n$ converge uniformly to $t_\gamma$, 
so $s_\gamma^{n} \rightarrow s_\gamma$ uniformly. Consequently, $H^n_\gamma \rightarrow H_\gamma$ uniformly. 
Hence $H_\gamma$ is Borel as well. Similar argument holds for $\gamma \mapsto G(\gamma)$.
\end{proof}
\begin{lemma}\label{lem:test_plan_mod_negligible}
Let \((\X,\sfd,\mm)\) be a metric measure space and let \(p\in [1,\infty)\).
Let \(f\in \mathcal L^p(\mm)\) and \(\rho\in \mathcal L^p_{\rm ext}(\mm)^+\) be given and let \(\Gamma_{f,\rho}^q\)
be defined as in \eqref{eq:gamma_f_rho}. The family \(\Gamma_{f,\rho}^q\) satisfies \(\ppi(\Gamma_{f,\rho}^q)=0\) 
for every \(q\)-test plan \(\ppi\) on \(\X\) if and only if \(\Modp^1(\Gamma_{f,\rho}^q)=0\).
\end{lemma}
\begin{proof}
{\bf Proof of \((\Longrightarrow)\):}
We argue by contradiction: assuming \(\Modp^1(\Gamma_{f,\rho}^q)>0\), we need to prove the existence of a $q$-test plan $\ppi$ such that
$\ppi(\Gamma_{f,\rho}^q)>0$. By Lemma~\ref{lem:gamma_f_rho_2}, there is $C > 0$ and a compact subfamily \(\Gamma\subseteq \Gamma_{f,\rho}^q\) of $C$-Lipschitz 
curves with \(0< \Modp^1(\Gamma)<+\infty\). In particular, their images lie in a compact set $K \subset \X$ by the continuity of the evaluation map.
Proposition~\ref{prop:pi_gamma}
provides us with \(\ppi_{\Gamma}\in \mathcal P(\Gamma)\) admitting 1-barycenter in \(L^q(\mm)\). 
The definition of $\Gamma_{f,\rho}^q$ implies that every $\gamma \in \Gamma_{f,\rho}^q$ has positive length. In particular, 
by the $\ppi_\Gamma$-integrability of the length functional and Lusin's theorem, there is a compact family 
$\Gamma' \subseteq \Gamma$ and $\ell\in (0,1)$ such that $\ell( \gamma ) \geq \ell$ for every $\gamma \in \Gamma'$ and 
$\pi_\Gamma( \Gamma' ) \geq (1-\ell)$. Observe that $\pi_{ \Gamma' } \coloneqq \pi|_{ \Gamma' }/\pi( \Gamma' ) \in \mathcal{P}( \Gamma' )$ 
admits a $1$-barycenter in $L^{q}( \mm )$.

\noindent
{\bf Case \(p=1\).} Since $\pi_{\Gamma'}(\Gamma^\infty_{f,\rho})\geq\pi_{\Gamma'}(\Gamma')>0$, Lemma \ref{lem:gamma_f_rho_1}  
yields the existence of an $\infty$-test plan $\ppi$ with $\ppi(\Gamma_{f,\rho}^\infty)>0$.

\noindent
{\bf Case \(p>1\).}
In this case, we first apply Theorem \ref{thm:AGS_8.5} to $\pi_{\Gamma'}$ in order to get a plan \(\sigma\in 
\mathcal P(C([0,1];\X))\) and a Borel reparametrization map \(\sf H\) so that \(\sigma\) is concentrated on the 
Souslin set \({\sf H}(\Gamma')\subseteq {\rm AC}^q([0,1];\X)\), has finite $q$-energy and has a parametric 1-barycenter in
$L^\infty(\mm)$. Since the condition determining the elements in the family \(\Gamma_{f,\rho}^q\) is reparametrization invariant, we have that \({\sf H}(\Gamma')\subseteq \Gamma_{f,\rho}^q\). In particular, $\sigma( \Gamma_{f,\rho}^q ) = 1$. Then, as in the case $p=1$, 
Lemma~\ref{lem:gamma_f_rho_1} implies that there is a $q$-test plan satisfying $\ppi( \Gamma_{f,\rho}^q ) > 0$.

\noindent
{\bf Proof of \((\Longleftarrow)\):} Since $q$-test plans have barycenter in $L^q(\mm)$, 
the conclusion follows directly from the estimate \eqref{eq:pi_less_modulus}.
\end{proof}

\begin{lemma}\label{lem:Gamma_tilde_negligible}
Let \((\X,\sfd,\mm)\) be a metric measure space and let \(p\in [1,\infty)\), \(f\in \mathcal L^p(\mm)\) and 
\(\rho\in \mathcal L^p_{\rm ext}(\mm)^+\) be given.
Let \(\Gamma_{f,\rho}^q\) be as in \eqref{eq:gamma_f_rho} for the Hölder conjugate $q$ of $p$. Set
\begin{equation}\label{eq:C_f_rho}
C_{f,\rho}\coloneqq \left\{\gamma\in \LIP([0,1];\X)\bigg |\,
\begin{array}{ll}
f\circ \gamma\text{ has an absolutely continuous representative } f_\gamma\\
\text{satisfying }|(f_\gamma)'_t|\leq\rho(\gamma_t)|\dot\gamma_t|<+\infty\text{ for }\mathcal L_1\text{-a.e.\ } t\in (0,1)
\end{array}
\right\}.
\end{equation}
If \(\Mod_p^1(\Gamma_{f,\rho}^q)=0\), then \(\Mod_p( \LIP([0,1]; \X)\setminus C_{f,\rho})=0\).
\end{lemma}
\begin{proof} 
By Lemma~\ref{lem:properties_Mod} (6), there exists \(h\in \mathcal L^p_{\rm ext}(\mm)^+\) such that 
\begin{equation*}
    \Gamma_{f,\rho}^{q}
    \subseteq
    \left\{
        \gamma\in AC([0,1];\X)
        \Big|\, h(\gamma_0)+h(\gamma_1)+\int_\gamma h\,\d s
        =
        +\infty
    \right\}
    \eqqcolon \widetilde H
\end{equation*}
We can, by simply replacing \(h\) by \(\rho+h \), assume that \(\rho\leq h\) everywhere in \(\X\).
We further set
\[H\coloneqq \left\{\gamma\in \AC([0,1];\X)\Big|\, \int_\gamma h\,\d s=+\infty\right\},\]
so that Lemma~\ref{lem:properties_Mod} (5) shows that \(\Modp(H)=0\).
We show that $\LIP( [0,1]; \X ) \setminus H \subseteq C_{f, \rho}$. Since $C_{f,\rho}$ contains all constant curves and is 
invariant under absolutely continuous reparametrizations (cf. \Cref{lem:csrep}), it suffices to prove the conclusion 
for nonconstant curves $\gamma \in \LIP( [0,1]; \X ) \setminus H$ that have constant speed. For such a $\gamma$ and notice 
that, since the speed is positive and constant, $h(\gamma)$ is finite $\mathcal L_1$-a.e. in $(0,1)$. Consequently, we have that 
\(h(\gamma_s)+h(\gamma_t)<+\infty\) for \(\mathcal L^2\)-a.e.\ \((s,t)\in (0,1)^2\). Given \(0<s<t<1\) 
we define \(\gamma^{(s,t)}_r \coloneqq \gamma_{ s + r(t-s) }\) for every \(r\in [0,1]\). 
Thus, for \(\mathcal L^2\)-a.e.\ \((s,t)\in (0,1)^2\) it holds
\[
h\big(\gamma^{(s,t)}_0\big)=h(\gamma_s)<+\infty,\quad h\big(\gamma^{(s,t)}_1\big)
=h(\gamma_t)<+\infty\quad\text{ and }\quad \int_{\gamma^{(s,t)}}h\,\d s\leq \int_\gamma h\,\d s <+\infty.
\]
This implies that \(\gamma^{(s,t)}\in \LIP([0,1];\X)\setminus \widetilde H\subseteq \LIP([0,1];\X)\setminus 
{\Gamma}_{f,\rho}^q\).
Then we have that 
\[\big|f(\gamma_t)-f(\gamma_s)\big|=\big|f(\gamma^{(s,t)}_1)-f(\gamma^{(s,t)}_0)\big|\leq 
\int_{\gamma^{(s,t)}} \rho\,\d s=\int_0^1\rho(\gamma^{(s,t)}_r)|\dot\gamma^{(s,t)}_r|\,\d r
=\int_s^t\rho(\gamma_r)|\dot\gamma_r|\,\d r.\]
We are now in a position to apply Lemma~\ref{lem:AGS13_Lemma2.1} to $f(\gamma)$
and deduce that \(\gamma\in C_{f,\rho}\). This concludes the proof.
\end{proof}
\begin{corollary}\label{cor:abs_cont_rep_for_BL_function}
Let \((\X,\sfd,\mm)\) be a metric measure space and let \(p\in [1,\infty)\).
Given \(f\in \mathcal L^p(\mm)\) and \(\rho\in \mathcal L^p_{\rm ext}(\mm)^+\), assume that 
\(\Mod_p^1( \Gamma_{f,\rho}^{q})=0\). Then
there exists an \(\mm\)-measurable representative \(\hat f\colon \X\to \R\)
of \(f\) such that for \(\Mod_p\)-a.e.\ curve \(\gamma\in C_{f,\rho}\) we have 
 \begin{equation}\label{eq:newtonian_rep}
 \hat f\circ \gamma \equiv f_\gamma \qquad\text{in $[0,1]$,}
\end{equation}
where $C_{f,\rho}$ is defined as in \eqref{eq:C_f_rho}.
\end{corollary}
\begin{proof}
Let us consider $h_0,\, h_1 \in \mathcal{L}^{p}_{\rm ext}( \mm )^+$ such that
\begin{align*}
    \LIP( [0,1]; \X ) \setminus C_{f,\rho}
    &\subset
    \Gamma_0
    \coloneqq
    \left\{ 
        \gamma \in C( [0,1]; \X )
        \mid
        \int_\gamma h_0 \d s = \infty
    \right\}
    \\
     \Gamma^{q}_{f,\rho}
    &\subset
    \Gamma_1
    \coloneqq
    \left\{ \gamma \in C( [0,1]; \X ) \mid
        h_1( \gamma_0 )
        +
        h_1( \gamma_1 )
        +
        \int_\gamma h_1 \d s = \infty
    \right\}
    \quad
\end{align*}
The existence of $h_1$ is guaranteed by Lemma \ref{lem:properties_Mod}(6) and the existence of $h_0$ by the same 
lemma and \Cref{lem:Gamma_tilde_negligible}. By considering $h_0 + h_1 + |f| + \rho$ in place of $h_0$ and $h_1$, 
we may assume that $h \equiv h_i$ for $i = 0,\,1$ and that $h \geq |f| + \rho$ everywhere in $\X$.
We assume these from now on, recalling also that
$\Mod_p(\Gamma_0)=0$ and $\Mod_p^1(\Gamma_1) = 0$ from \Cref{lem:properties_Mod}.
We denote
\begin{equation*}
    \mathcal{G}
    \coloneqq
    \left\{ \gamma \in \LIP( [0,1]; \X ) \mid {\rm essinf} |\dot{\gamma}| > 0 \right\}
    \setminus \Gamma_0.
\end{equation*}
Observe that if $\gamma \in \mathcal{G}$, the set of points $N( \gamma )$ for which $h( \gamma(t) ) = \infty$ or $f( \gamma(t) ) \neq f_\gamma(t)$ 
(for the absolutely continuous representative $f_\gamma$ of $f \circ \gamma$) is negligible for the Lebesgue measure. 
In particular, the set $[0,1] \setminus N( \gamma )$ is dense in $[0,1]$.

Consider $\theta^1, \theta^2 \in \mathcal{G}$ and nonconstant affine maps $A^1, A^2 \colon [0,1] \to \mathbb{R}$ for which the concatenation 
\begin{equation*}
    \theta_s
    \coloneqq
    \left\{
    \begin{aligned}
        &\theta^1 \circ A^1(s), \quad&&s \in [0,1/2]
        \\
        &\theta^2 \circ A^2(s), \quad&&s \in (1/2,1]
    \end{aligned}
    \right.
\end{equation*}
is well-defined and continuous. It is clear that $\theta \in \mathcal{G}$. Notice also that there are 
$s_n \in ( A^{1} )^{-1}( [0,1] \setminus N( \theta^1 ) ) \cap (0,1/2)$ and 
$t_n \in (A^{2})^{-1}( [0,1] \setminus N( \theta^2 ) ) ) \cap (1/2,1)$ converging to $1/2$. Consequently,
\begin{equation*}
    \gamma^{n}_s \coloneqq \theta_{ s_n + s(t_n-s_n) }, \quad s \in [0,1]
\end{equation*}
defines an element of $\mathcal{G} \setminus \Gamma_1$. Hence
\begin{align*}
    | (f_{ \theta^2 })_{ A^2(t_n) } - (f_{ \theta^1 })_{ A^1( s_n ) } |
    &=
    | f( \theta_{ t_n } ) - f( \theta_{ s_n } ) |
    =
    | f( \gamma^n(1) ) - f( \gamma^n(0) ) |
    \\
    &\leq
    \int_{ \gamma^n } \rho \,\d s
    =
    \int_{ [s_n, t_n] } ( \rho \circ \theta ) |\dot{\theta}| \,\d s
    \leq
    \int_{ [s_n, t_n] } ( h \circ \theta ) |\dot{\theta}| \,\d s.
\end{align*}
The upper bound converges to zero as $n \rightarrow \infty$ by the absolute continuity of integration. 
We deduce that the concatenation of $f_{ \theta^2 } \circ A^2$ and $f_{ \theta^1 } \circ A^1$ defines 
an absolutely continuous representative of $f \circ \theta$. Given this technical observation, we define a function \(\hat f\colon \X\to \R\) as
\[
\hat f(x)\coloneqq \left\{\begin{aligned} &( f_\gamma )_r,\quad \text{ if }x=\gamma_r \text{ for some }\gamma \in \mathcal{G} \text{ and } r\in [0,1]\\
& f(x),\quad \text{ otherwise.}
\end{aligned}\right.
\]
Well-posedness of the definition follows from the concatenation argument above.

We claim that $\hat f$ coincides with $f$ in the set $\left\{ h < +\infty \right\}$. 
Indeed, if we have $h(x) < \infty$ and $\widehat{f}(x) \neq f(x)$, there is $\gamma \in \mathcal{G}$ and $r \in [0,1]$ so that $x = \gamma_r$.
Consider a sequence of $s_n \in [0,1] \setminus N( \gamma )$, with $\lim_{ n }s_n = r$ and $s_n \neq r$ for every $n \in \mathbb{N}$. Then
\begin{equation*}
    \gamma^n_s = \gamma_{ s_n + s( r-s_n ) }, \quad s \in [0,1]
\end{equation*}
defines an element of $\mathcal{G} \setminus \Gamma_1$. 
Therefore, as $( f_{\gamma} )_r = ( f_{ \gamma^n } )_1$ and $( f_{ \gamma^n } )_0 = f( \gamma_{ s_n } ) = f( \gamma^n_0 )$, we have
\begin{align*}
    | ( f_{ \gamma } )_r - f( \gamma_r ) |
    &\leq
    | (f_\gamma)_r - ( f_{ \gamma^n } )_0 |
    +
    | f( \gamma_r ) - ( f_{ \gamma^n } )_0 |
    =
    | ( f_{ \gamma^n } )_1 - ( f_{ \gamma^n } )_0 |
    +
    | f( \gamma^n_1 ) - f( \gamma^n_0 ) |
    \\
    &\leq
    \int_{ \gamma^n } \rho \,\d s
    +
    \int_{ \gamma^n } \rho \,\d s
    \leq
    2 \int_{ \gamma^n } h \,\d s.
\end{align*}
As before, the term on the right converge to zero as $n \rightarrow \infty$ by absolute continuity of integration. Hence 
$f_{ \gamma }( r ) = f( \gamma(r) )$, a contradiction. Therefore $\widehat{f} \equiv f$ in 
$\left\{ h < +\infty \right\}$, so $\widehat{f}$ defines a representative of $f$. The proof is complete.
 \end{proof}
\subsection{Bibliographical notes}
\begin{itemize}
    \item The notion of \(p\)-modulus 
has been introduced in \cite{Fug:57} (in the more general setting of positive measures, instead of curves). For the \(p\)-modulus
(here denoted by \(\Modp\)) on the space of curves and its use in the metric Sobolev space theory, we refer to  \cite{HKST:15}.
Its properties were then studied in \cite{Amb:Mar:Sav:15}, in order to provide a relation with test plans.  For this reason,
in the same paper the notion of plan with parametric barycenter (corresponding to our notion of \(0\)-parametric barycenter
in Definition \ref{def:parametric_bar}) has been introduced.
\item The notion of \(\Modp^1\) has been introduced in \cite{Sav:22} (therein denoted by \(\widetilde {\Mod}_p\)). When $p > 1$, $\Modp^1$ was used in \cite{Sav:22} to establish the equivalence between the Newtonian Sobolev space and the Sobolev space defined via a notion of plans with nonparametric barycenter (here denoted by \({\rm Bar}(\ppi)\)).
\item Following \cite{Sav:22}, we introduce in this section the corresponding notions of plans with nonparametric barycenter 
\({\rm Bar}(\ppi)\) in \(L^q(\mm)\). Plans with \(q\)-integrable \(1\)-barycenter \({\rm Bar}_1(\ppi)\) correspond to the nonparametric 
\(q\)-test plans in \cite{Sav:22}. The notion of a \(q\)-test plan has been first introduced in \cite{Amb:Gig:Sav:13, Amb:Gig:Sav:14}.
\item In the case \(p\in (1,\infty)\), the duality relation between plans with barycenter in \(L^q(\mm)\) and the \(p\)-modulus as 
in Proposition \ref{prop:pi_gamma} has been first proven  in \cite{Amb:Mar:Sav:15} (relying on the use of Hahn--Banach theorem) and later in 
\cite{Sav:22}, by means of von Neumann min-max principle. 
\item The results about the negligibility of the family \(\Gamma^q_{f,\rho}\) with respect to different notions of plans and modulus follow 
mainly from the ideas of \cite{Amb:Mar:Sav:15} and \cite{Sav:22}. However, the novelty in our proof is the existence of a compact family in \(\Gamma_{f,\rho}^q\) in Lemma \ref{lem:gamma_f_rho_1}. We provide
a unified argument, which covers also the case \(p=1\). When \(p>1\), a different proof -- leveraging on the fact that
\(\Modp\) is a Choquet capacity -- was given in \cite{Amb:Mar:Sav:15}, but a different strategy is needed for \(p=1\)
since \({\rm Mod}_1\) is not a Choquet capacity (see \cite{Ex:Ka:Ma:Ma:21}).
\end{itemize}
\subsection{List of symbols}
\begin{center}
\begin{spacing}{1.2}
\begin{longtable}{p{2.2cm} p{11.8cm}}
\({\rm Adm}_\lambda(\Gamma)\) & set of \(\lambda\)-admissible functions for \(\Gamma\); Definition \ref{def:Adm_mod}\\
\({\rm Mod}_p^\lambda\) & \((p,\lambda)\)-modulus; \eqref{eq:Mod}\\
\({\rm Mod}_p\) & shorthand notation for \({\rm Mod}_p^0\); Remark \ref{rem:basicproperties}\\
\({\rm Bar}_\lambda(\ppi)\) & \(\lambda\)-barycenter of a plan \(\ppi\); \eqref{eq:def_barycenter}\\
\({\rm Bar}(\ppi)\) & shorthand notation for \({\rm Bar}_0(\ppi)\); Definition \ref{def_plan_barycenter}\\
\({\sf E}_q(\ppi)\) & \(q\)-energy of a plan \(\ppi\); \eqref{eq:def_energy_plan}\\
\({\rm Comp}(\ppi)\) & compression constant of a \(q\)-test plan \(\ppi\); Definition \ref{def:test_plan}\\
\({\rm pBar}_\lambda(\ppi)\) & parametric \(\lambda\)-barycenter of a plan \(\ppi\); Definition \ref{def:parametric_bar}
\end{longtable}
\end{spacing}
\end{center}
\section{Lipschitz derivations}\label{sec:derivations}
This section is devoted to the study of the notion of Lipschitz derivation in the sense of Di Marino \cite{DiMaPhD:14},
needed for the `integration-by-parts' definition of the Sobolev space. The latter has been inspired by the one introduced
by Weaver in \cite{Wea:00} -- we will also show the connection between the two in the last part of this section.
\subsection{Definitions and main properties}
\begin{definition}[Lipschitz derivation]\label{def:Lip_der}
Let \((\X,\sfd,\mm)\) be a metric measure space and let \(q\in[1,\infty]\).
Let \(\mathscr A\) be a subalgebra of \(\LIP_b(\X)\).
Let \(\varrho\colon\mathscr A\to L^\infty(\mm)^+\) be a given function,
which we call a \textbf{gauge function} on \(\X\), having the property that $\varrho(f)$ vanishes in the complement of the support of $f$.

Then by a \textbf{\((\varrho,q)\)-Lipschitz derivation} on \(\X\) we mean a linear operator \(b\colon\mathscr A\to L^q(\mm)\) such that:
\begin{itemize}
\item[\(\rm i)\)] \textsc{Weak locality.} There exists a function \(G\in L^q(\mm)^+\) such that
\begin{equation}\label{eq:def_Lip_der}
|b(f)|\leq G\,\varrho(f)\quad\text{ for every }f\in\mathscr A.
\end{equation}
\item[\(\rm ii)\)] \textsc{Leibniz rule.} \(b(fg)=f\,b(g)+g\,b(f)\) for every \(f,g\in\mathscr A\).
\end{itemize}
We denote by \({\rm Der}^q(\X;\varrho)\) the space of all \((\varrho,q)\)-Lipschitz derivations on \(\X\). 
\end{definition}
To any derivation \(b\in{\rm Der}^q(\X;\varrho)\) we associate the function \(|b|\in L^q(\mm)^+\) defined as follows:
\begin{equation}\label{eq:def_ptwse_norm_der}
|b|\coloneqq\bigwedge\big\{G\in L^q(\mm)^+\;\big|\;\eqref{eq:def_Lip_der}\text{ holds}\big\}=
\bigvee_{f\in\mathscr A}\1_{\{\varrho(f)\neq 0\}}^\mm\frac{|b(f)|}{\varrho(f)}.
\end{equation}
Moreover, given \(b,\tilde b\in{\rm Der}^q(\X;\varrho)\) and \(\lambda\in\R\), we define
\[\begin{split}
(b+\tilde b)(f)&\coloneqq b(f)+\tilde b(f)\quad\text{ for every }f\in\mathscr A,\\
(\lambda b)(f)&\coloneqq\lambda\,b(f)\quad\text{ for every }f\in\mathscr A.
\end{split}\]
It can be readily checked that \(b+\tilde b,\lambda b\in{\rm Der}^q(\X;\varrho)\) and that \({\rm Der}^q(\X;\varrho)\) is a vector space
with respect to such operations. Letting \(\|\cdot\|_{{\rm Der}^q(\X;\varrho)}\colon{\rm Der}^q(\X;\varrho)\to[0,\infty)\) be given by
\[
\|b\|_{{\rm Der}^q(\X;\varrho)}\coloneqq\||b|\|_{L^q(\mm)}\quad\text{ for every }b\in{\rm Der}^q(\X;\varrho),
\]
we also have that \(\big({\rm Der}^q(\X;\varrho),\|\cdot\|_{{\rm Der}^q(\X;\varrho)}\big)\) is a Banach space, as standard verifications show.
\medskip

Furthermore, given any \(b\in{\rm Der}^q(\X;\varrho)\) and \(\varphi\in L^\infty(\mm)\), we define \(\varphi b\colon\mathscr A\to L^q(\mm)\) as
\begin{equation}\label{eq:def_mult_der}
(\varphi b)(f)\coloneqq\varphi\,b(f)\quad\text{ for every }f\in\mathscr A.
\end{equation}
One can readily check that \(\varphi b\in{\rm Der}^q(\X;\varrho)\) and \(|\varphi b|=|\varphi||b|\). It follows that \({\rm Der}^q(\X;\varrho)\)
is a module over the ring \(L^\infty(\mm)\) if endowed with the multiplication \((\varphi,b)\mapsto\varphi b\) defined in \eqref{eq:def_mult_der}.
\medskip

In order to define the space of \((\varrho,q)\)-Lipschitz derivations \emph{with (measure-valued) divergence},
we need to assume further that the subalgebra
\(\mathscr A\) has the following property:
\begin{equation}\label{eq:hp_subalg}
\mathscr A_{bs}\coloneqq\mathscr A\cap\LIP_{bs}(\X)\quad\text{ is dense in \(L^q(\mu)\) for every \(\mu\in\mathfrak M_+(\X)\),}
\end{equation}
where we mean weak (or, equivalently, strong) density if \(q<\infty\), whereas weak\(^*\) 
density if \(q=\infty\). Notice that both \(\LIP_{bs}(\X)\) and \(\LIP_b(\X)\)
satisfy \eqref{eq:hp_subalg}.
\begin{definition}[Derivations with divergence]\label{def:derivations_with_divergence}
Let \((\X,\sfd,\mm)\) be a metric measure space and let \(q\in[1,\infty]\). Let \(\mathscr A\) be a 
subalgebra of \(\LIP_b(\X)\) satisfying \eqref{eq:hp_subalg}
and let \(\varrho\colon\mathscr A\to L^\infty(\mm)^+\) be a gauge function on \(\X\). Let \(b\in{\rm Der}^q(\X;\varrho)\)
be a given derivation. Then we say \(b\) has \textbf{(measure-valued) divergence} if there exists a measure \({\bf div}(b)\in\mathfrak M(\X)\) such that
\[
\int b(f)\,\d\mm=-\int f\,\d{\bf div}(b)\quad\text{ for every }f\in\mathscr A_{bs}.
\]
The divergence \({\bf div}(b)\) is uniquely determined by the strong density of \(\mathscr A_{bs}\) in \(L^1(|{\bf div}(b)|)\).
We denote by \({\rm Der}^q_{\mathfrak M}(\X;\varrho)\) the space of all those \((\varrho,q)\)-Lipschitz derivations on \(\X\) having divergence.

Moreover, for any exponent \(r\in[1,\infty]\) we define the space \({\rm Der}^q_r(\X;\varrho)\) as
\begin{equation}\label{eq:div_in_Lr}
{\rm Der}^q_r(\X;\varrho)\coloneqq
\bigg\{b\in{\rm Der}^q_{\mathfrak M}(\X;\varrho)\;\bigg|\;{\bf div}(b)\ll\mm,\,\div(b)\coloneqq\frac{\d{\bf div}(b)}{\d\mm}\in L^r(\mm)\bigg\}.
\end{equation}
\end{definition}

Both \({\rm Der}^q_{\mathfrak M}(\X;\varrho)\) and \({\rm Der}^q_r(\X;\varrho)\) 
are (not necessarily closed) vector subspaces of \({\rm Der}^q(\X;\varrho)\).
Moreover, \({\rm Der}^q_{\mathfrak M}(\X;\varrho)\ni b\mapsto{\bf div}(b)\in\mathfrak M(\X)\) and \({\rm Der}^q_r(\X;\varrho)\ni b\mapsto
\div(b)\in L^r(\mm)\) are linear.

\begin{proposition}[Leibniz rule for the divergence]\label{prop:leibniz:divergence}
Let \((\X,\sfd,\mm)\) be a metric measure space and let \(q\in[1,\infty]\).
Let \(\mathscr A\) be a subalgebra of \(\LIP_b(\X)\) satisfying \eqref{eq:hp_subalg}
and let \(\varrho\colon\mathscr A\to L^\infty(\mm)^+\) be a gauge function on \(\X\). Let \(\varphi\in\mathscr A_{bs}\)
and \(b\in{\rm Der}^q_{\mathfrak M}(\X;\varrho)\) be given. Then it holds \(\varphi b\in{\rm Der}^q_{\mathfrak M}(\X;\varrho)\) and
\begin{equation}\label{eq:Leibniz_div}
{\bf div}(\varphi b)=\varphi\,{\bf div}(b)+b(\varphi)\mm.
\end{equation}
Moreover, if in addition \(b\in{\rm Der}^q_r(\X;\varrho)\) for some \(r\in[1,\infty]\), 
then \(\varphi b\in\bigcap\limits_{s\in[1,q\wedge r]}{\rm Der}^q_s(\X;\varrho)\) and
\begin{equation}\label{eq:Leibniz_div_Lp}
\div(\varphi b)=\varphi\, \div(b)+b(\varphi).
\end{equation}
In particular, the spaces \({\rm Der}^q_{\mathfrak M}(\X;\varrho)\) and \({\rm Der}^q_q(\X;\varrho)\) are modules over the ring \(\mathscr A\).
\end{proposition}
\begin{proof}
Given any function \(f\in\mathscr A_{bs}\), we have that \(\varphi f\in\mathscr A_{bs}\) and \(b(\varphi f)=\varphi\,b(f)+f b(\varphi)\), so that
\[
\int(\varphi b)(f)\,\d\mm=\int b(\varphi f)\,\d\mm-\int f b(\varphi)\,\d\mm=-\int f\,\d(\varphi\,{\bf div}(b)+b(\varphi)\mm),
\]
which gives \(\varphi b\in{\rm Der}^q_{\mathfrak M}(\X;\varrho)\) and \eqref{eq:Leibniz_div}. If in addition \(b\in{\rm Der}^q_r(\X;\varrho)\),
then \(\varphi\,\div(b)\in L^1(\mm)\cap L^r(\mm)\) and \(b(\varphi)\in L^1(\mm)\cap L^q(\mm)\), 
so that \(\varphi\,\div(b)+b(\varphi)\in L^1(\mm)\cap L^{q\wedge r}(\mm)\).
It follows that \(\varphi b\) belongs to \({\rm Der}^q_s(\X;\varrho)\) for every \(s\in[1,q\wedge r]\) and that \eqref{eq:Leibniz_div_Lp} 
is verified. The proof is complete.
\end{proof}
\subsection{Derivations in the sense of Di Marino}
\begin{definition}[Derivations in the sense of Di Marino]\label{def:Lip_tg_mod}
Let \((\X,\sfd,\mm)\) be a metric measure space and \(q,r\in[1,\infty]\). Consider the gauge function \(\lip_a\colon\LIP_{bs}(\X)\to L^\infty(\mm)^+\). Then,
 with $\mathscr A=\LIP_{bs}(\X)$, we define
\begin{equation}\label{eq:DiMar_der}
{\rm Der}^q(\X)\coloneqq{\rm Der}^q(\X;\lip_a),\quad{\rm Der}^q_{\mathfrak M}(\X)\coloneqq{\rm Der}^q_{\mathfrak M}(\X;\lip_a),\quad{\rm Der}^q_r(\X)\coloneqq{\rm Der}^q_r(\X;\lip_a).
\end{equation}
Moreover, we define the \textbf{Lipschitz tangent module} of \((\X,\sfd,\mm)\) as
\begin{equation}\label{eq:def_Lip_tg_mod}
L^q_\Lip(T\X)\coloneqq{\rm cl}_{{\rm Der}^q(\X)}({\rm Der}^q_{\mathfrak M}(\X)).
\end{equation}
\end{definition}
Since \({\rm Der}^q_{\mathfrak M}(\X)\) is a module over \(\LIP_{bs}(\X)\), and \(\LIP_{bs}(\X)\) is weakly\(^*\) dense in \(L^\infty(\mm)\), one has
\[
L^q_\Lip(T\X)\,\text{ is an }
L^\infty(\mm)\text{-submodule of }{\rm Der}^q(\X).
\]
Also, since any $L^\infty(\mm)$ function is the pointwise $\mm$-a.e.\ limit of \(\LIP_{bs}(\X)\) functions uniformly bounded in 
$L^\infty(\mm)$ (by Lemma~\ref{lem:lip_dense_linfty}), using
Proposition~\ref{prop:leibniz:divergence} the proof of the following lemma is immediate:

\begin{lemma}\label{lemm:extension}
Let \((\X,\sfd,\mm)\) be a metric measure space and \(q\in[1,\infty]\). It holds that 
\begin{equation*}
    {\rm cl}_{{\rm Der}^q(\X)}( L^{\infty}( \mm )\text{\emph{-span of }} \mathrm{Der}^{q}_{q}( \X ) )
    =
    {\rm cl}_{{\rm Der}^q(\X)}( \mathrm{Der}^{q}_{q}( \X ) )    \subseteq
    L^{q}_{ \mathrm{Lip} }( T\X ).
\end{equation*}
In particular, the ${\rm Der}^{q}( \X )$-closure of $\mathrm{Der}^{q}_{q}( \X )$ is closed under taking $L^{\infty}( \mm )$-linear combinations. 
\end{lemma}

\begin{lemma}[Strong locality property]\label{lem:strong_loc_der}
Let \((\X,\sfd,\mm)\) be a metric measure space and \(q\in[1,\infty]\). Let \(b\in L^q_\Lip(T\X)\) be given. Then for every \(f,g\in\LIP_{bs}(\X)\) we have that
\begin{equation}\label{eq:strong_loc_der}
b(f)=b(g)\quad\text{ holds }\mm\text{-a.e.\ on }\{f=g\}.
\end{equation}
In particular, for every closed set \(C\subseteq\X\) we have that
\begin{equation}\label{eq:strong_loc_der_conseq}
|b(f)|\leq|b|\,\lip_a(f|_C)\quad\text{ holds }\mm\text{-a.e.\ on }C.
\end{equation}
\end{lemma}
\begin{proof}
First of all, observe that (by approximation) it is enough to check \eqref{eq:strong_loc_der} for \(b\in{\rm Der}^q_{\mathfrak M}(\X)\).
Define \(h\coloneqq f-g\in\LIP_{bs}(\X)\) and fix an arbitrary compact set \(K\subseteq\{h=0\}\). Consider the cut-off 
functions \(\varphi_n\coloneqq(1-n\,\sfd(\cdot,K))\vee 0\in\LIP_{bs}(\X)\) for every \(n\in\N\).
Notice that \(0\leq\varphi_n\leq 1\), \(\varphi_n(x)\to\1_K(x)\) for every \(x\in\X\), 
and \(\lip_a(\varphi_n)\leq n\1_{B_{1/n}(K)\setminus K}\). Hence, for \(\psi\in\LIP_{bs}(\X)\) one has
\begin{equation}\label{eq:strong_loc_der_aux}\begin{split}
\int\varphi_n\psi\,b(h)\,\d\mm&=\int b(\varphi_n\psi h)-\varphi_n h\,b(\psi)-\psi h\,b(\varphi_n)\,\d\mm\\
&=-\int\varphi_nh\,\d(\psi\,{\bf div}(b)+b(\psi)\mm)-\int\psi h\,b(\varphi_n)\,\d\mm.
\end{split}\end{equation}
Since \(\psi\,{\bf div}(b)+b(\psi)\mm\) has bounded support and \(\varphi_n h\to\1_K h=0\) everywhere, by Dominated Convergence Theorem we see that
\(\int\varphi_n h\,\d(\psi\,{\bf div}(b)+b(\psi)\mm)\to 0\) as \(n\to\infty\). Moreover, we have
\[\begin{split}
\bigg|\int\psi h\,b(\varphi_n)\,\d\mm\bigg|&\leq\|\psi\|_{C_b(\X)}\int|b||h|\lip_a(\varphi_n)\,\d\mm
\leq\|\psi\|_{C_b(\X)}\int_{B_{1/n}(K)\setminus K}n|b||h|\,\d\mm\\
&\leq\|\psi\|_{C_b(\X)}\Lip(h)\int_{B_{1/n}(K)\setminus K}|b|\,\d\mm\to 0\quad\text{ as }n\to\infty,
\end{split}\]
where we used the fact that \(|h(x)|\leq\Lip(h)\sfd(x,K)\leq\frac{1}{n}\Lip(h)\) for every \(x\in B_{1/n}(K)\) and the Dominated Convergence Theorem.
We also have that \(\int\varphi_n\psi\,b(h)\,\d\mm\to\int_K\psi\,b(h)\,\d\mm\) as \(n\to\infty\), again by the Dominated Convergence Theorem.
All in all, recalling also \eqref{eq:strong_loc_der_aux} we conclude that
\[
\int_K\psi\,b(h)\,\d\mm=\lim_{n\to\infty}\int\varphi_n\psi\,b(h)\,\d\mm=0\quad\text{ for every }\psi\in\LIP_{bs}(\X).
\]
Since \(\LIP_{bs}(\X)\) is weakly\(^*\) dense in \(L^\infty(\X)\), we deduce that \(b(f)-b(g)=b(h)=0\) \(\mm\)-a.e.\ on \(K\). 
Given that \(K\) was an arbitrary compact subset
of \(\{f=g\}\), by the inner regularity of \(\mm\) we can finally conclude that \(b(f)-b(g)=0\) holds \(\mm\)-a.e.\ on \(\{f=g\}\), 
thus proving the first claim \eqref{eq:strong_loc_der}.

Let us now pass to the verification of \eqref{eq:strong_loc_der_conseq}. Let \(k\in\N\) be fixed. Then we can find a 
sequence of points \((x_i)_i\subseteq C\) such that \(C\subseteq\bigcup_{i\in\N}B_{1/k}(x_i)\).
By McShane's Extension Theorem (and a standard cut-off argument), for any \(i\in\N\) we can find a Lipschitz 
function \(f_i\in\LIP_{bs}(\X)\) such that \(f_i|_{B_{1/k}(x_i)\cap C}=f|_{B_{1/k}(x_i)\cap C}\)
and \(\Lip(f_i)=\Lip(f;B_{1/k}(x_i)\cap C)\). Therefore, \eqref{eq:strong_loc_der} yields
\[
|b(f)|(x)=|b(f_i)|(x)\leq|b|(x)\,\lip_a(f_i)(x)\leq\Lip(f;B_{1/k}(x_i)\cap C)|b|(x)\leq\Lip(f;B_{2/k}(x)\cap C)|b|(x)
\]
for \(\mm\)-a.e.\ \(x\in B_{1/k}(x_i)\cap C\), whence it follows that \(|b(f)|(x)\leq|b|(x)\Lip(f|_C;B_{2/k}(x)\cap C)\) 
holds for \(\mm\)-a.e.\ \(x\in C\).
By the arbitrariness of \(k\in\N\), we can finally conclude \eqref{eq:strong_loc_der_conseq}.
\end{proof}
\begin{corollary}\label{cor:extens_Lip}
Let \((\X,\sfd,\mm)\) be a metric measure space and \(q\in[1,\infty]\). Fix any \(b\in L^q_\Lip(T\X)\). 
Then there exists a unique linear extension
\(\bar b\colon\LIP(\X)\to L^q(\mm)\) of \(b\) such that
\[
\bar b(fg)=f\,\bar b(g)+g\,\bar b(f)\quad\text{ for every }f,\,g\in\LIP(\X).
\]
Moreover, it holds that \(|\bar b(f)|\leq|b|\,\lip_a(f)\) for every \(f\in\LIP(\X)\). 
If, in addition, \(b\in{\rm Der}^q_{\mathfrak M}(\X)\), the divergence measure \({\bf div}(b)\) belongs to ${\mathcal M}(\X)$ and \(|b|\in L^1(\mm)\), then it holds that
\begin{equation}\label{eq:extens_Lip}
\int\bar b(f)\,\d\mm=-\int f\,\d{\bf div}(b)\quad\text{ for every }f\in\LIP_b(\X).
\end{equation}
Finally, if \(b\in L^q_\Lip(T\X)\), \(f,g\in\LIP(\X)\) and \(C\subseteq\X\) is closed, then the two strong locality properties
\eqref{eq:strong_loc_der} and \eqref{eq:strong_loc_der_conseq} of Lemma \ref{lem:strong_loc_der} hold for $\overline{b}$.
\end{corollary}
\begin{proof}
The uniqueness of $\overline{b}$ follow readily from the Leibniz rule. Indeed, if $\overline{b}$ and $\widehat{b}$ are two such extensions of $b$, $\psi \in \LIP_{bs}( \X )$ and $f \in \LIP(\X)$, then by the Leibniz rule and extension property,
\begin{align*}
    \overline{b}( \psi f ) &= \psi \overline{b}(f) + f b( \psi ),
    \\
    \widehat{b}( \psi f ) &= \psi \widehat{b}( f ) + f b( \psi ), \quad\text{and}
    \\
    \overline{b}( \psi f ) &= b( \psi f ) = \widehat{b}( \psi f ).
\end{align*}
So $\psi \overline{b}(f) = \psi \widehat{b}(f)$ for every $\psi \in \LIP_{bs}(\X)$. Thus $\overline{b}(f) = \widehat{b}(f)$ for every $f \in \LIP(\X)$, giving the claimed uniqueness.
It remains to argue existence of the extension. To 
this end, fix $x_0 \in \X$ and a strictly increasing sequence of radii $r_n>0$ converging to $\infty$. Let $\psi_n \in \LIP_{bs}( \X )$ such that 
$0 \leq \psi_n \leq 1$, $\overline{B}( x_0, r_n ) = \left\{ \psi_n = 1 \right\}$, supported on $\overline{B}( x_0, r_{n+1} )$ 
and that is $1/(r_{n+1}-r_n)$-Lipschitz. Then we have that
\begin{equation}\label{eq:nestedsupport}
    \left\{ \psi_{n} \neq 0 \right\}
    \subset
    \left\{ \psi_{m} = 1 \right\}
    \quad\text{for every $m > n$.}
\end{equation}
In the following, we consider $r_n = 2^n$ for $n \in \mathbb{N}$ and $\psi_n:\X\to [0,1]$ by
\begin{align*}
    \psi_n
    =
    \max\left\{
        0,
        \min\left\{
            \frac{ d( x_0, \cdot ) - r_n }{ r_{n+1} - r_n },
            1
        \right\}
    \right\}.
\end{align*}
Now we define the function
\begin{equation}\label{eq:pointwisedefinition}
    \bar{b}( f )
    =
    \lim_{ n \rightarrow \infty }
        b( \psi_n f )
    \quad\text{for $f \in \LIP(\X)$.}
\end{equation}
We argue well-posedness as follows. Observe that \eqref{eq:nestedsupport} yields that $\psi_m^2 \psi_n f = \psi_n f$ 
in $\left\{ \psi_{ n_0 } \neq 0 \right\}$ for $m > n > n_0$. This fact, \eqref{eq:strong_loc_der} and the Leibniz rule for $b$ give
\begin{align}\label{eq:pointwisedefinition:leibniz}
    b( \psi_n f )
    &=
    ( \psi_m \psi_n ) b( \psi_m f )
    +
    ( \psi_m f ) b( \psi_m \psi_n )
    \\ \notag
    &=
    \psi_n b( \psi_m f )
    \quad\text{in $\left\{ \psi_{ n_0 } \neq 0 \right\}$ for $m > n > n_0$.}
\end{align}
Here $\psi_m \psi_n = \psi_n$ by \eqref{eq:nestedsupport} and  $b( \psi_m \psi_n ) = b( 1 ) = 0$ in $\left\{ \psi_{ n_0 } \neq 0 \right\}$ 
by \eqref{eq:strong_loc_der} and \eqref{eq:strong_loc_der_conseq}. Thus $( b( \psi_n f ) )_{ n = n_0 }^{ \infty }$ has a pointwise limit in 
$\left\{ \psi_{ n_0 } \neq 0 \right\}$ for every $n_0 \in \mathbb{N}$. So \eqref{eq:pointwisedefinition} is well-defined in $L^{0}( \mm )$. 
Moreover, by taking the limit $m \rightarrow \infty$, we see from \eqref{eq:pointwisedefinition:leibniz} that
\begin{equation}\label{eq:pointwisedefinition:leibniz:1}
    b( \psi_n f )
    =
    \psi_n \bar{b}( f )
    \quad\text{in $\left\{ \psi_{ n_0 } \neq 0 \right\}$ for $n > n_0$.}
\end{equation}
By \eqref{eq:strong_loc_der_conseq}, we have that
\begin{equation}\label{eq:pointwisedefinition:leibniz:2}
    | b( \psi_n f ) | \leq |b|( |\psi_n|\lip_a(f) + |f|\lip_a(\psi_n) ) = |b|( |\psi_n|\lip_a(f) )
    \quad\text{in $\left\{ \psi_{ n_0 } \neq 0 \right\}$ for $n > n_0$.}
\end{equation}
So combining \eqref{eq:pointwisedefinition:leibniz:1} and \eqref{eq:pointwisedefinition:leibniz:2} and taking the limit 
$n \rightarrow \infty$ and then $n_0 \rightarrow \infty$ gives that
\begin{equation*}
    | \bar{b}(f) |
    \leq
    |b| \lip_a(f)
    \quad\text{for every $f \in \LIP(\X)$.}
\end{equation*}
Lastly, an argument similar to \eqref{eq:pointwisedefinition:leibniz} gives that
\begin{align*}
    b( \psi_n fg )
    =
    b( ( \psi_n f ) ( \psi_n g ) )
    =
    ( \psi_n f ) b( \psi_n g )
    +
    ( \psi_n g ) b( \psi_n f )
    =
    f b( \psi_n g )
    +
    g b( \psi_n f )
\end{align*}
in $\left\{ \psi_{ n_0 } \neq 0 \right\}$ for $n > n_0$. Passing to the limit $n \rightarrow \infty$ and then to 
$n_0 \rightarrow \infty$ gives the Leibniz rule for $\bar{b}$.

To deduce \eqref{eq:extens_Lip}, it suffices to recall \eqref{eq:pointwisedefinition} and apply dominated convergence. The fact that $\overline{b}( f ) = \overline{b}( g )$ $\mm$-a.e. on $\left\{ f = g \right\}$ for every $f,g \in \LIP(\X)$ follows from \eqref{eq:extens_Lip} and \eqref{eq:pointwisedefinition}. Similar reasoning based on \eqref{eq:pointwisedefinition:leibniz:2} and the validify of \eqref{eq:strong_loc_der_conseq} in $\LIP_{bs}(\X)$ extends \eqref{eq:strong_loc_der_conseq} for $f \in \LIP(\X)$ and closed $C \subseteq \X$.
\end{proof}

Given the uniqueness part of the above Corollary \ref{cor:extens_Lip}, we can unambiguously keep 
the same notation \(b\) to denote the extension \(\bar b\) of \(b\) to \(\LIP(\X)\).
\begin{proposition}\label{prop:cont_der_with_div}
Let \((\X,\sfd,\mm)\) be a metric measure space, \(q\in[1,\infty]\) and \(b\in L^q_{\Lip}(T\X)\). Let \((f_n)_n\subseteq\LIP(\X)\) be such that
\(\sup_n\Lip(f_n)<+\infty\) and \(f_n(x)\to f(x)\) for every \(x\in\X\). Then $f\in\LIP(\X)$. Moreover,
$b(f_n)$ converge weakly (weakly$^*$ in the case $q=\infty$) in $L^q(\mm)$ to $b(f)$.
\end{proposition}
\begin{proof} 
Let us denote \(L\coloneqq\sup_n\Lip(f_n)<+\infty\). Since $|b(f_n)|\leq L|b|$, we need only to check 
convergence in duality with $\varphi\in\LIP_{bs}(\X)$. In addition, as
$$
\biggl|\int\varphi\,b(f_n)\,\d\mm-\int\varphi\,c(f_n)\,\d\mm\biggr|\leq L\int |b-c||\varphi|\,\d\mm
\qquad\forall c\in {\rm Der}^q(\X)
$$
and an analogous property holds for the limit function $f$ or in the case $q=\infty$, 
by the definition of $L^q_{\Lip}(T\X)$ we can also assume that $b\in {\rm Der}^q_{{\mathfrak M}}(\X)$.
Under these additional assumptions it is sufficient to check the convergence properties
$$
\int f_n \varphi\,\d {\bf div} (b)\to\int f\varphi\,\d {\bf div}(b)
\quad\text{and}\quad
\int f_n\,b(\varphi)\,\d\mm\to\int f\,b(\varphi)\,\d\mm.
$$
Both follow by the dominated convergence theorem, since $f_n$ are uniformly bounded on ${\rm spt}(\varphi)$.
\end{proof}
\subsubsection{Derivations induced by plans}
\begin{lemma}\label{lem:plan_induces_der_tech}
Let \((\X,\sfd,\mm)\) be a metric measure space and let \(q\in(1,\infty]\). Let \(\ppi\) be a plan with barycenter in \(L^q(\mm)\).
For any \(f\in\LIP_{bs}(\X)\), the operator \(T_f\colon L^p(\mm)\to L^1(\ppi)\) given by
\begin{equation}\label{eq:plan_induces_der_tech_1}
T_f(g)(\gamma)\coloneqq\int g\circ\gamma\,\d\mu_{f\circ\gamma}\quad\text{ for every }g\in L^p(\mm)\text{ and }\ppi\text{-a.e.\ }\gamma
\end{equation}
is well-defined, linear and continuous. More precisely, for every \(g\in L^p(\mm)\) it holds that
\begin{equation}\label{eq:plan_induces_der_tech_2}
\|T_f(g)\|_{L^1(\sppi)}\leq\int|g|\,\lip(f){\rm Bar}(\ppi)\,\d\mm\leq\Lip(f)\|g\|_{L^p(\mm)}\|{\rm Bar}(\ppi)\|_{L^q(\mm)}.
\end{equation}
\end{lemma}
\begin{proof}
Given any \(g\in\mathcal L^p(\mm)\), we can find a sequence \((g_n)_n\subseteq\LIP_{bs}(\X)\) such that \(g_n(x)\to g(x)\) for \(\mm\)-a.e.\ \(x\in\X\)
and \(\sup_n|g_n|\leq h\) \(\mm\)-a.e.\ on \(\X\) for some \(h\in\mathcal L^p(\mm)^+\). Since \(\int\gamma_\# s_\gamma\,\d\ppi(\gamma)=\Pi^0_\sppi\ll\mm\),
and \(s_{f\circ\gamma}\ll s_\gamma\) for \(\ppi\)-a.e.\ \(\gamma\) by \eqref{eq:chain_rule_s_gamma}, from the \(\mm\)-a.e.\ convergence \(g_n\to g\) we
deduce that for \(\ppi\)-a.e.\ \(\gamma\) it holds that \(g_n\circ\gamma\to g\circ\gamma\) in the \(s_{f\circ\gamma}\)-a.e.\ sense. Similarly, from the
\(\mm\)-a.e.\ inequality \(\sup_n|g_n|\leq h\) it follows that \(\sup_n|g_n|\circ\gamma\leq h\circ\gamma\) 
holds \(s_{f\circ\gamma}\)-a.e.\ for \(\ppi\)-a.e.\ \(\gamma\). Given that
\[\begin{split}
\int\!\!\!\int h\circ\gamma\,\d s_{f\circ\gamma}\,\d\ppi(\gamma)&\overset{\eqref{eq:chain_rule_s_gamma}}\leq
\int\!\!\!\int(h\,\lip(f))\circ\gamma\,\d s_\gamma\,\d\ppi(\gamma)=\int h\,\lip(f)\,\d\Pi^0_\sppi\\
&=\int h\,\lip(f){\rm Bar}(\ppi)\,\d\mm\leq\Lip(f)\|h\|_{L^p(\mm)}\|{\rm Bar}(\ppi)\|_{L^q(\mm)}<+\infty,
\end{split}\]
we deduce that \(h\circ\gamma\in L^1(s_{f\circ\gamma})\) for \(\ppi\)-a.e.\ \(\gamma\). Hence, \eqref{eq:signed:vs:total} 
and the dominated convergence theorem give that
\(\int g_n\circ\gamma\,\d\mu_{f\circ\gamma}\to\int g\circ\gamma\,\d\mu_{f\circ\gamma}\) for \(\ppi\)-a.e.\ \(\gamma\), 
so that \(\gamma\mapsto\bar T_f(g)(\gamma)\coloneqq\int g\circ\gamma\,\d\mu_{f\circ\gamma}\)
is a \(\ppi\)-measurable function thanks to Corollary \ref{cor:aux_plan_gives_der}. 
Now observe that \eqref{eq:signed:vs:total} and \eqref{eq:chain_rule_s_gamma} yield
\[
|\bar T_f(g)(\gamma)|\leq\int|g|\circ\gamma\,\d s_{f\circ\gamma}\leq\int(|g|\,\lip(f))\circ\gamma\,\d s_\gamma\quad\text{ for every }\gamma\in R([0,1];\X).
\]
By integrating with respect to \(\ppi\), we thus obtain that
\[\begin{split}
\int|\bar T_f(g)|\,\d\ppi&\leq\int\!\!\!\int|g|\,\lip(f)\,\d\gamma_\# s_\gamma\,\d\ppi(\gamma)=\int|g|\,\lip(f){\rm Bar}(\ppi)\,\d\mm\\
&\leq\Lip(f)\|g\|_{L^p(\mm)}\|{\rm Bar}(\ppi)\|_{L^q(\mm)}<+\infty.
\end{split}\]
This shows that \(\bar T_f(g)\) is integrable with respect to \(\ppi\) and \(\ppi\)-a.e.\ 
invariant under modifications of the function \(g\) on \(\mm\)-negligible sets.
Therefore, \(\bar T_f\) induces a well-defined, linear, and continuous operator \(T_f\colon L^p(\mm)\to L^1(\ppi)\) 
as in \eqref{eq:plan_induces_der_tech_1} that also
satisfies \eqref{eq:plan_induces_der_tech_2}. The proof is complete.
\end{proof}
\begin{proposition}[Derivation induced by a plan]\label{prop:plan_induces_derivation}
Let \((\X,\sfd,\mm)\) be a metric measure space and let \(q\in (1,\infty]\). Let \(\ppi\) be a given plan on \(\X\) with barycenter in $L^{q}( \mm )$.
Then there exists a unique derivation $b_\sppi \in \mathrm{Der}^{q}_{ \mathfrak M}(\X)$ such that
\begin{equation}\label{eq:plan_induces_der}
\int g\,b_\sppi(f)\,\d\mm=\int\!\!\!\int g\circ\gamma\,\d\mu_{f\circ\gamma}\,\d\ppi(\gamma)\quad
\text{ for every }(f,g)\in\LIP_{bs}(\X)\times L^\infty(\mm).
\end{equation}
Moreover, it holds that $|b_\sppi|\leq {\rm Bar}(\ppi)$ and ${\bf div}(b_\sppi)=({\rm e}_0)_\#\ppi-({\rm e}_1 )_\#\ppi$.
\end{proposition}

\begin{proof}
Given any \(f\in\LIP_{bs}(\X)\), let us define \(T_f\colon L^p(\mm)\to L^1(\ppi)\) as in Lemma \ref{lem:plan_induces_der_tech}.
In particular, the function \(L^p(\mm)\ni g\mapsto\int T_f(g)\,\d\ppi\in\R\) is linear and continuous. Given that \(L^q(\mm)\)
is the dual of \(L^p(\mm)\), we deduce that there exists a unique function \(b_\sppi(f)\in L^q(\mm)\) such that
\begin{equation}\label{eq:plan_induces_derivation_1}
\int g\,b_\sppi(f)\,\d\mm=\int T_f(g)\,\d\ppi=\int\!\!\!\int g\circ\gamma\,\d\mu_{f\circ\gamma}\,\d\ppi(\gamma)
\quad\text{ for every }g\in L^p(\mm).
\end{equation}
Since \(\int|T_f(g)|\,\d\ppi\leq\int g\,\lip(f){\rm Bar}(\ppi)\,\d\mm\) holds for every \(g\in L^p(\mm)^+\) by
\eqref{eq:plan_induces_der_tech_2}, we have that \(|b_\sppi(f)|\leq{\rm Bar}(\ppi)\lip(f)\) and thus \(b_\sppi(f)\in L^1(\mm)\).
Now fix any \(g\in L^\infty(\mm)\) and take a sequence \((g_n)_n\subseteq\LIP_{bs}(\X)\) such that \(g_n(x)\to g(x)\)
for \(\mm\)-a.e.\ \(x\in\X\) and \(\sup_n\sup_\X|g_n|<+\infty\). As in the proof of Lemma \ref{lem:plan_induces_der_tech},
we have that \(g_n\circ\gamma\to g\circ\gamma\) in the \(s_{f\circ\gamma}\)-a.e.\ sense for \(\ppi\)-a.e.\ \(\gamma\), so that
\[
\int g\,b_\sppi(f)\,\d\mm=\lim_{n\to\infty}\int g_n\,b_\sppi(f)\,\d\mm\overset{\eqref{eq:plan_induces_derivation_1}}=
\lim_{n\to\infty}\int\int g_n\circ\gamma\,\d\mu_{f\circ\gamma}\,\d\ppi(\gamma)=\int\int g\circ\gamma\,\d\mu_{f\circ\gamma}\,\d\ppi(\gamma)
\]
by dominated convergence theorem. Hence, \eqref{eq:plan_induces_der} is proved. The map
\(b_\sppi\colon\LIP_{bs}(\X)\to L^1(\mm)\), which is linear by Remark \ref{rmk:prop_real_valued_curves} 4),
satisfies the Leibniz rule by Remark \ref{rmk:prop_real_valued_curves} 5). We also know that \(|b_\sppi(f)|\leq{\rm Bar}(\ppi)\lip(f)
\leq{\rm Bar}(\ppi)\lip_a(f)\) for every \(f\in\LIP_{bs}(\X)\), thus accordingly \(b_\sppi\in{\rm Der}^q(\X)\) and
\(|b_\sppi|\leq{\rm Bar}(\ppi)\). Finally, for every \(f\in\LIP_{bs}(\X)\), 
by using the characterization of the measure \(\mu_{f\circ \gamma}\) given in Definition \ref{def:signedvariation},
that
\[
\int b_\sppi(f)\,\d\mm\overset{\eqref{eq:plan_induces_der}}=\int\mu_{f\circ\gamma}([0,1])\,\d\ppi(\gamma)=\int f(\gamma_1)-f(\gamma_0)\,\d\ppi(\gamma),
\]
whence it follows that \(b_\sppi\in{\rm Der}^q_{\mathfrak M}(\X)\) and ${\bf div}(b_\sppi)=({\rm e}_0)_\#\ppi-({\rm e}_1)_\#\ppi$. The proof is complete.
\end{proof}

\begin{proposition}
[Derivation induced by a test plan]
\label{cor:plan_induces_derivation}
Let \((\X,\sfd,\mm)\) be a metric measure space and let \(q\in (1,\infty]\). Given a \(q\)-test plan \(\ppi\), 
the map \(b_\sppi\colon \LIP_{bs}(\X)\to L^q(\mm)\) given by 
\begin{equation}\label{eq:testplan_induces_der}
\int g\, b_{\sppi}(f)\,\d\mm=\int\int_0^1 g(\gamma_t)(f\circ \gamma)_t'\,\d t\,\d \ppi(\gamma),
\quad \text{ for every } (f,g) \in \LIP_{bs}(\X) \times L^p(\mm),
\end{equation}
defines an element of \({\rm Der}^q_{\infty}(\X)\), which satisfies $|b_\sppi|\leq {\rm Bar}(\ppi)$ and 
\begin{equation}\label{eq:derivation_testplan_properties}
    \int \big(f(\gamma_1)-f(\gamma_0)\big)\,\d \ppi(\gamma)=\int f\,\div(b_\sppi)\,\d \mm,\quad \text{ for every }f\in L^1(\mm).
\end{equation}
In measure-theoretic terms, ${\bf div}(b_\sppi)=(\e_0)_\#\ppi-(\e_1)_\#\ppi\ll\mm$ with bounded density.
\end{proposition}
\begin{proof}
We recall from \Cref{lem:ppi_bdd_compr} that $\ppi$ has a barycenter in $L^{q}( \mm )$, so the existence of $b_\ppi$ is 
immediate by \Cref{prop:plan_induces_derivation}. The divergence measure has a density in $L^{\infty}( \mm )$ by Definition 
\ref{def:test_plan} of 
test plans, namely, by the property (TP1).
The fact that \eqref{eq:testplan_induces_der} and \eqref{eq:derivation_testplan_properties} hold in 
$( f, g ) \in \LIP_{bs}( \X ) \times L^{p}( \mm )$ and $f \in L^{1}( \mm )$, respectively, follow from the 
corresponding identities in \Cref{cor:plan_induces_derivation} and an approximation of the elements 
in \(L^p(\mm)\) by those in \(L^\infty(\mm) \cap L^{p}(\mm)\).
\end{proof}
\subsection{Derivations in the sense of Weaver}\label{sec:Weaver_derivation}
We now present the definition of derivation in the sense of Weaver, and show the connection between the previously considered notions.
\begin{lemma}\label{lem:locality_Weaver}
Let \((\X,\sfd,\mm)\) be a metric measure space and \(b\colon\LIP_b(\X)\to L^\infty(\mm)\) a linear operator. Suppose that:
\begin{itemize}
\item[\(\rm i)\)] \(b\) is \textbf{weak\(^*\) continuous}, i.e.\ if a sequence \((f_n)_n\subseteq\LIP_b(\X)\)
of equi-Lipschitz functions converges pointwise to some \(f\in \LIP_b(\X)\), then \(b(f_n)\rightharpoonup b(f)\)
weakly\(^*\) in \(L^\infty(\mm)\).
\item[\(\rm ii)\)] The \textbf{Leibniz rule} holds, i.e.\ \(b(fg)=f\,b(g)+g\,b(f)\) for every \(f,g\in\LIP_b(\X)\).
\end{itemize}
Then for every \(f,g\in\LIP_b(\X)\) it holds that \(b(f)=b(g)\) \(\mm\)-a.e.\ on \(\{f=g\}\).
\end{lemma}
\begin{proof}
By linearity of \(b\), it suffices to show that \(b(f)=0\) holds \(\mm\)-a.e.\ on \(\{f=0\}\). Take \(M>0\) such
that \(|f|\leq M\) on \(\X\). For any \(n\in\N\), define \(\phi_n\colon\R\to\R\) and \(\psi_n\colon\R\to\R\) as
\(\phi_n(t)\coloneqq 1-e^{-nt^2}\) and \(\psi_n(t)\coloneqq t\phi_n(t)\) for every \(t\in\R\). Since
\(|\phi'_n(t)|=2n|t|e^{-nt^2}\leq 2nM\) for every \(t\in[-M,M]\), we see that \(\phi_n|_{[-M,M]}\) is \(2nM\)-Lipschitz
and thus \(\phi_n\circ f\in\LIP_b(\X)\). Moreover, easy computations show that \(\psi'_n(t)=2nt^2e^{-nt^2}-e^{-nt^2}+1\geq 0\) 
has maximum \(1+2e^{-3/2}=\psi'_n(\sqrt{3/2n})\), whence it follows that \(\Lip(\psi_n\circ f)\leq(1+2e^{-3/2})\Lip(f)\)
for every \(n\in\N\). Noticing also that \((\psi_n\circ f)(x)\to f(x)\) for every \(x\in\X\), we infer from the Leibniz
rule and the weak\(^*\) continuity of \(b\) that
\[
f\,b(\phi_n\circ f)+\phi_n\circ f\,b(f)=b(f\,\phi_n\circ f)=b(\psi_n\circ f)\rightharpoonup b(f)
\quad\text{ weakly\(^*\) in }L^\infty(\mm).
\]
Since \(\1_{\{f=0\}}^\mm(f\,b(\phi_n\circ f)+\phi_n\circ f\,b(f))=0\), we deduce that \(\1_{\{f=0\}}^\mm b(f)=0\), as desired.
\end{proof}
\begin{remark}\label{rmk:conseq_Leibniz_Weaver}{\rm
Let \(b\colon\LIP_b(\X)\to L^\infty(\mm)\) be a linear operator satisfying the Leibniz rule. Then
\[
b(\lambda\1_\X)=0\quad\text{ for every }\lambda\in\R.
\]
Indeed, \(b(\lambda^2\1_\X)=\lambda\,b(\lambda\1_\X)\) by linearity and \(b(\lambda^2\1_\X)=2\lambda\,b(\lambda\1_\X)\)
by the Leibniz rule, whence it follows that \(\lambda\,b(\lambda\1_\X)=0\). If \(\lambda\neq 0\), we deduce that
\(b(\lambda\1_\X)=0\). Finally, \(b(0\1_\X)=0\) by linearity.
}\end{remark}
\begin{corollary}\label{cor:equiv_Weaver}
Let \((\X,\sfd,\mm)\) be a metric measure space. Let \(b\colon\LIP_b(\X)\to L^\infty(\mm)\) be a weakly\(^*\)
continuous linear map satisfying the Leibniz rule. Fix \(G\in L^0(\mm)^+\). Then the following conditions are equivalent:
\begin{itemize}
\item[\(\rm i)\)] \(|b(f)|\leq(\Lip(f)+\|f\|_{C_b(\X)})G\) holds \(\mm\)-a.e.\ on \(\X\) for every \(f\in\LIP_b(\X)\).
\item[\(\rm ii)\)] \(|b(f)|\leq(\Lip(f)+\|f\|_{C_b(\X)})\1_{{\rm spt}(f)}^\mm G\) holds \(\mm\)-a.e.\ on \(\X\) for every \(f\in\LIP_b(\X)\).
\item[\(\rm iii)\)] \(|b(f)|\leq\Lip(f)G\) holds \(\mm\)-a.e.\ on \(\X\) for every \(f\in\LIP_b(\X)\).
\item[\(\rm iv)\)] \(|b(f)|\leq \lip_a(f)G\) holds \(\mm\)-a.e.\ on \(\X\) for every \(f\in\LIP_b(\X)\).
\end{itemize}
\end{corollary}
\begin{proof}
Trivially, \({\rm iv)}\Rightarrow{\rm ii)}\Rightarrow{\rm i)}\) and \({\rm iv)}\Rightarrow{\rm iii)}\Rightarrow{\rm i)}\). Let us now prove
that \({\rm i)}\Rightarrow{\rm iv)}\). Assume \(\rm i)\) holds and fix any \(f\in\LIP_b(\X)\).
Fix a sequence of radii \(r_n\searrow 0\). Since \((\X,\sfd)\) is a separable metric space, for any \(n\in\N\)
we can find a sequence \((x^n_k)_k\subseteq\X\) such that \(\X=\bigcup_{k\in\N}B_{r_n}(x^n_k)\).
By the McShane extension theorem, for \(k\in\N\) there exist functions \(f^n_k\in\LIP_b(\X)\) such that
\(f^n_k|_{B_{r_n}(x^n_k)}=f|_{B_{r_n/2}(x^n_k)}\), \(\Lip(f^n_k)=\Lip(f;B_{r_n}(x^n_k))\), and
\(\inf_{B_{r_n}(x^n_k)}f\leq f^n_k(x)\leq\sup_{B_{r_n}(x^n_k)}f\) for all \(x\in\X\). In particular,
\(\tilde f^n_k\coloneqq f^n_k-\inf_\X f^n_k\in\LIP_b(\X)\) satisfies \(\Lip(\tilde f^n_k)\leq\Lip(f;B_{r_n}(x^n_k))\).
Since \(\sup_{B_{r_n}(x^n_k)}f-\inf_{B_{r_n}(x^n_k)}f\leq\Lip(f){\rm diam}(B_{r_n}(x^n_k))\leq 2\,\Lip(f)r_n\),
we have \(\|\tilde f^n_k\|_{C_b(\X)}\leq 2\,\Lip(f)r_n\). Using Lemma \ref{lem:locality_Weaver}, Remark
\ref{rmk:conseq_Leibniz_Weaver}, and the fact \(B_{r_n}(x^n_k)\subseteq B_{2r_n}(x)\) for every \(x\in B_{r_n}(x^n_k)\), we get
\[\begin{split}
|b(f)|(x)&=|b(f^n_k)|(x)=|b(\tilde f^n_k)|(x)\leq G(x)(\Lip(\tilde f^n_k)+\|\tilde f^n_k\|_{C_b(\X)})\\
&\leq G(x)(\Lip(f;B_{r_n}(x^n_k))+2\,\Lip(f)r_n)\leq G(x)(\Lip(f;B_{2 r_n}(x))+2\,\Lip(f)r_n)
\end{split}\]
for \(\mm\)-a.e.\ \(x\in B_{r_n}(x^n_k)\). Hence, we deduce that \(|b(f)|\leq G\,\Lip(f;B_{2 r_n}(\cdot))+2G\,\Lip(f)r_n\) holds
\(\mm\)-a.e.\ on \(\X\) for every \(n\in\N\). Since \(\lip_a(f)(x)=\lim_n\Lip(f;B_{2 r_n}(x))\) holds for every \(x\in\X\),
by letting \(n\to\infty\) we can finally conclude that \(|b(f)|\leq G\,\lip_a(f)\) holds \(\mm\)-a.e.\ on \(\X\), proving iv).
\end{proof}
\begin{definition}[Derivations in the sense of Weaver]\label{def:Weaver_derivations}
Let \((\X,\sfd,\mm)\) be a metric measure space. Let the gauge function
\(\varrho_{\mathscr X}\colon\LIP_b(\X)\to L^\infty(\mm)^+\) be given by \(\varrho_{\mathscr X}(f)\coloneqq(\Lip(f)+\|f\|_{C_b(\X)})\1_{{\rm spt}(f)}^\mm\).
We say that \(b\in{\rm Der}^\infty(\X;\varrho_{\mathscr X})\) is a \textbf{Weaver \(\mm\)-derivation} on \(\X\)
if it is weakly\(^*\) continuous.

We denote by \(\mathscr X(\mm)\) the set of all Weaver \(\mm\)-derivations on \(\X\). 
Given any exponent \(q\in [1,\infty]\), we say that \(b\in \mathscr X(\mm)\) is \textbf{\(q\)-integrable} if
\(|b|\in L^q(\mm)\) and we write \(b\in \mathscr X^q(\mm)\).
\end{definition}

\begin{remark}
{\rm 
In view of Corollary \ref{cor:equiv_Weaver}, the spaces \(\mathscr X(\mm)\) and \(\mathscr X^q(\mm)\) could be equivalently
defined by considering the gauge function \(\LIP_b(\X)\ni f\mapsto\lip_a(f)\) instead of \(\varrho_{\mathscr X}\). Corollary \ref{cor:equiv_Weaver}
also implies that Definition \ref{def:Weaver_derivations} is equivalent to the original formulation by Weaver \cite{Wea:00}.
}
\end{remark}

\begin{proposition}
Let \((\X,\sfd,\mm)\) be a metric measure space and let \(q\in[1,\infty]\). 
Let us define the operator \(\iota_q\colon\mathscr X^q(\mm)\to{\rm Der}^q(\X)\) as
\(\iota_q(b)\coloneqq b|_{\LIP_{bs}(\X)}\) for every \(b\in\mathscr X^q(\mm)\), namely
\[
\iota_q(b)(f)\coloneqq b(f)\quad\text{ for every }b\in\mathscr X^q(\mm)\text{ and }f\in\LIP_{bs}(\X).
\]
Then \(\iota_q\) is well-defined, \(L^\infty(\mm)\)-linear, and satisfies \(|\iota_q(b)|=|b|\) for every \(b\in\mathscr X^q(\mm)\). Moreover,
\[
L^q_\Lip(T\X)\cap{\rm Der}^\infty(\X)\subseteq\iota_q(\mathscr X^q(\mm)).
\]
\end{proposition}
\begin{proof}
The first part of the statement follows from the equivalence \({\rm ii)}\Leftrightarrow{\rm iv)}\) of Corollary
\ref{cor:equiv_Weaver}. To prove the last part of the statement, fix any \(b\in L^q_\Lip(T\X)\cap{\rm Der}^\infty(\X)\). 
Consider the extension $\bar b$ of $b$ constructed in \Cref{cor:extens_Lip} and denote by $\bar b$ its 
restriction to $\LIP_b(\X)$. Since \(|\bar b(f)|\leq|b|\,\lip_a(f)\) for every \(f\in\LIP_b(\X)\)
and \(|b|\in L^\infty(\mm)\), we see that \(\bar b(\LIP_b(\X))\subseteq L^\infty(\mm)\). To prove that
\(\bar b\in\mathscr X^q(\mm)\), it only remains to show that \(\bar b\) is weakly\(^*\) continuous. When \(q=\infty\),
this follows directly from Proposition \ref{prop:cont_der_with_div}. When \(q<\infty\), Proposition \ref{prop:cont_der_with_div}
ensures that \(\bar b(f)\rightharpoonup\bar b(f)\) weakly in \(L^q(\mm)\) if \((f_n)_n\subseteq\LIP_b(\X)\) and
\(f\in\LIP_b(\X)\) satisfy \(L\coloneqq\sup_n\Lip(f_n)<+\infty\) and \(f_n\to f\) pointwise. Now fix any \(g\in L^1(\mm)\).
Since \(g|b|\in L^1(\mm)\), for any \(\varepsilon>0\) we can find a bounded Borel set \(B\subseteq\X\) such that
\(\1_B^\mm g\in L^p(\mm)\) and \(2 L\int_{\X\setminus B}|g||b|\,\d\mm\leq\varepsilon\). We can then find \(n_\varepsilon\in\N\)
such that \(\big|\int_B g\,\bar b(f_n)\,\d\mm-\int_B g\,\bar b(f)\,\d\mm\big|\leq\varepsilon\) for every \(n\geq n_\varepsilon\).
Therefore, for every \(n\geq n_\varepsilon\) we can estimate
\[\begin{split}
\bigg|\int g\,\bar b(f_n)\,\d\mm-\int g\,\bar b(f)\,\d\mm\bigg|&\leq
\bigg|\int_B g\,\bar b(f_n)\,\d\mm-\int_B g\,\bar b(f)\,\d\mm\bigg|+\int_{\X\setminus B}|g||\bar b(f_n-f)|\,\d\mm\\
&\leq\varepsilon+\Lip(f_n-f)\int_{\X\setminus B}|g||b|\,\d\mm\leq\varepsilon+2 L\int_{\X\setminus B}|g||b|\,\d\mm\leq 2\varepsilon,
\end{split}\]
which shows that \(\int g\,\bar b(f_n)\,\d\mm\to\int g\,\bar b(f)\,\d\mm\), and thus \(\bar b(f_n)\rightharpoonup\bar b(f)\)
weakly\(^*\) in \(L^\infty(\mm)\). All in all, we proved that \(\bar b\in\mathscr X^q(\mm)\), so that accordingly
\(b=\iota_q(\bar b)\in\iota_b(\mathscr X^q(\mm))\). The proof is complete.
\end{proof}

\subsection{Bibliographical notes}
\begin{itemize}
    \item The notion of Lipschitz derivation 
    has been introduced in \cite{Wea:00}, as a bounded linear and weakly\(^*\) continuous map from the space \((\LIP_b(\X), \|\cdot\|_\infty+\Lip(\cdot))\) to \(L^\infty(\mm)\), satisfying the Leibniz rule. The argument in Corollary \ref{cor:equiv_Weaver} for showing the equivalence of this notion with the one in Definition \ref{def:Weaver_derivations} follows closely the ideas presented in \cite[Lemma 3.157]{ Sch:16}. In the latter work the relation between Weaver's derivations and currents has been investigated. We will elaborate on this in the follow-up work \cite{AILP}.

\item In \cite{DiMar:14}, the notion of \(q\)-derivation with \(q\)-divergence has been introduced (here denoted by \({\rm Der}_q^q(\X)\)). We recall here its properties and generalize the results to the case of derivations with measure-valued divergence.  

\item In \cite{Gig:18}, also the notion of Sobolev derivation has been considered -- its relation with Lipschitz derivations will be studied in our follow-up work \cite{AILP}.
\end{itemize}
\subsection{List of symbols}
\begin{center}
\begin{spacing}{1.2}
\begin{longtable}{p{2.2cm} p{11.8cm}}
\(\mathscr A\) & a subalgebra of \(\LIP_b(\X)\); Definition \ref{def:Lip_der}\\
\(\varrho\) & a gauge function \(\varrho\colon\mathscr A\to L^\infty(\mm)^+\) on \(\X\); Definition \ref{def:Lip_der}\\
\({\rm Der}^q(\X;\varrho)\) & space of \((\varrho,q)\)-Lipschitz derivations on \(\X\); Definition \ref{def:Lip_der}\\
\(|b|\) & pointwise norm of a derivation \(b\in{\rm Der}^q(\X;\varrho)\); \eqref{eq:def_ptwse_norm_der}\\
\(\mathscr A_{bs}\) & shorthand notation for \(\mathscr A\cap\LIP_{bs}(\X)\); \eqref{eq:hp_subalg}\\
\({\bf div}(b)\) & measure-valued divergence of a derivation \(b\in{\rm Der}^q(\X;\varrho)\); Definition \ref{def:derivations_with_divergence}\\
\({\rm Der}^q_{\mathfrak M}(\X;\varrho)\) & \((\varrho,q)\)-Lipschitz derivations having measure-valued divergence;
Definition \ref{def:derivations_with_divergence}\\
\({\rm Der}^q_r(\X;\varrho)\) & space of \((\varrho,q)\)-Lipschitz derivations having divergence in \(L^r(\mm)\); \eqref{eq:div_in_Lr}\\
\({\rm div}(b)\) & the divergence \({\rm div}(b)\in L^r(\mm)\) of a derivation \(b\in{\rm Der}^q_r(\X;\varrho)\); \eqref{eq:div_in_Lr}\\
\({\rm Der}^q(\X)\) & space \({\rm Der}^q(\X)\coloneqq{\rm Der}^q(\X;\lip_a)\) of derivations in the sense of Di Marino; \eqref{eq:DiMar_der}\\
\({\rm Der}^q_{\mathfrak M}(\X)\) & shorthand notation for \({\rm Der}^q_{\mathfrak M}(\X;\lip_a)\); \eqref{eq:DiMar_der}\\
\({\rm Der}^q_r(\X)\) & shorthand notation for \({\rm Der}^q_r(\X;\lip_a)\); \eqref{eq:DiMar_der}\\
\(L^q_{\rm Lip}(T\X)\) & the Lipschitz tangent module 
\(L^q_{\rm Lip}(T\X)\coloneqq{\rm cl}_{{\rm Der}^q(\X)}({\rm Der}^q_{\mathfrak M}(\X))\) of \(\X\); \eqref{eq:def_Lip_tg_mod}\\
\(b_\sppi\) & derivation induced by a plan \(\ppi\); Proposition \ref{prop:plan_induces_derivation}\\
\(\mathscr X(\mm)\) & space of Weaver \(\mm\)-derivations on \(\X\); Definition \ref{def:Weaver_derivations}\\
\(\mathscr X^q(\mm)\) & space of \(q\)-integrable Weaver \(\mm\)-derivations on \(\X\); Definition \ref{def:Weaver_derivations}
\end{longtable}
\end{spacing}
\end{center}
\section{Definitions of Sobolev space}\label{sec:Sobolev_H}
In this section, we present the four notions of Sobolev space listed in Introduction. Here we only give the definitions of the spaces
and of the minimal objects associated with Sobolev functions. All the calculus rules and other fine properties will be proven in the
next section only for one of the notions (namely, for the Newtonian Sobolev functions). After proving the equivalence, we will then
know that the latter hold true for Sobolev functions defined via any of the other (equivalent) approaches.
\subsection{H-approach: via relaxation}\label{sec:H_def}
In order to introduce the notion of Sobolev space via the H-approach, we start by recalling the definition of 
relaxed upper slopes.

\begin{definition}[Relaxed upper slope and minimal relaxed upper slope]\label{def:relaxed slope}
Let \((\X, \sfd, \mm)\) be a metric measure space, let \(p\in [1,\infty)\) and \(f\in L^p(\mm)\).
A function \(G\in L^p(\mm)^+\) is said to be a  
\textbf{relaxed $p$-upper slope} of \(f\) if there exist \((f_n)_n\subseteq \LIP_{bs}(\X)\) and $G'\leq G$ such that 
\begin{equation}\label{eq:def_p_relaxed_grad}
f_n\to f\text{ in }L^p(\mm)\quad \text{ and }\quad \lip_a(f_n)\rightharpoonup G'\text{ weakly in }L^p(\mm).\end{equation}
Moreover, given \(f\in L^p(\mm)\) we denote by \({\rm RS}(f)\) the possibly empty set of all relaxed $p$-upper slopes of \(f\) and, whenever 
${\rm RS}(f)\neq\varnothing$, we call \textbf{minimal relaxed $p$-upper slope} the function
\begin{equation}\label{eq:def_Df_H}
|Df|_{H}\coloneqq \bigwedge\Big\{G\in L^p(\mm)^+\;\Big|\;G\in {\rm RS}(f)\Big\}.
\end{equation}
\end{definition}

\begin{definition}[Sobolev space \(H^{1,p}(\X)\)]\label{def:H_Sobolev_space}
Let \((\X, \sfd, \mm)\) be a metric measure space, let \(p\in [1,\infty)\) and \(f\in L^p(\mm)\).
The  \textbf{relaxation-type \(p\)-Sobolev space} \(H^{1,p}(\X)\) is defined as 
\[
H^{1,p}(\X)\coloneqq \{f\in L^p(\mm) \mid {\rm RS}(f)\neq \varnothing\}.
\]
Then the norm on \(H^{1,p}(\X)\) is given by
\[
\|f\|_{H^{1,p}(\X)}=\biggl(\|f\|_{L^p(\mm)}^p+\||Df|_{H}\|^p_{L^p(\mm)}\biggr)^{1/p}\qquad\forall f\in H^{1,p}(\X).
\]
\end{definition}
\begin{remark}\label{rem:def_H_Sobolev}
{\rm Some observations about Definition~\ref{def:relaxed slope} and Definition~\ref{def:H_Sobolev_space} 
are in order:
\begin{itemize}
     \item [i)] The convexity of $f\mapsto\lip_a(f)$ grants that ${\rm RS}(f)$ is a convex set, that $H^{1,p}(\X)$ is a vector space,
     and that $\|f\|_{H^{1,p}(\X)}$ is a norm. In addition, still convexity of $\lip_a(f)$ in combination with
     Mazur's lemma ensure that we can assume in \eqref{eq:def_p_relaxed_grad} that $\lip_a(f_n)\leq G_n'$ with $G_n'$ strongly
     convergent in $L^p(\mm)$ to $G'$. This refined convergence property, together with a diagonal argument, ensures the
     validity of the implication
     \begin{equation}\label{eq:diagonal_fundamental}
     \lim_{k\to +\infty}\|f_k-f\|_{L^p(\mm)}=0,\quad G_k\in {\rm RS}(f_k),\quad\lim_{k\to\infty}\|G_k-G\|_{L^p(\mm)}=0
     \Longrightarrow G\in {\rm RS}(f).
     \end{equation} 
    \item [ii)] From \eqref{eq:diagonal_fundamental} with $f_k=f$ we obtain the $L^p(\mm)$-closedness of ${\rm RS}(f)$. 
    In addition, ${\rm RS}(f)$ is a lattice (see Lemma~4.4 of \cite{Amb:Gig:Sav:14} for a proof, made in the case $p=2$, but extendable to
     all cases $p\in [1,\infty)$).
    \item [iii)] Completeness of the normed space \(\big(H^{1,p}(\X),\|\cdot\|_{H^{1,p}(\X)}\big)\) easily follows by \eqref{eq:diagonal_fundamental},
    considering a ``refined'' Cauchy sequence $(f_k)_k$ with the property
    \[\sum_k \| f_{k+1} - f_k \|_{ L^{p}( \mm ) } + \| |D(f_{k+1}-f_k)|_H \|_{L^p(\mm)}<+\infty.\]
   \end{itemize}
   }
\end{remark}

\begin{lemma} \label{lem_distinguished} Let $p\in [1,\infty)$ and $f\in H^{1,p}(\X)$.
Then $|Df|_H$ is the element with minimal $L^p(\mm)$ norm in
${\rm RS}(f)$ and there exist $(f_k)_k\subseteq\LIP_{bs}(\X)$ such that \(f_k\to f\) and \(\lip_a(f_n)\to|Df|_H\) strongly in \(L^p(\mm)\).
\end{lemma}  
\begin{proof} Since in $\sigma$-finite measure spaces the essential infimum is achieved by the infimum of a countable
subfamily, from the lattice property of ${\rm RS}(f)$ and $L^p(\mm)$-closure we obtain that $|Df|_H$ belongs to
${\rm RS}(f)$, while the property of being the element with minimal $L^p(\mm)$-norm is obvious from the definition.

In connection with the refined convergence property, let $(f_k)_k\subseteq\LIP_{bs}(\X)$ be convergent in $L^p(\mm)$ with
$\lip_a(f_k)\leq G_k$ and $G_k$ strongly convergent in $L^p(\mm)$  to $|Df|_H$. We claim that 
$\lip_a(f_k)$ strongly converges in $L^p(\mm)$ to $|Df|_H$ as well. Indeed, assuming with no loss
of generality that $\lip_a(f_k)$ weakly converges in $L^p(\mm)$ to $L$, from the inequality $L\leq |Df|_H$
and the minimality of $|Df|_H$, we obtain $L=|Df|_H$. In the case $p>1$, from
$$
\lims_{k\to\infty}\int\lip_a^p(f_k)\,\d\mm\leq\int |Df|_H^p\,\d\mm
$$
and the uniform convexity of the $L^p$ norm, we obtain strong convergence. In the case $p=1$, since 
$\lip_a(f_k)$ weakly converge to $|Df|_H$, we can assume with no loss of generality (possibly multiplying $f_k$ by
normalization constants $c_k$ convergent to 1) that $\int\lip_a(f_k)\,\d\mm=\int |Df|_H\,\d\mm$. Then
$$
\int |\lip_a(f_k)-|Df|_H|\,\d\mm=\frac 12\int\bigl(\lip_a(f_k)-|Df|_H\bigr)^+\,\d\mm\leq
\frac 12\int\bigl(G_k-|Df|_H\bigr)^+\,\d\mm\rightarrow 0.
$$ 
The claim follows.
\end{proof} 

In the case \(p\in (1,\infty)\) one can formulate the above definition of Sobolev space in terms of the domain of finiteness of 
an appropriate energy functional (known as Cheeger energy functional). Notice that if the support of $\mm$ coincides
with the whole space there is a unique $\LIP_{bs}(\X)$ representative in the Lebesgue equivalence class and then the inf in \eqref{eq:pre_energy}
is not necessary.
\begin{definition}[Cheeger energy functional]
Let \((\X,\sfd,\mm)\)
be a metric measure space and let \(p\in (1,\infty)\).
We define the \textbf{pre-Cheeger energy functional} \(\mathcal E_{p,\lip}\colon L^p(\mm)\to [0,\infty)\) as
\begin{equation}\label{eq:pre_energy}
    \mathcal E_{p,\lip}(f)\coloneqq \inf \left\{\frac{1}{p}\int 
    \lip_a^p(\bar f)\,\d \mm\Big |\, \bar f\in \LIP_{bs}(\X)\,\text{ and }\,f\in \pi_\mm(\bar f)\right\}
    \quad\text{ for every } f\in L^p(\mm),
   \end{equation}
   and the \textbf{Cheeger energy functional} \(\mathcal E_p\colon L^p(\mm)\to [0,\infty]\) as
\begin{equation}\label{eq:Cheeger_energy}
    \mathcal E_p(f)\coloneqq \inf\Big\{\limi_{n\to +\infty}\,
    \mathcal E_{p,\lip}(f_n)\,\Big |\;(f_n)_n\subseteq L^p(\mm), f_n\to f\text{ in }L^p(\mm)\Big\} \quad \text{ for every }f\in L^p(\mm).
\end{equation}
\end{definition}
\begin{proposition}\label{prop:H_finit_Ep}
Let \((\X, \sfd,\mm)\) be a metric measure space and let \(p\in (1,\infty)\). Then it holds that 
\[H^{1,p}(\X)={\rm Dom}(\mathcal E_p)\quad \text{ and }\quad \mathcal E_p(f)=\frac{1}{p}\int |Df|^p_H\,\d\mm.\]    
\end{proposition}
\begin{proof} 
The inclusion $\supseteq$ follows by the weak sequential compactness of (closed bounded subsets of) $L^p(\mm)$
(here we use that $p>1$), while the inequality $\geq$ follows by the lower semicontinuity of the norm under weak convergence
(recall that $|Df|_H$ is the least element in ${\rm RS}(f)$).
The inclusion $\subseteq$ follows by the very definition of relaxed $p$-upper slope, as weakly convergent
sequences are norm bounded, while the inequality $\leq$ follows by Lemma~\ref{lem_distinguished}.
\end{proof}
We also have notions of \(p\)-relaxed slopes obtaining via relaxation of functionals involving upper gradients.
\begin{definition}\label{def:upper_grad}
Let \((\X,\sfd)\) be a complete and separable metric space and let \(f\colon \X\to \R\) be a function.
A Borel function \(\rho\colon \X\to [0,\infty]\) is said to be an \textbf{upper gradient} of \(f\) if 
\[
|f(\gamma_{b_\gamma})-f(\gamma_{a_\gamma})|\leq \int_\gamma\rho\,\d s\quad \text{ holds for every } \gamma\in \mathscr R(\X).
\]
We will denote by \({\rm ug}(f)\) the (possibly empty) set of all upper gradients of the function \(f\).
\end{definition}

The next result follows from \cite[Lemma 6.2.6]{HKST:15}.

\begin{lemma}[Upper gradients of Lipschitz functions]\label{lem:ug_lipschitz}
Let \((\X,\sfd)\)  be a complete and separable metric space. Then for every \(f\in \LIP(\X)\) the function
$\lip(f)$ belongs to  ${\rm ug}(f)$. In particular, the same is true for $\lip_a(f)$.
\end{lemma}

\begin{definition}[More general relaxed slopes] Let \((\X,\sfd, \mm)\) be a metric measure space and \(p\in [1,\infty)\). 
We say that a function \(f\in L^p(\mm)\) admits a
\begin{itemize}
    \item [1)] \((\lip)\)-relaxed $p$-upper slope \(G\in L^p(\mm)^+\) if there exist \((f_n)_n\subseteq\LIP_{bs}(\X)\) and $G'\leq G$
    such that 
    \[f_n\to f\text{ in }L^p(\mm)\quad \text{ and }\quad \lip(f_n)\rightharpoonup G'\text{ weakly in }L^p(\mm);\]
    \item [2)] \(({\rm ug})\)-relaxed $p$-upper slope \(G\in L^p(\mm)^+\) if there exist \((f_n)_n\subseteq\LIP_{bs}(\X)\), \(\rho_n\in {\rm ug}(f_n)\) 
    and $G'\leq G$ such that 
    \[f_n\to f\text{ in }L^p(\mm)\quad \text{ and }\quad \rho_n\rightharpoonup G'\text{ weakly in }L^p(\mm).\]
As in \Cref{rem:def_H_Sobolev} iii) the weak convergences above to $G'$ can be equivalently asked to be strong in 
\(L^p(\mm)\) for functions $G_n'$ larger than $\lip(f_n)$ or $\rho_n$, respectively. 
\end{itemize}
\end{definition}
   
Recalling the inequality \(\lip(f)\leq \lip_a(f)\) for every \(f\in \LIP_{bs}(\X)\) and the fact that
\(\lip_a(f)\in {\rm ug}(f)\) (see Lemma \ref{lem:ug_lipschitz}), we get the following result.
\begin{lemma}\label{lem:lip_ug_slope}
Let \((\X,\sfd,\mm)\) be a metric measure space and \(p\in [1,\infty)\).  Every \(f\in H^{1,p}(\X)\) admits both \((\lip)\)-
and \(({\rm ug})\)-relaxed $p$-upper slopes.
\end{lemma}
We will see in Section \ref{sec:Equivalence} that for a function \(f\in L^p(\mm)\) having \((\lip)\)- or \(({\rm ug})\)-relaxed $p$-upper slope
is equivalent to being an element of the Sobolev space.
\subsection{W-approach: via integration-by-parts}\label{sec:W_def}
Next, we present the definition of the Sobolev space via the W-approach, by means of the notion of Di Marino's derivation presented in Section \ref{sec:H_def}.

\begin{definition}[The space \(W^{1,p}(\X)\)]\label{def:W_Sobolev_space}
Let \((\X,\d,\mm)\) be a metric measure space and  \(p\in [1,\infty)\). We say that a function \(f\in L^p(\mm)\) belongs
to the \textbf{integration-by-parts \(p\)-Sobolev space} \(W^{1,p}(\X)\) if there exists a 
\(\LIP_{bs}(\X)\)-linear map \({\sf L}_f\colon {\rm Der}^q_q(\X)\to L^1(\mm)\), continuous with respect to the ${\rm Der}^q(\X)$ norm, such that
\begin{equation}\label{eq:L_f}
\int {\sf L}_f(b)\,\d \mm=-\int f \div(b)\,\d\mm\quad \text{ holds for all }b\in {\rm Der}^q_q(\X).
\end{equation}
\end{definition}
\begin{remark}
{\rm Note that, whenever it exists, the map \({\sf L}_f\) is uniquely determined: suppose there exists another
    \(\LIP_{bs}(\X)\)-linear map \(\tilde {\sf L}_f\) satisfying \eqref{eq:L_f}. 
    Now fix \(h\in \LIP_{bs}(\X)\) and note that, by \(\LIP_{bs}(\X)\)-linearity, it holds that 
    \[
    \int h{\sf L}_f(b)\,\d \mm=\int {\sf L}_f(hb)\,\d \mm=-\int f\div(hb)\,\d \mm \quad \text{ for all }b\in {\rm Der}^q_q(\X). 
    \]
    Since the same holds for  \(\tilde {\sf L}_f\), we have that \(\int h{\sf L}_f(b)\,\d\mm=\int h\tilde{\sf L}_f(b)\,\d\mm\) holds for every \(h\in \LIP_{bs}(\X)\). By the arbitrariness of \(h\in \LIP_{bs}(\X)\) we deduce that \({\sf L}_f(b)=\tilde{\sf L}_f(b)\) holds $\mm$-a.e. in \(\X\).
}
\end{remark}

Taking into account the above remark, given \(b\in {\rm Der}^q_q(\X)\) we denote, for ease of notation,
\[b(f)\coloneqq {\sf L}_f(b)\quad\text{ for every }f\in W^{1,p}(\X)\]
and define $|Df|_W$ as the $\mm$-a.e. smallest function $g$ satisfying 
\begin{equation}\label{eq:wg}
|{\sf L}_f(b)|\leq g|b|\quad\text{$\mm$-a.e. in $\X$ for all $b\in {\rm Der}^q_q(\X)$,}
\end{equation}
namely
\begin{equation}\label{eq:minimal_W}
|Df|_W\coloneqq \bigvee_{b\in {\rm Der}^q_q(\X)} \1^{\mm}_{\{|b|\neq 0\}}\frac{|{\sf L}_f(b)|}{|b|}.
\end{equation}
We consider the following lemma, establishing well-posedness of \eqref{eq:wg}.
\begin{lemma}
Let \((\X,\sfd,\mm)\) be a metric measure space and \(q\in(1,\infty]\). Let \(L\colon{\rm Der}^q_q(\X)\to L^1(\mm)\)
be a \(\LIP_{bs}(\X)\)-linear \(\|\cdot\|_q\)-continuous operator. Then there exists a function \(g\in L^p(\mm)^+\) such that
\[
|L(b)|\leq g|b|\quad\text{ for every }b\in{\rm Der}^q_q(\X).
\]
\end{lemma}
\begin{proof}
The closure \(\mathbb V\) of \({\rm Der}^q_q(\X)\) in \(({\rm Der}^q(\X),\|\cdot\|_q)\) is a Banach space.
Due to the density of \(\LIP_{bs}(\X)\) in \(L^\infty(\mm)\) as in the point 2) of Lemma~\ref{lem:lip_dense_linfty} and the fact that \({\rm Der}^q_q(\X)\) is a \(\LIP_{bs}(\X)\)-module, we have that \(\mathbb V\) is also an \(L^\infty(\mm)\)-submodule of \({\rm Der}^q(\X)\). Indeed, by \Cref{lemm:extension}, we have that
\begin{equation*}
    {\rm cl}_{ {\rm Der}^q(\X)}( L^{\infty}( \mm )\text{-span of } \mathrm{Der}^{q}_{q}( \X ) )
    =
    {\rm cl}_{{\rm Der}^q(\X)}( \mathrm{Der}^{q}_{q}( \X ) )
    =
    \mathbb{V}.
\end{equation*}
Denote
\(\mathbb V_1\coloneqq\{b\in\mathbb V\,:\,|b|\leq 1\}\) for brevity. Given any derivation \(b\in\mathbb V_1\), we define
\[
\mm_b(E)\coloneqq\int L(\1_E^\mm b)\,\d\mm\quad\text{ for every }E\in\mathscr B(\X).
\]
Notice that \(|\mm_b(E)|\leq C\big(\int_E|b|^q\,\d\mm\big)^{1/q}\), where \(C\geq 0\) is the operator norm of \(L\).
By dominated convergence theorem, it follows that \(\mm_b\) is a finite signed Borel measure with \(\mm_b\ll\mm\). Then define
\[
g\coloneqq\bigvee_{b\in\mathbb V_1}\frac{\d\mm_b}{\d\mm}\in L^0_{\rm ext}(\mm)^+.
\]
First, we check that \(g\in L^p(\mm)\). We claim that there exists \((b_n)_n\subseteq\mathbb V_1\) such that
\(\frac{\d\mm_{b_n}}{\d\mm}\geq\big(1-\frac{1}{n}\big)g\) for every \(n\in\N\). To prove it, take a sequence
\((b_n^i)_i\subseteq\mathbb V_1\) and a Borel partition \((E_n^i)_i\) of \(\X\) such that
\(\sum_{i\in\N}\1_{E_n^i}^\mm\frac{\d\mm_{b_n^i}}{\d\mm}\geq\big(1-\frac{1}{n}\big)g\). Routine verifications show that,
letting \(b_n(f)\coloneqq\sum_{i\in\N}\1_{E_n^i}^\mm b_n^i(f)\) for every \(f\in\LIP_{bs}(\X)\), we obtain
a derivation \(b_n\in\mathbb V_1\) satisfying \(\frac{\d\mm_{b_n}}{\d\mm}\geq\big(1-\frac{1}{n}\big)g\), as desired.
Notice that \(\int h\frac{\d\mm_{b_n}}{\d\mm}\,\d\mm=\int L(hb_n)\,\d\mm\) holds for every \(h\in L^q(\mm)\cap L^\infty(\mm)\):
the case where \(h\) is a simple function is a direct consequence of the definition of \(\mm_{b_n}\), whence the general
case follows from the weak\(^*\) density of simple functions in \(L^\infty(\mm)\). In particular,
\(\big|\int h\frac{\d\mm_{b_n}}{\d\mm}\,\d\mm\big|\leq C\|hb_n\|_q\leq C\|h\|_{L^q(\mm)}\) for every \(h\in L^q(\mm)\cap L^\infty(\mm)\).
Since \(L^q(\mm)\cap L^\infty(\mm)\) is dense in \(L^q(\mm)\), and \(L^q(\mm)\) is the dual of \(L^p(\mm)\), we conclude that
\(\big(\int g^p\,\d\mm\big)^{1/p}\leq\lim_n\frac{n}{n-1}\big\|\frac{\d\mm_{b_n}}{\d\mm}\big\|_{L^p(\mm)}\leq C\),
thus in particular \(g\in L^p(\mm)\).

Finally, given any \(b\in{\rm Der}^q_q(\X)\), we set \(\phi_k\coloneqq|b|\wedge k\in L^\infty(\mm)\)
and \(\psi_k\coloneqq\1_{\{|b|\geq 1/k\}}^\mm/|b|\in L^\infty(\mm)\) for every \(k\in\N\). Notice that
\(b^k\coloneqq\psi_k b\in\mathbb V_1\). Hence, by dominated convergence theorem we get
\[\begin{split}
\int_E L(b)\,\d\mm&=\lim_{k\to\infty}\int_E\1_{\{|b|\geq 1/k\}}^\mm\frac{|b|\wedge k}{|b|}L(b)\,\d\mm
=\lim_{k\to\infty}\int_E\phi_k L(b^k)\,\d\mm=\lim_{k\to\infty}\int_E\phi_k\frac{\d\mm_{b^k}}{\d\mm}\,\d\mm\\
&\leq\lim_{k\to\infty}\int_E\phi_k g\,\d\mm\leq\int_E g|b|\,\d\mm\quad\text{ for every }E\in\mathscr B(\X),
\end{split}\]
whence it follows that \(L(b)\leq g|b|\). Replacing \(b\) with \(-b\), we conclude that \(|L(b)|\leq g|b|\).
\end{proof}
\subsection{B-approach: via a notion of plan}\label{sec:B_def}
The definition of the Sobolev space we present next is the one obtained via the B-approach. 
Namely, we look at the behaviour of the function along curves selected by using notions of plans.
\begin{definition}[\(B\)-weak \(p\)-upper gradient]\label{def:B_upper_gradients}
Let \((\X,\sfd,\mm)\) be a metric measure space, \(p\in[1,\infty)\) and \(f\in L^p(\mm)\). Then we say that a function \(G\in L^p(\mm)^+\) is a
B-\textbf{weak \(p\)-upper gradient} of \(f\) provided 
for any \(q\)-test plan \(\ppi\) on \(\X\), it holds that 
\(
f\circ\gamma\in W^{1,1}(0,1)\) for \(\ppi\)-a.e.\ curve \(\gamma\) and
\begin{equation}\label{eq:def_wug}
|( f\circ\gamma)'_t|\leq G(\gamma_t)|\dot\gamma_t|,\quad\text{ for }(\ppi\otimes\mathcal L_1)\text{-a.e.\ }(\gamma,t)\in
C([0,1];\X)\times (0,1).
\end{equation}
We define the set of all \(B\)-weak \(p\)-upper gradients of \(f\) as
\begin{equation}\label{eq:WUG_B}
{\rm WUG}_B(f)\coloneqq \Big\{G\in L^p(\mm)^+\;\Big|\;G\text{ is a B-weak }p\text{-upper gradient of }f\Big\}
\end{equation}
and, whenever ${\rm WUG}_B(f)\neq\varnothing$, we denote by
\(|Df|_B\in L^p(\mm)^+\) the \textbf{minimal} B-\textbf{weak \(p\)-upper gradient} of \(f\), defined as
\begin{equation}\label{eq:minimal_B}
|Df|_B\coloneqq\bigwedge\Big\{G\in L^p(\mm)^+\;\Big|\;G\in {\rm WUG}_B(f)\Big\}.
\end{equation}
\end{definition}

\begin{definition}[Sobolev space \(B^{1,p}(\X)\)]\label{def:BL_space}
Let \((\X,\sfd,\mm)\) be a metric measure space and \(p\in[1,\infty)\).
The \textbf{Beppo Levi \(p\)-Sobolev space} \(B^{1,p}(\X)\) is defined as the space of all functions
 in \(L^p(\mm)\) admitting a B-weak \(p\)-upper gradient.
The norm on \(B^{1,p}(\X)\) is given by
\[
\|f\|_{B^{1,p}(\X)}\coloneqq\Big(\|f\|_{L^p(\mm)}^p+\big\||Df|_B\big\|_{L^p(\mm)}^p\Big)^{1/p},\quad\text{ for every }f\in B^{1,p}(\X).
\]
\end{definition}
\begin{remark}{\rm
Some comments on Definition~\ref{def:BL_space} are in order:
\begin{itemize}
\item[i)]
By the bounded compression property of $\ppi$ and Lemma~\ref{lem:ppi_bdd_compr}, the validity of the properties \eqref{eq:def_wug} and $ f\circ\gamma\in W^{1,1}(0,1)$ for $\ppi$-a.e. $\gamma$ are in fact independent of the chosen Borel representatives in the $L^p(\mm)$ class. For this reason, they have been stated directly at the $L^p$ level, rather than at the $\mathcal L^p$ level.
A similar discussion applies to items ii), iii), and iv) of Theorem \ref{thm:equiv_BL} below.
\item[ii)] The set ${\rm WUG}_B(f)$ is easily seen to be a lattice. Hence, arguing as in Lemma~\ref{lem_distinguished}
we obtain that the minimal B-weak \(p\)-upper gradient \(|Df|_B\) is indeed a B-weak \(p\)-upper gradient of \(f\).
\item[iii)] One can easily infer from the definition that \(\big(B^{1,p}(\X),\|\cdot\|_{B^{1,p}(\X)}\big)\) is a normed space.
In fact, it is also complete, but we do not prove it, as it will follow from Theorem \ref{thm:equivalence_Sobolev_spaces}.
\end{itemize}
}\end{remark}
\begin{theorem}[Properties of \(B^{1,p}(\X)\)]\label{thm:prop_BL}
Let \((\X,\sfd,\mm)\) be a metric measure space, \(p\in[1,\infty)\), \(f\in B^{1,p}(\X)\) and \(\ppi\) a $q$-test plan. Then
$$
\big|f(\gamma_1)-f(\gamma_0)\big|\leq\int_0^1 |Df|_B(\gamma_t)|\dot\gamma_t|\,\d t
\qquad\text{for $\ppi$-a.e. $\gamma\in C([0,1];\X)$}.
$$
\end{theorem}
\begin{proof}
Without loss of generality we can assume that all curves in the support of $\ppi$ are contained
in a bounded subset of $\X$, so that the bounded compression property ensures the $L^1(\ppi)$ integrability of $\gamma\mapsto f(\gamma_t)$,
$t\in [0,1]$. First of all, it is sufficient to prove the integrated form of the inequality, namely
$$
\int \big|f(\gamma_1)-f(\gamma_0)\big|\,\d\ppi(\gamma)\leq\int\int_0^1 |Df|_B(\gamma_t)|\dot\gamma_t|\,\d t\,\d\ppi(\gamma).
$$
Indeed, by applying the integrated form to any restricted test plan as in Proposition~\ref{prop:restr_x_ppi}, one obtains the inequality as stated. Set now 
$$
a_\varepsilon(\gamma)\coloneqq\frac 1\varepsilon\int_0^\varepsilon f(\gamma_s)\,\d s,\qquad
b_\varepsilon(\gamma)\coloneqq\frac 1 \varepsilon\int_{1-\varepsilon}^1 f(\gamma_s)\,\d s\quad \text{ for }\ppi\text{-a.e.\ } \gamma
$$
and notice that, since $f\circ\gamma\in W^{1,1}(0,1)$ for $\ppi$-a.e. $\gamma$, one has
$$
\int \big|a_\varepsilon(\gamma)-b_\varepsilon(\gamma)\big|\,\d\ppi(\gamma)\leq
\int\int_0^1 |(f\circ\gamma)'|\,\d t\,\d\ppi(\gamma)\leq\int\int_0^1 |Df|_B(\gamma_t)|\dot\gamma_t|\,\d t\,\d\ppi(\gamma).
$$
Hence, to conclude, it is sufficient to prove that $a_\varepsilon\to f(\gamma_0)$ and $b_\varepsilon\to f(\gamma_1)$ in $L^1(\ppi)$,
respectively. We prove it for $a_\varepsilon$, as the proof for $b_\varepsilon$ is analogous, as a consequence of the bounded
compression property. Indeed,
$$
\int |a_\varepsilon(\gamma)-f(\gamma_0)|\,\d\ppi(\gamma)\leq
\frac{1}{\varepsilon}\int |f(\gamma_\varepsilon)-f(\gamma_0)|\,\d\ppi(\gamma),
$$
so the operators $f\mapsto \int |f(\gamma_\varepsilon)-f(\gamma_0)|\,\d\ppi(\gamma)$ are uniformly bounded in $L^1(\mm)$ and
obviously converge to $0$ when $f\in \LIP_{bs}(\X)$. By density, they converge to $0$ in $L^1(\mm)$.
\end{proof}

\begin{theorem} [Equivalent characterization of $B^{1,p}(\X)$]\label{thm:equiv_BL}
Let \((\X,\sfd,\mm)\) be a metric measure space, \(p\in[1,\infty)\), $f\in L^p(\mm)$ and $G\in L^p(\mm)^+$. Then $f\in B^{1,p}(\X)$
and $|Df|_B\leq G$ $\mm$-a.e. in $X$ if and only if 
\begin{equation}\label{eq:weak_ine}
f(\gamma_1)-f(\gamma_0)\leq\int_0^1 G(\gamma_t)|\dot\gamma_t|\,\d t\qquad\text{for $\ppi$-a.e. $\gamma\in C([0,1];\X)$}
\end{equation}
for any $q$-test plan $\ppi$.
\end{theorem}
\begin{proof} 
One implication is obvious by Theorem~\ref{thm:prop_BL}. By applying \eqref{eq:weak_ine} to the restricted $q$-test
plans as in Proposition~\ref{prop:restr_t_ppi}, we obtain
$$
f(\gamma_t)-f(\gamma_s)\leq\int_s^t G(\gamma_r)|\dot\gamma_r|\,\d r,\quad\text{$\ppi$-a.e.\ $\gamma\in C([0,1];\X)$, for all }(s,t)\in T,
$$
where we denote \(T\coloneqq\big\{(s,t)\in[0,1]\,:\,s<t\big\}\). Hence, by Fubini's theorem, for \(\ppi\)-a.e.\ \(\gamma\) we have that
\begin{equation}\label{eq:equiv_BL_aux2}
f(\gamma_t)-f(\gamma_s)\leq\int_s^t G(\gamma_r)|\dot\gamma_r|\,\d r,\quad\text{ for }\mathcal L^2\text{-a.e.\ }(s,t)\in T.
\end{equation}
By applying the same inequality to the image of $\ppi$ under the map $\gamma_t\mapsto\gamma_{1-t}$, we easily obtain
that the same inequality holds $\ppi$-a.e. with $|f(\gamma_t)-f(\gamma_s)|$ in the left-hand side. Hence, from Lemma~\ref{lem:AGS13_Lemma2.1}
it follows that \(f\circ\gamma\in W^{1,1}(0,1)\) for \(\ppi\)-a.e.\ \(\gamma\) and that
\[
|(f\circ\gamma)'_t|\leq G(\gamma_t)|\dot\gamma_t|,\quad\text{ for }(\ppi\otimes\mathcal L_1)\text{-a.e.\ }(\gamma,t).
\]
Therefore, the proof is complete.
\end{proof}
\subsection{N-approach: via a notion of modulus}\label{sec:N_def}
The last definition of the Sobolev space we present is the one obtained via the N-approach, the so-called Newtonian Sobolev space. 
Namely, we look at the behaviour of the function along curves selected by using notions of modulus.

\begin{definition}[N-weak $p$-upper gradient]\label{def:WUGN}
Let \((\X,\sfd,\mm)\) be a metric measure space and \(p\in[1,\infty)\). 
A function $f \colon \X \to [-\infty,\infty]$ has an \textbf{N-weak p-upper gradient} if $f(\gamma_{a_\gamma})$ and $f(\gamma_{b_\gamma})$ are finite for $\Modp$-a.e.\ nonconstant curve $\gamma \in \mathscr{R}( \X )$ and there is a Borel function \( \rho\in \mathcal L^p_{\rm ext}(\mm)^+\) such that
\begin{equation}\label{eq:uppergradient}
    |f(\gamma_{b_\gamma}) -  f(\gamma_{a_\gamma})|
    \leq
    \int_\gamma \rho\,\d s
    \quad\text{for \(\Modp\)-a.e.\ nonconstant curve \(\gamma\in \mathscr R(\X)\)}.
\end{equation}
Under these conditions we say that $\rho$ is an \textbf{N-weak \(p\)-upper gradient of \(f\)}.
We define the set of N-weak \(p\)-upper gradients of \(f\) as
\begin{equation}\label{eq:WUG_N}
\mathrm{WUG}_{N}(f)\coloneqq 
\big\{
\rho\in \mathcal L^p_{\rm ext}(\mm)^+\big|\, \rho \text{ is an }N\text{-weak } p\text{-upper gradient of } f
\big\}.
\end{equation}
Whenever $\mathrm{WUG}_N(f)\neq\varnothing$, we denote by 
\(\rho_{f}\in \mathcal L^p_{\rm ext}(\mm)^+\) 
any element of \({\rm WUG}_{N}(f)\) of minimal \(\mathcal L^p(\mm)\)-seminorm.
Then we define the \textbf{minimal N-weak \(p\)-upper gradient} of \(f\) as
\begin{equation}\label{eq:minimal_N}
|Df|_N\coloneqq \pi_{\mm}(\rho_{f})\in L^p(\mm)^+.
\end{equation}
\end{definition}
\begin{remark}\label{rem:WUGN}
{\rm
Some comments on Definition~\ref{def:WUGN} are in order:
\begin{itemize}
    \item [i)] The set \({\rm WUG}_{N}(f)\) is a closed convex lattice in \(\mathcal L^p_{\rm ext}(\mm)^+\) and thus 
    admits a unique (up to \(\mm\)-a.e.\ equality) element of \(\mathcal L^p(\mm)\)-minimal seminorm, therefore $|Df|_N$ is 
    uniquely determined as an element of $L^p(\mm)$. The lattice property ensures also that such an element is minimal in the 
    \(\mm\)-a.e.\ sense. We prove these statements below; see also \cite[Theorem 6.3.20]{HKST:15}.
    \item [ii)] The minimal \(N\)-weak \(p\)-upper gradient satisfies several calculus rules which we formulate below in 
    \Cref{thm:calc_rules_mwug:newt}. For example, if $f, g \colon \X \rightarrow [-\infty, \infty]$ are $\mm$-measurable with 
    $\mathrm{WUG}_N( f )$ and $\mathrm{WUG}_N( g )$ non-empty, then $\mathrm{WUG}_{N}(f-g) \neq \varnothing$ and $\rho_{(f-g)} = 0$ and \(\rho_f = \rho_g\) \(\mm\)-almost everywhere in $\left\{ f = g \right\}$.
   \item [iii)] Since constant curves have infinite modulus, our finiteness assumption on $f$ at the endpoint of $\Modp$-almost every nonconstant curve does not entail the global finiteness of $f$ but, rather, that the set where $f$ is infinite is $p$-exceptional according to Definition~\ref{def:exceptional}. Moreover, notice that $p$-integrability of $\rho$ entails that the right-hand side in \eqref{eq:uppergradient} is finite as well for $\Modp$-a.e.\ curve $\gamma$.
\item[\(\rm iv)\)] We highlight that in Definition~\ref{def:WUGN} we did not make any measurability assumption on the function \(f\).
In some cases, for example on doubling metric measure spaces supporting a weak \(p\)-Poincar\'{e} inequality (see
\cite[Theorem 1.11]{Ja:Ja:Ro:Ro:Sha:07}), it holds that
\begin{equation}\label{eq:WUG_nonempty_meas}
{\rm WUG}_N(f)\neq\varnothing\quad\Longrightarrow\quad f\text{ is }\mm\text{-measurable.}
\end{equation}
However, this implication is not valid on every metric measure space. For example, let \((\X,\sfd,\mm)\) be a metric
measure space where non-\(\mm\)-measurable functions exist and consider the `snowflake distance' \(\sqrt\sfd\); then
\((\X,\sqrt\sfd,\mm)\) is a metric measure space where all rectifiable curves are constant, thus every function
\(f\) has the zero function as an \(N\)-weak \(p\)-upper gradient and in particular \eqref{eq:WUG_nonempty_meas} fails. In Remark
\ref{rmk:example_cone}, we construct a more sophisticated example of a space where \eqref{eq:WUG_nonempty_meas} fails and
any two points are joined by a rectifiable curve.
\end{itemize}}
\end{remark}
\begin{remark}\label{rmk:example_cone}{\rm
We endow \(C\coloneqq\big\{(\lambda,t)\in\R^2\;\big|\;0\leq t\leq\lambda\leq 1\big\}\) with the unique distance such that
\[
\sfd((\lambda_1,t_1),(\lambda_2,t_2))^2=\lambda_1^2+\lambda_2^2-2\lambda_1\lambda_2\cos\sqrt{|\lambda_1^{-1}t_1-\lambda_2^{-1}t_2|}
\quad\forall(\lambda_1,t_1),(\lambda_2,t_2)\in C\setminus\{(0,0)\}.
\]
Notice that \(\sfd\) induces the Euclidean topology on \(C\), that \((C,\sfd)\) is isometric to a (truncated) Euclidean cone over the interval \([0,1]\) equipped with
the `snowflake distance' \((t_1,t_2)\mapsto\sqrt{|t_1-t_2|}\), and that any two points of \(C\) can be joined by a rectifiable curve (passing through the tip \((0,0)\) of the cone).
Finally, let us endow the space \((C,\sfd)\) with the restriction \(\mm\) of the \(2\)-dimensional Lebesgue measure.
We claim that the space \((C,\sfd,\mm)\) does not satisfy \eqref{eq:WUG_nonempty_meas}. To prove it, fix some non-\(\mathcal L^1\)-measurable function \(f\colon[0,1]\to[0,1]\). Let us define
\[
F(\lambda,t)\coloneqq\lambda f(\lambda^{-1}t)\quad\text{ for every }(\lambda,t)\in C\setminus\{(0,0)\}
\]
and \(F(0,0)\coloneqq 0\). Since \(F(\lambda,\cdot)\colon[0,\lambda]\to[0,\lambda]\) is non-\(\mathcal L^1\)-measurable for
every \(\lambda\in(0,1]\), we have that the function \(F\colon C\to[0,1]\) is not \(\mm\)-measurable. Nevertheless, one can
readily check that for any \(p\in[1,\infty)\) the function \(F\) admits the constant function \(1\) as an \(N\)-weak \(p\)-upper gradient.
}\end{remark}
We define two spaces using functions with $N$-weak $p$-upper gradients.

\begin{definition}[Dirichlet spaces \( \overline{D}^{1,p}(\X)\), \(D^{1,p}(\X)\)]\label{def:D_Sobolev_space}
Let \((\X,\sfd, \mm)\) be a metric measure space and \(p\in [1,\infty)\). Let $\bar{D}^{1,p}(\X)$ denote the collection of $\mm$-measurable $f \colon \X \rightarrow [-\infty,\infty]$ with ${\rm WUG}_{N}( f ) \neq \varnothing$. Consider the function $\| \cdot \| \colon \bar{D}^{1,p}( \X ) \rightarrow [0,\infty)$ defined by
\begin{equation}\label{eq:seminorm}
    \| f \|_{ D^{1,p}(\X) } \coloneqq \| \rho_f \|_{ L^{p}(\mm) }.
\end{equation}
Identify $f, g \in \bar{D}^{1,p}(\X)$ if $\| f - g \|_{ D^{1,p}(\X) } = 0$. The associated quotient space is the \textbf{Dirichlet space} $D^{1,p}(\X)$.
\end{definition}
In the following section, cf. \Cref{sec:calculus_rules}, we establish calculus rules for Dirichlet functions guaranteeing that $D^{1,p}(\X)$ becomes a vector space and \eqref{eq:seminorm} induces a norm on the vector space.

When we include an integrability condition on the function $f$, we may define the pre-space of Newtonian \(p\)-Sobolev functions $\bar N^{1,p}(\X)$ and a seminorm $\|f\|_{N^{1,p}(\X)}$ in the vector space of $\mm$-measurable functions in such a way that $f=0$ $\mm$-a.e. in $\X$ implies $\|f\|_{N^{1,p}(\X)}=0$.
Hence, it makes sense to pass this definition and the norm to the quotient in $L^p(\mm)$, defining the Newtonian $p$-Sobolev space
$N^{1,p}(\X)$ (in a way consistent with the definition given in \cite[(7.1.26)]{HKST:15}).
\begin{definition}[Sobolev spaces \(\bar N^{1,p}(\X),\, N^{1,p}(\X)\)]\label{def:barN_Sobolev_space}
Let \((\X,\sfd, \mm)\) be a metric measure space and \(p\in [1,\infty)\). 
The space \(\bar N^{1,p}(\X)\) is defined as the space of 
\(\mm\)-measurable \(p\)-integrable functions \(f \colon \X \rightarrow \R\) for which \({\rm WUG}_N(f)\neq \varnothing\).
The seminorm on \(\bar N^{1,p}(\X)\) is given by
\[
\|f\|_{N^{1,p}(\X)}\coloneqq\Big(\|f\|_{L^p(\mm)}^p+\|f\|_{D^{1,p}(\X)}^p\Big)^{1/p}.
\]
The \textbf{Newtonian \(p\)-Sobolev space} $N^{1,p}(\X)\subseteq L^p(\mm)$ is the quotient
\begin{equation}\label{eq:def_N1p}
N^{1,p}(\X)\coloneqq \pi_{\mm}
\big(\bar N^{1,p}(\X)\big),
\end{equation}
keeping the same notation for the quotient norm $\|\cdot\|_{N^{1,p}(\X)}$ in $N^{1,p}(\X)$.
\end{definition}
\subsection{Bibliographical notes}
The content of this section is of expository nature, where we recall four different notions of metric Sobolev space present in the literature. More precisely: 
\begin{itemize}
\item The content of Section \ref{sec:H_def} is mainly taken from \cite{Amb:Gig:Sav:13}. In the latter paper, the H-approach presented in this section has been shown to be equivalent to the original Cheeger's definition in \cite{Ch:99}.
\item Section \ref{sec:W_def} mainly follows \cite{DiMaPhD:14}. 
Let us mention that this kind of approach has been employed also in \cite{BBS} to define the Sobolev space on the Euclidean space endowed with an arbitrary Radon measure.
\item The main references for Section \ref{sec:B_def} are \cite{Amb:Gig:Sav:13, Amb:Gig:Sav:14}. Another (equivalent in case \(p>1\)) notion has been recently introduced in \cite{Sav:22}, expressing the weak upper gradient property in terms of nonparametric plans with \(q\)-integrable barycenters.
\item We follow the standard reference \cite{HKST:15} in Section \ref{sec:N_def}; see also \cite{Bj:Bj:11}. The definition itself has been introduced in \cite{Sha:00}, after \cite{Kos:Mac:98}.
\end{itemize}
\subsection{List of symbols}
\begin{center}
\begin{spacing}{1.2}
\begin{longtable}{p{2.2cm} p{11.8cm}}
\({\rm RS}(f)\) & set of relaxed \(p\)-upper slopes of \(f\); Definition \ref{def:relaxed slope}\\
\(|Df|_H\) & minimal relaxed \(p\)-upper slope of \(f\); \eqref{eq:def_Df_H}\\
\(H^{1,p}(\X)\) & relaxation-type \(p\)-Sobolev space; Definition \ref{def:H_Sobolev_space}\\
\(\mathcal E_{p,\lip}\) & pre-Cheeger energy functional; \eqref{eq:pre_energy}\\
\(\mathcal E_p\) & Cheeger energy functional; \eqref{eq:Cheeger_energy}\\
\({\rm Dom}(\mathcal E_p)\) & finiteness domain of \(\mathcal E_p\); Proposition \ref{prop:H_finit_Ep}\\
\({\rm ug}(f)\) & set of upper gradients of \(f\); Definition \ref{def:upper_grad}\\
\(W^{1,p}(\X)\) & integration-by-parts \(p\)-Sobolev space; Definition \ref{def:W_Sobolev_space}\\
\(|Df|_W\) & minimal function associated to a function \(f\in W^{1,p}(\X)\); \eqref{eq:minimal_W}\\
\({\rm WUG}_B(f)\) & set of \(B\)-weak \(p\)-upper gradients of \(f\); \eqref{eq:WUG_B}\\
\(|Df|_B\) & minimal \(B\)-weak \(p\)-upper gradient of \(f\); \eqref{eq:minimal_B}\\
\(B^{1,p}(\X)\) & Beppo Levi \(p\)-Sobolev space; Definition \ref{def:BL_space}\\
\({\rm WUG}_N(f)\) & set of \(N\)-weak \(p\)-upper gradients of \(f\); \eqref{eq:WUG_N}\\
\(\rho_f\) & element of \({\rm WUG}_N(f)\) of minimal \(\mathcal L^p(\mm)\)-seminorm; Definition \ref{def:WUGN}\\
\(|Df|_N\) & minimal \(N\)-weak \(p\)-upper gradient of \(f\); \eqref{eq:minimal_N}\\
\(\bar{D}^{1,p}(\X)\) & \(\mm\)-measurable functions  with an \(N\)-weak \(p\)-upper gradient; Definition \ref{def:D_Sobolev_space}\\
\(D^{1,p}(\X)\) & Dirichlet space; Definition \ref{def:D_Sobolev_space}\\
\(\bar N^{1,p}(\X)\) & \(p\)-integrable functions with an \(N\)-weak \(p\)-upper gradient; Definition \ref{def:barN_Sobolev_space}\\
\(N^{1,p}(\X)\) & Newtonian \(p\)-Sobolev space; \eqref{eq:def_N1p}
\end{longtable}
\end{spacing}
\end{center}
\section{Calculus rules and fine properties of Newtonian Sobolev functions}\label{sec:calculus_rules}
In this section, we prove that minimal \(N\)-weak \(p\)-upper gradients satisfy the standard calculus rules, formulated in \Cref{thm:calc_rules_mwug:newt} below. In particular, this will justify the well-posedness 
of the definition of Newtonian Sobolev space (cf.\ Remark~\ref{rem:WUGN}) and its consistency with the one in \cite{HKST:15}.

We will prove calculus rules for Dirichlet functions and apply them to show some standard fine properties of Newtonian 
Sobolev functions in terms of the Sobolev capacity. The former are motivated by our follow-up work \cite{AILP} and the 
latter by the energy density of Lipschitz functions in the Sobolev space, a result we report from \cite{EB:20:published}.
\subsection{Calculus rules for Dirichlet functions}
\label{sec:calculus_rules_dirichlet}
We establish the following theorem in this section.
\begin{theorem}[Calculus rules for minimal \(N\)-weak \(p\)-upper gradients]\label{thm:calc_rules_mwug:newt}
Let \((\X,\sfd,\mm)\) be a metric measure space and \(p\in[1,\infty)\). Then the following properties hold:
\begin{itemize}
\item[\(\rm i)\)] Let $f \colon \X \rightarrow [-\infty,\infty]$ be $\mm$-measurable with $\mathrm{WUG}_{N}(f) \neq \varnothing$. 
Then
\[
    \rho_f = 0 \quad\mm\text{-a.e. on $\left\{ x \in \X \mid |f(x)| = +\infty \right\}$.}
\]
and
\[
    \rho_f = 0 \quad\mm\text{-a.e.\ on }f^{-1}(N)
\]
whenever a Borel set \(N\subseteq\R\) satisfies \(\mathcal L^1(N)=0\).
\item[\(\rm ii)\)] \textsc{Chain rule.} Let $f \colon \X \rightarrow \mathbb{R}$ be $\mm$-measurable with $\mathrm{WUG}_{N}(f) \neq \varnothing$. If $\varphi \colon \mathbb{R} \rightarrow \mathbb{R}$ is Lipschitz and $g = \varphi \circ f$, then $\mathrm{WUG}_{N}(g) \neq \varnothing$ and
\begin{equation*}
    \rho_g
    =
    ( \lip( \varphi ) \circ f ) \rho_f,
    \quad\text{$\mm$-almost everywhere.}
\end{equation*}
More generally, suppose that $f \colon \X \rightarrow [-\infty,\infty]$ is 
$\mm$-measurable with $\mathrm{WUG}_{N}(f) \neq \varnothing$. Then, given a Lipschitz $\varphi \colon \mathbb{R}\rightarrow \mathbb{R}$, consider
\begin{equation*}
    g(x)
    =
    \left\{
    \begin{aligned}
        &\varphi( f(x) ),
        \quad&&\text{if $f(x) \in \mathbb{R}$},
        \\
        &0,\quad&&\text{if $|f(x)| = +\infty$}.
    \end{aligned}
    \right.
\end{equation*}
Then $\mathrm{WUG}_{N}(g) \neq \varnothing$ and $\mm$-almost everywhere
\begin{equation*}
    \rho_g
    =
    \left\{
    \begin{aligned}
        &( \lip( \varphi ) \circ f ) \rho_f,
        \quad&&\text{in $\left\{ x \in \X \mid |f(x)| < \infty \right\}$},
        \\
        &0,
        \quad&&\text{in $\left\{ x \in \X \mid |f(x)| = +\infty\right\}$.}
    \end{aligned}
    \right.
\end{equation*}
\item[\(\rm iii)\)] \textsc{Locality property.} Let \(f,g \colon \X \rightarrow [-\infty, \infty]\) be $\mm$-measurable 
such that $\mathrm{WUG}_N( f ) \neq \varnothing $ and $ \mathrm{WUG}_N( g )\neq  \varnothing$. 
Then $\mathrm{WUG}_N( f-g )\neq\varnothing$ and
\begin{equation*}
    \rho_{ (f - g) } = 0
    \quad\text{$\mm$-almost everywhere on $\left\{ f = g \right\}$.}
\end{equation*}
\item[\(\rm iv)\)] \textsc{Leibniz rule.}  Let \(f,g \colon \X \rightarrow [-\infty, \infty]\) be 
$\mm$-measurable such that $\mathrm{WUG}_N( f ) \neq \varnothing$ and  
$ \mathrm{WUG}_N( g )\neq \varnothing$ and $|f| + |g| \leq C$ $\mm$-almost everywhere in $\X$. Then $\mathrm{WUG}_N( f g ) \neq \varnothing$ and
\[
    \rho_{fg}
    \leq
    |f|\rho_g
    +
    |g|\rho_f
    \quad\mm\text{-a.e.\ on }\X.
\]
\end{itemize}
\end{theorem}

The proof of \Cref{thm:calc_rules_mwug:newt} is split into several sublemmas some of which hold in a slightly more general setting. 
We prove \Cref{thm:calc_rules_mwug:newt} at the end of this section.
The key ingredient in the proof of \Cref{thm:calc_rules_mwug:newt} is the following lemma that functions with $N$-weak $p$-weak 
upper gradients are absolutely continuous along $\Modp$-almost every nonconstant absolutely continuous curve.

\begin{lemma}\label{lemm:integrationbyparts}
Let $p \in [1,\infty)$ and $f \colon \X \rightarrow [-\infty,\infty]$ with $\mathrm{WUG}_N( f ) \neq \varnothing$ and $\rho \in \mathrm{WUG}_N( f )$. Then there exists a $\Modp$-negligible curve family $\Gamma_0$ such that for every nonconstant $\gamma \in \mathscr R(\X) \setminus \Gamma_0$,
\begin{enumerate}
    \item $f \circ \gamma$ is continuous and rectifiable;
    \item the total variation measures $s_{ f \circ \gamma }$ of $f \circ \gamma$ and $s_\gamma$ of $\gamma$ satisfy
    \begin{equation}\label{eq:metricspeedcontrol}
        s_{ f \circ \gamma } \ll s_\gamma
        \quad\text{and}\quad
        \frac{ \d s_{ f \circ \gamma } }{ \d s_\gamma }
        \leq
        \rho \circ \gamma
        \in
        \mathcal L^1_{\rm ext}( s_\gamma )^+.
    \end{equation}
\end{enumerate}
\end{lemma}
\begin{proof}
Let $\Gamma_1$ be a $\Modp$-negligible family such that \eqref{eq:uppergradient} holds for every nonconstant $\gamma \in \mathscr R(\X) \setminus \Gamma_1$ for the pair $( f, \rho )$ and that $f( a_\gamma )$ and $f( b_\gamma )$ are finite. Then, by \Cref{lem:properties_Mod} (6), there there exists $h \in \mathcal L^p_{\rm ext}(\mm)^+$ for which $\rho \leq h$ everywhere and
\begin{equation*}
    \Gamma_1 
    \subseteq
    \left\{
        \gamma \in \mathscr C(\X)
        \mid
       \int_\gamma h \,\d s=+  \infty 
    \right\}
    \eqqcolon
    \Gamma_0.
\end{equation*}
Recall that $\Gamma_0$ is $\Mod_p$-negligible by \Cref{lem:properties_Mod} (5). We prove that $\Gamma_0$ is a curve family for which the conclusion of the Lemma holds.

To this end, observe that $\Gamma_0$ satisfies the following monotonicity principle: if $\gamma \in \mathscr{R}(\X)$ has a subinterval $I \subseteq [a_\gamma,b_\gamma]$ such that $\gamma|_{I} \in \Gamma_0$, then $\gamma \in \Gamma_0$. Indeed, by monotonicity of the path integration,
\begin{equation*}
   + \infty = \int_{ \gamma|_{I} } \rho \,\d s \leq \int_{ \gamma } \rho \,\d s.
\end{equation*}

Now, if $\gamma \in \mathscr{R}( \X ) \setminus \Gamma_0$, then $\gamma|_{I} \in \mathscr{R}( \X ) \setminus \Gamma_0 \subseteq \mathscr{R}( \X ) \setminus \Gamma_1$ holds for every interval $I \subseteq [a_\gamma, b_\gamma]$ for which $\ell( \gamma|_{I} ) > 0$. Hence, by definition of $\Gamma_1$ and the inclusion $\Gamma_1 \subset \Gamma_0$, \eqref{eq:uppergradient} holds for every subinterval $I \subseteq [a_\gamma,b_\gamma]$ with $\ell( \gamma|_{I} ) > 0$.

Consequently, whenever $[a, b] \subseteq [a_\gamma,b_\gamma]$ is such that $\ell( \gamma|_{ [a,b] } ) > 0$, we have that
\begin{equation*}
    | f( \gamma(b) ) - f( \gamma(a) ) | \leq \int_{ \gamma|_{ [a,b] } } \rho \,\d s < +\infty.
\end{equation*}
Consider a nonconstant $\gamma \in \mathscr{R}( \X ) \setminus \Gamma_0$ and its constant speed reparametrization $\gamma^{\sf cs} \colon [0,1] \rightarrow \X$. Here $\gamma^{\sf cs} \in \mathscr{R}( \X ) \setminus \Gamma_0$ by the invariance of path integral under reparametrization, cf. \Cref{lem:integral_rep_invariant}.

Then, for every partition $P$ of $[0,1]$, we have the total variation bound
\begin{equation*}
    V( f \circ \gamma^{\sf cs}; P )
    \leq
    \int_{ \gamma^{\sf cs} } \rho \,\d s
    <
    +\infty.
\end{equation*}
Taking the supremum over $P$ yields that $f \circ \gamma^{\sf cs}$ is bounded, continuous and rectifiable. In fact, for every 
subinterval $I \subseteq [0,1]$ of positive diameter,
\begin{equation*}
    \ell( f \circ \gamma^{\sf cs}|_{I} )
    \leq
    \int_{ \gamma^{\sf cs}|_I } \rho \,\d s
    \leq
    \int_{ \gamma^{\sf cs} } \rho \,\d s
    <
    +\infty.
\end{equation*} 
This implies that the total variation measure $s_{ f \circ \gamma^{\sf cs} }$ is absolutely continuous with respect to $s_{ \gamma^{\sf cs} }$, with density bounded from above by $\rho \circ \gamma \in \mathcal L^1_{\rm ext}( s_{ \gamma^{\sf cs}} )^+$. We may replace $\gamma^{\sf cs}$ by $\gamma$ in this last conclusion by Lemmas \ref{lem:csrep} and \ref{lem:integral_rep_invariant}.
\end{proof}

The following lemma implies that products of functions with $N$-weak $p$-upper gradients are well-defined along $p$-almost every rectifiable curve.

\begin{lemma}[Leibniz rule for paths]\label{lemm:newt:leibniz}
Consider $f_1, f_2 \colon \X \rightarrow [-\infty, \infty]$, $\rho_1, \rho_2 \colon \X \rightarrow [-\infty,\infty]$, and 
$\Gamma_0^{1}$ and $\Gamma_0^2$ are families such that $( f_i, \rho_i, \Gamma_0^i )$ satisfies the conclusion of 
\Cref{lemm:integrationbyparts} for $i = 1,2$. Then, for every nonconstant curve
$\gamma \in \mathscr R(\X) \setminus ( \Gamma_0^1 \cup \Gamma_0^2 )$, the integrals
\begin{gather*}
    T_\gamma( f_1, f_2 )
        \coloneqq
        \int_{ [a_\gamma, b_\gamma] }
            ( f_1 \circ \gamma )
        \,d\mu_{ f_2 \circ \gamma },
        \quad
    T_\gamma( f_2, f_1 )
        \coloneqq
        \int_{ [a_\gamma, b_\gamma] }
            ( f_2 \circ \gamma )
        \,d\mu_{ f_1 \circ \gamma }
    \quad\text{and}
    \\
    \partial T_\gamma( f_1 f_2 )
    =
    ( f_1 f_2 )( \gamma_{ b_\gamma } )
    -
    ( f_1 f_2 )( \gamma_{ a_\gamma } )
\end{gather*}
are well-defined and satisfy
\begin{gather*}
    | T_\gamma( f_1, f_2 ) |
    \leq
    \int_\gamma |f_1| \rho_2 \,\d s
    <
    +\infty,
    \quad
    | T_\gamma( f_2, f_1 ) |
    \leq
    \int_\gamma |f_2| \rho_1 \,\d s
    <
    +\infty
    \quad\text{and}
    \\
    \partial T_\gamma( f_1 f_2 )
    -
    T_\gamma( f_1, f_2 )
    = 
    T_\gamma( f_2, f_1 ).
\end{gather*}
\end{lemma}
\begin{proof}
We denote $\Gamma_0 = \Gamma_0^1 \cup \Gamma_0^2$ and consider $\gamma \in \mathscr R(\X) \setminus \Gamma_0$. Then
$f_1 \circ \gamma$ and $f_2 \circ \gamma$ are integrable over the signed measures $\mu_{ f_2 \circ \gamma }$ and 
$\mu_{ f_1 \circ \gamma }$, respectively, and the integral estimates follow from the fact that the total variation of 
$\mu_{ f_i \circ \gamma }$ is bounded from above by $s_{ f_i \circ \gamma }$, cf. \eqref{eq:signed:vs:total}, for $i = 1,2$. 
Thus, by definition,
\begin{equation*}
    T_\gamma( f_1, f_2 )
    =
    \int_{ [a,b] } f_1 \circ \gamma \,\d\mu_{ f_2 \circ \gamma }
    \quad\text{and}\quad
    T_\gamma( f_2, f_1 )
    =
    \int_{ [a,b] } f_2 \circ \gamma \,\d\mu_{ f_1 \circ \gamma }.
\end{equation*}
We may apply the integration by parts formula for the rectifiable paths $f_1 \circ \gamma$ and $f_2 \circ \gamma$ to show
\begin{equation*}
    T_\gamma( f_1, f_2 )
    =
    -
    T_\gamma( f_2, f_1 )
    +
    ( f_1 f_2 )( \gamma_{ b_\gamma } ) - ( f_1 f_2 )( \gamma_{ a_\gamma } ).
\end{equation*}
The rest of the claim follows from $\mu_{ f_i \circ \gamma }$ having density with respect to $s_{ f_i \circ \gamma }$, with 
absolute value of the density bounded from above by one and that the density of $s_{ f_i \circ \gamma }$ 
relative to $s_{\gamma}$ is bounded from above by $(\rho_i \circ \gamma) \in \mathcal{L}^{1}_{{\rm ext}}( s_\gamma )$ for $i = 1,2$.
\end{proof}
\begin{corollary}[Leibniz rule for functions]\label{cor:leibnizrule}
Consider $\mm$-measurable $f_1, f_2 \colon \X \rightarrow [-\infty, \infty]$ and 
$\rho_1, \rho_2 \colon \X \rightarrow [-\infty,\infty]$ with $\rho_i \in \mathrm{WUG}_N( f_i )$ for $i = 1,2$.
If $f_1$ and $f_2$ are in $\mathcal{L}^{\infty}_{\mathrm{ext}}( \mm )$, then any Borel representative of 
$|f_1| \rho_2 + |f_2| \rho_1$ belongs to $\mathrm{WUG}_N( f_1 f_2 )$.
\end{corollary}
\begin{proof}
Consider a Borel representative $\rho$ of $|f_1| \rho_2 + |f_2| \rho_1$. 
Let $B$ be a Borel set with $\mm( B ) = 0$ and $B \supseteq\left\{ \rho \neq |f_1|\rho_2 + |f_2|\rho_1 \right\}$. 
Using the notation from Lemma~\ref{lemm:integrationbyparts} and
\Cref{lemm:newt:leibniz}, let $$\Gamma_0 = \Gamma^{1}_0 \cup \Gamma^2_0 \cup \Gamma^{+}_{B},$$ where
\begin{equation}\label{eq:defGamma_B}
    \Gamma_B^{+}
    =
    \left\{
        \gamma \in \mathscr C(\X)
        \colon
        \int_\gamma \infty \cdot \1_B = +\infty
    \right\},
\end{equation}
so that $\Gamma_B^+$ consists of the $\Modp$-negligible set of unrectifiable curves and rectifiable curves that have positive length in $B$.

We deduce from \Cref{lemm:newt:leibniz} that for every nonconstant $\gamma \in \mathscr R(\X) \setminus \Gamma_0$, we have that
\begin{align*}
    | (f_1 f_2)( \gamma_{b_\gamma} ) - ( f_1 f_2 )( \gamma_{a_\gamma} ) |
    \leq
    \int_{\gamma} |f_1|\rho_2 \,ds
    +
    \int_{\gamma} |f_2|\rho_1 \,ds
    =
    \int_\gamma \rho \,ds.
\end{align*}
The claim follows.
\end{proof}
We obtain the following immediate corollaries from Lemmas \ref{lemm:integrationbyparts} and \ref{lemm:newt:leibniz}.
\begin{corollary}\label{cor:summability}
Let $p \in [1,\infty)$ and consider $f_1, f_2 \colon \X \rightarrow \left[-\infty, \infty\right]$ 
with $\rho_1 \in \mathrm{WUG}_N( f_1 )$ and $\rho_2 \in \mathrm{WUG}_N( f_2 )$. 
Then $\rho_1 + \rho_2 \in \mathrm{WUG}_N( f_1 + f_2 )$.
\end{corollary}
\begin{proof}
We apply \Cref{lemm:newt:leibniz} and consider the family $\Gamma_0 = \Gamma_0^1 \cup \Gamma_0^2$. 
Then $\Mod_p(\Gamma_0)= 0$ and for every nonconstant $\gamma \in \mathscr{R}(\X) \setminus \Gamma_0$, the compositions 
$f_1 \circ \gamma$ and $f_2 \circ \gamma$ are continuous and rectifiable, and satisfy
\begin{align*}
    \ell( ( f_1 + f_2 ) \circ \gamma )
    \leq
    \ell( f_1 \circ \gamma ) + \ell( f_2 \circ \gamma )
    \leq
    \int_{ \gamma } \rho_1 \,\d s + \int_\gamma \rho_2 \,\d s
    =
    \int_\gamma \rho_1 + \rho_2 \,\d s
    < +\infty.
\end{align*}
\end{proof}
\begin{corollary}\label{cor:minimum:gradient}
Let $p \in [1,\infty)$ and consider $f \colon \X \rightarrow \left[-\infty, \infty\right]$ 
with $\rho_1, \,\rho_2 \in \mathrm{WUG}_N( f )$. Then $\min\left\{ \rho_1, \rho_2 \right\} \in \mathrm{WUG}_N( f )$.
\end{corollary}
\begin{proof}
We apply \Cref{lemm:newt:leibniz} for $f = f_1 = f_2$ and $\rho_1$ and $\rho_2$, and consider the family 
$\Gamma_0 = \Gamma_0^1 \cup \Gamma_0^2$. Then $\Mod_p(\Gamma_0)= 0$ and every nonconstant
$\gamma \in \mathscr{R}(\X) \setminus \Gamma_0$ is such that
\begin{align*}
    s_{ f \circ \gamma } \ll s_\gamma
    \quad\text{and}\quad
    \frac{ \d s_{ f \circ \gamma } }{ \d s_\gamma }
    \leq
    \min\left\{ \rho_1, \rho_2 \right\} \circ \gamma
    \in
    \mathcal L^1_{\rm ext}( s_\gamma )^+.
\end{align*}
Therefore $\min\left\{ \rho_1, \rho_2 \right\} \in \mathrm{WUG}_N( f )$.
\end{proof}
\Cref{cor:minimum:gradient} implies the following:
\begin{proposition}\label{prop:minimaluppergradient}
Let $p \in [1,\infty)$ and consider $f \colon \X \rightarrow \left[-\infty, \infty\right]$. 
If $\mathrm{WUG}_N( f ) \neq \varnothing$, there is $\rho \in \mathrm{WUG}_N( f )$ such that 
$\rho \leq \rho'$ $\mm$-almost everywhere for every $\rho' \in \mathrm{WUG}_N( f )$. 
In particular, $\rho = \rho_f$ in the sense of \Cref{def:WUGN}.
\end{proposition}
\begin{proof}
We consider a sequence $(\rho_n)_n\subseteq\mathrm{WUG}_N( f )$ such that
\begin{equation*}
    \lim_{ n \rightarrow \infty } \| \rho_n \|_{ L^p(\mm) }
    =
    \inf_{ \rho' \in \mathrm{WUG}_N( f ) } \| \rho' \|_{ L^p(\mm) }.
\end{equation*}
By applying \Cref{cor:minimum:gradient} iteratively, and by using the fact that a countable union of $\Modp$-negligible curve 
families is $\Modp$-negligible, there is $\Gamma_0$ such that $\Mod_p(\Gamma_0)= 0$ and for every nonconstant
$\gamma \in \mathscr{R}( \X ) \setminus \Gamma_0$ and every $n \in \mathbb{N}$, we have
\begin{align}\label{eq:totalvariation}
    s_{ f \circ \gamma } \ll s_\gamma
    \quad\text{and}\quad
    \frac{ \d s_{ f \circ \gamma } }{ \d s_\gamma }
    \leq
    \rho_n \circ \gamma
    \in
    \mathcal L^1_{\rm ext}( s_\gamma )^+.
\end{align}
In particular, if we denote $\widetilde{\rho}_n = \min_{ i \leq n } \rho_i$, then \eqref{eq:totalvariation} implies that
\begin{equation*}
    \rho = \lim_{ n \rightarrow \infty } \widetilde{\rho}_n = \inf \widetilde{\rho}_n
\end{equation*}
satisfies
\begin{equation*}
    s_{ f \circ \gamma } \ll s_\gamma
    \quad\text{and}\quad
    \frac{ \d s_{ f \circ \gamma } }{ \d s_\gamma }
    \leq
    \rho \circ \gamma
    \in
    \mathcal L^1_{\rm ext}( s_\gamma )^+
\end{equation*}
for every nonconstant $\gamma \in \mathscr{R}( \X ) \setminus \Gamma_0$. Thus $\rho \in \mathrm{WUG}_N(f)$ and Fatou's lemma yield
\begin{equation*}
    \| \rho \|_{ L^p(\mm) }
    =
    \inf_{ \rho' \in \mathrm{WUG}_N( f ) } \| \rho' \|_{ L^p(\mm) }.
\end{equation*}
This equality and \Cref{cor:minimum:gradient} imply that $\rho \leq \rho'$ $\mm$-almost everywhere for every $\rho' \in \mathrm{WUG}_N( f )$.
\end{proof}
The following uniqueness result holds for the minimal weak upper gradients.
\begin{lemma}[Strong locality]\label{lemm:minimaluppergradient:unique}
Consider $f_1, f_2 \colon \X \to [-\infty, \infty]$ such that both $\mathrm{WUG}_N( f_1 )$ and  $\mathrm{WUG}_N( f_2 )$ are not empty. 
If $E \subset \X$ is Borel and $N= E \setminus \left\{ f_1 = f_2 \right\}$ is $\mm$-negligible, then 
$\rho_{ (f_2 - f_1) } = 0$ and $\rho_{f_1} = \rho_{ f_2 }$ $\mm$-almost everywhere in $E$.
\end{lemma}
\begin{proof}
Fix $\rho_i \in \mathrm{WUG}_N( f_i )$ for $i = 1,2$. Let $B \supseteq N$ be an $\mm$-negligible Borel set. We consider the $\Modp$-negligible family
in \eqref{eq:defGamma_B} and we apply \Cref{lemm:newt:leibniz}. Then, there exists a $\Modp$-negligible family $\Gamma_0$ such that for every nonconstant
$\gamma \in \mathscr R(\X) \setminus \Gamma_0$,
\begin{equation*}
    \text{$f_1 \circ \gamma$ and $f_2 \circ \gamma$ are continuous and rectifiable}
\end{equation*}
and the conclusion \eqref{eq:metricspeedcontrol} holds for the pairs $( f_1, \rho_1 )$ and $( f_2, \rho_2 )$.

In particular, if $\gamma \in \mathscr R(\X) \setminus ( \Gamma_0 \cup \Gamma^+_{B} )$ is nonconstant, then $f_1 \circ \gamma = f_2 \circ \gamma$ everywhere
in $\gamma^{-1}(E\setminus B)$ and $s_{\gamma}( \gamma^{-1}(B) ) = 0$. Using this fact we deduce that $s_{ (f_2 - f_1) \circ \gamma } = 0$ $s_{\gamma}$-almost everywhere in $\gamma^{-1}(E)$ by using property 3) 
in Remark \ref{rmk:prop_real_valued_curves}.

Next, by arguing as in the proof of \Cref{cor:summability}, we see that
\begin{equation*}
    \frac{ \d s_{ (f_2 - f_1) \circ \gamma  } }{ \d s_\gamma }
    \leq
    \1_{ \X \setminus E} ( \rho_{ f_1 } + \rho_{ f_1 } )
    \in
    \mathcal{L}^{1}_{{\rm ext}}( s_\gamma ).
\end{equation*}
In particular, $\rho_{ (f_{2}-f_1) } = 0$ $\mm$-almost everywhere in $E$ by \Cref{prop:minimaluppergradient}.  \Cref{cor:summability} 
implies that $| \rho_{ f_2 } - \rho_{ f_1 } | \leq \rho_{ (f_2 - f_1) }$ $\mm$-almost everywhere, so we also deduce the equality 
$\rho_{f_2} = \rho_{f_1}$ $\mm$-almost everywhere on $E$.
\end{proof}
Next, we prove a chain rule for functions with $N$-weak $p$-upper gradients. We establish the claim even for non-measurable functions which causes further technical issues. Indeed, since N-weak p-upper gradient are Borel, we need the following construction and observation. Let $h \colon \X \to [0,\infty)$ be any function for which there is at least one Borel function $\rho' \colon \X \to [0,\infty)$ satisfying 
$h \leq \rho'$ $\mm$-almost everywhere. Then the class of such $\rho' \colon \X \to [0,\infty)$ is non-empty, closed under taking essential infima, 
and closed under pointwise convergence $\mm$-almost everywhere. These facts imply that there is a Borel function
$\rho \colon \X \rightarrow [0,\infty)$ such that $\rho \geq h$ $\mm$-almost everywhere and every Borel
$\rho' \colon \X \rightarrow [0,\infty)$ satisfying $\rho' \geq h$ $\mm$-almost everywhere satisfies
$\rho' \geq \rho$ $\mm$-almost everywhere. We refer to such a $\rho$ as a \emph{minimizer} for the majorization problem associated to $h$.
\begin{lemma}[Chain rule]\label{lemm:chainrule}
Consider $f \colon \X \rightarrow [-\infty,\infty]$ such that $\mathrm{WUG}_N( f ) \neq \varnothing$. 
If $\varphi \colon \mathbb{R} \rightarrow \mathbb{R}$ is Lipschitz and $\eta \colon \X \rightarrow [0,\infty]$ is a Borel function satisfying
\begin{equation}\label{eq:dominance}
   ( \lip( \varphi ) \circ f ) \rho_f\leq \eta\rho_f \leq \Lip(\varphi)\rho_f \quad\text{$\mm$-almost everywhere,}
\end{equation}
then
\begin{equation*}
    \eta\rho_f \in \mathrm{WUG}_N( g )
\end{equation*}
for the function
\begin{equation*}
    g(x)
    =
    \left\{
    \begin{aligned}
        &\varphi( f(x) )
        \quad&&\text{if $f(x) \in \mathbb{R}$}
        \\
        &0\quad&&\text{if $|f(x)| = +\infty$}.
    \end{aligned}
    \right.
\end{equation*}
Furthermore, $\rho_g = \eta \rho_f$ $\mm$-almost everywhere whenever $\eta$ is an essential infima of nonnegative Borel functions satisfying \eqref{eq:dominance}.
\end{lemma}
\begin{proof} We apply the construction above as follows. Consider a minimal $p$-weak upper gradient $\rho_f \in \mathrm{WUG}_N( f )$, and let
\begin{equation}\label{eq:first}
    h(x)
    = 
    \left\{
    \begin{aligned}
        &\lip( \varphi ) \circ f(x),
        \quad&&\text{when $0 < \rho_{f}(x) < \infty$}
        \\
        &0,
        \quad&&\text{otherwise}.
    \end{aligned}
    \right.
\end{equation}
Let $\rho$ denote a minimizer for the majorization problem associated to $h$ from \eqref{eq:first}.
Up to modifying $\rho$ in an $\mm$-negligible Borel set, we lose no generality in assuming that
\begin{equation}\label{eq:pointwiselips:keyidentity}
    h
    \leq
    \rho
    \leq
    \Lip( \varphi )
    \quad\text{everywhere in $\X$.}
\end{equation}
We show that $\rho \rho_f \in \mathrm{WUG}_N( g )$ and if $\rho' \in \mathrm{WUG}_N( g )$, then $\rho' \geq \rho \rho_f$ $\mm$-almost everywhere.

By \Cref{lemm:integrationbyparts}, there is a $\Modp$-negligible curve family $\Gamma_0$ such that for every nonconstant
$\gamma \in \mathscr R(\X) \setminus \Gamma_0$, $f \circ \gamma$ is continuous and rectifiable, and
\begin{align}\label{eq:totalvariation:key}
    \1_{ \gamma^{-1}(E) }
    s_{ f \circ \gamma }
    \leq
    \1_{ \gamma^{-1}(E) } \rho_f s_\gamma
    \quad\text{for every Borel set $E \subset \X$.}
\end{align}
For such a $\gamma$, 
we may apply the chain rule for rectifiable paths to deduce that the total variation measure 
$s_{ \varphi \circ ( f \circ \gamma ) }$ (from \eqref{eq:lengthmeasure}) satisfies
\begin{equation}\label{eq:chainrule:newt}
    s_{ \varphi \circ ( f \circ \gamma ) }
    =
    \lip( \varphi ) \circ ( f \circ \gamma ) s_{ f \circ \gamma }
    \leq
    ( \lip( \varphi ) \circ f \rho_f ) \circ \gamma s_\gamma.
\end{equation}
Indeed, the equality is readily verified using the 
length speed reparametrization of $f \circ \gamma$ and the chain rule for Lipschitz functions.
The inequality follows from \eqref{eq:totalvariation:key}. 

Since $f \circ \gamma$ is continuous, 
the path $\gamma$ does not intersect the set $\left\{ x \in \X \colon |f(x)| =+ \infty \right\}$,
so $\varphi \circ ( f \circ \gamma ) \equiv g \circ \gamma$. Then the inequalities \eqref{eq:pointwiselips:keyidentity}
and \eqref{eq:chainrule:newt} imply $\rho \rho_f \in \mathrm{WUG}_N( g )$.

Consider $\rho_{ g } \in \mathrm{WUG}_N( g )$. Since $\rho \rho_f \in \mathrm{WUG}_N( g )$, \Cref{prop:minimaluppergradient} implies
that $\rho_g \leq \rho \rho_f$ $\mm$-almost everywhere. The proof is complete after we show that $\rho_g \geq \rho \rho_f$ $\mm$-almost everywhere. 
To this end, apply \Cref{lemm:integrationbyparts} to the pair of $g$ and $\rho_g$, and let $\Gamma_1$ denote the $\Modp$-negligible
curve family from the conclusion.

First, if $\sup h = 0$, the claim is clear by the previous paragraph, 
so it suffices to consider the case $\sup h > 0$. Let $\varepsilon > 0$ and consider the Borel set
\begin{equation*}
    B
    \coloneqq
    \left\{ x \in \X|\;
    \rho_{ g }(x)
    +
    \varepsilon \sup h
    < \rho(x)\rho_f(x)<+\infty
    \right\}.   
\end{equation*}
We claim that $\mm( B ) = 0$ for every $\varepsilon>0$. To this end, since $0 < \rho_f <+ \infty$ in $B$, the minimality of 
$\rho$ implies that $h \rho_f > \rho_{ g } + \varepsilon \sup h$ $\mm$-almost everywhere in $B$. 
In particular, there is a Borel set $B_0 \subset B$ with $\mm( B_0 ) = 0$ and
\begin{equation}\label{eq:h:lowerbound}
    h(x)\rho_f(x) > \rho_g + \varepsilon \sup h
    \quad\text{for every $x \in B \setminus B_0$.}
\end{equation}
We claim that
\begin{equation*}
    \widetilde{\rho}
    =
    \left\{
    \begin{aligned}
        &\rho_f,
        \quad&&\text{if $x \in B_0 \cup (\X \setminus B)$}
        \\
        &\rho_f - \varepsilon,
        \quad&&\text{if $x \in B \setminus B_0$}
    \end{aligned}
    \right.
\end{equation*}
is an $N$-weak $p$-upper gradient of $f$. Unless $\mm( B ) = 0$, this leads to a contradiction with the mimimality of $\rho_f$, 
To prove the claim, notice that  if a nonconstant $\gamma$ belongs to
$\mathscr R(\X) \setminus ( \Gamma^{+}_{ B } \cup \Gamma_0 \cup \Gamma_1 )$, with $\Gamma^+_B$ as in \eqref{eq:defGamma_B},
it is clear that the $N$-weak upper gradient inequality holds by \eqref{eq:totalvariation:key}, 
and it remains to verify the upper gradient inequality for rectifiable 
$\gamma \in \Gamma^{+}_{ B } \setminus ( \Gamma_0 \cup \Gamma_1 \cup \Gamma^{+}_{ B_0 } )$. 
Then \eqref{eq:totalvariation:key}, \eqref{eq:chainrule:newt} and \eqref{eq:h:lowerbound} imply
\begin{align*}
    \1_{ \gamma^{-1}( B \setminus B_0 ) } h s_{ f \circ \gamma }
    &=
    \1_{ \gamma^{-1}( B \setminus B_0 ) }
    s_{ g \circ \gamma }
    \leq
    \1_{ \gamma^{-1}( B \setminus B_0 ) }
    ( \rho_{ g } \circ \gamma ) s_{ \gamma }
    \\
    &\leq
    \1_{\gamma^{-1}( B \setminus B_0 )}
    \left(
        -
        \varepsilon \sup h
        +
        ( h \rho_{ f } \circ \gamma )
    \right)
    s_\gamma
    \\
    &\leq
    \1_{\gamma^{-1}( B \setminus B_0 )}
    \left(
        -
        \varepsilon
        +
        \rho_{ f } \circ \gamma
    \right)
    h
    s_\gamma
\end{align*}
Dividing both sides by $h$, which is possible since \eqref{eq:h:lowerbound} implies that $h > 0$ in $B \setminus B_0$, yields that
\begin{align*}
    \1_{ \gamma^{-1}( B \setminus B_0 ) } s_{ f \circ \gamma }
    \leq
    \1_{\gamma^{-1}( B \setminus B_0 )}
    \left(
        -
        \varepsilon
        +
        \rho_{ f } \circ \gamma
    \right)
    s_\gamma
    =
    \1_{\gamma^{-1}( B \setminus B_0 )} ( \widetilde{\rho} \circ \gamma ) s_\gamma.
\end{align*}
Combining this with \eqref{eq:totalvariation:key} shows
\begin{align*}
    s_{ f \circ \gamma }
    \leq
    ( \widetilde{\rho} \circ \gamma )
    s_\gamma.
\end{align*}
Then $\widetilde{\rho} \in \mathrm{WUG}_N( g )$ follows from the fact that $\Gamma_{ B_0 }^{+} \cup \Gamma_0 \cup \Gamma_1$ is $\Modp$-negligible. 
Hence $\mm( B ) = 0$. Since $\varepsilon > 0$ was arbitrary in the definition of $B$, we deduce that $\rho_{g} \geq \rho \rho_f$ $\mm$-almost everywhere.
\end{proof}
\begin{proof}[Proof of \Cref{thm:calc_rules_mwug:newt}]
\Cref{cor:summability} and \Cref{lemm:minimaluppergradient:unique} imply iii). The Leibniz rule iv) follows from \Cref{cor:leibnizrule}.

We prove i). Consider a Borel set $B_1 \supseteq f^{-1}(N)$ with $\mm( B_1 \setminus f^{-1}(N) ) = 0$, and let a Borel set $B \supseteq B_1 \setminus f^{-1}(N)$ satisfy $\mm( B ) = 0$. We apply \Cref{lemm:integrationbyparts} to $f$ and $\rho_f \in \mathrm{WUG}_{N}(f)$ and let $\Gamma_0$ denote the $\Modp$-negligible curve 
family obtained from the conclusion. 

Consider $\theta = f \circ \gamma$, for nonconstant $\gamma \in \mathscr{R}( \X ) \setminus ( \Gamma_0 \cup \Gamma_B^{+} )$. By property
2) of Remark \ref{rmk:prop_real_valued_curves}, we have that $s_{ f \circ \gamma }( ( f \circ \gamma )^{-1}(N) ) = 0$,
while \eqref{eq:metricspeedcontrol} gives that $s_{ f \circ \gamma }( \gamma^{-1}( B ) ) = 0$. Overall, we deduce $s_{ f \circ \gamma }( \gamma^{-1}( B_1 ) ) = 0$. Since $\Modp ( \Gamma_0 \cup \Gamma^{+}_{B} ) = 0$, this implies that $\1_{ \X \setminus B_1 } \rho_f$ is a 
$N$-weak $p$-upper gradient of $f$, so $\rho_f$ $\mm$-almost everywhere in $f^{-1}(N) \subset B_1$ by \Cref{lemm:minimaluppergradient:unique}. 
The claim follows.

Lastly, we verify ii). The claim that $\rho_f = 0$ $\mm$-almost everywhere in $\left\{ x \in \X \mid |f(x)| = +\infty \right\}$ follows from two facts.
First, by \Cref{lemm:integrationbyparts} 1., outside a $\Modp$-negligible family $\Gamma_0$, the path $\gamma$ does not intersect the set 
$\gamma^{-1}( \left\{ |f| = +\infty \right\} )$, given the convention in the definition \eqref{eq:uppergradient}. Next, there are Borel sets 
$B \supseteq \left\{ |f| = +\infty \right\}$ and $B' \supseteq B \setminus \left\{ |f| = +\infty \right\}$ such that $\mm( B' ) = 0$. 
Then, keeping the notation \eqref{eq:defGamma_B},
we have $\mathscr{R}( \X ) \setminus ( \Gamma_0 \cup \Gamma^{+}_{ B' } ) \subset \mathscr{R}( \X ) \setminus \Gamma^{+}_{ B }$. 
Thus $\1_{\X \setminus B}\rho_f$ is an $N$-weak $p$ upper gradient of $f$, so $\rho_{f} = 0$ $\mm$-almost everywhere in 
$\left\{ |f| = +\infty \right\}$ by \Cref{prop:minimaluppergradient}. The latter part of the claim is immediate from \Cref{lemm:chainrule}. 
\end{proof}

\subsection{Sobolev capacity and fine properties of Newtonian Sobolev functions}
In this section, we define a Sobolev capacity which uses the fact that elements of Newtonian $p$-Sobolev spaces have, {\it a priori},
finer properties than Sobolev functions relying on other approaches.
\begin{definition}[Sobolev \(p\)-capacity]\label{def:p_capacity}
Let \((\X,\sfd,\mm)\) be a metric measure space, let \(p\in [1,\infty)\) and 
let \(E\) be any subset of \(\X\). The {\bf \(p\)-capacity} of the set \(E\) is given by the quantity 
\begin{equation}\label{eq:def_Cap_p}
{\rm Cap}_p(E)
\coloneqq
\inf\left\{\|f\|_{N^{1,p}(\X)}^p\big|\, f\in \bar N^{1,p}(\X)\text{ such that } f \geq 1 \text{ on }E\right\},
\end{equation}
taking values in the interval \([0,\infty]\). The value is understood to be infinite if such functions do not exist.
\end{definition}
\begin{remark}\label{Rem_6.11}
{\rm 
We recall some basic properties of ${\rm Cap}_p$.
\begin{enumerate}
    \item[1)] We have
    \begin{equation*}
    {\rm Cap}_p(E)
    =
    \inf\left\{
        \|f\|_{N^{1,p}(\X)}^p\big|\, f\in \bar N^{1,p}(\X)\text{ such that $0 \leq f \leq 1$ and } f = 1 \text{ on }E\right\}.
\end{equation*}
    Indeed, the chain rule \Cref{thm:calc_rules_mwug:newt} i) implies that the same infimum is 
    obtained if we only consider functions $f \in \bar N^{1,p}( \X )$ for which $0 \leq f \leq 1$ everywhere and $f = 1$ on $E$.
    \item[2)] The set function ${\rm Cap}_p$ is an outer measure for every $p \in [1,\infty)$. 
    The subadditivity follows from 1) and a suitable monotonicity argument. We refer the reader to \cite[Lemma 7.2.4]{HKST:15} for a proof.
    \item[3)] Sets of null \(p\)-capacity can be characterized by using the concept of \(p\)-exceptional sets (cf.\ Definition~\ref{def:exceptional})
    as follows (for the proof see \cite[Proposition 7.2.8]{HKST:15}): Let \((\X,\sfd, \mm)\) be a metric measure space, 
    \(p\in [1,\infty)\) and \(E\) an \(\mm\)-measurable subset of \(\X\) . 
    Then
    \begin{equation*}
        {\rm Cap}_p(E)=0\quad \text{ if and only if }\quad \bar\mm(E)=0\,\text{ and }\, E\text{ is }p\text{-exceptional}.
    \end{equation*}
    \item[4)] If $E \subset \X$ is any set, then 
    \begin{equation*}
        {\rm Cap}_p( E )
        =
        \inf {\rm Cap}_p( U ),
    \end{equation*}
    where the infimum is taken over open sets $U$ containing $E$; see \cite[Theorem 1.8]{EB:PC:23}.    
\end{enumerate}
}
\end{remark}
The following proposition implies that representatives of a given equivalence class $f \in N^{1,p}( \X )$ 
having $N$-weak $p$-upper gradients coincide in a set negligible for ${\rm Cap}_p$, cf. \cite[Corollary 7.2.10]{HKST:15}.
\begin{proposition}\label{prop:Newton:representative}
Let \((\X,\sfd,\mm)\) be a metric measure space and let $f \in \bar N^{1,p}( \X )$ for \(p\in [1,\infty)\). If $\mm( \left\{ f \neq 0 \right\} ) = 0$, then
\[
{\rm Cap}_p\big(\{x\in\X\, \mid f(x)\neq 0\}\big)=0.
\]
\end{proposition}
\begin{proof} \Cref{lemm:minimaluppergradient:unique} implies that $\rho_f = 0$ $\mm$-almost everywhere so $|Df|_N = 0$. The claim follows from Remark~\ref{Rem_6.11} 3).
\end{proof}
The proof of the proposition below follows the same ideas as those of Fuglede's lemma (\Cref{thm:Fuglede}).
The claim is standard (see e.g. \cite[Proposition 7.3.7]{HKST:15}), but we include the proof for completeness.

\begin{proposition}[\({\rm Cap}_p\)-representative of a Newtonian \(p\)-Sobolev function]\label{prop:Cap_representative}
Let \((\X,\sfd,\mm)\) be a metric measure space and let \(p\in [1,\infty)\). 
Let \((f_n)_n\subseteq \bar N^{1,p}(\X)\) and let \((\rho_n)_n\subseteq \mathcal L^p_{\rm ext}(\mm)^+\) 
be a sequence of corresponding \(N\)-weak \(p\)-upper gradients. Assume that 
\begin{equation}\label{eq:convergence}
    \lim_n\|f_n-f\|_{L^p(\mm)}=0\quad \text{ and }\quad \lim_n\|\rho_n-\rho\|_{L^p(\mm)}=0,
\end{equation}
for some \(f\in L^p(\mm)\) and \(\rho\in \mathcal L^p_{\rm ext}(\mm)^+\). Then there exists 
an \(\mm\)-measurable representative $\bar f$ in the equivalence class of $f$ such that 
\(\rho\in {\rm WUG}_{N}(\bar f)\).
Moreover, up to a subsequence, it holds that 
\(\lim_nf_n(x)=\bar f(x)\) for \({\rm Cap}_p\)-a.e.\ \(x\in \X\).
\end{proposition}
\begin{proof}
First of all, we can extract from the given sequences the subsequences (without relabeling them, for the simplicity of the notation) a fastly converging in \(L^p(\mm)\) -- namely, so that 
\begin{equation*}\label{eq:fastconvergence}
    \| { \rho }_{ n+1 } - { \rho }_{ n } \|_{ L^{p}( \mm ) }^p
    +
    \| { f }_{ n+1 } - { f }_{ n } \|_{ L^{p}( \mm ) }^p
    \leq
    2^{-np}.
\end{equation*}
In particular,  it holds that
\begin{equation*}
    H
    \coloneqq
    | f_1 |
    +
    \rho_1
    +
    \sum_{ n = 1 }^{ \infty } | f_{n+1} - f_n |
    +
    \sum_{ n = 1 }^{ \infty } | \rho_{n+1} - \rho_n |
\end{equation*}
is $p$-integrable.
Let $E = \left\{ x \in \X |\; (  f_n(x) )_{ n = 1 }^{ \infty }
\text{ is not Cauchy in $\mathbb{R}$} \right\}$. Since $H(x) < \infty$ $\mm$-almost everywhere, we have that $\bar\mm( E ) = 0$. We define
\begin{equation*}
    \bar f(x)
    =
    \left\{
    \begin{aligned}
        &\lim_{ n \rightarrow \infty } f_n(x),
        \quad&&\text{if $x \in \X \setminus E$,}
        \\
        &0,
        \quad&&\text{if $x \in E$,}
    \end{aligned}
    \right.
    \qquad \text{and}\qquad 
       \bar \rho(x)
    =
    \left\{
    \begin{aligned}
        &\lim_{ n \rightarrow \infty } \rho_n(x),
        \quad&&\text{if $H(x) < \infty$,}
        \\
        &0,
        \quad&&\text{otherwise.}
    \end{aligned}
    \right.
\end{equation*}
As \eqref{eq:convergence} holds, we have that $( \bar f, \bar \rho )$ is a representative of $( f, \rho )$. 
We claim that $\bar f \in \bar N^{1,p}( \X )$ with $N$-weak $p$-upper gradient $\bar \rho$. To this end, we apply 
\Cref{lemm:integrationbyparts} to every pair $( f_n, \rho_n )$ and obtain a $\Modp$-negligible curve family 
$\Gamma_0^n$ from the conclusion. We consider the $\Modp$-negligible family
\begin{equation*}
    \Gamma_0
    =
    \left\{ \gamma \in \mathscr{R}( \X ) \mid \int_\gamma \infty \cdot \1_{ \left\{ H = \infty \right\} } \,\d s = \infty \right\}
    \cup
    \left\{ \gamma \in \mathscr{R}( \X ) \mid \int_\gamma H \,\d s = \infty \right\}
    \cup
    \bigcup_{ i = 1 }^{ \infty } \Gamma^i_0.
\end{equation*}
Notice that if $\gamma \in \mathscr{R}( \X ) \setminus \Gamma_0$ is nonconstant, 
then $(  f_n \circ \gamma(t) )_{ n = 1 }^{ \infty }$ is Cauchy for $s_\gamma$-almost every $t$,
namely, for every $t \in \gamma^{-1}( \X \setminus E )$. For such $t$, the limit will be $\bar f \circ \gamma(t)$ by definition of $\bar f$.

Next, by passing to a constant-speed reparametriztion, we lose no generality in assuming that the domain of $\gamma$ is $[0,1]$ and $\gamma$ has constant speed.
Then for every $s \in [0,1]$ and $\delta > 0$, there is $t \in [0,1] \cap (s-\delta,s+\delta) \setminus \gamma^{-1}( E )$. 
For every $m, n \in \mathbb{N}$, we deduce
\begin{align*}
    | {f}_m( \gamma(s) ) - {f}_n( \gamma(s) ) |
    &\leq
    | {f}_m( \gamma(s) ) - {f}_m( \gamma(t) ) |
    +
    | {f}_{m}( \gamma(t) ) - {f}_n( \gamma(t) ) |
    +
    | {f}_n( \gamma(s) ) - {f}_n( \gamma(t) ) |
    \\
    &\leq
    2 \int_{ \gamma|_{ [0,1] \cap (s-\delta,s+\delta) } } H \,\d s
    +
    | {f}_{m}( \gamma(t) ) - {f}_n( \gamma(t) ) |.
\end{align*}
Next, taking $\lims$ with respect to $m$, then $n$ and lastly over $\delta \rightarrow 0$, implies that 
$( {f}_m( \gamma(s) ) )_{ m = 1 }^{ \infty }$ is Cauchy. In particular, $\gamma$ does not intersect the set 
$E$, so $\gamma \in \mathscr{R}( \X ) \setminus \Gamma_E$. 
Together with the fact that \(\bar\mm(E)=0\) this show that \(E\) has null \(p\)-capacity and therefore the last part of the statement follows.
By Arzela--Ascoli theorem, the pointwise convergence of $( {f}_m \circ \gamma )_{ m = 1 }^{ \infty }$ to $\bar{f} \circ \gamma$ 
improves to uniform convergence. Then, by lower semicontinuity of length, we deduce that
\begin{align*}
    \ell( \bar f \circ \gamma|_{ [s,t] } )
    \leq
    \limi_{ n \rightarrow \infty }
    \ell( f_n \circ \gamma|_{[s,t]} ).
\end{align*}
Observe that
\begin{align*}
    \limi_{ n \rightarrow \infty }
    \ell( f_n \circ \gamma|_{ [s,t] } )
    \leq
    \lim_{ n \rightarrow \infty }
    \int_{ \gamma|_{ [s,t] } } \rho_n \,\d s
    =
    \int_{ \gamma|_{ [s,t] } } \bar \rho \,\d s
\end{align*}
by dominated convergence in $\mathcal{L}^{1}_{\mathrm{ext}}( s_\gamma )$. 
Thus $\bar \rho$ is an $N$-weak $p$-upper gradient of $\bar f$ as claimed. The proof is complete.
\end{proof}
\begin{remark}
\rm
Observe that the \Cref{prop:Cap_representative} and its proof imply the following: 
\begin{enumerate}
    \item [1)] If the functions $f_n$ are Borel, then the limit function $f$ has a Borel representative $\bar f \in \bar N^{1,p}( \X )$
    (since the set $E$ in the proof is Borel and pointwise limits of sequences of Borel functions are Borel).
    \item [2)] $N^{1,p}( \X )$ endowed with the quotient norm $\| \cdot \|_{ N^{1,p}( \X ) }$ is a Banach space.
\end{enumerate}
\end{remark}
We also have the following stability result.
\begin{lemma}\label{lem:convergence_of_wug}
Let \((\X,\sfd, \mm)\) be a metric measure space and \(p\in [1,\infty)\). 
Let \((f_n)_n\subseteq N^{1,p}(\X)\), \(f\in N^{1,p}(\X)\) and \((g_n)_n\subseteq L^p(\mm)\) be such that 
\[
f_n\to f\,\text{ in }L^p(\mm)\quad \text{ and }\quad g_n\to |Df|_N \, \text{ in }L^p(\mm).
\]
Assume also that \(|Df_n|_N\leq g_n\) for every \(n\in \N\). Then \(|Df_n|_N\to |Df|_N\) in \(L^p(\mm)\).
\end{lemma}
\begin{proof}
First, let us check that \(|Df_n|_N\rightharpoonup|Df|_N\) weakly in \(L^p(\mm)\). 
Since \(|Df_n|_N\leq g_n\to|Df|_N\), we have (up to a non-relabelled subsequence)
that \(|Df_n|_N\leq h\) for every \(n\in\N\), for some \(h\in L^p(\mm)^+\). 
By Dunford--Pettis theorem (or, when \(p>1\), just the reflexivity of \(L^p(\mm)\)),
up to a further subsequence we have that \(|Df_n|_N\rightharpoonup G\) for some
\(G\in L^p(\mm)^+\) with \(G\leq|Df|_N\). Mazur's lemma implies that, for any \(k\in\N\),
we can find \(N(k)\in\N\) with \(N(k)\geq k\) and \((\alpha^k_n)_{n=k}^{N(k)}\subseteq[0,1]\) such that \(\sum_{n=k}^{N(k)}\alpha^k_n=1\)
and \(\sum_{n=k}^{N(k)}\alpha^k_n|Df_n|_N\to G\)
in \(L^p(\mm)\). Since \(\tilde f_k\coloneqq\sum_{n=k}^{N(k)}\alpha^k_n f_n\in N^{1,p}(\X)\) satisfies 
\(\|\tilde f_k-f\|_{L^p(\mm)}\leq\sum_{n=k}^{N(k)}\alpha^k_n\|f_n-f\|_{L^p(\mm)}\to 0\)
as \(k\to\infty\) and \(|D\tilde f_k|_N\leq\sum_{n=k}^{N(k)}\alpha^k_n|Df_n|\) by Corollary \ref{cor:summability}, 
we deduce from Proposition \ref{prop:Cap_representative}
that \(|Df|_N\leq G\). Therefore, \(G=|Df|_N\) and thus the original sequence \((f_n)_n\) has the property that
\(|Df_n|_N\rightharpoonup|Df|_N\) weakly in \(L^p(\mm)\).

Next, observe that the weak lower semicontinuity of \(\|\cdot\|_{L^p(\mm)}\) yields
\[
\||Df|_N\|_{L^p(\mm)}\leq\limi_{n\to\infty}\||Df_n|_N\|_{L^p(\mm)}\leq\lim_{n\to\infty}\|g_n\|_{L^p(\mm)}=\||Df|_N\|_{L^p(\mm)},
\]
which implies \(\||Df_n|_N\|_{L^p(\mm)}\to\||Df|_N\|_{L^p(\mm)}\). In the case \(p>1\), 
we deduce that \(|Df_n|_N\to|Df|_N\) strongly in \(L^p(\mm)\) by the uniform convexity of \(L^p(\mm)\).
Finally, in the case \(p=1\),
\[\begin{split}
\lims_{n\to\infty}\||Df_n|_N-|Df|_N\|_{L^1(\mm)}
&\leq\lims_{n\to\infty}\||Df_n|_N-g_n\|_{L^1(\mm)}+\lim_{n\to\infty}\|g_n-|Df|_N\|_{L^1(\mm)}\\
&=\lims_{n\to\infty}\bigg(\int g_n\,\d\mm-\int|Df_n|_N\,\d\mm\bigg)=0.
\end{split}\]
Therefore, the proof is complete.
\end{proof}
\subsection{Energy density of Lipschitz functions in the Newtonian Sobolev space}\label{sec:energy_density}
We provide a sketch of proof for the recent energy density result from \cite{EB:20:published}. More precisely, we prove the following.

\begin{theorem}[Density in energy of \(\LIP_{bs}(\X)\) in \(N^{1,p}(\X)\)]
\label{thm:Energy_density_EB}
Let \((\X,\sfd,\mm)\) be a metric measure space and let \(p\in [1,\infty)\). Given \(f\in N^{1.p}(\X)\), 
there exist \((f_n)_n\subseteq \LIP_{bs}(\X)\) such that 
\[
\lim_n\|f_n-f\|_{L^p(\mm)}=0\quad \text{ and }\quad \lim_n\|\lip_a(f_n)-|Df|_N\|_{L^p(\mm)}=0.
\]
\end{theorem}
The proof uses some ideas from discrete convolution. All the notions we introduce below,
as  well as the proof of the above theorem are contained in \cite{EB:20:published}.

\begin{definition}[Discrete paths]
Let \((\X,\sfd)\) be a metric space. Given any \(n\in \N\), a \textbf{discrete} \(n\)-\textbf{path} (or briefly \textbf{path})
in \(\X\) is a finite ordered set 
\(P=(p_0,\cdots,p_n)\subseteq \X\). Given a path \(P\), we introduce the following notions:
\begin{align}\label{eq:diam_path}
    {\rm diam}(P)\coloneqq & \max \big\{\sfd(p_k,p_l)|\, k,l\in \{0,\ldots, n\} \big\};\\
    \label{eq:mesh_path}
    {\rm Mesh}(P)\coloneqq & \max \big\{\sfd(p_k,p_{k+1})|\, k\in \{0,\ldots,n-1\}\big\};\\
    \label{eq:len_path}
    {\rm Len}(P)\coloneqq & \sum_{k=0}^{n-1}\sfd(p_k,p_{k+1}).
\end{align}
We say that a path \(Q=(q_0,\ldots,q_n)\) is a \textbf{sub-path} of \(P=(p_0,\ldots,p_m)\) with \(n\leq m\), and we write \(Q\leq P\), if
there exists \(0\leq l\leq m-n\) such that \(p_{k+l}=q_k\) for every \(k\in \{0,\ldots,n\}\).
\end{definition}
Given \(\delta>0\), a closed set \(C\subseteq \X\) and a point \(x\in \X\), a path \(P\) is said to be \((\delta, C,x)\)-\textbf{admissible} 
if \({\rm Mesh}(P)\leq \delta\), \(p_0\in C\) and \(p_n=x\). We will denote by \({\rm Adm}(\delta, C,x)\) the collection of all 
\((\delta, C,x)\)-admissible paths.
\begin{lemma}\label{lem:tilde_f}
Let \((\X,\sfd)\) be a metric space. Fix \(M, \delta>0\) and a closed set 
\(C\subseteq \X\). Let \(f\colon \X\to [0,M]\) be any function, and let \(g\colon \X\to [0,\infty)\) be bounded and continuous.
Define \(\tilde f\colon \X\to [0,M]\) as
\[
\tilde f(x)\coloneqq \min\Big\{M, \inf_{P\in {\rm Adm}(\delta, C,x)}f(p_0)+\sum_{k=0}^{n-1}g(p_k)\sfd(p_k,p_{k+1})\Big\},
\quad \text{ for every } x\in \X.
\]
The function \(\tilde f\) satisfies the following properties:
\begin{itemize}
    \item [\rm a)] \(\tilde f(x)\leq f(x)\) for every \(x\in C\);
    \item [\rm b)] \(|\tilde f(x)-\tilde f(y)|\leq \max\{g(x),g(y)\}\sfd(x,y)\) for every \(x,y\in \X\) with \(\sfd(x,y)\leq\delta\);
    \item [\rm c)] \(\lip_a(\tilde f)(x)\leq g(x)\) for every \(x\in \X\);
    \item [\rm d)] \(\Lip(\tilde f)\leq\max\{M/\delta,\|g\|_\infty\}\).
\end{itemize}
\end{lemma}
In what follows we fix an isometric embedding \(\iota\colon \X\hookrightarrow\ell^\infty\) and identify \(X\) with the image of \(\iota\).
As observed in \cite{EB:20:published}, the notions we introduce next do not depend on the choice of \(\iota\).

\begin{definition}[Linearly interpolating curves]\label{def:lin_interp_curve}
Let \((\X,\sfd)\) be a metric space. Let \(n\in \N\) and let \(P=(p_0,\cdots,p_n)\) be a given discrete \(n\)-path in \(\X\). 
We define the \textbf{linearly interpolating curve} (associated with \(P\)) \(\gamma_P\colon [0,1]\to \ell^\infty\) as follows: 
if \({\rm Len}(P)=0\), we set \(\gamma_P(t)\coloneqq p_0\) for every \(t\in [0,1]\); if \({\rm Len}(P)>0\), we first define 
the sequence of interpolating times \(T_P\coloneqq (t_0,\ldots,t_n)\) by setting
\[
t_0\coloneqq 0\quad \text{ and }\quad t_k\coloneqq \sum_{i=0}^{k-1}\frac{\sfd (p_i,p_{i+1})}{{\rm Len}(P)},\, \text{ for }k\in \{1,\ldots, n\}.
\]
Then we define for every \(k\in \{0,\ldots, n\}\)
\[
\gamma_P(t_k)\coloneqq p_k\quad \text{ and }\quad \gamma_P(t)\coloneqq 
\frac{(t-t_k)p_{k+1}+(t_{k+1}-t)p_{k}} {t_{k+1}-t_k} \,\text{ for every }t\in [t_k,t_{k+1}].
\]    
\end{definition}
We next introduce a notion of convergence for discrete paths.
\begin{definition}[Convergence of discrete paths]\label{def:conv_discrete_paths}
Let \((\X,\sfd)\) be metric space. Let \((P_i)_i\) be a given sequence of discrete \(n_i\)-paths, \(n_i\in \N\). We say that 
\((P_i)_i\) converges to a curve \(\gamma\colon [0,1]\to \ell^\infty\) if \((\gamma_{P_i})_i\) converge uniformly to \(\gamma\)
and \({\rm Mesh}(P_i)\to 0\) as \(i\to +\infty\).
\end{definition}

\begin{lemma}[A compactness result for discrete paths]\label{lem:compactness_discrete_paths}
Let \((\X, \sfd)\) be a metric space. Let \(M,\, L,\, D>0\) be given real constants and let \((K_k)_k\) be an increasing
sequence of non-empty compact sets in \(\X\). Given \(n\in \N\), define \(h_n\colon\X\to [0,\infty]\) as
\[
h_n(x)\coloneqq \sum_{k=1}^n\min\big\{n\,\sfd(x,K_k),1\big\},\quad \text{ for every }x\in \X.
\]
Then, given any sequence \((P_i)_i\) of discrete \(n_i\)-paths \(P_i=(p^i_0,\ldots,p^i_{n_i})\) satisfying 
\[
\lim_{i\to +\infty}{\rm Mesh}(P_i)=0,\quad \sup_{i\in \N}{\rm Len}(P_i)\leq L,\quad \inf_{i\in \N}{\rm diam}(P_i)\geq D
\]
as well as
\[
\sum_{j=0}^{n_i-1}h_i(p^i_j)\sfd(p^i_j,p^i_{j+1})\leq M\qquad\forall i\in\N,
\]
there exist a subsequence of \((P_i)_i\) converging to some \(\gamma\colon [0,1]\to \ell^\infty\) in the sense of Definition~\ref{def:conv_discrete_paths}. 
\end{lemma}

\begin{lemma}[A lower semicontinuity result for discrete paths]\label{lem:lsc_discrete_paths}
Let \((\X,\sfd)\) be a metric space. Let \(g\colon \X\to [0,\infty]\) be a lower semicontinuous function and let \((g_i)_i\) be an increasing sequence of continuous functions \(g_i\colon \X\to [0,\infty)\)
converging to \(g\) pointwise. Given a sequence \((P_i)_i\) of discrete \(n_i\)-paths  \(P_i=(p^i_0,\ldots,p^i_{n_i})\) 
with \(\sup_{i\in \N}{\rm Len}(P_i)<+\infty\) and which converges to a curve \(\gamma\colon [0,1]\to \ell^\infty\), it holds that 
\[
\int_\gamma g\,\d s\leq \limi_{i\to +\infty}\sum_{k=0}^{n_i-1}g_i(p^i_k)\sfd(p^i_k,p^i_{k+1}).
\]
\end{lemma}
\begin{proof}[Proof of \Cref{thm:Energy_density_EB}]
\ \\
{\bf Step 1.} First of all, we notice that it is enough to prove the statement for functions \(f\in N^{1,p}(\X)\) such that
\(0\leq f\leq M\) for some \(M>0\) and \(f|_{\X\setminus B_R(x_0)}=0\) for some \(x_0\in \X\) and \(R>0\). As standard 
truncation and cut-off techniques provide the statement for an arbitrary element of \(N^{1,p}(\X)\). 
Thus, we fix \(M>0\), \(x_0\in \X\) and \(R>2\) until the end of the proof.
\smallskip

\noindent
{\bf Step 2.}
Fix \(\varepsilon >0\). Denote by \(\rho_1\colon \X\to [0,+\infty]\) a lower semicontinuous upper gradient of \(f\) such that 
\(\rho_1\geq |Df|_N\) and \(\|\rho_1-|Df|_N\|_{L^p(\mm)}^p\leq \varepsilon 4^{-2p}\). Its existence is provided by the Vitali-Carath\' eodory Theorem. 
We may, without loss of generality, assume that \(\rho_1=\rho_f=0\) on \(\X\setminus B_R(x_0)\) for some \(\rho_f\in \ppi_\mm^{-1}(|Df|_N)\). 

Choose now an increasing sequence \((K_n)_n\) of compact sets in \(\X\) so that \(K_n\subseteq B_{2R}(x_0)\)
and \(\mm(B_{2R}(x_0)\setminus K_n)\leq \varepsilon n^{-p}4^{-n-2p}\) and \(f|_{K_n}\) is continuous. Fix \(\sigma>0\) such that
\(\mm(B_{2R}(x_0))\sigma^p\leq \varepsilon 4^{-2p}\). Define \(\psi_{2R}(x)\coloneqq \max\big\{0,\min\{1,2R-\sfd(x,x_0)\}\big\} \)
for every \(x\in \X\). Finally, set
\[
\rho_\varepsilon(x)\coloneqq \rho_1(x)+\sigma \psi_{2R}(x)+\sum_{n=1}^{+\infty}\1_{B_{2R}(x_0)\setminus K_n}(x),\quad \text{for every }x\in \X.
\]
It can be verified that \(\rho_\varepsilon\) is lower semicontinuous and satisfies \(\|\rho_\varepsilon-|Df|_N\|_{L^p(\mm)}^p\leq \varepsilon\).
\smallskip

\noindent
{\bf Step 3.} Pick an increasing sequence of bounded continuous functions \((\tilde \rho_n)_n\) converging to \(\rho_1\) pointwise. 
Define further
\[
\rho_n(x)\coloneqq \tilde \rho_n(x)+\sigma \psi_{2R}(x)+\sum_{k=1}^n\min\big\{n\sfd(x,K_k),\1_{B_{2R}(x_0)}(x)\big\},\quad \text{ for every }x\in \X.
\]
The sequence \((\rho_n)_n\) is increasing, converges to \(\rho_\varepsilon\) pointwise and for every \(n\in \N\) it holds that 
\(0\leq \rho_n\leq \rho_\varepsilon\).  Choose an \(N>0\) such that \(\mm(B_{2R}(x_0)\setminus K_N)\leq \varepsilon (2M)^{-p}\) and set 
\(A\coloneqq K_N\cup \X\setminus B_R(x_0)\).  Note that \(f|_A\) is continuous. 
\smallskip

\noindent
{\bf Step 4.} Now, given \(n\in \N\) and the corresponding data \((f,\rho_n,A,M,n^{-1})\) introduced above, 
we define 
\[
f_n(x)\coloneqq \min\left\{M, \inf_{P\in {\rm Adm}(n^{-1}, A, x)}f(p_0)+\sum_{k=0}^{N_P-1}\rho_n(p_k)\sfd(p_k,p_{k+1})\right\},
\quad \text{ for every } x\in \X.
\]
We observe the following properties of each function \(f_n\) and of the sequence \((f_n)_n\).
From Lemma \ref{lem:tilde_f} we have that the function \(f_n\) is Lipschitz and satisfies 
\(\lip_a(f_n)\leq \rho_n\leq \rho_\varepsilon\), \(0\leq f_n|_A\leq f|_A\) and \(f_n(x)=0\) for every \(x\in \X\setminus B_{R}(x_0)\). 
Also, \(f_n(x)\leq f(x)\) for all \(x\in A\). Since \(\rho_n\geq \rho_m\) whenever 
\(n\geq m\) and since for each \(x\in \X\) an \((n^{-1}, A, x)\)-admissible path is also 
\((m^{-1}, A, x)\)-admissible, the sequence \((f_n(x))_n\) is increasing. 

In particular, it follows from the above observations that
for every \(x\in K_N\subseteq A\) the sequence \((f_n(x))_n\) has a limit. \emph{We claim that \(\lim_n f_n(x) =f(x)\) for every \(x\in K_N\).}
To prove it, we argue by contradiction: suppose there exists \(x\in K_N\) and \(\delta>0\) such that \(\lim_nf_n(x)\leq f(x)-\delta\). 
From the definition of \(f_n\) we get the sequence of discrete \((n^{-1}, A, x)\)-admissible paths \(P_n=(p_0^n, \ldots, p_{N_n}^n)\) with 
the property 
\[
f(p^n_0)+\sum_{k=0}^{N_n-1}\rho_n(p^n_k)\sfd(p_n^k,p^n_{k+1})<f(x)-\frac{\delta}{2}\leq M.
\]
Further, we denote by \(Q_n=(q^n_0,\ldots,q^n_{M_n})\) the largest sub-path of \(P^n\) with \(q^{M_n}_n=x\) and \(Q^n\subseteq B_{3R/2}(x_0)\).
One can check that \(Q_n\) is \((n^{-1}, A, x)\) admissible and satisfies the estimate
\[
f(q^n_0)+\sum_{k=0}^{M_n-1} \rho_n(q^n_k)\sfd(q^n_k,q_{k+1}^n)\leq f(p^n_0)+\sum_{k=0}^{N_n-1}\rho_n(p^n_k)\sfd(p^n_k,p^n_{k+1}) <f(x)-\delta\leq M.
\]
Using those properties, one can verify that the sequence \((Q_n)_{n\in \N}\) satisfies the assumptions of 
Lemma \ref{lem:compactness_discrete_paths}. Thus, taking a  (non-relabeled) subsequence, 
we have that \((Q_n)_{n\in \N}\) converges to some \(\gamma\colon [0,1]\to \ell^\infty\).

Now, notice that \(q^n_0\in A\) for all \(n\in \N\). Thus, as \(A\) is closed we get that 
\(\lim_n q^n_0=\gamma(0)\in A\). One can show in a similar way that \(\gamma(1)=x\). 
Recall that \(f|_A\) is continuous and that \((\rho_n)_n\) is an increasing sequence 
pointwise converging to \(\rho_\varepsilon\). Thus, we may apply the lower semicontinuity 
result for discrete paths in Lemma \ref{lem:lsc_discrete_paths}  and get that
\[
f(\gamma_0) +\int_\gamma \rho_\varepsilon\,\d s\leq \limi_{n\to +\infty}f(q^n_0)+\sum_{k=0}^{M_n-1}\rho_n(q^n_k)\sfd(q_k^n,q_{k+1}^n)
\leq f(\gamma_1)-\frac{\delta}{2}.
\]
It follows that \(f(\gamma_1)-f(\gamma_0)>\int_\gamma\rho_\varepsilon\,\d s\), 
contradicting the fact that \(\rho_\varepsilon\) is an upper gradient of \(f\). Thus, the claim is proven.
\\
{\bf Step 5.} Thanks to the previous step and the fact that \(A\setminus K_N\subseteq \X\setminus B_R(x_0)\) we have 
\(f_n(x)=f(x)=0\) on \(A\setminus K_N\). Thus, by Dominated Convergence Theorem we may find \(n_0\in \N\)  such that
\[\int_{A}|f_{n}-f|^p\,\d\mm=\int_{K_N}|f_{n}-f|^p\,\d \mm\leq \frac{\varepsilon}{2}\quad\text{ for all }n\geq n_0.\]
Taking into account the choice of \(K_N\) and the fact that \(X=A\cup (B_R(x_0)\setminus K_N)\), we may find \(n_1\geq n_0\) so that 
\[
\int_\X |f_{n}-f|^p\,\d\mm=\int_A|f_{n}-f|^p\,\d\mm+\int_{B_R(x_0)\setminus K_N}|f_{n}-f|^p\,\d\mm 
\leq \frac{\varepsilon}{2}+M^p\mm(B_R(x_0)\setminus K_N)\leq \varepsilon,
\]
for all \(n\geq n_1\). 
Thus, taking into account that \(\rho_{f_n}\leq \lip_a(f_n)\leq \rho_n\leq \rho_\varepsilon\), 
\(\|\rho_\varepsilon-|Df|_N\|_{L^p(\mm)}\leq \varepsilon\) and that \(\rho_n\to \rho_\varepsilon\) 
pointwise as \(n\to +\infty\), we deduce from Lemma \ref{lem:convergence_of_wug} that the statement of the theorem
is satisfied by the sequence \((f_n)_{n\geq n_1}\subseteq \LIP_{bs}(\X)\).
\end{proof}
\subsection{Bibliographical notes}
\begin{itemize}
    \item The results in Section \ref{sec:calculus_rules_dirichlet} are 
    new and extend the validity of the calculus rules for Newtonian Sobolev functions proven e.g.\ in \cite{HKST:15} 
    to the case of Dirichlet functions.
    \item The results about the Sobolev \(p\)-capacity and \({\rm Cap}_p\)-representatives of 
    Sobolev functions can be found in \cite{HKST:15}; for some recent results see also \cite{EB:PC:23}.
    \item Lemma \ref{lem:convergence_of_wug} and the proof of the equivalence are taken from \cite{EB:20:published}.
\end{itemize}
\subsection{List of symbols}
\begin{center}
\begin{spacing}{1.2}
\begin{longtable}{p{2.2cm} p{11.8cm}}
\({\rm Cap}_p\) & Sobolev \(p\)-capacity; \eqref{eq:def_Cap_p}\\
\({\rm diam}(P)\) & diameter of a discrete \(n\)-path \(P\); \eqref{eq:diam_path}\\
\({\rm Mesh}(P)\) & mesh of a discrete \(n\)-path \(P\); \eqref{eq:mesh_path}\\
\({\rm Len}(P)\) & length of a discrete \(n\)-path \(P\); \eqref{eq:len_path}\\
\({\rm Adm}(\delta, C,x)\) & set of \((\delta, C,x)\)-admissible paths; Section \ref{sec:energy_density}\\
\(\gamma_P\) & linearly interpolating curve of a discrete \(n\)-path \(P\); Definition \ref{def:lin_interp_curve}
\end{longtable}
\end{spacing}
\end{center}
\section{Equivalence of the four approaches}\label{sec:Equivalence}
In this section we show, as a consequence of all considerations made up to this point, that the four notions of metric
Sobolev space presented in Section \ref{sec:Sobolev_H} coincide. We then discuss several functional properties of the Sobolev space.
\subsection{Proof of the equivalence}
In the next theorem we state and prove the main result of this paper, namely the 
equivalence between different notions of metric Sobolev space presented in Section \ref{sec:Sobolev_H}.
\begin{theorem}\label{thm:equivalence_Sobolev_spaces}
Let \((\X,\sfd,\mm)\) be a metric measure space and let \(p\in [1,\infty)\). Then 
\[
H^{1,p}(\X)=W^{1,p}(\X)=B^{1,p}(\X)=N^{1,p}(\X)
\]   
and, for every \(f\in W^{1,p}(\X)\), the equalities \(|Df|_H=|D f|_W=|Df|_B=|Df|_N\) hold in $L^p(\mm)$.
\end{theorem}
\begin{proof} The strategy of proof is the following: we will prove in the first three steps the natural inclusions 
$$
H^{1,p}(\X)\subseteq W^{1,p}(\X)\subseteq B^{1,p}(\X)\subseteq N^{1,p}(\X)
$$
with the corresponding inequalities 
between the relaxed gradients associated to the various definitions. Eventually, the circle will be closed in Step 4, proving the 
inequality $|Df|_H\leq |Df|_N$ associated to the most technical inclusion $N^{1,p}(\X)\subseteq H^{1,p}(\X)$.
\\
{\bf Step 1.} \(H^{1,p}(\X)\subseteq W^{1,p}(\X)\) with $|D f|_W\leq |Df|_H$. 
\\
Let \(f\in H^{1,p}(\X)\). Pick a sequence \((f_n)_n\subseteq \LIP_{bs}(\X)\) 
such that \(f_n\to f\) and \(\lip_a(f_n)\to |Df|_H\) in \(L^p(\mm)\) as \(n\to +\infty\). 
Then, given \(b\in {\rm Der}^q_q(\X)\), for every \(h\in \LIP_{bs}(\X)\) we set
\[
\mathcal L_f(b)(h)\coloneqq -\int  f\,\div(h b)\,\d\mm=\lim_{n\to +\infty}-\int f_n\div(h b)\,\d\mm=\lim_{n\to +\infty}\int hb(f_n)\,\d\mm.
\]
Then, it holds that
\[|\mathcal L_f(b)(h)|\leq\lims_{n\to+\infty}\int |h b(f_n)|\,\d \mm\leq \lim_{n\to +\infty}\int |h||b|\lip_a(f_n)\,\d\mm
=\int|h||b||Df|_H\,\d\mm.\]
We further define a finite measure \(\mu\coloneqq |b||Df|_H\mm\in \mathcal M_+(\X)\). Notice that 
\[|\mathcal L_f(b)(h)|\leq \|h\|_{L^1(\mu)},\quad\text{ for every }h\in \LIP_{bs}(\X).\] 
Since \(h\mapsto \mathcal L_f(b)(h)\) is linear, the above estimate implies that the existence of 
\(\tilde{\mathcal L_f}(b)\in L^\infty(\mu)\) with \(\|\tilde{\mathcal L_f}(b)\|_{L^\infty(\mu)}\leq 1\) such that 
\[
\int h\tilde{\mathcal L_f}(b)\,\d\mu=\mathcal L_f(b)(h).
\]
Finally, we set \({\sf L}_f(b)\coloneqq \tilde{\mathcal L_f}(b)|b||Df|_H\in L^1(\mm)\). The \(\LIP_{bs}(\X)\)-linearity   
is a consequence of the following observation: pick \(\psi\in \LIP_{bs}(\X)\) and note that \(\psi b\in {\rm Der}^q_q(\X)\). Then 
\[
\int h{\sf L}_f(\psi b)\,\d \mm=\int h \tilde{\mathcal L_f}(\psi b)\,\d\mu=\mathcal L_f(\psi b)(h)=\mathcal L_f(b)(h\psi)
=\int h \psi\tilde{\mathcal L_f}(b)\,\d\mu=\int h\psi{\sf L}_f(b)\,\d \mm
\]
holds for every \(h\in \LIP_{bs}(\X)\), and thus \(\psi {\sf L}_f(b)={\sf L}_f(\psi b)\) is satisfied \(\mm\)-a.e.\ on \(\X\).
We now check the validity of the integration-by-parts formula \eqref{eq:L_f}. 
Again, for every \(\psi\in \LIP_{bs}(\X)\) and \(b\in {\rm Der}^q_q(\X)\) we have that 
\[\int \psi{\sf L}_f(b)\,\d \mm=\mathcal L_f(b)(\psi)=-\int f\div(\psi b)\,\d\mm.\]
Now fix \(\bar x\in \X\) and let \((\psi_n)_n\subseteq \LIP_{bs}(\X;[0,1])\) be a sequence of \(1\)-Lipschitz functions
satisfying \(\psi_n=1\) on \(B_n(\bar x)\). In particular, \(\psi_n\to 1\) and \(\lip_a(\psi_n)\to 0\) everywhere on \(\X\).
Thus, the \(\mm\)-a.e.\ estimate \(|f \div(\psi_n b)|\leq |f|(|\psi_n\div (b)|+|b(\psi_n)|)\leq |f|(|\div(b)|+|b|)\in L^1(\mm)\)
and the dominated convergence theorem imply that
\[
\int {\sf L}_f(b)\,\d\mm=\lim_{n\to +\infty}\int \psi_n{\sf L}_f(b)\,\d\mm =\lim_{n\to +\infty}-\int f\div (\psi_n b)\,\d\mm = -\int f\div(b)\,\d\mm,
\]
proving \eqref{eq:L_f}. To verify the continuity of the map 
\({\rm Der}_q^q(\X)\ni b\mapsto {\sf L}_f(b)\in L^1(\mm)\) with
respect to the ${\rm Der}^q(\X)$ topology, by the very definition of \({\sf L}_f(b)\), the inequality
\(|{\sf L}_f(b)|\leq |b||Df|_H\) is satisfied \(\mm\)-a.e.\ in \(\X\). This shows that 
\(\|{\sf L}_f(b)\|_{L^1(\mm)}\leq \|b\|_{{\rm Der}^q(\X)}\||Df|_H\|_{L^p(\mm)}\), proving the desired continuity and 
therefore that \(f\in W^{1,p}(\X)\). At the same time, the \(\mm\)-a.e.\ inequality \(|{\sf L}_f(b)|\leq |b||Df|_H\)  proves that 
\(|Df|_W\leq |Df|_H\).
\\
{\bf Step 2.} \(W^{1,p}(\X)\subseteq B^{1,p}(\X)\) with $|D f|_B\leq |Df|_W$. 
\\
Let \(f\in W^{1,p}(\X)\). To prove that \(f\in B^{1,p}(\X)\) we will show that 
\(|Df|_W\in {\rm WUG}_B(f)\). To this aim, fix a \(q\)-test plan \(\ppi\) on \(\X\) and assume with no loss 
of generality that $\ppi$ is supported in a family of curves supported in a bounded set. 
Recall from Proposition~\ref{cor:plan_induces_derivation} that \(\ppi\) induces a derivation \(b_\ppi\in {\rm Der}^q_\infty(\X)\), which also belongs
to ${\rm Der}^q_q(\X)$ by our boundedness assumption on $\ppi$. Thus, recalling \eqref{eq:L_f} and \eqref{eq:wg}, we have
\[
\int f\,\div(b_\ppi)\,\d\mm=-\int b_\ppi(f)\,\d\mm \leq \int |b_\ppi||Df|_W\,\d\mm.
\]
Taking into account that 
\((\e_0)_\#\ppi-(\e_1)_\#\ppi=\div(b_\ppi)\)
and the inequality $|b_\sppi|\leq {\rm Bar}(\ppi)$ we get 
\[\int f(\gamma_1)-f(\gamma_0)\,\d \ppi(\gamma)=\int f\div(b_\ppi)\,\d\mm\leq \int |b_\ppi||Df|_W\,\d\mm\leq \int \int_\gamma |Df|_W\,\d s\,\ppi(\gamma).\]
Since the same holds for any restricted plan $\ppi\vert_\Gamma/\ppi(\Gamma)$ as in 
Proposition~\ref{prop:restr_x_ppi}, we obtain that the characterisation of $B^{1,p}(\X)$ given 
in Theorem~\ref{thm:equiv_BL} is satisfied with  \(G=|Df|_W\). This proves that \(f\in B^{1,p}(\X)\)
and that \(|Df|_B\leq |Df|_W\).
\\
{\bf Step 3.} \(B^{1,p}(\X)\subseteq N^{1,p}(\X)\) with $|D f|_N\leq |Df|_B$. 
\\
Let \(f\in B^{1,p}(\X)\). 
Pick Borel representatives \(\bar f\in \mathcal L^p(\mm)\) and \(\bar G\in \mathcal L^p(\mm)^+\) of \(f\) and \(|Df|_B\), respectively.
Then by Theorem \ref{thm:prop_BL}
and the definition of the family \(\Gamma^q_{\bar f, \bar G}\) in \eqref{eq:gamma_f_rho} we have that
\(\ppi(\Gamma^q_{\bar f,\bar G})=0\) holds for every \(q\)-test plan \(\ppi\) on \(\X\). 
Then, Lemma~\ref{lem:test_plan_mod_negligible} grants $\Mod_p^1(\Gamma^q_{\bar f,\bar G})=0$.
Hence, by Corollary \ref{cor:abs_cont_rep_for_BL_function}, there exists an \(\mm\)-measurable
representative \(\hat f\) of \(f\) such that for \(\Modp\)-a.e.\ curve \(\gamma\in \LIP([0,1];\X)\) it holds that 
\[
\hat f\circ \gamma\in {\rm AC}([0,1];\X)\quad \text{ and }\quad |(\hat f\circ \gamma)'_t|
\leq \bar G(\gamma_t) |\dot\gamma_t| \quad \text{ for }\mathcal L_1\text{-a.e.\ }t\in (0,1).
 \] 
Taking into account that the integrals over the curves are independent on the reparametrizations
(cf.\ Lemma \ref{lem:integral_rep_invariant}), the above implies that 
\[
|\hat f(\gamma_1)-\hat f(\gamma_0)|\leq \int_\gamma \bar G\,\d s,\quad \text{ for }\Modp\text{-a.e.\ }\gamma\in \mathscr R(\X),
\]
and consequently that \(\hat f\in \bar N^{1,p}(\X)\) with \(\bar G\in {\rm WUG}_N(\hat f)\). Therefore, we have that \(f\in N^{1,p}(\X)\) and 
\(|Df|_N\leq G=|Df|_B\) holds \(\mm\)-a.e.\ in \(\X\).
\\
{\bf Step 4.} \(N^{1,p}(\X)\subseteq H^{1,p}(\X)\) with $|D f|_H\leq |Df|_N$. 
\\ 
Let $f\in N^{1,p}(\X)$. By Theorem \ref{thm:Energy_density_EB}, there exists a sequence \((f_n)_n\subseteq \LIP_{bs}(\X)\)
such that \(f_n\to f\) and \(\lip_a(f_n)\to |Df|_N\)  in \(L^p(\mm)\), as \(n\to +\infty\), and thus
\(f\in H^{1,p}(\X)\). Moreover, \(|Df|_N\) is a \(p\)-relaxed slope of \(f\) and thus \(|Df|_H\leq |Df|_N\), which concludes the proof.
\end{proof}
\begin{remark}
{\rm 
\begin{itemize}

\item [\rm 1)] It is straightforward to check that also via relaxation of upper gradients 
one obtains the space \(H^{1,p}(\X)\), namely: a function \(f\in L^p(\mm)\) belongs to the space
\(H^{1,p}(\X)\) if and only if it admits \(({\rm ug}, p)\)-relaxed slope.  
Indeed, it follows from  Lemma \ref{lem:lip_ug_slope} that \(H^{1,p}(\X)\subseteq H^{1,p}_{\rm ug}(\X)\), where \(H^{1,p}_{\rm ug}(\X)\) is defined via relaxation of any upper gradient of \(f\), in place of \(\lip_a(f)\). By means of Proposition \ref{prop:Cap_representative}, one can prove that \(H^{1,p}_{\rm ug}(\X)\subseteq N^{1,p}(\X)\).  Then the equality \(H^{1,p}(\X)=H^{1,p}_{\rm ug}(\X)=N^{1,p}(\X)\) follows from {\bf Step 4} in the above proof.

\item [\rm 2)] Note that one can also prove directly that \(B^{1,p}(\X)=N^{1,p}(\X)\) by using the result 
about Borel representative with respect to capacity (see Proposition \ref{prop:Cap_representative}).
\end{itemize}
}
\end{remark}
\subsection{Functional properties of metric Sobolev spaces}\label{sec:Functional_properties}
In  what follows we will use the notation \(W^{1,p}(\X)\) to indicate the space of \(p\)-Sobolev functions with \(p\in [1,\infty)\),
being defined via any of the above presented equivalent approaches. 
The set of all weak \(p\)-upper gradients of a given \(p\)-Sobolev function \(f\) will be denoted by 
\({\rm WUG}(f)\), while the \textbf{minimal weak \(p\)-upper gradient} of  \(f\) will be denoted by \(|Df|\coloneqq |Df|_{N}\). Notice, indeed, that the monotonicity property of the WUG notions we introduced, together with the coincidence of the minimal ones proved in Theorem~\ref{thm:equivalence_Sobolev_spaces}, show that the notation we use in this section, both for the
space and for {\rm WUG}, are not ambiguous.

\begin{remark}
{\rm 
 Even if we do not keep track of it in the notation \(|Df|\), we point out the important fact that, in general, minimal weak \(p\)-upper gradients are 
 dependent on the exponent \(p\). A proof of this phenomenon is given in \cite{DiMa:Sp}, showing that for any given any $n\geq 1$ and \(\alpha>0\) one 
 can find a density \(w\) on \(\R^n\) so that the minimal \(p\)-weak upper gradient of any Sobolev function \(f\) on 
 \((\R^n, \sfd_{Eucl}, w\mathcal L^n)\) equals \(0\) for every \(p\leq 1+\alpha\), while for
 \(p>1+\alpha\) it equals \(\lip f\) for any Lipschitz function \(f\).
There are however instances in which \(|Df|\) is \(p\)-independent:
\begin{itemize}
    \item In PI spaces the equality \(|Df|=\lip(f)\) holds \(\mm\)-a.e.\ for every \(p\in [1,\infty)\) (see \cite{Ch:99}), thus guaranteeing the 
    independence of \(p\).
    \item The above is true in the setting of \({\sf RCD}(K,\infty)\) spaces, due to \cite{Gig:Han}.
    \item In the paper \cite{Gig:Nob}, a first-order condition on the metric-measure structure -- called \emph{bounded interpolation property} -- has been 
    shown to be sufficient for having the independence of the exponent \(p\).
\end{itemize}
} 
\end{remark}

The linearity properties follow directly from \Cref{thm:calc_rules_mwug:newt} (valid even for Dirichlet functions):
\begin{subequations}\begin{align}\label{eq:W1p_ptwse_vs1}
|D(\lambda f)|=|\lambda||Df|&\quad\text{ for every }\lambda\in\R\text{ and }f\in W^{1,p}(\X),\\
\label{eq:W1p_ptwse_vs2}
|D(f+g)|\leq|Df|+|Dg|&\quad\text{ for every }f,g\in W^{1,p}(\X).
\end{align}\end{subequations}

We recall from \Cref{rem:def_H_Sobolev} (or \Cref{prop:Cap_representative}) the conclusion 
that $W^{1,p}( \X )$ is complete. Combined with \eqref{eq:W1p_ptwse_vs1} and \eqref{eq:W1p_ptwse_vs2}, we conclude the following.
\begin{lemma}\label{lemm:Banach}
The Sobolev spaces $W^{1,p}( \X )$ are Banach spaces for $p \in [1,\infty)$.
\end{lemma}
We also obtain the following stability result.
\begin{lemma}\label{lem:basic_conseq_RS}
Let \((\X,\sfd,\mm)\) be a metric measure space and \(p\in[1,\infty)\). Let \((f_n)_n\subseteq W^{1,p}(\X)\) and 
\(f,G\in L^p(\mm)\) be such that
\(f_n\to f\) strongly in \(L^p(\mm)\) and \(|Df_n|\rightharpoonup G\) weakly in \(L^p(\mm)\). 
Then we have that \(f\in W^{1,p}(\X)\) and \(|Df|\leq G\) holds \(\mm\)-a.e.\ on \(\X\).
\end{lemma}
\begin{proof}
Given any \(n\in\N\), there exists a sequence \((f^k_n)_k\subseteq\LIP_{bs}(\X)\) such that \(f^k_n\to f_n\) and 
\(\lip_a(f^k_n)\) strongly converges 
to \(|Df_n|\) as \(k\to\infty\).
Therefore, by a diagonalisation argument we can find \((k_n)_n\subseteq\N\) such that the elements 
\(\tilde f_n\coloneqq f^{k_n}_n\) satisfy \(\tilde f_n\to f\) and
\(\lip_a(\tilde f_n)\rightharpoonup G\). It follows that \(G\in{\rm RS}(f)\), so that \(f\in W^{1,p}(\X)\) and 
\(|Df|\leq G\) holds \(\mm\)-a.e., as desired.
\end{proof}
\begin{corollary}[Lower semicontinuity properties of the Sobolev norms]\label{cor:lsc_Sob_norm}
Let \((\X,\sfd,\mm)\) be a metric measure space. Then the following properties hold:
\begin{itemize}
\item[\(\rm i)\)] Let \((f_n)_n\subseteq W^{1,1}(\X)\) and \(f\in L^1(\mm)\) be such that \(f_n\to f\) in \(L^1(\mm)\). 
Suppose there exists \(h\in L^1(\mm)\)
such that \(|Df_n|\leq h\) holds \(\mm\)-a.e.\ on \(\X\) for every \(n\in\N\). Then \(f\in W^{1,1}(\X)\) and
\[
\|f\|_{W^{1,1}(\X)}\leq\limi_{n\to\infty}\|f_n\|_{W^{1,1}(\X)}.
\]
\item[\(\rm ii)\)] Given any \(p\in(1,\infty)\), it holds that \(\|\cdot\|_{W^{1,p}(\X)}\colon L^p(\mm)\to[0,\infty]\) is lower semicontinuous,
where we adopt the convention that \(\|f\|_{W^{1,p}(\X)}\coloneqq+\infty\) for every \(f\in L^p(\mm)\setminus W^{1,p}(\X)\).
\end{itemize}
\end{corollary}
\begin{proof}
\ \\
{\bf i)} First, notice that \(\limi_n\|f_n\|_{W^{1,1}}(\X)\leq\|f\|_{L^1(\mm)}+\|h\|_{L^1(\mm)}<+\infty\). 
Take a subsequence \((n_i)_i\subseteq\N\) such that
\(\lim_i\|f_{n_i}\|_{W^{1,1}(\X)}=\limi_n\|f_n\|_{W^{1,1}(\X)}\). An application of the Dunford--Pettis theorem ensures that, 
up to a further non-relabelled subsequence,
it holds that \(|Df_{n_i}|\rightharpoonup G\) weakly for some \(G\in L^1(\mm)\). Hence, Lemma \ref{lem:basic_conseq_RS} 
implies that \(f\in W^{1,1}(\X)\) and
\(|Df|\leq G\) \(\mm\)-a.e.\ on \(\X\). Since the \(L^p\)-norm is weakly lower semicontinuous, we thus conclude that
\[
\|f\|_{W^{1,1}(\X)}\leq\|f\|_{L^1(\mm)}+\|G\|_{L^1(\mm)}\leq\lim_{i\to\infty}\|f_{n_i}\|_{W^{1,1}(\X)}=\limi_{n\to\infty}\|f_n\|_{W^{1,1}(\X)}.
\]
{\bf ii)} Given a converging sequence \(f_n\to f\) in \(L^p(\mm)\), we claim that \(\|f\|_{W^{1,p}(\X)}\leq\limi_n\|f_n\|_{W^{1,p}(\X)}\).
If \(\limi_n\|f_n\|_{W^{1,p}(\X)}=+\infty\), then there is nothing to prove. If \(\limi_n\|f_n\|_{W^{1,p}(\X)}<+\infty\), then we can argue as for i),
by just using the reflexivity of \(L^p(\mm)\) instead of the Dunford--Pettis theorem.
\end{proof}
The following calculus rules are immediate from the calculus rules for Dirichlet functions in \Cref{thm:calc_rules_mwug:newt}, 
after we are mindful of the integrability of the function.

\begin{theorem}[Calculus rules for minimal weak upper gradients]\label{thm:calc_rules_mwug}
Let \((\X,\sfd,\mm)\) be a metric measure space and \(p\in[1,\infty)\). Then the following properties are verified:
\begin{itemize}
\item[\(\rm i)\)] Let \(f\in W^{1,p}(\X)\) be given. Let \(N\subseteq\R\) be a Borel set with \(\mathcal L^1(N)=0\). Then it holds that
\[
|Df|=0\quad\mm\text{-a.e.\ on }f^{-1}(N).
\]
\item[\(\rm ii)\)] \textsc{Chain rule.} Let \(f\in W^{1,p}(\X)\) and \(\varphi\in\LIP(\R)\) with \(\varphi(0)=0\). 
Then \(\varphi\circ f\in W^{1,p}(\X)\) and
\[
|D(\varphi\circ f)|= (\lip( \varphi ) \circ f)|Df|\quad\mm\text{-a.e.\ on }\X.
\]
\item[\(\rm iii)\)] \textsc{Locality property.} Let \(f,g\in W^{1,p}(\X)\) be given. Then it holds that
\[
|Df|=|Dg|\quad\mm\text{-a.e.\ on }\{f=g\}.
\]
\item[\(\rm iv)\)] \textsc{Leibniz rule.} Let \(f,g\in W^{1,p}(\X)\cap L^\infty(\mm)\) be given. Then it holds that \(fg\in W^{1,p}(\X)\) and
\[
|D(fg)|\leq|f||Dg|+|g||Df|\quad\mm\text{-a.e.\ on }\X.
\]
\end{itemize}
\end{theorem}
\begin{remark}
\rm 
Suppose that \(f\in W^{1,p}(\X)\) and \(\varphi\in\LIP(\R)\) with \(\varphi(0)=0\). Then, as a Corollary of i) and ii) in \Cref{thm:calc_rules_mwug}, we get that
\[
|D(\varphi\circ f)|= (|\varphi'|\circ f )|Df|\quad\mm\text{-a.e.\ on }\X,
\]
where we adopt the convention that \(|\varphi'|\circ f\equiv 0\) on the set \(f^{-1}(\{t\in\R\,:\,\varphi'(t)\text{ does not exist}\})\).
\end{remark}
\begin{lemma}[Localisation of the Sobolev space]\label{lem:localisation_Sob}
Let \((\X,\sfd,\mm)\) be a metric measure space and let \(p\in[1,\infty)\). Let \((\Omega_n)_n\) be an increasing sequence of open
subsets of \(\X\) with \(\X=\bigcup_{n\in\N}\Omega_n\). Define \(\mm_n\coloneqq\mm|_{\Omega_n}\) for every \(n\in\N\). Given any
\(f\in L^p(\mm)\) and \(n\in\N\), we denote by \(f_n\in L^p(\mm_n)\) the equivalence class of the function \(f\)
up to \(\mm_n\)-a.e.\ equality. Then it holds that
\[
f\in W^{1,p}(\X,\sfd,\mm)\quad\Longleftrightarrow\quad f_n\in W^{1,p}(\X,\sfd,\mm_n)\text{ for all }n\in\N\text{ and }
\sup_{n\in\N}\|f_n\|_{W^{1,p}(\X,\sfd,\mm_n)}<+\infty.
\]
Moreover, if \(f\in W^{1,p}(\X,\sfd,\mm)\), then for every \(n\in\N\) it holds that
\begin{equation}\label{eq:localisation_Sob_cl}
|Df|=|Df_n|\quad\mm\text{-a.e.\ on }\Omega_n.
\end{equation}
\end{lemma}
\begin{proof}
Assume \(f\in W^{1,p}(\X,\sfd,\mm)\). Fix \(n\in\N\). Since \(\mm_n\leq\mm\), one can
see that \(f_n\in W^{1,p}(\X,\sfd,\mm_n)\) and \(|Df_n|\leq|Df|\) \(\mm\)-a.e.\ on \(\Omega_n\).
In particular, \(\sup_n\|f_n\|_{W^{1,p}(\X,\sfd,\mm_n)}\leq\|f\|_{W^{1,p}(\X,\sfd,\mm)}<+\infty\).

Conversely, assume that \(f_n\in W^{1,p}(\X,\sfd,\mm_n)\) for every \(n\in\N\) and \(\sup_n\|f_n\|_{W^{1,p}(\X,\sfd,\mm_n)}<+\infty\).
Define \(g\in L^p(\mm)^+\) as \(g\coloneqq|Df_n|\) \(\mm\)-a.e.\ on \(\Omega_n\) for every \(n\in\N\). The well-posedness of \(g\)
follows from a cut-off argument, by taking into account the locality properties of minimal weak upper gradients.
Now fix a \(q\)-test plan \(\ppi\) on \(\X\). Given any \(n\in\N\) and \(q\in(0,1)\cap\mathbb Q\), we consider the Borel
set \(\Gamma_{n,q}\coloneqq\big\{\gamma\in C([0,1];\X)\,:\,\gamma([0,q])\subseteq\Omega_n\big\}\) and the
\(q\)-test plan \(\ppi_{n,q}\coloneqq\ppi(\Gamma_{n,q})^{-1}\ppi|_{\Gamma_{n,q}}\) (when \(n\), \(q\) are such that
\(\ppi(\Gamma_{n,q})>0\)). For any \(k\geq n\) we have that \(\ppi_{n,q}\) is a test plan on \((\X,\sfd,\mm_k)\), thus
\[
|(f\circ\gamma)'_t|=|(f_k\circ\gamma)'_t|\leq|Df_k|(\gamma_t)|\dot\gamma_t|=g(\gamma_t)|\dot\gamma_t|\quad\text{ for }
\ppi_{n,q}\otimes\mathcal L_1\text{-a.e.\ }(\gamma,t).
\]
Since \(C([0,1];\X)=\bigcup_{n\in\N}\bigcup_{q\in(0,1)\cap\mathbb Q}\Gamma_{n,q}\), we obtain
\(|(f\circ\gamma)'_t|\leq g(\gamma_t)|\dot\gamma_t|\) for \(\ppi\otimes\mathcal L_1\)-a.e.\ \((\gamma,t)\),
whence it follows that \(f\in W^{1,p}(\X,\sfd,\mm)\) and \(|Df|\leq g\). Consequently, also \eqref{eq:localisation_Sob_cl} is proved.
\end{proof}
\subsubsection{Reflexivity and Hilbertianity of metric Sobolev spaces}
Metric Sobolev spaces of exponent \(p\in(1,\infty)\) are not necessarily reflexive Banach spaces. 
Indeed, examples of (compact) metric measure spaces \((\X,\sfd,\mm)\)
for which \(W^{1,p}(\X)\) is not reflexive are known, see \cite[Subsection 12.5]{Hei:07} or
\cite[Proposition 44]{Amb:Col:DiMa:15}. Nevertheless, many metric measure
spaces of interest are known to have reflexive Sobolev spaces. For example:
\begin{itemize}
\item \emph{\(p\)-PI spaces}, i.e.\ doubling spaces supporting a weak local \((1,p)\)-Poincar\'{e} inequality \cite{Ch:99}.
\item More generally, all those metric measure spaces for which \(({\rm spt}(\mm),\sfd)\) is \emph{metrically doubling} \cite{Amb:Col:DiMa:15}.
Ideas from \cite{Amb:Col:DiMa:15} have been later adapted to provide simpler proofs of the reflexivity of \(W^{1,p}(\X)\) on \(p\)-PI spaces,
see \cite{DCS:13} and \cite{Alv:Haj:Mal:23}.
\item Finsler \cite{Lu:Pa:20} and sub-Finsler manifolds \cite{LeDo:Lu:Pa:23} endowed with a Radon measure.
\item Metric measures spaces that can be expressed as a countable union of Lipschitz differentiability spaces, cf. \cite{Iko:Pas:Sou:22}. 
This includes e.g.\ rectifiable spaces \cite{Bate:Li:17} and $p$-PI spaces \cite{Ch:99}.
\item Metric measure spaces admitting a \emph{\(p\)-weak differentiable structure} \cite[Corollary 6.7]{EB:S:21}.
This class of spaces includes all spaces \emph{of finite Hausdorff dimension} \cite[Theorem 1.5]{EB:S:21}, 
thus in particular all those classes of spaces discussed in the first three previous bullet points. It also 
contains the spaces from the fourth bullet-point.
\item All separable reflexive Banach spaces endowed with a boundedly-finite Borel measure \cite{Sav:22} (see also \cite{Pa:Ra:23}).
\item Weighted \(p\)-Wasserstein spaces defined over a reflexive Banach space, equipped with a finite Borel measure \cite{Sod:23}.
\item In the case where \(p=2\), all \emph{infinitesimally Hilbertian} spaces. The concept of infinitesimal Hilbertianity,
introduced in \cite{Gig:15}, amounts to requiring that \(W^{1,2}(\X)\) is a Hilbert space.
\end{itemize}
The distinguished subclass of infinitesimally Hilbertian spaces includes the following spaces:
\begin{itemize}
\item Euclidean spaces equipped with a Radon measure \cite{Gig:Pas:21} (see also \cite{DiMa:Lu:Pa:20} and \cite{Lu:Pa:Ra:21}).
\item Riemannian manifolds equipped with a Radon measure \cite{Lu:Pa:20}.
\item Sub-Riemannian manifolds equipped with a Radon measure \cite{LeDo:Lu:Pa:23}.
\item Separable Hilbert spaces equipped with a boundedly-finite Borel measure \cite{Sav:22} (see also \cite{DiMar:Gig:Pas:Sou:20} and \cite{Pa:Ra:23}).
\item Weighted \(2\)-Wasserstein spaces defined over the Euclidean space, a Riemannian manifold, or a Hilbert space, 
equipped with a finite Borel measure \cite{Fo:Sa:So:23}.
\item Locally \({\sf CAT}(\kappa)\)-spaces equipped with a boundedly-finite Borel measure \cite{DiMar:Gig:Pas:Sou:20}.
\end{itemize}
Moreover, if the dual of a separable Banach space \(\mathbb B\) is uniformly convex, then \(p\)-Sobolev spaces 
defined on \(\mathbb B\) (or on the \(p\)-Wasserstein
space over \(\mathbb B\)) are uniformly convex as well \cite{Sod:23}.
\subsubsection{Separability of metric Sobolev spaces}
Reflexive metric Sobolev spaces have two important features: they are separable (Corollary \ref{cor:Sob_reflex_impl_sep}) and
boundedly-supported Lipschitz functions are \emph{strongly dense} in them (Theorem \ref{thm:Lip_strong_dense}).
\begin{lemma}\label{lem:reflex_Sob_aux}
Let \((\X,\sfd,\mm)\) be a metric measure space and \(p\in(1,\infty)\). Suppose that \(W^{1,p}(\X)\) is reflexive. Let \(f\in W^{1,p}(\X)\)
and \((f_n)_n\subseteq W^{1,p}(\X)\) be such that \(f_n\rightharpoonup f\) and \(|Df_n|\rightharpoonup|Df|\) weakly in \(L^p(\mm)\). Then
\(f_n\rightharpoonup f\) weakly in \(W^{1,p}(\X)\).
\end{lemma}
\begin{proof}
Since \(f_n\rightharpoonup f\) and \(|Df_n|\rightharpoonup|Df|\) weakly in \(L^p(\mm)\), we know from the Banach--Steinhaus theorem that
\((\|f_n\|_{W^{1,p}(\X)})_n\) is bounded. Hence, from any subsequence \((n_i)_i\) we can extract a further subsequence \((n_{i_j})_j\)
such that \(f_{n_{i_j}}\rightharpoonup g\) weakly in \(W^{1,p}(\X)\) for some function \(g\in W^{1,p}(\X)\). In particular, 
\(f_{n_{i_j}}\rightharpoonup g\)
weakly in \(L^p(\mm)\), whence it follows that \(g=f\) and thus \(f_{n_{i_j}}\rightharpoonup f\) weakly in \(W^{1,p}(\X)\). We can finally
conclude that \(f_n\rightharpoonup f\) weakly in \(W^{1,p}(\X)\), proving the statement.
\end{proof}
\begin{corollary}\label{cor:Sob_reflex_impl_sep}
Let \((\X,\sfd,\mm)\) be a metric measure space and \(p\in(1,\infty)\). Suppose that \(W^{1,p}(\X)\) is reflexive. Then \(W^{1,p}(\X)\) is separable.
\end{corollary}
\begin{proof}
Since \(L^p(\mm)\times L^p(\mm)\) is separable in the product topology, we can find a countable subset \(D\) of \(W^{1,p}(\X)\) such that
\(\{(f,|Df|)\,:\,f\in D\}\subseteq L^p(\mm)\times L^p(\mm)\) is dense in \(\{(f,|Df|)\,:\,f\in W^{1,p}(\X)\}\). This means exactly that \(D\)
is dense in energy in \(W^{1,p}(\X)\), in the sense that for every \(f\in W^{1,p}(\X)\) there exists \((f_n)_n\subseteq D\) such that \(f_n\to f\)
and \(|Df_n|\to|Df|\) strongly in \(L^p(\mm)\). Hence, Lemma \ref{lem:reflex_Sob_aux} implies that \(D\) is weakly dense in \(W^{1,p}(\X)\).
The strong density of \(D\) in \(W^{1,p}(\X)\) then follows thanks to Mazur's lemma.
\end{proof}
In many cases of interest, also the \(1\)-Sobolev space \(W^{1,1}(\X)\) is known to be separable (even if it is typically
not reflexive), for example when \((\X,\sfd,\mm)\) is a \(1\)-PI space \cite{Alv:Haj:Mal:23}.
\begin{theorem}[Strong density of \(\LIP_{bs}(\X)\) in \(W^{1,p}(\X)\)]\label{thm:Lip_strong_dense}
Let \((\X,\sfd,\mm)\) be a metric measure space and \(p\in(1,\infty)\). Suppose that \(W^{1,p}(\X)\) is reflexive. Let \(f\in W^{1,p}(\X)\) be given.
Then there exists a sequence \((f_n)_n\subseteq\LIP_{bs}(\X)\) such that
\[
\lim_n\|f_n-f\|_{W^{1,p}(\X)}=0\quad\text{ and }\quad\lim_n\|\lip_a(f_n)-|Df||\|_{L^p(\mm)}=0.
\]
\end{theorem}
\begin{proof}
Thanks to Theorem~\ref{thm:equivalence_Sobolev_spaces} and Lemma \ref{lem_distinguished}, 
we can find a sequence \((g_n)_n\subseteq\LIP_{bs}(\X)\) such that
\(g_n\to f\) and \(\lip_a(g_n)\to|Df|\) strongly in \(L^p(\mm)\).

Since \(|Dg_n|\leq\lip_a(g_n)\) for all \(n\in\N\), we deduce that \((g_n)_n\) is bounded
with respect to the \(W^{1,p}(\X)\)-norm. Being \(W^{1,p}(\X)\) reflexive by assumption, 
up to passing to a non-relabeled subsequence we have that \(g_n\rightharpoonup g\)
weakly in \(W^{1,p}(\X)\) for some \(g\in W^{1,p}(\X)\). In particular, we have that \(g_n\rightharpoonup g\) weakly in \(L^p(\mm)\), so that \(g=f\).
By Mazur's lemma, for every \(n\in\N\) we can choose \(k_n\in\N\) with \(k_n\geq n\) and \((\alpha_i^n)_{i=n}^{k_n}\subseteq[0,1]\) 
with \(\sum_{i=n}^{k_n}\alpha_i^n=1\)
such that \(f_n\coloneqq\sum_{i=n}^{k_n}\alpha_i^n g_i\to f\) strongly in \(W^{1,p}(\X)\). In particular, \(|Df_n|\to|Df|\) strongly in \(L^p(\mm)\).
Given that \(|Df_n|\leq\lip_a(f_n)\leq\sum_{i=n}^{k_n}\alpha_i^n\lip_a(g_i)\) for every \(n\in\N\) and
\[
\bigg\|\sum_{i=n}^{k_n}\alpha_i^n\lip_a(g_i)-|Df|\bigg\|_{L^p(\mm)}
\leq\sum_{i=n}^{k_n}\alpha_i^n\|\lip_a(g_i)-|Df|\|_{L^p(\mm)}\leq\sup_{i\geq n}\|\lip_a(g_i)-|Df|\|_{L^p(\mm)}\overset{n}{\to}0,
\]
we can finally conclude that \(\lip_a(f_n)\to|Df|\) strongly in \(L^p(\mm)\). The statement is achieved.
\end{proof}
We are not aware of any example of a metric measure space whose \(p\)-Sobolev space is separable but not reflexive.
Moreover, we do not know whether on every metric measure space Lipschitz functions are strongly dense in the Sobolev space.
\subsection{Bibliographical notes}
\begin{itemize}
    \item 
    The proof of the inclusions 
    \(H^{1,p}(\X)\subseteq W^{1,p}(\X)\subseteq B^{1,p}(\X)\) follows the approach from \cite{DiMaPhD:14}, where the equivalence has been proven in the case \(p>1\), by using the density in energy of Lipschitz functions in the space \(B^{1,p}(\X)\) proven in \cite{Amb:Gig:Sav:13,Amb:Gig:Sav:14}. The inclusion \(B^{1,p}(\X)\subseteq N^{1,p}(\X)\) is based on the relation modulus--plans contained in Section \ref{sec:mod_plan}, and is inspired by the approach in \cite{Amb:Mar:Sav:15}, where the equivalence between \(B^{1,p}(\X)\) and \(N^{1,p}(\X)\) is proven in the case \(p>1\). In order to cover \(p=1\), we appeal to the density result in \cite{EB:20:published}, which gives \(N^{1,p}(\X)\subseteq H^{1,p}(\X)\).
   
    \item Several notions of Sobolev space in the case \(p=1\) have been studied in \cite{APS} and the equivalence between some of them has been proven in the case of PI spaces. 
    \item In addition to the list of papers mentioned above, let us mention that reflexivity and separability properties of the Sobolev space have been investigated in \cite{Gig:18}, in relation to the same properties of the associated (co)tangent modules. We will elaborate on the latter in \cite{AILP}.
\end{itemize}
%
%
%
%
%
%

%
%
\end{document}